\documentclass[a4paper, reqno]{amsart}
\usepackage[margin=2.5cm]{geometry}
\usepackage{appendix}
\usepackage{mathtools}
\usepackage{quiver}
\usepackage{amssymb,amsmath, mathtools, mathabx}  
\usepackage{amsthm}
\usepackage{thmtools}
\usepackage[colorlinks, linkcolor=blue, citecolor=green!70!black]{hyperref}

\usepackage{cleveref}
\usepackage[mathscr]{euscript}
\usepackage{enumitem}

\DeclareMathAlphabet{\mathpzc}{OT1}{pzc}{m}{it}

\usepackage[utf8]{inputenc}
\usepackage{tikz}

\usetikzlibrary{matrix,arrows,decorations.pathmorphing,decorations.pathreplacing,decorations.markings,  calc, nfold}
\tikzset{
	dot/.style={circle,draw,fill,inner sep=1pt},
	arrow/.style={->,thick,shorten <=2pt,shorten >=2pt},
	every label/.append style = {font = \small}
}
\newtheorem{theorem}{Theorem}[section]
\newtheorem{prop}[theorem]{Proposition}
\newtheorem{cor}[theorem]{Corollary}
\newtheorem{lemma}[theorem]{Lemma}
\newtheorem{defnprop}[theorem]{Definition/Proposition}

\newtheorem{theoremabc}{Theorem}

\theoremstyle{definition}
\newtheorem{warning}[theorem]{Warning}
\newtheorem{defn}[theorem]{Definition}
\newtheorem{ex}[theorem]{Example}
\newtheorem{remark}[theorem]{Remark}
\newtheorem{notation}[theorem]{Notation}
\newtheorem{construction}[theorem]{Construction}

\usepackage[utf8]{inputenc}
\DeclareFontFamily{U}{min}{}
\DeclareFontShape{U}{min}{m}{n}{<-> udmj30}{}

\usepackage{bbm}

\newcommand\cA{\mathscr A} 
 
\newcommand\cC{\mathscr C} 
 
\newcommand\cD{\mathscr D}
\newcommand\cE{\mathscr E} 
\newcommand\cF{\mathscr F}
\newcommand\cG{\mathscr G}
\newcommand\cI{\mathscr I}

\newcommand\cL{\mathscr L}

\newcommand\cO{\mathscr O}

\newcommand\cS{\mathscr S}

\newcommand\cV{\mathscr V}

\newcommand\rB{\mathrm B}

\newcommand\bD{\mathbb D}

\newcommand\bS{\mathbb S}

\newcommand\NN{\mathbb N} \newcommand\bN\NN

\newcommand\bZ{\mathbb Z}

\newcommand\id{\mathrm{id}}

\newcommand\stab{\mathrm{Stab}}

\newcommand\mor{\mathrm{Hom}}

\newcommand\dk{\mathrm{DK}}
\newcommand\chn{\mathrm{Ch}}
\newcommand\ima{\mathrm{Im}}

\newcommand\tot{\mathrm{tot}}
\newcommand\inj{\mathrm{inj}}
\newcommand\Stn{\mathrm{Stein}}
\newcommand\stn{\mathpzc{Stein}}

\newcommand\op{\mathrm{op}}
\newcommand\cat{\mathrm{Cat}}
\newcommand\ctree{\mathpzc{Tree}}

\newcommand\pfib{\mathrm{PFib}}
\newcommand\eq{=\joinrel=}
\newcommand\ccat{\mathpzc{Cat}}
\newcommand\arr{\mathrm{Ar}}

\newcommand\sq{\mathrm{Sq}}
\newcommand\colim{\mathrm{colim}}
\newcommand\bydef{\overset{\mathrm{def}}{=}}
\newcommand\spc{\mathrm{Sp}}

\newcommand\lax{\mathrm{lax}}
\newcommand\psh{\mathrm{PSh}}
\newcommand\prc{\mathrm{PrCat}}

\newcommand\twar{\mathrm{TwAr}}
\newcommand\obj{\mathrm{Ob}}

\newcommand\act{\mathrm{act}}

\newcommand\Act{\mathrm{Act}}
\newcommand\inrt{\mathrm{int}}
\newcommand\el{\mathrm{el}}

\newcommand\seg{\mathrm{Seg}}

\newcommand\cok{\mathrm{CoKer}}
\newcommand\pt{\mathrm{pt}}

\newcommand{\xinert}[1]{\overset{#1}{\rightarrowtail}}
\newcommand{\lxinert}[1]{\overset{#1}{\leftarrowtail}}
\newcommand{\xactive}[1]{\overset{#1}{\twoheadrightarrow}}
\newcommand{\lxactive}[1]{\overset{#1}{\twoheadleftarrow}}

\title{deformation theory for $(\infty,n)$-categories}
\author{Roman Kositsyn}
\date{April 2025}

\begin{document}
	
	\begin{abstract}
		For an $(\infty,n)$-category $\cE$ we define an $(\infty,1)$ category $\twar(\cE)$ and provide an isomorphism between the stabilization of the overcategory of $\cE$ in $\cat_{(\infty,n)}$ and the $\infty$-category of spectrum-valued functors on $\twar(\cE)$. We use this to develop the deformation theory of $(\infty,n)$-categories and apply it to given an $\infty$-categorical characterization of lax-idempotent monads.
	\end{abstract}
	\maketitle
	\tableofcontents
	
	\section{Introduction}
	Assume we have a (discrete) $n$-category $C$, then we can associate to it a chain complex 
	\[C_n\xrightarrow{\partial}C_{n-1}\xrightarrow{\partial }...\xrightarrow{\partial}C_0\]
	of length $n$ such that $C_i$ is the factor of the free abelian group on the set of $i$-morphisms modulo relations
	\begin{align}\label{eq:rel_intro}
		[f*_k g] = [f] + [g]\\
		[\id_t] = 0,
	\end{align}
	where $[f]$ denotes the basis element corresponding to the morphism $f$, $*_k$ for $k<i$ denotes the operation of $k$-composition (so $f$ and $g$ in \eqref{eq:rel_intro} are assumed to be $k$-composable) and $\id_{t}$ is the identity $i$-morphism on some $j$-morphism $t$ for $j<i$ and $\partial[f]\bydef [s(f)] - [t(f)]$, where $s$ and $t$ denote the source and target of the morphism respectively. This is easily seen to be an isomorphism invariant of $C$, however possibly not a very good one: denote by $\widetilde{C}$ the groupoid obtained by inverting all $i$-morphisms in $C$ for all $i$, then $C_k = \widetilde{C}_k$ for all $k$: indeed, a general $i$-morphism in $\widetilde{C}$ is a formal composition of morphisms in $C$ and their inverses, note that by \eqref{eq:rel_intro} we have
	\[0 = [\id_{s(f)}] = [f^{-1}*_i f] = [f^{-1}] + [f]\Rightarrow [f^{-1}]=[f],\]
	which easily implies the claim.\par 
	This construction is the 1-categorical avatar of the stabilization construction from $\infty$-categories. More specifically, it is possible to show that $\stab(\seg_n)\cong \chn_n(\spc)$, where $\seg_n$ is the $\infty$-category of $n$-fold Segal spaces and $\chn_n(\spc)$ denotes the (appropriately defined) stable $\infty$-category of chain complexes in spectra of length $n$. Given an $n$-fold Segal space $\cE$ (for example, an ordinary $n$-category) we can associate to it $\Sigma^\infty_\seg\cE\in\chn_n(\spc)$. If we consider the full subcategory $\cat_n\hookrightarrow\seg_n$ on \textit{complete} Segal spaces, then its stabilization and given by $\spc$ and moreover the completion of a Segal space corresponds to the totalization of the corresponding complex. Just as in the discrete case, this construction only depends on the homotopy type of $\cE$; in fact, we can say slightly more: the following diagram
	\[\begin{tikzcd}[sep=huge]
		{\seg_n} & {\chn_n(\spc)} \\
		\cS & \spc
		\arrow["{\Sigma^\infty_\seg}", from=1-1, to=1-2]
		\arrow["{|-|}"', from=1-1, to=2-1]
		\arrow["{\tot(-)}", from=1-2, to=2-2]
		\arrow["{\Sigma^\infty}"', from=2-1, to=2-2]
	\end{tikzcd}\]
	commutes up to homotopy, where the left vertical map is the functor of geometric realization and the bottom map associates to a space $X\in\cS$ its suspension spectrum.\par 
	Consider again the case of ordinary $n$-categories, but assume now that we have a functor $F:C\rightarrow D$ to some ordinary $n$-category $D$ that we will view as fixed. In that case for any $k$-morphism $f$ in $D$ we can consider the complex $C^f_n\xrightarrow{\partial}...\xrightarrow{\partial}C^f_k$ where each terms is obtained by imposing relations \eqref{eq:rel_intro} on the free abelian group on the set of morphisms lying over $f$. Note that for every $j$-composable pair $(f,g)$ the operation of composition defines an operation
	\[m:C_i^f\times C_i^g\rightarrow C_i^{F*_j g},\]
	which endows $C^f_i$ with the structure of the \textit{local system of abelian groups} in the sense of \cite[Definition 3.5.10.]{lurie2008classification}. The $\infty$-categorical counterpart of this observation is \Cref{prop:seg} which identifies the stabilization of $\seg_{n,/\cE}$ for $\cE\in\cat_n$ with spectrum-values presheaves on $\Theta_{n,/\cE}$ that satisfy the Segal condition.\par 
	The bulk of this work is concerned with providing a simpler description of $\stab(\cat_{n,/\cE})$, before moving on to it let us first consider the results already present in the literature: the case $\dim(\cE)=1$ was considered in \cite{harpaz2018abstract} where it was shown that $\stab(\cat_{/\cE})\cong \mor_\cat(\twar(\cE),\spc)$, where $\twar$ denotes the ordinary ($\infty$-categorical) twisted arrows category. The case $\dim(\cE)=2$ was treated in \cite{nuiten2019quillen} and is somewhat more difficult to present: in that case we also have $\stab(\cat_{2,/\cE})\cong \mor_\cat(\twar_2(\cE),\spc)$, where $\twar_2(\cE)$ is an $\infty$-category that can be described informally as having objects given by 2-morphisms in $\cE$ and morphisms by geometric realizations of $\infty$-categories with objects  
		\begin{equation}\label{eq:2d_twar_intro}
		\begin{tikzcd}[sep=huge]
			x & z & w & y
			\arrow[""{name=0, anchor=center, inner sep=0}, "{a_1}"{description}, curve={height=12pt}, from=1-1, to=1-2]
			\arrow[""{name=1, anchor=center, inner sep=0}, "{a_0}"{description}, curve={height=-12pt}, from=1-1, to=1-2]
			\arrow[""{name=2, anchor=center, inner sep=0}, "f"{description}, curve={height=-30pt}, from=1-1, to=1-4]
			\arrow[""{name=3, anchor=center, inner sep=0}, "g"{description}, curve={height=30pt}, from=1-1, to=1-4]
			\arrow[""{name=4, anchor=center, inner sep=0}, "{b_1}"{description}, curve={height=12pt}, from=1-2, to=1-3]
			\arrow[""{name=5, anchor=center, inner sep=0}, "{b_0}"{description}, curve={height=-12pt}, from=1-2, to=1-3]
			\arrow[""{name=6, anchor=center, inner sep=0}, "{c_1}"{description}, curve={height=12pt}, from=1-3, to=1-4]
			\arrow[""{name=7, anchor=center, inner sep=0}, "{c_0}"{description}, curve={height=-12pt}, from=1-3, to=1-4]
			\arrow["\alpha"{description}, shorten <=3pt, shorten >=3pt, Rightarrow, from=1, to=0]
			\arrow["\epsilon"', shorten <=2pt, shorten >=2pt, Rightarrow, from=2, to=5]
			\arrow["\beta"{description}, shorten <=3pt, shorten >=3pt, Rightarrow, from=5, to=4]
			\arrow["\eta", shorten <=2pt, shorten >=2pt, Rightarrow, from=4, to=3]
			\arrow["\gamma"{description}, shorten <=3pt, shorten >=3pt, Rightarrow, from=7, to=6]
		\end{tikzcd}
		\end{equation}
	such that the source of the above morphism is $\beta$ and the target is $\epsilon*_1(\alpha*_) \beta*_\gamma)*_1 \eta$, which we will denote by $(\epsilon|\alpha,\beta|\eta)$, and with morphisms given by isomorphisms $\alpha\cong \alpha'_1 *_1\alpha'*_1\alpha'_0$ and $\gamma\cong \gamma'_1 *_1\gamma'*_1\gamma'_0$ such that the source of this morphism is given by $(\epsilon*_1(\alpha'_0*_0 b_0*_0\gamma'_0)|\alpha',\beta'|(\alpha'_1*_0 b_1*_0\gamma'_1)*_1\eta)$ and the target by $(\epsilon|\alpha,\beta|\eta)$. The main result of the present work is 
	\begin{theoremabc}\label{thm:A}
		For $\cE\in\cat_n$ there is an isomorphism
		\[\stab(\cat_{n,/\cE})\cong \mor_\cat(\twar(\cE),\spc),\]
		where $\twar(\cE)$ is the $\infty$-category described in \Cref{constr:twar_final}.
	\end{theoremabc}
	Unfortunately, the complexity of $\twar(\cE)$ grows with dimension of $\cE$, but we have given a rough description of it in \Cref{ex:lowD}.\par 
	Observe that already in dimension 2 the diagram \eqref{eq:2d_twar_intro} defining a morphism in $\twar(\cE)$ is not corepresentable by an object of $\Theta_2$. It is a more general form of pasting diagram, in order to handle those we use the formalism of Steiner complexes introduced (under a different name) in \cite{steiner2004omega}. The actual definition of Steiner complexes is rather technical and reserved for \Cref{sect:stein}, however the general idea is that (strong) Steiner complexes are complexes $C_n\xrightarrow{\partial}...\xrightarrow{\partial} C_0\xrightarrow{e}\bZ$ of free abelian groups $C_i\cong \bZ^{\oplus m_i}$ with augmentation $e$ such that $e(b_0)=1$ for every basis element $b_0$ on $C_0$ satisfying an additional condition that the basis is "strongly loop-free". If the conditions are satisfied we can associate to it an $n$-category $C^*$ such that the basis elements of $C_i$ correspond to "elementary $i$-cells" and every $k$-morphism in it can be uniquely presented as a composition of those elementary cells We use this formalism in order to ultimately define $\twar(\cE)$ in \Cref{sect:twar}, however we also obtain an important result concerning them:
 	\begin{theoremabc}\label{thm:B}
 		Denote by $Stn_n$ the category of strong Steiner complexes and by $x$ its object, then considered as an $(\infty,n)$-category it is a free category on its elementary cells.
 	\end{theoremabc}
 	A more precise version of this statement can be found at the end of \Cref{sect:stein}. Note that previously \Cref{thm:B} has been proved for discrete $n$-categories in \cite{ara2023categorical}, while an $\infty$-categorical version of it was also obtained in \cite{campion2023infty}.\par 
 	The purpose for calculating the stabilization of the overcategories for us is deformation theory. A general formalism for deformation theories has been developed in \cite[Section 12]{lurie2018spectral}, we will recall the basic definitions of it. First, given any presentable category $\cC$, a morphism $f:x\rightarrow y$ in $\cC$, an object $M\in\stab(\cC_{/y})$ and $g':\Sigma^\infty x\rightarrow M$ we can consider the pullback 
 	\[\begin{tikzcd}[sep=huge]
 		{x'} & y \\
 		x & {\Omega^\infty M}
 		\arrow[from=1-1, to=1-2]
 		\arrow[from=1-1, to=2-1]
 		\arrow["\ulcorner"{anchor=center, pos=0.125}, draw=none, from=1-1, to=2-2]
 		\arrow["0", from=1-2, to=2-2]
 		\arrow["g"', from=2-1, to=2-2]
 	\end{tikzcd},\]
 	where $g$ is obtained from $g'$ using the adjunction $\Sigma^\infty\dashV \Omega^\infty$ and the right vertical arrow is the zero section. We will call such an object $x'$ a \textit{small extension} of $x$ and we will call the objects of $\cC_{/y}$ that can be obtained by iterated small extensions from the terminal object $\id_y$ \textit{Artinian objects}. The category $\cat_n$ is well-suited for deformation theory since by the main result of \cite{harpaz2020k} every object $\cE\in\cat_n$ can be viewed as an Artinian object of $\cat_{n,/\theta_{\leq n+1}\cE}$, where $\theta_{\leq n+1}: \cat_n\rightarrow \cat_{(n+1,n)}$ is the left adjoint to the inclusion of $(n+1,n)$-categories into $(\infty,n)$-categories. Formal deformation theory then allows us to reduce the study of objects in $\cat_n$ to the study of discrete categories in $\cat_{(n+1,n)}$ and their deformation theory. More specifically, denote $L_\cE\bydef\Sigma^\infty(*)\in\stab(\cat_{n,/\cE})$, with this notation we have
 	 \begin{theoremabc}\label{thm:C}
 	 	Assume that $f:\cE\rightarrow \cD$ in $\cat_n$ is such that 
 	 	\[\cok(f_! L_\cE\rightarrow L_\cD)\cong 0\]
 	 	and $\tau_{\leq n+1}f$ induces a monomorphism
 	 	\[\tau_{\leq n+1}f^*:\mor_{\cat_{(n+1,n)}}(\tau_{\leq n+1}\cD,A)\rightarrow \mor_{\cat_{(n+1,n)}}(\tau_{\leq n+1}\cE,A)\]
 	 	for any $A\in\cat_{(n+1,n)}$, then $f^*$ is also a monomorphism and moreover we have a pullback square
 	 	\begin{equation}
 	 		\begin{tikzcd}[sep=huge]
 	 			{\mor_{\cat_n}(\cD,\cA)} & {\mor_{\cat_{(n+1,n)}}(\tau_{\leq n+1}\cD,\tau_{\leq n+1}\cA)} \\
 	 			{\mor_{\cat_n}(\cE,\cA)} & {\mor_{\cat_{(n+1,n)}}(\tau_{\leq n+1}\cE,\tau_{\leq n+1}\cA)}
 	 			\arrow["{\tau_{\leq n+1}}", from=1-1, to=1-2]
 	 			\arrow["{f^*}"', from=1-1, to=2-1]
 	 			\arrow["\ulcorner"{anchor=center, pos=0.125, rotate=45}, draw=none, from=1-1, to=2-2]
 	 			\arrow["{\tau_{\leq n+1}f^*}", from=1-2, to=2-2]
 	 			\arrow["{\tau_{\leq n+1}}"', from=2-1, to=2-2]
 	 		\end{tikzcd}
 	 	\end{equation}
 	 	for any $\cA\in\cat_n$.
 	 \end{theoremabc}
 	 Finally, in the last section we provide an application of this theorem: we denote by $\rB\Delta^\act_\lax$ the tricategory associated to the monoidal bicategory $\Delta^\act_\lax$ which is a faithful subcategory of $\cat$ containing finite ordinals $[m]$ and endpoint-preserving functors between them. There is a natural morphism $\cI:\rB \Delta^\act\hookrightarrow \rB\Delta^\act_\lax$ from the underlying bicategory of $\rB\Delta^\act_\lax$ which is a "walking comonad". The tricategory $\rB\Delta^\act_\lax$ is supposed to be the "walking lax-idempotent comonad", and we prove it in the last section using the deformation theory of \Cref{thm:C}. More specifically, we prove 
 	 \begin{theoremabc}\label{thm:D}
 	 	For $\cE\in\cat_3$ the space of morphisms $F:\rB\Delta^\act_\lax\rightarrow\cE$ is isomorphic to the subspace of $\rB\Delta^\act\rightarrow\cE$ for which the image of the 2-morphism $\delta^2_1:[1]\xactive{}[2]$ is left adjoint to $\sigma_0^1:[2]\xactive{}[1]$.
 	 \end{theoremabc}
 	 \subsection*{Notation and conventions.} In this paper we will call $(\infty,n)$-categories simply $n$-categories, whenever a certain result only holds for ordinary (discrete) $n$-categories we will specifically mention it. We will use \cite{joyal2007quasi} to identify $n$-categories with $n$-fold Segal spaces, we will denote by $\cat_n$ the category of $n$-categories and by $\seg_n$ the category of $n$-fold Segal spaces. We will use the term "space" as a synonym for $\infty$-groupoid and denote by $\cS$ the category of spaces. For $\cC$ and $\cD$ in $\cat$ we will denote $\psh_\cD(\cC)\bydef \mor_\cat(\cC^\op,\cD)$, so in particular the ordinary category of presheaves is denoted as $\psh_\cS(\cC)$. We will denote by $\Theta_n$ the disk category of \cite{joyal1997disks} and identify $\seg_n$ with a subcategory of $\psh_\cS(\Theta_n)$ on functors satisfying the Segal condition. Finally, we will make light use of basic notions and some basic results of \cite{chu2019homotopy} on algebraic patterns without further mention. 
  	\section{Stabilization of $\cat_{n,/\theta}$}\label{sect:stab}
	Our goal in this section is to calculate the stabilization of $\cat_{n,/\theta}$ for $\theta\in\Theta_n$. This is important since by \Cref{lem:globn} we can (at least in theory) express $\stab(\cat_{n,/\cE})$ for any $\cE\in\cat_n$ in terms of $\stab(\cat_{n,/\theta})$, and this property will be used numerous times in what follows to provide a more explicit definition of this stabilization.\par 
	Our argument will consist of two parts -- first we will use the Dold-Kan correspondence for $\Theta_{n,/\theta}$ explored in the appendix to identify $\stab(\cat_{n,/\theta})$ with $\stab(\seg^\inj_{n,/\theta})$ and then further employ the Verdier duality of \cite{aoki2023posets} to identify this latter category with $\psh_\spc(\Theta^{\inrt,\op}_{n,/\theta})$. 
	\begin{prop}\label{prop:seg}
		For an algebraic pattern $\cO$ the stabilization of $\seg_\cO(\cS_*)$ is isomorphic to $\seg_\cO(\spc)$.
	\end{prop}
	\begin{proof}
		Indeed, by definition we have the following pullback square
		\[\begin{tikzcd}[sep=huge]
			{\seg_\cO(\cS_*)} & {\mor_\cat(\cO,\cS_*)} \\
			{\mor_\cat(\cO^\el,\cS_*)} & {\mor_\cat(\cO^\inrt,\cS_*)}
			\arrow[from=1-1, to=2-1]
			\arrow[from=1-1, to=1-2]
			\arrow["{j^*}", from=1-2, to=2-2]
			\arrow["{i_*}"', from=2-1, to=2-2]
		\end{tikzcd},\]
		where $j:\cO^\inrt\hookrightarrow\cO$ and $i:\cO^\el\hookrightarrow\cO^\inrt$ are natural inclusions. This can equivalently be described as the pushout 
		\[\begin{tikzcd}[sep=huge]
			{\seg_\cO(\cS_*)} & {\mor_\cat(\cO,\cS_*)} \\
			{\mor_\cat(\cO^\el,\cS_*)} & {\mor_\cat(\cO^\inrt,\cS_*)}
			\arrow[from=2-1, to=1-1]
			\arrow[from=1-2, to=1-1]
			\arrow["{j_!}"', from=2-2, to=1-2]
			\arrow["{i^*}", from=2-2, to=2-1]
		\end{tikzcd}\]
		in the category $\mathrm{PrCat}$ of presentable categories and left adjoint functors. Since the stabilization of $\cV$ is given by $\spc\otimes \cV$ (since stabilization is a smashing localization, see \cite{gepner2016universality}) and $\otimes_{\mathrm{PrCat}}$ preserves colimits (since $\mathrm{PrCat}$ is closed monoidal) we get the pushout square 
		\[\begin{tikzcd}[sep=huge]
			{\mathrm{Stab}(\seg_\cO(\cS))} & {\mor_\cat(\cO,\spc)} \\
			{\mor_\cat(\cO^\el,\spc)} & {\mor_\cat(\cO^\inrt,\spc)}
			\arrow[from=2-1, to=1-1]
			\arrow[from=1-2, to=1-1]
			\arrow["{j_!}"', from=2-2, to=1-2]
			\arrow["{i^*}", from=2-2, to=2-1]
		\end{tikzcd},\]
		which by definition means that $\mathrm{Stab}(\seg_\cO(\cS))\cong \seg_\cO(\spc)$.
	\end{proof}

	\begin{construction}\label{constr:str_T}
		It follows from \Cref{prop:seg} that for $\cE\in \seg_n$ the stabilization of $\seg_{n,/\cE}$ can be identified with the category of functors $\Theta^\op_{n,/\cE}\rightarrow\spc$ satisfying the Segal condition, here we will introduce a stratification $T_k$ of that category.\par
		For $k\leq n$ denote by $T_{\leq k}(\cE)$ the full subcategory of $\seg_{\ccat_{n,/\cE}}(\spc)$ on functors $\cF$ for which $\cF(c_p\xrightarrow{f}\cE)\cong 0$ for $p>k$, this is a full subcategory and the forgetful functor from $\stab(\cat_{n,/\cE})$ to $T_{\leq k}(\cE)$ preserves limits and colimits, hence $T_{\leq k}(\cE)$ defines a stratification. 
	\end{construction}
	\begin{lemma}\label{lem:globn}
		For $\cE\in\seg_n$ there is an equivalence
		\[\stab(\seg_{n,/\cE})\cong \underset{f:\theta\rightarrow\cE}{\lim}\stab(\seg_{n,/\theta}).\]
		Assuming $\cE$ is complete, we also have an equivalence
		\[\stab(\cat_{n,/\cE})\cong \underset{f:\theta\rightarrow\cE}{\lim}\stab(\cat_{n,/\theta}).\]
	\end{lemma}
	\begin{proof}
		Denote by $\widetilde{\cE}$ the left fibration over $\Theta_n$ corresponding to $\cE$, then $\seg_{n,/\cE}$ is a subcategory of $\psh_\cS(\widetilde\cE)$ on functors satisfying the Segal condition. Note that $\widetilde{\cE}\cong \underset{f:\theta\rightarrow\cE}{\colim}\theta$, so 
		\[\psh_\cS(\widetilde{\cE})\cong \underset{f:\theta\rightarrow\cE}{\lim}\psh_\cS(\widetilde{\theta}).\]
		This equivalence preserves the Segal condition and so defines an isomorphism
		\begin{equation}\label{eq:cat_lim}
			\seg_{n,/\cE}\cong\underset{f:\theta\rightarrow\cE}{\lim}\seg_{n,/\theta}.
		\end{equation}
		Note that the limit in \eqref{eq:cat_lim} is taken over pullback morphisms $g^*:\seg_{n,/\theta}\rightarrow\seg_{n,/\theta'}$ for $g:\theta\rightarrow\theta'$, they are right adjoint functors (with left adjoint $g_!$) between presentable categories, so \eqref{eq:cat_lim} can be viewed as a colimit diagram in $\prc$. The claim now follows since stabilization is given by tensoring with $\spc$ in $\prc$ and the monoidal structure in $\prc$ is closed, so $\otimes$ preserves colimit in both variables.\par
		To prove the second claim note that, since $\cE$ is assumed to be complete, an object $F:\cD\rightarrow\cE$ belongs to $\cat_{n,/\cE}$ if and only if its fibers over identity morphisms are complete (since those are the only invertible morphisms in $\cE$). It follows that we have
		\begin{equation}\label{eq:cat_lim_comp}
			\cat_{n,/\cE}\cong\underset{f:\theta\rightarrow\cE}{\lim}\cat_{n,/\theta},
		\end{equation}
		and the claim follows by the same argument.
	\end{proof}
	\begin{prop}\label{prop:seg_dk}
		Under the equivalence of \Cref{cor:dk_theta} the objects of $\stab(\seg_{n,/\theta})$ correspond to functors $\cF:\widetilde{\Theta}_{n,/\theta}\rightarrow\spc$ satisfying the following conditions:
		\begin{enumerate}[label=\alph*]
			\item\label{it:con1} the restriction of $\cF$ to $\Theta^\inj_{n,/\theta}$ satisfies the Segal condition;
			\item\label{it:con2} denote by $\widehat{\Theta}_{n,/\theta}$ the full subcategory of $\widetilde{\Theta}_{n,/\theta}$ on morphisms $\theta_f\xactive{s}\theta_j\xrightarrow{j}\theta$ with surjective $s$ and injective $j$ such that $s^{-1}(\ima(i_e))\cong [0]$ for all except possibly one $i_e:c_k\xinert{}\theta_j$, in the latter case the sole non-trivial fiber is isomorphic to $c_p$ for some $p$ with $k+p\leq n$, then $\cF(\theta_g)\cong 0$ unless $\theta_g\in \widehat{\Theta}_{n,/\theta}$;
			\item \label{it:con3} given an object $\theta_f\in \widehat{\Theta}_{n,/\theta}$ such that $c_p\xinert{i_f}\theta_f$ is the sole elementary cell sent to a cell of lower dimension by $f$, a morphism $\theta_f\xrightarrow{h}\theta_g$ is sent to an isomorphism by $\cF$ if the restriction of $h$ along $i_f$ is an identity morphism.
		\end{enumerate} 
	\end{prop}
	\begin{proof}
		For the duration of the proof we will denote by $X$ the subcategory of $\psh_\spc(\widetilde{\Theta}_{n,/\theta})$ on $\cF$ satisfying the conditions of the proposition. Assume $\cF\in \psh_\spc(\Theta_{n,/\theta})$ satisfies the Segal condition, we need to show that the corresponding object $\dk(\cF)\in X$. We have seen in \Cref{constr:str_T} that $\stab(\cat_{n,/\theta})$ admits a stratification $T_{\leq k}(\theta)$, by \cite[Theorem A]{ayala2019stratified} we can express $\cF$ as a finite colimit of object lying in the fibers $T_k(\theta)$, since all of the conditions are stable under finite colimits we may assume that $\cF\in T_k(\theta)$, so that 
		\[\cF(\theta_f\xrightarrow{f}\theta)\cong \bigoplus_{c_k\xinert{i}\theta_f}\cF(f\circ i).\]
		Condition \ref{it:con1} is obvious, to prove \ref{it:con2} first assume $\theta_f\in \widehat{\Theta}_{n,/\theta}$ and denote by $\theta_f\xactive{s}\theta_j\xrightarrow{j}\theta$ its surjective/injective factorization (which exists by \Cref{prop:fact_inj}). If $f$ is injective then $\dk\cF(\theta_f)\cong \cF(\theta_f)$, so assume that it is not and moreover that $c_k\xinert{i_e}\theta_f$ is the sole non-trivial fiber of $s$. By definition we have
		\begin{equation}\label{eq:dk0}
			\dk\cF(\theta_f)\cong \ker(\cF(\theta_f)\rightarrow\cF(\theta_j))\cong \ker(\bigoplus_{c_k\xinert{i}\theta_f}\cF(f\circ i)\rightarrow\bigoplus_{c_k\xinert{i}\theta_f, \; i\neq i_e}\cF(f\circ i))\cong \cF(f\circ i_e).
		\end{equation}
		Finally, if $\dim(\ima(i_e))\neq k$, then the preceding calculation shows that $\dk\cF(\theta_f)\cong 0$.\par
		Now assume that $\theta_f$ does not lie in $\widehat{\Theta}_{n,/\theta}$, denote by $\theta_f\xactive{s}\theta_j\xrightarrow{j}\theta$ its active/inert factorization. By definition
		\begin{equation}\label{eq:dk_eq}
			\dk\cF(\theta_f)\cong \ker(\cF(\theta_f)\rightarrow\underset{(\theta_f\xrightarrow{e}\theta_g)\in E_n(\theta)_{\neq, \theta_f/}}{\lim}\cF(\theta_f)),
		\end{equation}
		where the limit is taken over all non-identity morphisms in $E^\lor_n(\theta)$ in the notation of \Cref{constr:dk_fact}. Denote by $\widetilde{\cF}:E_{n}(\theta)_{\theta_f/}\rightarrow\spc$ the functor sending sending $\theta_f\xrightarrow{e}\theta_g$ to $\cF(\theta_g)$ and by $i:C\hookrightarrow E_n(\theta)_{\theta_f/}$ the full subcategory on morphisms with target in $\widehat{\Theta}_{n,/\theta}$, we claim that $\widetilde{\cF}\cong i_* i^*\widetilde{\cF}$. Indeed, first observe that that $C\cong \{c_k\xinert{i_e}\theta_f,\;\dim(f\circ i)<k\}^\triangleright$, where the cone point is given by $\theta_f\xactive{s}\theta_j$ and to the element $i_e$ corresponds a surjective morphism $\theta_f\xactive{s_e}\theta_{f_e}$ sending all elementary cells in the fibers of $s$ to identities except for $i_e$. It follows that we need to prove that for any $\theta_f\xrightarrow{e}\theta_g$ 
		\begin{equation}\label{eq:dk_eq2}
			\cF(\theta_g)\cong \cF(\theta_{f_{e,0}})\times_{\cF(\theta_j)}...\times_{\cF(\theta_{j})}\cF(\theta_{e_m}),
		\end{equation}
		where $e_0,... e_m$ is an ordering of the subset of $\{c_k\xinert{i_e}\theta_f,\;\dim(f\circ i)<k\}$ containing morphisms that factor through $\theta_g$. Using the fact that $\cF\in T_k(\theta)$ \eqref{eq:dk_eq2} can be rewritten as
		\[\cF(\theta_j)\oplus \bigoplus_{0\leq t\leq m}\cF(f\circ i_{e_t})\cong (\cF(\theta_j)\oplus \cF(f\circ i_{e,0}))\times_{\cF(\theta_j)}...\times_{\cF(\theta_j)}(\cF(\theta_j)\oplus \cF(f\circ i_{e,m})),\]
		which is obvious. It now follows by transitivity of right Kan extension that both terms in the right-hand side of \eqref{eq:dk_eq} are isomorphic to $i_*i^*\widetilde{\cF}(\id_{\theta_f})$, hence $\dk\cF(\theta_f)\cong 0$. The remaining condition \ref{it:con3} now follows from \eqref{eq:dk0} since the value $\dk\cF(\theta_f)$ clearly only depends on $f\circ i_e$, hence the morphisms of the kind described in \ref{it:con3} induce isomorphisms as required.\par
		We now need to prove for any $\cF\in X$ we have $\dk'\cF\in \stab(\cat_{n,/\theta})$. Similarly to \Cref{constr:str_T} we can define a stratification $T'_{\leq k}(\theta)$ on $X$ such that $\cF\in T_{\leq k}(\theta)$ if and only if $\cF(\theta_f)\cong 0$ if $\dim(\theta_f)>k$, then as before it suffices to assume $\cF\in T'_k(\theta)$ for some $k$. By definition it follows that $\cF(\theta_j)\cong \bigoplus_{c_k\xinert{i}\theta_j}\cF(j\circ i)$ for injective $j$ and for non-injective $\theta_f\in \widehat{\Theta}_{n,/\theta}$ with a corresponding morphism $i_e:c_p\xinert{}\theta_f$ we have $\cF(\theta_f)\cong 0$ unless $k=p$. It now follows from this and \eqref{eq:dk_sum} that
		\[\dk'\cF(\theta_f)\cong \bigoplus_{c_k\xinert{i}\theta_f} \dk\cF(f\circ i),\]
		which obviously satisfies the Segal condition.
	\end{proof}
	To state the next result we will need some notation that will not be used elsewhere.
	\begin{notation}
		Given $\theta'\in\Theta_k$ and $m\geq 0$ denote by $\Sigma^m_\theta \theta'\in \Theta_{k+m}$ inductively by setting $\Sigma^0_\theta\theta'\bydef \theta'$ and $\Sigma^m_\theta\theta'$ to be the category with two objects $\{0,1\}$ such that $\mor_{\Sigma^{m}_\theta \theta'}(0,1) = \Sigma^{m-1}_\theta \theta'$. It is easy to see from this description that any morphism $f:\theta'\rightarrow \theta''$ induces a morphism $\Sigma^m_\theta f:\Sigma^m_\theta\theta'\rightarrow \Sigma^m_\theta \theta''$ making $\Sigma^m_\theta$ a functor.
	\end{notation}
	\begin{lemma}\label{lem:exp_dk}
		The category $\widehat{\Theta}_{n,/\theta}$ admits the following explicit description: it is a Cartesian fibration over $\Theta^\inj_{n,/\theta}$ whose fiber over $\theta_j\xrightarrow{j}\theta$ is the pointed category $K_j$ with objects given by $*_j$ as well as objects of the form $(c_l\xinert{i}\theta_j, k)$, where $1\leq k \leq n-l$, the morphisms in the fibers are of the following types:
		\begin{enumerate}
			\item\label{it:m1} for any $(i,k)$ a morphism $d_k:(i,k)\rightarrow (i,k+1)$ as well as a morphism $d_0:*_j\rightarrow (i,0)$ for all $i$;
			\item\label{it:m2} for any pair $(i,k)$, $(i',s)$ with $i\neq i'$ such that there is a factorization $c_l\xinert{i_0}c_r\xinert{i'}\theta_j$ of $c_l\xinert{i}\theta_j$ and $s+r = k+l$ a morphism $\gamma^{s,k}_{i_0}:(i',s)\rightarrow (i,k)$;
			\item\label{it:m3} for $(i,k)$ identify $\ima(i)$ with an object $x$ in $\theta^j_i\bydef \mor_{\theta_j}(i\circ i_-, i\circ i_+)$, where $i_\pm:c_{l-1}\xinert{} c_l$ are the inert inclusions (or with an object of $\theta_j$ if $l=0$), denote $\theta_i^{j,+}\bydef \mor_{\theta_i^j}(x,x+1)$ and $\theta_i^{j,-}\bydef \mor_{\theta_i^j}(x-1,x)$ then for any $k>1$ and any surjection $s_\pm:\theta_i^{j,\pm}\xactive{}c_{k-1}$ we have a morphism $\delta_i^{s_\pm}: *_j\rightarrow (i,k)$
		\end{enumerate}
		as well as their compositions. For any $(i,k)$ we have a relation
		\begin{equation}\label{eq:d}
			d_{k+1}\circ d_k\cong 0.
		\end{equation}
		Additionally, for $c_l\xinert{i_0}c_r$ as in the case \ref{it:m2} we have 
		\begin{equation}\label{eq:d1}
			d_{s}\circ \gamma_{i_0}^{s,k}\cong \gamma^{s,k+1}_{i_0}\circ d_k.
		\end{equation}
		For a composable pair $c_l\xinert{i_0}c_r\xinert{i'_0} c_p$ with $k+l=s+r=p+q$ we have
		\begin{equation}\label{eq:gamma}
			\gamma_{i'_0\circ i_0}^{q,k}\cong \gamma_{i'_0}^{q,s}\circ \gamma^{s,k}_{i_0}.
		\end{equation}
		Finally, for a pair $(c_l\xinert{i}\theta_j,k)$ and $(c_r\xinert{i'}\theta_j,s)$, a factorization $c_l\xinert{i_0}c_r\xinert{i'}\theta_j$ of $i$ and a surjection $s_\pm:\theta_{i'}^{j,\pm}\xactive{}c_{s-1}$ we have 
		\begin{equation}\label{eq:delta}
			\gamma^{s,k}_{i_0}\circ \delta_{i'}^{s_\pm}\cong \delta_i^{s'_\pm},
		\end{equation}
		where $s'_\pm$ is the composition
		\begin{equation}\label{eq:rel2}
			\theta_{j,\pm}^{i'}\xactive{s_0}\Sigma_\theta^{r-l} \theta^i_{j,\pm}\xactive{\Sigma_\theta^{r-l}s_\pm} c_{s+r-l-1}\cong c_{l-1},
		\end{equation}
		where the first morphism is the natural section to the inclusion $\Sigma_\theta^{r-l} \theta^i_{j,\pm}\hookrightarrow \theta^{i'}_{j,\pm}$, the second is induced by $s_\pm$ and the last equivalence follows since $s+r\geq k+l$.\par
		For any composable pair of injective morphisms $\theta_{j\circ v}\xrightarrow{v}\theta_j$ the functor $v^*:C_j\rightarrow C_{j\circ v}$ sends $*_j$ to $*_{j\circ v}$, sends $(i,k)$ to $0$ if $\ima(i)\nsubseteq \ima(v)$ and to the unique $(c_l\xinert{i'}\theta_{j\circ v}, k)$ such that $\ima(i)\subseteq \ima(v\circ i')$ otherwise.
	\end{lemma}
	\begin{proof}
		We first claim that there are no non-trivial morphisms in $M_n([0])$ between $c_l$ and $c_k$ unless $l=k$ or $l=k-1$, in which case those are unique. We first claim that there are no non-trivial morphisms $[0]\cong c_0\xrightarrow{m} c_k$ for $k>1$: indeed, by condition \ref{it:u2} we must have $m(0)=1$ and also that such morphism lies in $M_n([0])$ only if $k\leq 1$. For general $l$ we again use \ref{it:u2} and the injectivity of $m$ to conclude $m(0)=0$ and $m(1)=1$, in which case the morphism is uniquely defined by the functor of morphism categories $c_{l-1}\cong \mor_{c_l}(0,1)\rightarrow c_{k-1}$, at which point we conclude by induction. In fact, it is easy to see that such morphism is in fact given by $i_+:c_{l}\xinert{} c_{l+1}$.\par
		To relate this to our situation we first need to relate the objects of $\widehat{\Theta}_{n,/\theta}$ to the objects described in the statement of the lemma: to $*_j$ corresponds the object $\theta_j\xrightarrow{j}\theta$, to a pair $(i,k)$ corresponds an object $\theta_{f_{(i,k)}}\xactive{s_{(i,k)}}\theta_j\xrightarrow{j}\theta$, where $s_{(i,k)}$ is the surjective morphism with the sole non-trivial fiber $s_{(i,k)}^{-1}(\ima(i))\cong c_k$. Note also that any morphism $\theta_{f_0}\xrightarrow{g}\theta_{f_1}$ induces a commutative diagram
		\begin{equation}\label{eq:inj_diag}
			\begin{tikzcd}[sep=huge]
				{\theta_{f_0}} && {\theta_{f_1}} \\
				{\theta_{j_0}} && {\theta_{j_1}} \\
				& \theta
				\arrow["g", from=1-1, to=1-3]
				\arrow["{s_0}"', from=1-1, to=2-1]
				\arrow["{s_1}", from=1-3, to=2-3]
				\arrow["{g_\inj}", from=2-1, to=2-3]
				\arrow["{j_0}"', from=2-1, to=3-2]
				\arrow["{j_1}", from=2-3, to=3-2]
			\end{tikzcd}
		\end{equation}
		upon taking the surjective/injective factorization of \Cref{prop:fact_inj}, the morphisms in the fiber $K_j$ correspond to morphisms $g$ as above with $g_\inj\cong \id$.\par
		Any morphism from $(i,l)$ to $(i,k)$ in $K_j$ then must induce a morphism of fibers $c_l\rightarrow c_k$, and we have seen above that there is a unique such non-identity morphism for $l=k-1$. This also gives the relation \eqref{eq:d} since the composition of such morphism is necessarily 0 as there are no non-zero morphisms from $c_k$ to $c_{k+1}$.\par
		Assume now that we have a pair of objects $(c_l\xinert{i}\theta_j,k)$ and $(c_r\xinert{i'}\theta_j,s)$ with distinct $i$ and $i'$, denote by $i'_1:c_{r+s}\xinert{}\theta_{f_{i',s}}$ the inclusion of the cell such that $s_{(i',s)}\circ i_0\cong i'$ and similarly denote $i_1:c_{k+l}\xinert{} \theta_{f_{(i,k)}}$ the corresponding morphism for $i$. A morphism $(i',s)\xrightarrow{m} (i,k)$ is uniquely determined by the image of the cell $c_{r+s}\xinert{i'_1}\theta_{f_{i',s}}$, since $m$ must also be injective this image must necessarily contain the image of $i_1$. Assume first that $l=0$ and the inclusion $c_l\cong [0]\xinert{i_0}c_{r}$ is the inclusion of the minimal element $i_-$, denote by $\theta'\xinert{u'}\theta_{f_{(i,k)}}$ the inclusion of a category with 3 objects such that $\mor_{\theta'}(0,1)\cong c_{k+l-1}$ and $\mor_{\theta'}(1,2)\cong c_{r-1}$ such that the composition $c_{k+l}\xinert{}\theta'\xinert{u'}\theta_{f_{(i,k)}}$ is isomorphic to $i_1$ and the composition $c_r\xinert{}\theta'\xinert{u'}\theta_{f_{(i,k)}}$ to $i'$, in this case the image of $c_{r+s}$ must necessarily factor as
		\[c_{r+s}\xrightarrow{w}\theta'\xinert{u'}\theta_{f_{(i,k)}},\]
		where $w$ is a morphism in $M_{r+s}([0])$, moreover the restriction to morphism categories $w_1:c_{r+s-1}\rightarrow\mor_{\theta'}(1,2)\cong c_{r-1}$ must be the unique surjection. Since $w\in M_{r+s}([0])$ we see by \ref{it:u3} that $w_0:c_{r+s-1}\rightarrow\mor_{\theta'}(0,1)\cong c_{k+l-1}$ must be of the form $q\circ s_1$ for a surjective $s_1$ and $q\in M_{r+s-1}([0])$. Since $w$ must also be injective, $s_1$ should be an identity morphism, so $w$ is uniquely determined by $q:c_{r+s-1}\rightarrow c_{k+l-1}$ in $M_{r+s-1}([0])$. We have seen above that such a morphism exists if and only if either $r+s = k+l$ or $r+s = k+l-1$, the morphism $\gamma^{s,k}_{i_0}$ corresponds to the former case, while the latter is given by $d_{k-1}\circ \gamma^{s,k-1}_{i_0}$. Finally, if $i_0$ is (say) $c_l\xinert{i_-} c_r$, then we can consider the morphisms between $(i',s)$ and $(i,k)$ as morphisms between $(i',s-l)$ and $(i,k-l)$ over $\mor_{\theta_j}(c_{l-1}\xinert{i_-}c_l\xinert{i}\theta_j, c_{l-1}\xinert{i_+}c_l\xinert{i}\theta_j)$ and apply our previous considerations. The relations \eqref{eq:d1} and \eqref{eq:gamma} now follow by construction.\par
		For \ref{it:m3} by passing to $\mor_{\theta_j}(c_{l-1}\xinert{i_-}c_l\xinert{i}\theta_j, c_{l-1}\xinert{i_+}c_l\xinert{i}\theta_j)$ we may again assume that $l=0$, so $i$ has the form $i:[0]\xinert{}\theta_j$. The morphism $*_j\rightarrow (i,k)$ is uniquely determined by its restriction to $\theta_i^{j,\pm}$ since its restrictions to all other cells are identities, the image of one of $\theta^{j,-}_i$ and $\theta^{j,+}_i$ must contain the image of $i_1:c_{k+l}\xinert{}\theta_j$, assume it is $\theta^{j,+}_i$. In this case the restriction of the morphism to $\theta^{j,+}_i$ must factor as 
		\[\theta_i^{j,+}\xrightarrow{w'}\theta''\xinert{u''}\theta_{f_{(i,k)}},\]
		where $w'\in M_n([0])$ and $\theta''$ is the category with three objects such that $\mor_{\theta''}(0,1)\cong \theta_i^{j,+}$ and $\mor_{\theta''}(1,2)\cong c_{k-1}$, moreover the restriction to morphism categories $w_0:\theta_i^{j,+}\rightarrow\theta_i^{j,+}$ must be identity. It follows that the morphism is uniquely determined by $w_1:\theta_i^{j,+}\rightarrow c_{k-1}$, this needs to be of the form $v\circ s_+$ for some surjective $s_+$ and $v\in M_{k-1}([0])$. We have seen above that there are 2 possible morphisms in $M_{k-1}([0])$ with this target, the morphism $\delta_i^{s_+}$ corresponds to $v=\id$, morphisms corresponding to the non-identity morphism $v$ are given by $d_{k-1}\circ \delta_i^{s'_+}$ for some $s'_+:\theta_i^{j,+}\xactive{}c_{k-2}$. The relation \eqref{eq:delta} now follows by construction. This concludes the description of $K_j$, it remains to prove that $\widehat{\Theta}_{n,/\theta}$ is a Cartesian fibration over $\Theta^\inj_{n,/\theta}$. however this follows immediately from \Cref{lem:theta_pull} below.
	\end{proof}
	\begin{lemma}\label{lem:theta_pull}
		Using the notation of \Cref{lem:exp_dk}, assume we have $\theta_{j_1}\in \Theta^\inj_{n,/\theta}$, an object $\theta_f\in K_{j_1}$ and an injective morphism $\theta_{j_0}\xrightarrow{j}\theta_j$, then there exists a pullback 
		\[\begin{tikzcd}[sep=huge]
			{j^*\theta_f} & {\theta_f} \\
			{\theta_{j_0}} & {\theta_{j_1}} & \theta
			\arrow[from=1-1, to=1-2]
			\arrow[from=1-1, to=2-1]
			\arrow["\lrcorner"{anchor=center, pos=0.125}, draw=none, from=1-1, to=2-2]
			\arrow["s", two heads, from=1-2, to=2-2]
			\arrow["f", from=1-2, to=2-3]
			\arrow["j", from=2-1, to=2-2]
			\arrow["{j_1}", from=2-2, to=2-3]
		\end{tikzcd}\]
		such that $j^* \theta_f\in K_{j_0}$.
	\end{lemma}
	\begin{proof}
		Note first that the forgetful functor $\Theta_{n,/\theta}\rightarrow\Theta_n$ preserves pullbacks, so we may assume $\theta\cong [0]$. The problem is trivial for the object $*_j\in K_{j_1}$ corresponding to the identity morphism, so we may assume that $\theta_f\cong \theta_{f_{(i,k)}}$ for some $(c_l\xinert{i}\theta_{j_1}, k)$. Additionally, we may assume that $\ima(i)\subset \ima(j)$, since otherwise the pullback in question is once again an identity. We claim that we may further assume $l=0$. Indeed, assume that $l>0$ and that the image of $i$ lies in some $\theta^m_{j_0}\bydef \mor_{\theta_{j_1}}(m,m+1)$, denote $i'$ the induced morphism $c_{l-1}\xinert{}\theta^m_{j_1}$; note that in this case $\theta_{j_1}$ and $\theta_{f_{(i,k)}}$ have the same set of objects, so we also denote $\theta^m_f\bydef \mor_{\theta_{f_{(i,k)}}}(m,m+1)$ so that we have a surjective morphism $s_m:\theta^m_f\xactive{}\theta^m_{j_1}$ with a sole non-trivial fiber over the image of $i'$. Since $\ima(i)\subset \ima(j)$, there is an object $p\in \theta_{j_0}$ such that $j(p)\leq m<m+1\leq j(p+1)$. Assume we have a morphism $\theta'\xrightarrow{g}\theta_{j_0}$ such that there is an object $q\in \theta'$ such that $g(q)\leq p<p+1\leq g(q+1)$, which is necessary for factoring the morphism $j\circ g$ through $f_{(i,k)}$. Providing a morphism $h:\theta'\rightarrow \theta_{f_{(i,k)}}$ such that $s\circ h\cong j\circ g$ is then equivalent to providing a commutative diagram
		\[\begin{tikzcd}[sep=huge]
			{\theta'_q} & {\theta^m_f} \\
			{\theta^p_{j_0}} & {\theta^m_{j_1}}
			\arrow[from=1-1, to=1-2]
			\arrow["{g_q^p}"', from=1-1, to=2-1]
			\arrow["{s_m}", two heads, from=1-2, to=2-2]
			\arrow["{j_p^m}", from=2-1, to=2-2]
		\end{tikzcd},\]
		where $g_q^p:\theta'_q\bydef \mor_{\theta'}(q,q+1)\rightarrow \theta^p_{j_0}\bydef \mor_{\theta_{j_0}}(p,p+1)$ and $j_p^m:\theta^p_{j_0}\rightarrow \theta^m_{j_1}$ are the morphisms induced by $g$ and $j$ respectively on morphism categories. It thus suffices to find a pullback $j_p^{m,*} \theta_f^m$, which is the same kind of problem as the original one, but with $l$ replaced by $(l-1)$, continuing in this manner we may assume $l=0$. In this case assume that the image of $i$ is the object $x\in \theta_{j_1}$ and $y\in \theta_{j_0}$ is the unique object such that $j(y)=x$, denote $i':[0]\xinert{\{y\}}\theta_{j_0}$ the corresponding inert inclusion. Then it is easy to see that $j^*(i,k)\cong (i', k)$, concluding the proof. 
	\end{proof}
	So far we were working with the stabilization of the larger category $\seg_{n,/\theta}$ of Segal spaces over $\theta$, we will now identify the stabilization of a smaller category $\cat_{n,/\theta}$ of \textit{complete} Segal spaces. For that we will first need the lemma below, which compares the images of a walking $i$-morphism $c_i$ and a walking invertible $i$-morphism $d_i$. 
	\begin{lemma}\label{lem:theta_inv}
		Denote by $d_i$ for $0\leq i\leq n$ the $n$-category obtained from the elementary cell $c_i$ by inverting all morphisms, then we have
		\[\Sigma^\infty c_i\cong \Sigma^\infty d_i\]
		in $\stab(\seg_n)$.
	\end{lemma}
	\begin{proof}
		Using \Cref{prop:seg_dk} we see that
		\begin{equation}\label{eq:seg_chn}
			\stab(\seg_n)\cong \chn_{\leq n+1}(\spc),
		\end{equation}
		where $\chn_{\leq n+1}(\spc)$ denotes the category of chain complexes of spectra of length $(n+1)$. Using the Dold-Kan correspondence of \cite{walde2022homotopy} we also get an isomorphism
		\begin{equation}\label{eq:dk_wals}
			\chn_{\geq 0}(\spc)\cong \psh_\spc(\Delta^\op),
		\end{equation}
		where $\chn_{\geq 0}(\spc)$ denotes the category of chain complexes of spectra concentrated in non-negative degrees with $\mor_\cat(\Delta^\op,\spc)$. Using the Dold-Kan correspondence of \cite[Theorem 1.2.3.7.]{lurie2009higher} we may further identify 
		\begin{equation}\label{eq:dk_lur}
			\mor_\cat(\Delta^\op,\spc)\cong \mor_\cat(\bN, \spc).
		\end{equation}
		Under this equivalence to a functor $\cF:\Delta^\op\rightarrow \spc$ corresponds a string
		\[\underset{\Delta^\op_{\leq 0}}{\colim}j_0^*\cF\rightarrow\underset{\Delta^\op_{\leq 1}}{\colim}j_1^*\cF\rightarrow...\rightarrow\underset{\Delta^\op_{\leq k}}{\colim}j_k^*\cF\rightarrow...\]
		where $j_k:\Delta^\op_{\leq k}\hookrightarrow\Delta^\op$ denotes the inclusion of the full subcategory on $[l]$ with $l\leq k$. Restricting \eqref{eq:dk_lur} to the subcategory $\chn_{\leq n+1}(\spc)$ we obtain an isomorphism
		\begin{equation}\label{eq:dk_lur_n}
			\stab(\seg_n)\cong \mor_\cat([n], \spc).
		\end{equation}
		Finally, note that we have a commutative diagram 
		\[\begin{tikzcd}[sep=huge]
			{\psh_\cS(\Delta^\op_{\leq k})} & {\psh_\spc(\Delta^\op_{\leq k})} \\
			\cS & \spc
			\arrow["{\Sigma^\infty}"{description}, from=1-1, to=1-2]
			\arrow["\colim"{description}, from=1-1, to=2-1]
			\arrow["\colim"{description}, from=1-2, to=2-2]
			\arrow["{\Sigma^\infty}"{description}, from=2-1, to=2-2]
		\end{tikzcd}\]
		(which commutes since the corresponding diagram of right adjoints obviously does), so it follows that for $\cE\in\seg_n$ the object $\Sigma^\infty\cE\in \stab(\seg_n)$ corresponds to the string
		\[\Sigma^\infty|i_0^*\cE|\rightarrow\Sigma^\infty|i_1^*\cE|\rightarrow...\rightarrow\Sigma^\infty|\cE|,\]
		where $i_k:\Theta_k\hookrightarrow \Theta_n$ denotes the natural inclusion. The claim now follows since $|i^*_k c_n|\cong |i^*_k d_n|$ for all $k$.
	\end{proof}
	The above result essentially shows that stabilization does not make a difference between a category and a groupoid obtained from it by inverting all morphisms; this greatly limits the phenomena it can detect, but from our perspective it also simplifies the description of $\stab(\cat_{n,/\theta})$.
	\begin{prop}\label{prop:theta_comp}
		There is an isomorphism
		\[\stab(\cat_{n,/\theta})\cong \seg_\spc(\Theta^\inj_{n,/\theta}),\]
		where $\seg_\spc(\Theta^\inj_{n,/\theta})\hookrightarrow\psh_\spc(\Theta^\inj_{n,/\theta})$ denotes the subcategory of functors satisfying the Segal condition.
	\end{prop}
	\begin{proof}
		Note that an $n$-fold Segal space $\cE\xrightarrow{f} \theta$ is complete if and only if for any $j:c_l\rightarrow \theta$ the $(n-l)$-fold Segal space $f^{-1}(j)$ is complete. It follows from this and \Cref{prop:seg_dk} that $\stab(\cat_{n,/\theta})\hookrightarrow\stab(\seg_{n,/\theta})$ corresponds to $\cF:\widehat{\Theta}^{\op}_{n,/\theta}\rightarrow\spc$ such that for all $j:c_l\rightarrow\theta$ the chain complex
		\[\cF(c_l\eq c_l, n-l)\xrightarrow{d_{n-l}}...\xrightarrow{d_0}\cF(*_j)\]
		(in the notation of \Cref{lem:exp_dk}) lies in $\stab(\cat_n)\hookrightarrow\stab(\seg_n)\cong \chn_n$. The subcategory $\stab(\cat_n)\hookrightarrow \stab(\seg_n)$ corresponds to $\cF\in\seg_\spc(\Theta_n)$ for which $\cF(d_k)\cong \cF(c_0)$ for all $k$, however by \Cref{lem:theta_inv} this implies $\cF(c_k)\cong \cF(c_0)$, meaning that $\cF$ is a constant functor, so that $\stab(\cat_n)\cong \spc$. It follows from this and \ref{it:con1} that $\cF\cong i_{\inj,!}i_\inj^*\cF$, where $i_\inj:\Theta^\inj_{n,/\theta}\hookrightarrow \widehat{\Theta}_{n,/\theta}$ is the natural inclusion, and that the restriction $i_\inj^*\cF$ satisfies the Segal condition.
	\end{proof}
	\begin{lemma}\label{lem:theta_push}
		$\Theta_n$ admits pushouts of active morphisms along inert morphisms.
	\end{lemma}
	\begin{proof}
		We will prove the claim by induction on $n$ starting with the case of $\Theta_1\cong \Delta$. We claim that, given an inert morphism $[l]\xinert{i}[n]$ and an active morphism $[l]\xactive{a}[m]$, we can construct the pushout as follows:
		\[\begin{tikzcd}[sep=huge]
			{[l]} & {[n]} \\
			{[m]} & {[n+m-l]}
			\arrow["i", tail, from=1-1, to=1-2]
			\arrow["a"', two heads, from=1-1, to=2-1]
			\arrow["{a'}", two heads, from=1-2, to=2-2]
			\arrow["{i'}"', tail, from=2-1, to=2-2]
			\arrow["\urcorner"{anchor=center, pos=0.125, rotate=180}, draw=none,       from=2-2, to=1-1]
		\end{tikzcd},\]
		where $i'$ sends $[m]$ to the subinterval $[i(0), i(0)+m]$ and 
		\[
		a'(j) = 
		\begin{cases}
			j & \text{for }j\leq i(0))\\    
			i(0) + a(j - i(0)) & \text{for }i(0)< j\leq i(l)\\
			i(0) + m + (j - i(l)) & \text{for }j > i(l)
		\end{cases}.
		\]
		Indeed, given a pair of morphisms $f:[n]\rightarrow[s]$ and $g:[m]\rightarrow[s]$ which agree on the image of $[l]$ we can form 
		$h:[n + m - l]\rightarrow[s]$
		by setting $h(i'(j)) = g(j)$ and $h(a'(k)) = f(k)$, it is easy to see that this definition is consistent since $f$ and $g$ agree on the image of $i$ and that such morphism is unique since all objects of $[n + m - l]$ lie in the image of $i'$ or $a'$.\par
		Now, assume we have proved the claim for $\Theta_{N-1}$, assume we are given morphisms $\theta'\xinert{i}\theta$ and $\theta'\xactive{a}\theta''$ in $\Theta_N$. The object $\theta$ is given by the set of points $\{0,1,...,n\}$ and objects $\theta_i\in\Theta_{N-1}$ for $0\leq i\leq n-1$ such that $\mor_\theta(i,i+1)\cong\theta_i$, the inert morphism $i$ is given by an inert morphism $[m]\xinert{\overline{i}}[n]$ and inert morphisms $\theta'_s\xinert{i_s} \theta_{\overline{i}(s)}$, similarly the active morphism $a$ is given by an active morphism $[l]\xactive{\overline{a}}[m]$ together with active morphisms $\theta'_s\xactive{a_k}\theta''_k$ for $\overline{a}(s)\leq k < \overline{a}(s+1)$. We define the pushout as the object $\theta'''$ as follows: define the underlying interval to be $[n + m - l]$, define $\theta'''_j$ to be $\theta_j$ for $j < \overline{i}(0)$, for $\overline{i}(0) + \overline{a}(s)\leq j < \overline{i}(0) + \overline{a}(s + 1)$ to be given by the pushout 
		\begin{equation}\label{eq:theta_push}
			\begin{tikzcd}[sep=huge]
				{\theta'_s} & {\theta_{\overline{i}(s)}} \\
				{\theta''_{j - i(0)}} & {\theta'''_j}
				\arrow["{i_s}", tail, from=1-1, to=1-2]
				\arrow["{a_{j - i(0)}}"', two heads, from=1-1, to=2-1]
				\arrow["{a'_j}", two heads, from=1-2, to=2-2]
				\arrow["{i'_j}"', tail, from=2-1, to=2-2]
				\arrow["\ulcorner"{anchor=center, pos=0.125, rotate=180}, draw=none,       from=2-2, to=1-1]
			\end{tikzcd}
		\end{equation}   
		in $\Theta_{N-1}$, which was assumed to exist by the inductive hypothesis, and for $j\geq \overline{i}(0) + \overline{a}(l)$ by $\theta_{j + m - \overline{a}(l)}$, then $\theta'''$ admits obvious morphisms $i':\theta''\xinert{}\theta'''$ and $a':\theta\xactive{}\theta'''$ with components given by the appropriate $a'_k$ or $i'_k$ in the notation of \eqref{eq:theta_push}. To prove that this is indeed a pushout note that for $\zeta\in\Theta_N$ with the underlying interval $[s]$ a pair of morphisms $f:\theta\rightarrow\zeta$ and $g:\theta''\rightarrow\zeta$ that agree on the image of $\theta'$ are given by morphisms $\overline{f}:[n]\rightarrow[s]$ and $\overline{g}:[m]\rightarrow[s]$ together with morphisms $f^k_j:\theta_j\rightarrow \zeta_k$ for $\overline{f}(j-1)\leq k < \overline{f}(j)$ and $g^k_j:\theta''_j\rightarrow\zeta_k$ for $\overline{g}(j-1)\leq k < \overline{g}(j)$ such that the following diagrams
		\[\begin{tikzcd}[sep=huge]
			{\theta'_j} & {\theta_{\overline{i}(j)}} \\
			{\theta''_{q}} & {\zeta_k}
			\arrow["{i_j}", tail, from=1-1, to=1-2]
			\arrow["{a_{q}}"', two heads, from=1-1, to=2-1]
			\arrow["{f^k_{i(j)}}", from=1-2, to=2-2]
			\arrow["{g^k_{q}}"', from=2-1, to=2-2]
			\arrow["\ulcorner"{anchor=center, pos=0.125, rotate=180}, draw=none,       from=2-2, to=1-1]
		\end{tikzcd}\]
		commute for $\overline{a}(j)\leq q < \overline{a}(j+1)$ and $g(q)\leq k < g(q+1)$. The fact that this induces a unique morphism $h:\theta'''\rightarrow\zeta$ now follows from the definition of $\theta'''$ and the universal property of pushouts \eqref{eq:theta_push}.
	\end{proof}
	\begin{warning}
		The pushouts described in \cref{lem:theta_push} are not in general preserved by the inclusion $\Theta_n\hookrightarrow\cat_n$ for $n > 1$.
	\end{warning}
	\begin{lemma}\label{lem:act}
		For $q\leq l\leq n$ and any inert morphism $c_q\xinert{i} c_l\cong \theta_f$ over $\theta\in \Theta_n$ the induced morphism
		\[\Theta^\act_{n,/\theta,\theta_f/}\xrightarrow{p_f}\Theta^\act_{n,/\theta,\theta_{f\circ i}/}\]
		is a coCartesian fibration, moreover all the fibers $p^{-1}_f(a)$ over $a:c_q\xactive{}\theta_g$ contain initial objects and the transition functors $p_f^{-1}(a)\rightarrow p_f^{-1}(a')$ preserve those initial objects.
	\end{lemma}
	\begin{proof}
		Given an active morphism $c_l\xactive{a}\theta_g$ which restricts to $c_q\xactive{a'}\theta_{g\circ i'}$ and a further active morphism $\theta_{g\circ i'}\xactive{a_0}\theta_{h'}$ we can construct the following commutative diagram
		\[\begin{tikzcd}[sep=huge]
			{c_q} & {\theta_{g\circ i'}} & {\theta_{h'}} \\
			{c_l} & {\theta_g} & {\theta_h} \\
			&&& \theta
			\arrow["{a'}", two heads, from=1-1, to=1-2]
			\arrow["i"', tail, from=1-1, to=2-1]
			\arrow["{a_0}", two heads, from=1-2, to=1-3]
			\arrow["{i'}"', tail, from=1-2, to=2-2]
			\arrow[tail, from=1-3, to=2-3]
			\arrow["{h'}"{description}, curve={height=-12pt}, from=1-3, to=3-4]
			\arrow["a"', two heads, from=2-1, to=2-2]
			\arrow[two heads, from=2-2, to=2-3]
			\arrow["g"{description}, curve={height=12pt}, from=2-2, to=3-4]
			\arrow["\ulcorner"{anchor=center, pos=0.125, rotate=180}, draw=none, from=2-3, to=1-2]
			\arrow["h"{description}, from=2-3, to=3-4]
		\end{tikzcd},\]
		where the left square is a factorization square and the right square is a pushout square, which exists by \Cref{lem:theta_push}. This defines a functor $a_{0,!}:p_f^{-1}(a')\rightarrow p_f^{-1}(a_0\circ a')$ and a morphism $a\rightarrow a_{0,!}(a)$ and the universal property of the pushout proves that this morphism is coCartesian.\par
		Observe that $p_f^{-1}(\id_{c_q})$ admits an initial object given by $\id_{c_l}$ and also that any object $c_q\xactive{a}\theta_g$ admits a unique morphism from $\id_{c_q}$, it follows that it suffices to prove that that $a_!(\id_{c_l})$ is the initial object of $p_f^{-1}(a)$, however note that for any $c_l\xactive{a'}\theta_h$ in $p^{-1}_f (a)$ we get a unique commutative diagram 
		\[\begin{tikzcd}[sep=huge]
			{c_i} & {\theta_g} \\
			{c_l} & {a_{!}(\id_{c_l})} \\
			&& {\theta_h}
			\arrow["a", two heads, from=1-1, to=1-2]
			\arrow["i"', tail, from=1-1, to=2-1]
			\arrow[tail, from=1-2, to=2-2]
			\arrow[curve={height=-12pt}, tail, from=1-2, to=3-3]
			\arrow[two heads, from=2-1, to=2-2]
			\arrow["{a'}"{description}, curve={height=12pt}, two heads, from=2-1, to=3-3]
			\arrow["\ulcorner"{anchor=center, pos=0.125, rotate=180}, draw=none, from=2-2, to=1-1]
			\arrow["{a_0}"{description}, from=2-2, to=3-3]
		\end{tikzcd}\]
		by the universal property of the pushout.
	\end{proof}
	\begin{lemma}\label{lem:C_cont}
		Fix $c_l\xrightarrow{u}\theta$, $\theta_g\xrightarrow{g}\theta$ and $c_l\xinert{j}\theta_g$ such that $\ima(u)\xinert{}\ima(g\circ j)$, denote by $C(g,u,j)$ the category with objects given by spans $\theta_g\xrightarrow{h}\theta_f\lxinert{i_0}\theta_u\cong c_l$ such that $c_l\xinert{i_0}\theta_f$ factors through $\ima(h\circ j)\xinert{}\theta_f$ and morphisms by commutative diagrams
		\[\begin{tikzcd}[sep=huge]
			& {\theta_f} \\
			{\theta_g} && {c_l} \\
			& {\theta_{f'}}
			\arrow["s"{description}, from=1-2, to=3-2]
			\arrow["h", from=2-1, to=1-2]
			\arrow["{h'}"', from=2-1, to=3-2]
			\arrow["{i_0}"', tail, from=2-3, to=1-2]
			\arrow["{i'_0}", tail, from=2-3, to=3-2]
		\end{tikzcd},\]
		then $C(g,u,j)$ is contractible. Similarly, the category $C'(g,u,j)$ defined as above but inside the full subcategory $\Theta^\inj_{l/\theta}$ is also contractible.
	\end{lemma}
	\begin{proof}
		Assume first that $g$ and $u$ are injective, then we claim that the natural inclusion $C'(g,u,j)\hookrightarrow C(g,u,j)$ admits a left adjoint. Indeed, it is explicitly given by sending a span $(h,i_0)$ to $(s\circ h, s\circ i_0)$ in the diagram
		\[\begin{tikzcd}[sep=huge]
			& {\theta_f} \\
			{\theta_g} & {\theta_w} & {c_l} \\
			& \theta
			\arrow["s", two heads, from=1-2, to=2-2]
			\arrow["h", from=2-1, to=1-2]
			\arrow[from=2-1, to=2-2]
			\arrow["g"', from=2-1, to=3-2]
			\arrow["w"', from=2-2, to=3-2]
			\arrow["{i_0}"', tail, from=2-3, to=1-2]
			\arrow[tail, from=2-3, to=2-2]
			\arrow["u", from=2-3, to=3-2]
		\end{tikzcd},\]
		where $f\cong w\circ s$ is the surjective/injective factorization of \Cref{prop:fact_inj} (note that $s\circ i_0$ is still inert since $u$ was assumed to be injective). In particular, it follows that $C(g,u,j)$ and $C'(g,u,j)$ are homotopy equivalent, so it suffices to show contractibility of one of them.\par
		We begin with the case of $C(g,u,j)$, denote by $C^\act(g,u,j)$ the full subcategory of $C(g,u,j)$ on spans of the form $\theta_g\xactive{a}\theta_f\lxinert{i_0}c_l$, we claim that its inclusion admits a right adjoint. Indeed, by the existence of active/inert factorization we can extend any span $\theta_g\xrightarrow{h}\theta_f\lxinert{i_0}c_l$ to a commutative diagram
		\[\begin{tikzcd}[sep=huge]
			& {\theta_f} \\
			{\theta_g} && {c_l} \\
			& {\theta_{f'}}
			\arrow["i"{description}, tail, from=1-2, to=3-2]
			\arrow["a", two heads, from=2-1, to=1-2]
			\arrow["h"', from=2-1, to=3-2]
			\arrow["{j_0}"', dotted, from=2-3, to=1-2]
			\arrow["{i_0}", tail, from=2-3, to=3-2]
		\end{tikzcd},\]
		note also that the dotted line $j_0$ exists because the image of $i_0$ lies in the image of $h$ by construction, then it is easy to see that $(a,j_0)$ constitutes the required right adjoint. In particular, it follows that $C(g,u,j)$ and $C^\act(g,u,j)$ are homotopy equivalent, so it suffices to show contractibility of $C^\act(g,u,j)$.\par
		Recall from \cite[Lemma 9.14.]{chu2019homotopy} that, since $\Theta_n$ is an extendable algebraic pattern by \cite[Example 8.15]{chu2019homotopy} and \cite[Corollary 9.17.]{chu2019homotopy}, we have an equivalence
		\begin{equation}\label{eq:act_seg}
			(\Theta^\act_{l,/\theta})_{f/}\cong \underset{c_i\xinert{i}\theta_f}{\lim}(\Theta^\act_{l,/\theta})_{f\circ i/}
		\end{equation}
		for any $\theta_f\xrightarrow{f}\theta$ in $\Theta_{l,/\theta}$. Using \eqref{eq:act_seg} and the definition of $C^\act$ we see that
		\begin{equation}\label{eq:act_C}
			C(g,u,j)^\act\cong \underset{c_i\xinert{i}\theta_g}{\lim} C^\act_i,
		\end{equation}
		where $C^\act_i\cong (\Theta^\act_{l,/\theta})_{g\circ i/}$ for $i\neq j$ and $C^\act_j\cong C^\act(g\circ j,u,\id)$. Our goal is to prove that \eqref{eq:act_C} is contractible, we will in fact prove that it admits an initial object. We will do so by induction on the number of cells in $\theta_g$, starting with the minimal case of $\theta_g\cong c_l$.\par
		We will in fact prove a slightly more general claim with $\theta$ replaced by $\prod_{t\in T}\theta_t$ with $\theta_t\in \Theta_n$ and $T$ a finite set. This claim  will also be demonstrated by induction, this time on $l$, with the case $l=0$ being trivial. Assume it has been proven for $(l-1)$, denote 
		\[\mor_{\prod_t \theta_t}(u(0),u(1))\bydef \prod_{m\in M}\theta'_m\]
		for some $\theta'_m\in \Theta_{n-1}$ and $\theta'_f\bydef \mor_{\theta_f}(i_0(0), i_0(1))$. The morphism $g$ then induces a morphism 
		\begin{align*}
			g':c_{l-1}&\rightarrow\mor_{\prod_t \theta_t}(g(0), g(1))\cong \mor_{\prod_t \theta_t}(g(0), u(0))\times\mor_{\prod_t \theta_t}(u(0), u(1))\times\mor_{\prod_t \theta_t}(u(1), g(1))\rightarrow\\
			&\rightarrow\mor_{\prod_t \theta_t}(u(0), u(1))\cong \prod_m \theta'_m,
		\end{align*}
		where the first morphism is induced by the action of $g$ on morphism $(n-1)$-categories, the first isomorphism follows from the definition of $\theta_t$, the second morphism is just the projection to the middle term and the last isomorphism follows by definition. Similarly, $c_l\xinert{i_0}\theta_f$ induces $c_{l-1}\xinert{i'_0}\theta'_f$, also denote $u'$ the composition $c_{l-1}\xinert{i'_0}\theta'_f\xrightarrow{f}\prod_m \theta'_m$. This gives us an object $c_{l-1}\xactive{}\theta'_f\lxinert{i'_0}c_{l-1}$ in $C^\act(g', u',\id)$, by inductive assumption this category admits an initial object, denote it by $c_{l-1}\xactive{a'_r} \theta'_r\xinert{i'_r}c_{l-1}$ and denote $\theta'_r\xactive{a'_f}\theta'_f$ the unique morphism in $C^\act(g', u',\id)$. Finally, define $\theta_r$ to have objects $\{0,1,2,3\}$ such that $\mor_{\theta_r}(0,1)\cong \mor_{\theta_r}(2,3)\bydef c_{l-1}$ and $\mor_{\theta_r}(1,2)\bydef \theta'_r$, we define the morphism $\theta_r\xactive{a_f}\theta_f$ by sending $1<2$ to $i_0(0)\leq i_0(1)$ (and endpoints to endpoints), defining the restriction to $[0,1]$ (resp. to $[2,3]$) to be the unique active morphism $c_{l-1}\xactive \mor_{\theta_f}(0, i_0(0))$ (resp. $c_{l-1}\xactive \mor_{\theta_f}(i_0(1), a_f(3))$) and the restriction to $[1,2]$ to be given by the unique morphism $a'_f:\theta'_r\xactive{}\theta'_f$ defined above, it is clear that such a morphism is unique.\par
		Now we need to prove the claim in the general case, assume that $\theta_g$ is obtained by gluing a $k$-cell to $\theta_{g_0}$ along a boundary $s$-cell, i.e. assume we have a pushout
		\[\begin{tikzcd}[sep=huge]
			{c_s} & {c_k} \\
			{\theta_{g_0}} & {\theta_g}
			\arrow["{j_1}"', tail, from=1-1, to=1-2]
			\arrow["{j_0}", tail, from=1-1, to=2-1]
			\arrow["j'_0"',tail, from=1-2, to=2-2]
			\arrow["j'_1"',tail, from=2-1, to=2-2]
			\arrow["\ulcorner"{anchor=center, pos=0.125, rotate=180}, draw=none, from=2-2, to=1-1]
		\end{tikzcd}\]
		It follows from this and \eqref{eq:act_C} that we have a pullback diagram
		\[\begin{tikzcd}[sep=huge]
			{C(g,u,i_0)} & {(\Theta^\act_{l,/\theta})_{g\circ j'_0}} \\
			{C(g_0, u, i_0)} & {(\Theta^\act_{l,/\theta})_{g_0\circ j_0}}
			\arrow[from=1-1, to=1-2]
			\arrow[from=1-1, to=2-1]
			\arrow["\ulcorner"{anchor=center, pos=0.125}, draw=none, from=1-1, to=2-2]
			\arrow[from=1-2, to=2-2]
			\arrow[from=2-1, to=2-2]
		\end{tikzcd}.\]
		By induction, $C(g_0, u, i_0)$ admits an initial object, by \Cref{lem:act} the right vertical morphism in the diagram is a coCartesian fibration whose corresponding functor $(\Theta^\act_{l,/\theta})_{g_0\circ j_0}\rightarrow\cat$ lands in the subcategory of categories with initial objects and morphisms preserving them. It follows that its pullback $C(g,u,i_0)\rightarrow C(g_0,u,i_0)$ satisfies the same property, it is easy to see from this that it admits an initial object.
	\end{proof}
	\begin{construction}\label{constr:seg_aux}
		Denote by $C'$ the total category of the coCartesian fibration over $\Theta^\inj_{n,/\theta}$ sending $\theta_{j_0}\xrightarrow{j_0}\theta$ to $\Theta^{\inrt,\op}_{k,/\theta_{j_0}}$ and a morphism $j:\theta_{j_0}\rightarrow \theta_{j_1}$ to the functor $j_!$ sending $\theta_0\xinert{i_0}\theta_{j_0}$ to $\theta_1\xinert{i_1}\theta_{j_1}$ appearing in the factorization square 
		\[\begin{tikzcd}[sep=huge]
			{\theta_0} && {\theta_1} \\
			{\theta_{j_0}} && {\theta_{j_1}} \\
			& \theta
			\arrow["a", two heads, from=1-1, to=1-3]
			\arrow["{i_0}"', tail, from=1-1, to=2-1]
			\arrow["{i_1}", tail, from=1-3, to=2-3]
			\arrow["j", from=2-1, to=2-3]
			\arrow["{j_0}"', from=2-1, to=3-2]
			\arrow["{j_1}", from=2-3, to=3-2]
		\end{tikzcd}.\]
		Denote by $C'_\el$ the full subcategory of $C'$ on objects of the form $c_l\xinert{}\theta_j$ for some $l\leq k$.
	\end{construction}
	\begin{lemma}\label{lem:inv0}
		Denote by $S_\act$ the subcategory of $\twar(\Theta^\inj_{n,/\theta})$ on morphisms of the form 
		\[\begin{tikzcd}[sep=huge]
			{\theta_{j\circ v}} & {\theta_{j\circ v'}} \\
			{\theta_j} & {\theta_j} \\
			& \theta
			\arrow["a", two heads, from=1-1, to=1-2]
			\arrow["v"', from=1-1, to=2-1]
			\arrow["{v'}", from=1-2, to=2-2]
			\arrow["j"', from=2-1, to=3-2]
			\arrow[equal, from=2-2, to=2-1]
			\arrow["j", from=2-2, to=3-2]
		\end{tikzcd}\]
		with $a$ active, then 
		\[\twar(\Theta^\inj_{n,/\theta})[S_\act^{-1}]\cong C',\]
		where the category $C'$ is defined in \Cref{constr:seg_aux}.
	\end{lemma}
	\begin{proof}
		We can view $\twar(\Theta^\inj_{n,/\theta})$ as the total category of a coCartesian fibration over $\Theta^\inj_{n,/\theta}$ with fiber $\Theta^{\inj,\op}_{n,/\theta_j}$ over $\theta_j\xrightarrow{j}\theta$ and with pushforward functors $v_!: \Theta^{\inj,\op}_{n,/\theta_{j\circ v}}\rightarrow \Theta^{\inj,\op}_{n,/\theta_{j}}$ given by postcomposition with $v$. Note that all morphisms in $S_\act$ lie in some fiber $\Theta^{\inj,\op}_{n,/\theta_j}$, denote by $S_{j,\act}$ the intersection of $S_\act$ with the fiber over $j$, then it is immediate that the pushforward functors $v_!$ take morphisms in $S_{j\circ v,\act}$ to morphisms in $S_{j,\act}$. It follows that $\twar(\Theta^\inj_{n,/\theta})$ is isomorphic to a coCartesian fibration with fibers $\Theta^{\inj,\op}_{n,/\theta_j}[S^{-1}_{j,\act}]$, so it suffices to produce an equivalence between functors 
		\[F_0:j\mapsto \Theta^{\inj,\op}_{n,/\theta_j}[S^{-1}_{j,\act}]\]
		and 
		\[F_1:j\mapsto \Theta^{\inrt,\op}_{n,/\theta_j}.\]
		For this note that there is a functor $p_\inrt:\Theta^{\inj,\op}_{n,/\theta_j}\rightarrow \Theta^{\inrt,\op}_{n,/\theta}$ that sends $\theta_j\xrightarrow{j}\theta$ to $\theta_i\xinert{i}\theta$, where $\theta_j\xactive{a}\theta_i\xinert{i}\theta$ is the active/inert decomposition of $j$. Postcomposing with the natural morphism $\Theta^{\inj,\op}_{n,/\theta}\rightarrow \Theta^{\inj,\op}_{n,/\theta}[S_{j,\act}^{-1}]$ defines a natural transformation between $F_0$ and $F_1$, it remains to show that each of its components is an isomorphism, for that note that $p_\inrt$ is the right adjoint to the natural inclusion $\Theta^{\inrt,\op}_{n,/\theta}\xhookrightarrow{i_\inrt}\Theta^{\inj,\op}_{n,/\theta}$ and the natural transformation $i_\inrt\circ p_\inrt\rightarrow \id$ has components given by morphisms $\theta_j\xactive{a}\theta_i$ from the active/inert factorization, and so becomes an isomorphism after inverting $S_{j,\act}$.
	\end{proof}
	\begin{lemma}\label{lem:inv1}
		Denote by $S$ the subcategory of coCartesian morphisms in $C'$ and set $S_\el \bydef S\bigcap C'_\el$, then we have
		\[\psh_\spc(C'_\el[S_\el^{-1}])\cong \psh_\spc(\Theta^{\inrt,\op}_{k,/\theta})\]
	\end{lemma}
	\begin{proof}
		We will in fact prove that $\psh_\cS(C'_\el[S_\el^{-1}])\cong \psh_\cS(\Theta^{\inrt,\op}_{k,/\theta})$. We can represent a morphism in $C'_\el$ by the following commutative diagram
		\begin{equation}\label{eq:mor_diag}
			\begin{tikzcd}[sep=huge]
				{c_{l_0}} & {\ima(j\circ i_0)} & {c_{l_1}} \\
				{\theta_{j_0}} && {\theta_{j_1}} \\
				& \theta
				\arrow["a", two heads, from=1-1, to=1-2]
				\arrow["{i_0}"', tail, from=1-1, to=2-1]
				\arrow["{i'}"', tail, from=1-2, to=2-3]
				\arrow["{i''}"', tail, from=1-3, to=1-2]
				\arrow["{i_1}", tail, from=1-3, to=2-3]
				\arrow["j", from=2-1, to=2-3]
				\arrow["{j_0}"', from=2-1, to=3-2]
				\arrow["{j_1}", from=2-3, to=3-2]
			\end{tikzcd}.
		\end{equation}
		Note that such a diagram in particular induces an inclusion $\ima(j_1\circ i_1)\xinert{i_j}\ima(j_0\circ i_0)$, we define a functor $F:C'_\el\rightarrow \Theta^{\inrt,\op}_{k,/\theta}$ by sending $(i_0,j_0)$ to $\ima(j_0\circ i_0)$ and a morphism $j$ as in \eqref{eq:mor_diag} to $i_j$, we claim that this morphism establishes $\psh_\spc(\Theta^{\inrt,\op}_{k,/\theta})$ as a localization of $\psh_\spc(C'_\el)$. To prove that it suffices to show that $F^* F_!\cong \id$ as an endofunctor of $\psh_\cS(\Theta^{\inrt,\op}_{t,/\theta})$, which in turn follows if we prove that $F_!F^* h_i\rightarrow h_i$ is an isomorphism for any representable presheaf $i$. By untangling the definitions we see that $F_!F^* h_i(i')$ is the geometric realization of a category $B$ with objects given by factorizations $i\xinert{}F((i_0,j_0))\xinert{}i'$ and morphisms by morphisms $(i_0,j_0)\xrightarrow{h}(i_1,j_1)$ making the obvious diagram commute, we need to prove that $B$ is contractible. Denote by $A$ the full subcategory of $B$ on factorizations of the form $i\eq F((i,j))\xinert{}i'$, we claim that it is cofinal in $B$. For that we must prove that for any $(i_0, j_0)$ in $B$ the category $(i_0,j_0)/A$ is contractible, we first assume that $(i_0, j_0)$ is of the form $(c_l\eq c_l, c_l\xrightarrow{j_1}\theta)$ such that $\ima(i)\subseteq \ima(j_1)$. Denote by $c_k\xrightarrow{j_1}\theta$ the unique injective morphism such that $\ima(j_1) = \ima(i)$, then an object of $(\id, j)/A$ is given by a cospan $\theta_{j_1}\cong c_k\xinert{\widetilde{i}}\theta_{j'}\lxactive{a}c_l\cong \theta_{j_0}$ over $\theta$ and morphisms are induced by morphisms $\theta_{j'}\rightarrow \theta_{j''}$ making the diagram commute. In other words, it is isomorphic to the category $C'(j_0, j_1,\id)$ of \Cref{lem:C_cont}, which is contractible by the conclusion of the lemma. In the general case an object of $(j_0,i_0)/A$ is given by a diagram \eqref{eq:mor_diag}, note that it admits a natural forgetful functor $G:(j_0,i_0)/A\rightarrow (j_0\circ i_0,\id)/A$ obtained by taking the top row of \eqref{eq:mor_diag}, we also claim that $G$ admits a left adjoint. Indeed, assume we have a cospan $ c_k\xinert{\widetilde{i}}\theta_{j'}\lxactive{a}c_l$, then we can form a diagram 
		\begin{equation}\label{eq:diag_push}
			\begin{tikzcd}[sep=huge]
				{c_l} & {\theta_{j'}} & {c_k} \\
				{\theta_{j_0}} & {\theta_{j''}} & {\theta_{j''}} \\
				& \theta
				\arrow["a", two heads, from=1-1, to=1-2]
				\arrow["{i_0}"', tail, from=1-1, to=2-1]
				\arrow["{\widehat{i}}"{description}, tail, from=1-2, to=2-2]
				\arrow["{\widetilde{i}}"', tail, from=1-3, to=1-2]
				\arrow["{\widehat{i}\circ \widetilde{i}}", from=1-3, to=2-3]
				\arrow[two heads, from=2-1, to=2-2]
				\arrow["{j_0}"', from=2-1, to=3-2]
				\arrow["\lrcorner"{anchor=center, pos=0.125, rotate=180}, draw=none, from=2-2,     to=1-1]
				\arrow["{j''}"{description}, from=2-2, to=3-2]
				\arrow[equal, from=2-3, to=2-2]
				\arrow["{j''}", from=2-3, to=3-2]
			\end{tikzcd},
		\end{equation}
		where the left square is a pushout diagram, which exists by \Cref{lem:theta_push}. That this is a left adjoint follows from the universal property of the pushout, so in particular it induces a homotopy equivalence of categories, which concludes the proof of our claim since we have already shown that $(j_0\circ i_0,\id)/A$ is contractible. We now need to show that $A$ is contractible, however this follows since it contains an initial object given by $(\id, c_k\widetilde{j}\theta)$, where $\widetilde{j}$ is a morphism such that $\ima(\widetilde{i})=\ima(i)$.\par
		We have thus established $\psh_\cS(\Theta^{\inrt,\op}_{/\theta})$ as a localization of $\psh_\cS(C'_\el)$, it remains to prove that it is precisely the localization described in the proposition. Note that every object of the form $F^*\cF$ for $\cF\in \psh_\cS(\Theta^{\inrt,\op}_{/\theta})$ sends morphisms in $S_\el$ to isomorphisms (since $F$ sends them to identity), so $\psh_\cS(\Theta^{\inrt,\op}_{n,/\theta})\hookrightarrow\psh_\cS(C'_\el[S^{-1}_\el]$, to prove the converse we must show that for $\cG\in \psh_\cS(C'_\el[S_\el^{-1}])$ we have
		\begin{equation}\label{eq:iso_s}
			\cG(h_j)\cong \cG(F^*F_! h_{(i,j)})
		\end{equation}
		for any representable presheaf $h_{(i,j)}$. By construction we have
		\[F^*F_! h_{(i,j)}(i',j')\cong \mor_{\Theta^\inrt_{/\theta}}(\ima(j\circ i), \ima(j'\circ i')).\]
		Note that $\cG(c_l\xinert{i}\theta_j,j)\cong \cG(c_l\eq c_l, j\circ i)$. We claim that 
		\begin{equation}\label{eq:form_F}
			F^* F_! h_{(i,j)}\cong \underset{(\id_{c_l},j\circ i)\xrightarrow{s}(i',j')}{\colim}h_{(i',j')},
		\end{equation}
		where the colimit is taken over morphisms $s\in S_\el$. Note that this claim immediately implies \eqref{eq:iso_s} since $\cG$ sends all $s\in S_\el$ to isomorphisms so 
		\[\cG(\underset{(\id_{c_l},j\circ i)\xrightarrow{s}(i',j')}{\colim}h_{(i',j')})\cong \cG(\id_{c_l},j\circ i)\cong \cG(i,j).\]
		However note that the value of the right-hand side of \eqref{eq:form_F} on $(i',j')$ is isomorphic to $(i',j')/A$ in the notation from earlier in the proof, and we have shown that this category is either empty or contractible and the latter holds if and only if $\ima(j\circ i)\subseteq \ima(i',j')$, which means it is isomorphic to $F^*F_! h_{(i,j)}$.
	\end{proof} 
	\begin{lemma}\label{lem:kan}
		Denote $\cI:C'_\el\hookrightarrow C'$ the natural inclusion, then for $\cF\in \psh_\spc(C'_\el)$ and $\theta_{j\circ i}\xinert{i}\theta_j$ in $C'$ we have
		\[\cI_!\cF(i,j)\cong \underset{e\xinert{i_0}\theta_{j\circ i}}{\colim}\cF(i\circ i_0,j).\]
	\end{lemma}
	\begin{proof}
		By definition the value $\cI_!\cF(i,j)$ is given by the colimit of $\cF(\widetilde{i},j')$ over morphisms 
		\begin{equation}\label{eq:mor_kan}
			\begin{tikzcd}[sep=huge]
				{\theta_{j\circ i}} & {\ima(v\circ i)} & e \\
				{\theta_{j}} && {\theta_{j'}} \\
				& \theta
				\arrow[two heads, from=1-1, to=1-2]
				\arrow["i"', tail, from=1-1, to=2-1]
				\arrow[tail, from=1-2, to=2-3]
				\arrow["{i'_0}"', tail, from=1-3, to=1-2]
				\arrow["{\widetilde{i}}", tail, from=1-3, to=2-3]
				\arrow["v", from=2-1, to=2-3]
				\arrow["j"', from=2-1, to=3-2]
				\arrow["{j'}", from=2-3, to=3-2]
			\end{tikzcd}.
		\end{equation}  
		Denote by $A$ the full subcategory of $((i,j)/\cI)$ on diagrams \eqref{eq:mor_kan} in which $v\cong \id$, to prove our claim it then suffices to show that $A$ is a cofinal subcategory. Denote by $v$ the morphism represented by \eqref{eq:mor_kan}, we need to prove that $(v/A)$ is contractible. The objects in that category are given by diagrams 
		\begin{equation}\label{eq:mor_kan2}
			\begin{tikzcd}[sep=huge]
				{e_0} && {\ima(v\circ i\circ \widehat{i})} \\
				& {\theta_{j\circ i}} & {\ima(v\circ i)} & e \\
				\\
				& {\theta_{j}} && {\theta_{j'}} \\
				&& \theta
				\arrow[two heads, from=1-1, to=1-3]
				\arrow["{\widehat{i}}"', from=1-1, to=2-2]
				\arrow[tail, from=1-3, to=2-3]
				\arrow[two heads, from=2-2, to=2-3]
				\arrow["i"', tail, from=2-2, to=4-2]
				\arrow[tail, from=2-3, to=4-4]
				\arrow["{i'_1}"', tail, from=2-4, to=1-3]
				\arrow["{i'_0}", tail, from=2-4, to=2-3]
				\arrow["{\widetilde{i}}", tail, from=2-4, to=4-4]
				\arrow["v", from=4-2, to=4-4]
				\arrow["j"', from=4-2, to=5-3]
				\arrow["{j'}", from=4-4, to=5-3]
			\end{tikzcd}
		\end{equation}
		with morphisms given by morphisms $e_0\xinert{\widehat{i}'}e'_0$ making the obvious diagram commute. We need to prove that this category is contractible ,however note that for any inert morphism $e\xinert{i'_0}\ima(v\circ i)$ there is a unique inert morphism $e_0\xinert{\widehat{i}}\theta_{j\circ i}$ of minimal dimension such that $i'_0$ factors through $\ima(v\circ i\circ \widehat{i})\xinert{}\ima(v\circ i)$, and this defines an initial object of this category.
	\end{proof}
	\begin{construction}\label{constr:D}
		Denote by $M_\theta$ the category $\twar(\Theta^\inrt_{n,/\theta})$ and by $M^\el_{/\theta}$ the subcategory of $M_\theta$ on arrows between elementary objects, denote by $p_t:M\rightarrow \Theta^\inrt_{n,/\theta}$ and $p_s:M\rightarrow \Theta^{\inrt,\op}_{n/\theta}$ the projection to target and source respectively, denote $\bD^\theta\bydef p_{t,*}p_s^*:\psh_\spc(\Theta^{\inrt,\op}_{n,/\theta})\rightarrow \psh_\spc(\Theta^\inrt_{n,/\theta})$. Also denote $M'\bydef M_\theta^\op$, $p
		_t:M'\rightarrow \Theta^{\inrt,\op}_{n,/\theta}$ and $p_s':M'\rightarrow \Theta^\inrt_{n,/\theta}$ the corresponding projections and $\bD^{\theta,'}\bydef p_s^{',*}p'_{t,!}$. Finally, denote $\bD^\theta_\el$ and $\bD^{',\theta}_\el$ the restrictions to the subcategories of elementary objects.
	\end{construction}
	\begin{lemma}\label{lem:verd}
		In the notation of \Cref{constr:D} we have a commutative diagram
		\begin{equation}\label{eq:diag_d}
			\begin{tikzcd}[sep=huge]
				{\psh_\spc(\Theta^{\el,\op}_{n,/\theta})} & {\psh_\spc(\Theta^{\inrt,\op}_{n,/\theta})} \\
				{\psh_\spc(\Theta^{\el}_{n,/\theta})} & {\psh_\spc(\Theta^{\inrt}_{n,/\theta})} \\
				{\psh_\spc(\Theta^{\el,\op}_{n,/\theta})} & {\psh_\spc(\Theta^{\inrt,\op}_{n,/\theta})}
				\arrow["{\cI_!}", from=1-1, to=1-2]
				\arrow["{\bD_\el}"', from=1-1, to=2-1]
				\arrow["\bD", from=1-2, to=2-2]
				\arrow["{\cI_*}", from=2-1, to=2-2]
				\arrow["{\bD'_\el}"', from=2-1, to=3-1]
				\arrow["{\bD'}"', from=2-2, to=3-2]
				\arrow["{\cI_!}", from=3-1, to=3-2]
			\end{tikzcd},
		\end{equation}
		where $\cI$ denotes the natural inclusion, moreover the left vertical map in the diagram is an isomorphism with inverse $\bD'_\el$.
	\end{lemma}
	\begin{proof}
		By untangling the definitions we see that for $\cF:\Theta^\inrt_{n,/\theta}\rightarrow\spc$ we have
		\begin{equation}
			\bD\cF(\theta_i\xinert{i}\theta)\cong  \underset{(\theta_{i'}\xinert{}\theta_i)\in \Theta^\inrt_{n,/\theta_i}}{\lim}\cF(\theta_{i'}).
		\end{equation}
		We claim that the inclusion $\cI_{\theta_i}:\Theta^\el_{n,/\theta_i}\hookrightarrow\Theta^{\inrt}_{n,/\theta_i}$ is coinitial: indeed, for a given $\theta_{i'}\xinert{}\theta_i$ we have $\cI/\theta_{i'}\cong \Theta^\el_{n,/\theta_{i'}}$, and the geometric realization of the latter category is isomorphic to the geometric realization of $\theta_{i'}$, which is contractible. It follows that $\bD\cong \cI_* \cI^* \bD$, so we have
		\[\bD\cI_!\cong \cI_*\cI^*\bD\cI_!\cong \cI_*\bD_\el,\]
		meaning that the top square in the diagram \eqref{eq:diag_d} commutes, the commutativity of the bottom square follows by a dual argument.\par
		To prove the last claim observe that $\bD^\theta_\el$ and $\bD^\theta$ described in \Cref{constr:D} are in fact the Verdier duality functors of \cite{aoki2023posets}. Note that for $c_k\xinert{i}\theta$ we have an isomorphism
		\[|(\Theta^\el_{n/\theta})_{<i}|\cong |\Theta^\el_{/\partial c_k}|\cong S^{k-1},\]
		so this is a Verdier poset in the notation of loc. cit., it now follows from \cite[Theorem A]{aoki2023posets} that $\bD_\el$ is an isomorphism.\par
		This does not yet prove that $\bD'_\el$ is its inverse, we will prove $\bD'_\el \bD_\el \cF(c_k\xinert{}\theta)\cong \cF(c_k\xinert{}\theta)$ by induction on $k$, the case $k=0$ being trivial (since $\Theta^\inrt_{/c_0}\cong \{\id_{c_0}\}$). Assume we have proved the claim for $k<n$, denote $P\bydef \Theta^\el_{/c_n}\setminus\{\id_{c_n}\}$, since $P^\triangleright \cong \Theta^\el_{/c_n}$ is Verdier, it follows that $P$ is Gorenstein. By definition we have
		\begin{equation}\label{eq:D_cone}
			\bD_\el\cF(c_n\xinert{i}\theta)\cong \underset{(c_k\xinert{i'}c_n)\in\Theta^\el_{n,/c_n}}{\lim}\cF(c_k\xinert{i\circ i'}\theta).
		\end{equation}
		We can view \eqref{eq:D_cone} as a limit cone $\Theta^{\el,\triangleleft}_{n,/c_n}\cong P^{\triangleleft\triangleright}\rightarrow \spc$, by \cite[Theorem B]{aoki2023posets} it is also a colimit cocone, meaning that we have
		\begin{equation}\label{eq:verd1}
			\cF(c_n\xinert{i}\theta)\cong \underset{\binom{\bot}{c_k\xinert{i'}c_n}\in P^\triangleleft}{\colim}\binom{\bD_\el\cF(c_n\xinert{i}\theta)}{\cF(c_k\xinert{i\circ i'}\theta)},
		\end{equation}
		where the notation $\binom{\bot}{c_k\xinert{i'}c_n}\in P^\triangleleft$ suggests that the elements of $P^\triangleleft$ are either of the form $\bot$ or $(c_k\xinert{i'}c_n)\in P$ and in the expression $\binom{\bD\cF(c_n\xinert{i}\theta)}{\cF(c_k\xinert{i\circ i'}\theta)}$ the top element shows the value of the colimit diagram on $\bot$, while the bottom -- on elements of $P$. Using our inductive assumption we also have for $k<n$
		\begin{equation}\label{eq:verd2}
			\cF(c_k\xinert{u}\theta)\cong \bD'_\el \bD_\el\cF(c_k\xinert{u}\theta)\cong \underset{(c_q\xinert{u'}c_k)\in\Theta^{\el,\op}_{n,/c_k}}{\colim}\;\bD_\el \cF(c_q\xinert{u\circ u'}\theta).
		\end{equation}
		Denote by $D$ the coCartesian fibration over $P$ with fiber over $c_k\xinert{i}c_n$ given by $\Theta^{\el,\op}_{/c_k}$. In other words, it is a full subcategory of $\twar(\Theta^\el_{n/c_n})$ on $c_q\xinert{i_0}c_k$ with $k<n$. Denote by $p_0:D\rightarrow P^\op$ the projection to the source of the arrow. Putting everything together we get
		\begin{align*}
			\cF(c_n\xinert{i}\theta)&\cong \underset{\binom{\bot}{c_q\xinert{i''}c_k\xinert{i'}c_n}\in D^\triangleleft}{\colim}\binom{\bD_\el\cF(c_n\xinert{i}\theta)}{\bD_\el\cF(c_q\xinert{i\circ i'\circ i''}\theta)}\\
			&\cong \underset{\binom{\bot}{c_q\xinert{i_0}c_n}\in (P^\op)^\triangleleft}{\colim}\binom{\bD_\el\cF(c_n\xinert{i}\theta)}{\underset{c_q\xinert{i'}c_k\xinert{i''}c_n,\; i'\circ i''\cong i_0}{\colim}\bD_\el\cF(c_q\xinert{i\circ i_0}\theta)}\\
			&\cong \underset{\binom{\top}{c_q\xinert{i_0}c_n}\in (P^\triangleright)^\op}{\colim}\binom{\bD_\el\cF(c_n\xinert{i}\theta)}{\bD_\el\cF(c_q\xinert{i\circ i_0}\theta)}\\
			&\cong \underset{(c_q\xinert{i_0}c_n)\in\Theta^{\el,\op}_{n,/c_n}}{\colim}\bD_\el\cF(c_q\xinert{i\circ i_0}\theta)\cong \bD'_\el \bD_\el\cF(c_n\xinert{i}\theta),
		\end{align*}
		where the first isomorphism follows by substituting \eqref{eq:verd2} into \eqref{eq:verd1}, the second by computing the colimit over $D^\triangleright$ as a composition of left Kan extension along $p_0$ and colimit over $(P^\op)^\triangleleft$, the third since the colimit over $\{c_q\xinert{i'}c_k\xinert{i''}c_n,\; i'\circ i''\cong i_0\}$ is a cotensor with a contractible category and $(P^\op)^\triangleright\cong (P^\triangleleft)^\op$, the fourth by using $(P^\triangleleft)^\op\cong \Theta^{\el,\op}_{/c_n}$ and the last isomorphism follows by definition.
	\end{proof}
	\begin{prop}\label{prop:seg_inj}
		There is an equivalence
		\[\seg_\spc(\Theta^\inj_{n,/\theta})\cong\psh_\spc(\Theta^{\inrt,\op}_{n,/\theta}).\]
	\end{prop}
	\begin{proof}
		We will first define a functor
		\[F:\seg_\spc(\Theta^\inj_{n,/\theta})\rightarrow \psh_\spc(\Theta^\inrt_{n,/\theta}).\]
		We denote by $p:D\rightarrow \twar(\Theta^\inj_{n,/\theta})$ the Cartesian fibration whose fiber over $\theta_{j\circ u}\xrightarrow{u} \theta_j$ is $\Theta^{\inrt,\op}_{n,/\theta_{j\circ u}}$ and such that the Cartesian morphism over the morphism 
		\[\begin{tikzcd}[sep=huge]
			& {\theta_{j'\circ u'\circ i'}} \\
			{\theta_{j\circ u}} & {\theta_{j'\circ u'}} \\
			{\theta_j} & {\theta_{j'}} \\
			& \theta
			\arrow["{i'}", tail, from=1-2, to=2-2]
			\arrow["u"', from=2-1, to=3-1]
			\arrow["{v'}"', from=2-2, to=2-1]
			\arrow["{u'}"', from=2-2, to=3-2]
			\arrow["v", from=3-1, to=3-2]
			\arrow["j"', from=3-1, to=4-2]
			\arrow["{j'}", from=3-2, to=4-2]
		\end{tikzcd}\]
		with target $(\theta_{j\circ u'\circ i'}\xinert{i'}\theta_{j'\circ u'})$ is given by 
		\[
		\begin{tikzcd}[sep=huge]
			{\ima(v'\circ i')} & {\theta_{j'\circ u'\circ i'}} \\
			{\theta_{j\circ u}} & {\theta_{j'\circ u'}} \\
			{\theta_j} & {\theta_{j'}} \\
			& \theta
			\arrow["{i''}", tail, from=1-1, to=2-1]
			\arrow["a"', two heads, from=1-2, to=1-1]
			\arrow["{i'}", tail, from=1-2, to=2-2]
			\arrow["u"', from=2-1, to=3-1]
			\arrow["{v'}"', from=2-2, to=2-1]
			\arrow["{u'}"', from=2-2, to=3-2]
			\arrow["v", from=3-1, to=3-2]
			\arrow["j"', from=3-1, to=4-2]
			\arrow["{j'}", from=3-2, to=4-2]
		\end{tikzcd}.
		\]
		In particular, a general morphism in $D$ is given by a diagram 
		\begin{equation}\label{eq:mor_d}
			\begin{tikzcd}[sep=huge]
				{\theta_{j\circ u\circ i}} & {\ima(v'\circ i')} & {\theta_{j'\circ u'\circ i'}} \\
				& {\theta_{j\circ u}} & {\theta_{j'\circ u'}} \\
				& {\theta_j} & {\theta_{j'}} \\
				&& \theta
				\arrow["{\widetilde{i}}", from=1-1, to=1-2]
				\arrow["i"', tail, from=1-1, to=2-2]
				\arrow["{i''}", tail, from=1-2, to=2-2]
				\arrow["a"', two heads, from=1-3, to=1-2]
				\arrow["{i'}", tail, from=1-3, to=2-3]
				\arrow["u"', from=2-2, to=3-2]
				\arrow["{v'}"', from=2-3, to=2-2]
				\arrow["{u'}"', from=2-3, to=3-3]
				\arrow["v", from=3-2, to=3-3]
				\arrow["j"', from=3-2, to=4-3]
				\arrow["{j'}", from=3-3, to=4-3]
			\end{tikzcd}
		\end{equation}
		and it has source $\theta_{j\circ u\circ i}$ and target $\theta_{j'\circ u'\circ i'}$. We define $q:D\rightarrow \Theta^\inj_{n,/\theta}$ by sending $\theta_{j\circ u\circ i}\xinert{i}\theta_{j\circ u}$ to $\ima(u\circ i)\xinert{}\theta_j\xrightarrow{j}\theta$ and a morphism \eqref{eq:mor_d} to the composition
		\[\ima(u\circ i)\xinert{\widetilde{i}}\ima(u\circ i'')\cong \ima(u\circ i''\circ a)\cong \ima(u\circ v'\circ i')\xactive{a_v}\ima(v\circ u\circ v'\circ i')\cong \ima(u'\circ i'),\]
		where the first isomorphism follows since $a$ is active and the rest by commutativity of \eqref{eq:mor_d} and the morphism $a_v$ is induced by $v$. We first define $\cF'\bydef p_!q^*\cF$. Explicitly, we have
		\begin{equation}\label{eq:val_prel}
			\cF'(\theta_{j\circ u}\xrightarrow{u}\theta_j)\cong \underset{(\theta_{j\circ u\circ i}\xinert{i}\theta_{j\circ u})\in \Theta^\inrt_{n,/\theta_{j\circ u}}}{\colim}\cF(\ima(u\circ i)\xrightarrow{j|_{\ima(u\circ i)}}\theta).
		\end{equation}
		Denote by $\cF^j$ the restriction of $\cF$ to $\Theta^\inrt_{n,/\theta_{j}}$, then it is easy to see from \eqref{eq:val_prel} and definition of $\bD'$ that 
		\begin{equation}\label{eq:val_d}
			\cF'(\theta_{j\circ u}\xrightarrow{u}\theta_j)\cong \bD'\cF^j(\ima(u)\xinert{}\theta_j).
		\end{equation}
		In particular, \eqref{eq:val_d} implies that
		\[\cF'(\theta_{j\circ u}\xrightarrow{u}\theta_j)\cong \cF'(\ima(u)\xinert{}\theta_j),\]
		meaning that $\cF'$ factors through $\twar(\Theta^\inj_{n,/\theta})[S_\act^{-1}]\cong C'$, where the isomorphism follows from \Cref{lem:inv0}.\par 
		Note that since $\cF\in \seg_\spc(\Theta^\inj_{n,/\theta})$, $\cF^j$ is the right Kan extension of its restriction to $\Theta^\el_{n,/\theta_j}$, so it follows from the first claim of \Cref{lem:verd} and \eqref{eq:val_d} that $\cF'$ viewed as an object of $\psh_\spc(\Theta^{\inrt,\op}_{n,/\theta_j})$ is the left Kan extension of its restriction to $\Theta^{\el,\op}_{n,/\theta_j}$. In other words, we have
		\begin{equation}\label{eq:left_kan}
			\cF'(\ima(u)\xinert{i_0}\theta_j)\cong\underset{e\xinert{i}\ima(u)}{\colim}\;\cF'(e\xinert{i_0\circ i}\theta_j).
		\end{equation}
		It follows from \eqref{eq:left_kan} and \Cref{lem:kan} that $\cF'$ is the left Kan extension of its restriction to $C'_\el$, so it can be identified with an object $\cF''\in \psh_\spc(C'_\el)$. Finally, it follows once again from \eqref{eq:left_kan} that $\cF''$ inverts $S_\el$ (in the notation of \Cref{lem:inv1}), hence factors through $\psh_\spc(C'_\el[S_\el^{-1}])\cong \psh_\spc(\Theta^{\inrt,\op}_{n,/\theta})$, where the isomorphism follows from \Cref{lem:inv1}; we define $F(\cF)$ to be the image of $\cF'$ in $\psh_\spc(\Theta^{\inrt,\op}_{n,/\theta})$.\par
		Denote by $p':C'\rightarrow \Theta^\inj_{n,/\theta}$ the natural projection and by $q':C'\rightarrow \Theta^{\inrt,\op}_{n,/\theta}$ the morphism that sends $\theta_{j\circ i}\xinert{i}\theta_j\xrightarrow{j}\theta$ to $\ima(j\circ i)$ and a morphism
		\begin{equation}\label{eq:mor_c'}
			\begin{tikzcd}[sep=huge]
				{\theta_{j\circ i}} & {\ima(v\circ i)} & {\theta_{j'\circ i'}} \\
				{\theta_j} && {\theta_{j'}} \\
				& \theta
				\arrow["a", two heads, from=1-1, to=1-2]
				\arrow["i"', from=1-1, to=2-1]
				\arrow["{i_0}"', tail, from=1-2, to=2-3]
				\arrow["{\widetilde{i}}"', from=1-3, to=1-2]
				\arrow["{i'}", from=1-3, to=2-3]
				\arrow["v"', from=2-1, to=2-3]
				\arrow["j"', from=2-1, to=3-2]
				\arrow["{j'}", from=2-3, to=3-2]
			\end{tikzcd}
		\end{equation}
		to the composition
		\[\ima(j\circ i)\cong \ima(j'\circ v\circ i)\cong \ima(j'\circ i_0\circ a)\xleftarrow{\widetilde{i}}\ima(j'\circ i').\]
		We define 
		\[G:\psh_\spc(\Theta^{\inrt,\op}_{n,/\theta})\rightarrow \seg_\spc(\Theta^\inj_{n,/\theta})\]
		by $G\bydef p'_*q^{',*}$. By definition we have
		\begin{equation}\label{eq:G}
			G\cG(\theta_j\xrightarrow{j}\theta)\cong \underset{\theta_{j\circ i}\xinert{i}\theta_j}{\lim}\cG(\ima(j\circ i))
		\end{equation}
		for $\cG\in\psh_\spc(\Theta^{\inrt,\op}_{n,/\theta})$. Denote by $\cG^j:\Theta^{\inrt,\op}_{n,/\theta_j}\rightarrow\spc$ the functor sending $\theta_{j\circ i}\xinert{i}\theta_j$ to $\cG(\ima(j\circ i))$, then it follows from \eqref{eq:G} that 
		\begin{equation}\label{eq:G_d}
			G\cG(\theta_j)\cong \bD\cG^j(\theta_j\eq \theta_j).
		\end{equation}
		It follows from the proof of \Cref{lem:verd} that $\bD$ factors through $\psh_\spc(\Theta^\el_{n,/\theta_j})\xrightarrow{\cI_*}\psh_\spc(\Theta^\inrt_{n,/\theta_j})$, so in particular $G\cG$ does indeed belong to $\seg_\spc(\Theta^\inj_{n,/\theta})$.\par
		Finally, to show that $F$ and $G$ are inverse to one another we can use \eqref{eq:G_d} and \eqref{eq:val_d} to show that
		\[G\circ F\cF(\theta_j\xrightarrow{j}\theta)\cong \bD\bD'\cF^j(\theta_j\eq \theta_j)\cong \cF(\theta_j\xrightarrow{j}\theta),\]
		where the last isomorphism follows from the second claim of \Cref{lem:verd} using that $\cF^j$ lies in the image of $\cI_*$, the proof of $G\circ F\cong \id$ is the same, this time using $\bD'\bD\cong \id$.
	\end{proof}
	\begin{theorem}\label{thm:main_theta}
		There is an equivalence
		\[\stab(\cat_{n,/\theta})\cong \psh_\spc(\Theta^{\inrt,\op}_{n,/\theta}),\]
		moreover for any $\cE\in \cat_n$ we have
		\[\stab(\cat_{n,/\cE})\cong \underset{f:\theta\rightarrow\cE}{\lim}\psh_\spc(\Theta^{\inrt,\op}_{n,/\theta}).\]
	\end{theorem}
	\begin{proof}
		Combine \Cref{prop:theta_comp} and \Cref{prop:seg_inj} for the first claim and use the second claim of \Cref{lem:globn} for the second.
	\end{proof}
	\section{Steiner complexes}\label{sect:stein}
	For a collection $\{f_1,...,f_l\}$ of 1-morphisms in a category, there is an essentially unique way in which they might be composable -- namely, if the target of one always coincides with the source of another. This is no longer the case for higher categories -- objects of $\Theta_n$ represent ways in which morphisms might be composed, however they no longer cover all the options as for example the orientals of \cite{street1987algebra} or the lax cubes of \cite{manin1989arrangements} do not belong to $\Theta_n$. These more general arrangements of morphisms are called \textit{pasting schemes}, and there have been numerous formalisms aimed at making this vague idea precise -- see \cite{forest2019unifying} for a review of some of those formalisms.\par 
	Our work will also make use of pasting schemes outside of $\Theta_n$ in the definition of the twisted arrows category in \Cref{sect:twar}, so we will also need to pick a formalism to describe them. We will employ the augmented directed chain complexes with unital loop-free basis of \cite{steiner2004omega}, which we will simply call \textit{Steiner complexes} following \cite{ara2023categorical}. Our goal in this section is to properly introduce this notion and prove some of its basic properties, culminating with \Cref{prop:stn_sat}. Admittedly, we will prove more that is necessary for the applications in \Cref{sect:twar}, however this will also help us prove one of the main results of the paper -- \Cref{thm:stein}.
	\begin{defn}\label{def:stein}
		Recall from \cite{steiner2004omega} that an augmented directed complex (\textit{ADC} for short) is a chain complex (in the classical sense) $(...\rightarrow K_n\xrightarrow{\partial_{n-1}}K_{n-1}\xrightarrow{\partial_{n-2}}...\xrightarrow{\partial_0}K_0)$ of abelian groups and an augmentation $e:K_0\rightarrow\bZ$ such that $e\circ \partial_0\cong 0$, we also require that each $K_n$ admits a distinguished commutative submonoid $K_n^+$ such that $K_n$ is the group completion of $K^+_n$. A morphism $f:K\rightarrow L$ of ADCs is by definition a morphism of underlying chain complexes such that $f_j(K_j^+)\subset L^+_j$. We will call the \textit{dimension} of $K$ denoted $\dim K$ the maximum $n$ such that $K_n\neq 0$. \par
		A \textit{basis} for an ADC is the data of elements $P_n\subset K_n^+$ for all $n$ such that $K_n^+$ is a free commutative monoid on $P_n$. Given a basis $P_n$, we can write any element $u\in K_n$ as $u=\sum _{g\in P_n} u_g *g$ with $u_g \in \bZ$, this induces a partial order on $K_n$ where  we declare $u\geq v$ just in case $u_g\geq v_g$ for $g\in P_n$. We can then write $\partial_{n-1}u= \partial^+ u - \partial ^- u$, where $\partial^\pm u\geq 0$. We also denote 
		\[u\wedge v\bydef \sum_{g\in P_n}\min(u_g, v_g) g.\]
		For $u\in K_n$ we denote $[u]=([u]^-_0,[u]^+_0,...,[u]^-_{n-1}, [u]^+_{n-1},[u]_n)$ the sequence of elements such that $[u]_n\bydef u$ and inductively $[u]^\pm_k\bydef \partial^\pm[u]^\pm_{k+1}$ and call it an \textit{atom} associated to $u$. We denote by $<_i$ the transitive closure of a relation on $P\bydef \bigcup_n P_n$ such that for $i<\min(l,k)$, $u\in P_l$ and $v\in P_k$ we have $u<_j v$ if $[u]^+_i\wedge [v]^-_i> 0$. We also write $u<_\bN v$ if either $u\leq\partial^-v$ or $v\leq\partial^+u$.\par
		Finally, we call a basis $P$ \textit{unital} if for all $u\in P$ we have $e([u]^\pm_0)=1$ and \textit{loop-free} if $<_j$ is non-reflexive for $i\in \bN$, i.e. if its transitive closure is a partial order. We will call a basis \textit{strongly loop-free} if $<_\bN$ is not reflexive (note that by \cite[Propositin 3.7.]{steiner2004omega} every strongly loop-free basis is loop-free). Following \cite{ara2023categorical} we will call an ADC with a strongly loop-free unital basis a \textit{strong Steiner complex} and denote by $\Stn$ the category of Steiner complexes and morphisms of ADCs between them and by $\Stn_n$ for $n\in \bN$ the subcategory of complexes of dimension $\leq n$. 
	\end{defn}
	\begin{construction}\label{constr:stein}
		We associate to an ADC $K$ the strict $\omega$-category $K^*$ such that its set of $n$-morphisms (called \textit{cells}) is the set of sequences $(X_0^-, X_0^+,..., X^-_{n-1}, X^+_{n-1}, X_n)$ such that $X_i^\pm\geq 0$,
		\begin{equation}\label{eq:unital}
			e(X_0^\pm)=1
		\end{equation}
		and
		\begin{equation}\label{eq:boundary}
			\partial_i X^+_{i+1}\cong \partial_i X^-_{i+1}= X^+_i - X^-_i,
		\end{equation}
		we will write $u\in X$ for a basis element $u$ if it appears in the decomposition of $X_n$ with non-zero coefficient. In particular, atoms $[u]$ associated to $u\in P_n$ define an $n$-morphism in $K^*$. The $i$-composition $X*_i Y$ of $X$ and $Y$ is defined to be the cell $Z$ such that $Z_{j, \pm}\bydef X_{j,\pm}+ Y_{j,\pm}$ for $j>i$, $Z_{i,-}\bydef X_{i,-}$, $Z_{i,+}\bydef Y_{i,+}$ and $Z_{j,\pm}\bydef X_{j,\pm}=Y_{j,\pm}$ for $j<i$. It is proved in \cite{steiner2004omega} that those morphisms do form an $\omega$-category, and moreover, if $K$ had a unital loop-free basis, then $K^*$ is a free $\omega$-category in an appropriate sense. 
	\end{construction}
	\begin{remark}\label{rem:theta}
		By \cite{steiner2006simple} we can associate to $\theta\in \Theta_n$ an object of $\stn_n$ with $l$-dimensional basis elements given by inert morphisms $c_l\xinert{}\theta$, the main result of that work is that its basis is strongly loop-free and the set of morphisms of $n$-categories between $\theta$ and $\theta'$ is equivalent to the set of morphisms of the corresponding ADCs with basis. It follows that we can identify $\Theta_n$ with a full subcategory of $\stn_n$ in this manner, which we will do from now on without further mention. Note that under this identification for $x\in\Stn_n$ the $l$-morphisms of the $n$-category $x^*$ correspond to morphisms $c_l\xrightarrow{f}x$ of ADCs. 
	\end{remark}
	\begin{notation}
		Given $x\in\Stn_n$ and $k\leq n$ we will denote by $x_k$ the set of basis elements of $x$ of dimension $k$. For $k<n$ and $\sigma\in\{-,+\}$ we will denote by $d_k^\sigma x\in\Stn_k$ the subcomplex of $x$ generated by the basis elements of dimension $k$ which are minimal (resp. maximal) with respect to $\leq_{k}$ if $\sigma = -$ (resp. $\sigma = +$), it is easy to see that such a complex indeed lies in $\Stn_k$. 
	\end{notation}
	\begin{prop}\label{prop:unit}
		Assume that we have a morphism $f:c_l\rightarrow x$ for some $x\in\Stn$ and $(a,b)$ is a pair of elements of $x_l$ such that $a\leq f(\id_{c_l})$ and $b\leq f(\id_{c_l})$, then $\partial^\epsilon_{l-1}a \bigcap \partial^\epsilon_{l-1} b = \varnothing$. In particular, if
		\begin{equation}\label{eq:unit}
			f(\id_{c_l}) = \sum_{i\in\cI} n_i*b_i
		\end{equation}
		with $n_i\in \bN$, then all $n_i=1$.
	\end{prop}
	\begin{proof}
		The first claim follows from \cite[Lemma 3.4.5.]{forest2019unifying}, the second follows from the first since if for some $i$ we have $n_i>1$, then the pair $(b_i,b_i)$ contradicts the first claim.
	\end{proof}
	\begin{remark}\label{rem:nozero}
		Note that for a strong Steiner complex $x$ and $b\in x_k$ we have $\partial^-b>0$ and $\partial^+b>0$: indeed, assume that $\partial^- b=0$, then 
		\[0= \partial\circ \partial b= \partial(\partial^+b) = \partial^+\partial^+b - \partial^-\partial^+b,\]
		hence we have $\partial^+\partial^+b = \partial^-\partial^+b$. It follows that the induced morphism $c_{k-1}\xrightarrow{f}x$ sending $[\id_{c_{k-1}}]$ to $\partial^+ b$ defines an endomorphism, which contradicts \Cref{prop:unit}.
	\end{remark}
	\begin{defn}\label{def:act}
		Given a morphism $f:x\rightarrow y$ in $\Stn_n$, it follows from \Cref{prop:unit} that the image of any basis element $c_j\xinert{b}x$ in $y$ can be identified with a collection of basis elements $\{c^b_0,...,c^b_k\}$ of $y$. We will call the morphism $f$ \textit{active} if for every basis element $c$ of $y$ there is a basis element $b$ of $x$ such that $c= c^b_j$ for some $j$ and \textit{inert} if $f$ takes basis elements to basis elements of the same dimension. Finally, we will call $x\in \Stn_n$ \textit{elementary} if there exists a basis element $b\in x_m$ for some $m\leq n$ such that every other basis element belongs to some $d^\pm_l b$. We will denote by $\stn_n\hookrightarrow\Stn_n$ the full subcategory containing objects $x$ that admit an active morphism $a:c_n\xactive{}x$ (which is necessarily unique by \Cref{prop:unit} again). Given $x\in \Stn_n$ we will denote by $\stn^\inrt_{n,/x}$ the full subcategory of $\stn_{n,/x}$ on inert morphisms $y\xinert{i}x$ and by $\stn^\el_{n,/x}$ a further subcategory containing $e\xinert{}x$ with $e$ elementary (note that such morphisms may be identified with basis elements in $x$). 
	\end{defn}
	\begin{remark}\label{rem:k_bound}
		If $e\in\stn^\el_n$ with $\dim(e)=n$, then for every $k<n$ we have $(d^-_k e)_k\cap(d^+_k e)_k=\varnothing$. Indeed, if $c\in(d^-_k e)_k\cap(d^+_k e)_k$ and $e_n$ denotes the unique basis element of $e$ in dimension $n$, then we would have $c\leq_\bN e_n\leq_\bN c$, which contradicts the fact that the basis of $e$ was assumed to be strongly loop-free.
	\end{remark}
	\begin{defn}\label{def:jact}
		For $x\in\stn^c_n$ and $0\leq j\leq n-1$ we will call a morphism $c_n\xrightarrow{f}x$ $j$\textit{-active} if the image of $i^\pm_j:c_j\xinert{} c_n$ is $d^\pm_j x$, additionally we define every morphism to be \textit{(-1)-active}.
	\end{defn}
	\begin{lemma}\label{lem:n-1act}
		A morphism $f:c_n\rightarrow x$ for $x\in\stn_n^c$ is $(n-1)$-active in the sense of \Cref{def:jact} if and only if it is active in the sense of \Cref{def:act}.
	\end{lemma}
	\begin{proof}
		Assume that there is some basis element $b_0$ of dimension $n$ that is not in the image of $f$. Consider some chain $b_0 <_{n-1}b_1 <_{n-1}...<_{n-1}b_m$ of maximal length, since it is maximal we must have $d^+_{n-1}b_m\leq d^+_{n-1} x$. Since $f$ was assumed $(n-1)$-active, there must be some basis element $e_0$ in the image of $f$ such that $d^+_{n-1}b_m\cap d^+_{n-1}e_0\neq \varnothing$, it now follows from \Cref{prop:unit} that $b_m = e_0$. By construction $d^+_{n-1}b_{m-1}\leq d^-_{m-1}e_0$, since it clearly does not lie in $d^-_{n-1}x$ there must be some $e_1$ in the image of $f$ such that $d^+_{n-1} e_1$ intersects $d^+_{n-1}b_{m-1}$, repeating the same argument we get $b_{m-1} = d_1$, iterating this process we eventually prove that $b_0$ lie in the image of $f$.
	\end{proof}
	\begin{lemma}\label{lem:jact}
		Given a $k$-active $f:c_n\rightarrow x$ morphism for $k<(n-1)$ with $x\in\stn_n^c$, there exists a $(k+1)$-active morphism $f^\act_{k+1}:C^{k+1}_n(1)\rightarrow x$ in the notation of \Cref{constr:C_delta} such that
		\begin{equation}\label{eq:comp_act}
			f^\act_{k+1}\circ i_{k+1}\cong f.
		\end{equation}
	\end{lemma}
	\begin{proof}
		Recall the order $\leq_p$ for $0\leq p\leq n-1$ from \Cref{def:stein}, define $x'\xinert{i'}x$ to be the image of $f$. For $\epsilon\in \{-,+\}$ define $S^\epsilon_n\subset x_n$ to contain elements $b$ such that there is $b'\in (d^\epsilon_{k+1}x')_{k+1}$ for which $b\geq_{k+1} b'$ if $\epsilon=+$ and $b\leq_{k+1}b'$ if $\epsilon=-$, for $k+1\leq q<n$ and $\epsilon'\in\{-,+\}$ define $S^{\epsilon,\epsilon'}_q\subset x_q$ to contain elements $b_0$ such that there is $b_0'\in (d^\epsilon_{k+1}x')_{k+1}$ for which $b_0\geq_{k+1} b_0'$ if $\epsilon=+$ and $b_0\leq_{k+1}b_0'$ if $\epsilon=-$ and such that $b_0$ is maximal (resp. minimal) with respect to $\leq_{q}$ if $\epsilon'=+$ (resp. $\epsilon'=-$). Now define 
		\[y^\epsilon_n\bydef \sum_{b\in S^\epsilon_n} b,\]
		for $k+1\leq q<n$ define 
		\[(y^\epsilon)^{\epsilon'}_q\bydef \sum_{b\in S^{\epsilon,\epsilon'}_q} b\]
		and for $q\leq k$ define
		\[(y^\epsilon)^{\epsilon'}_q\bydef \sum_{b\in (d^{\epsilon'}_qx')_q} b.\]
		We need to prove that $(y_n,y_{n-1}^+, y_{n-1}^-,...,y_0^+, y_0^-)$ defines a morphism in $x^*$. First note that for $k\geq 0$ \eqref{eq:unital} holds because $e(d^\pm_0x')=1$ and if $k=-1$ because $e(d^\pm_0x)=1$, so it remain to prove \eqref{eq:boundary}, i.e. that
		\begin{equation}\label{eq:y_bound}
			\partial y^{\epsilon,\epsilon'}_q = y^{\epsilon,+}_{q-1} - y^{\epsilon,-}_{q-1}
		\end{equation}
		. For $q\leq k$ it holds because it holds for $x'$, for $q>k+1$ we can write
		\begin{equation}\label{eq:sum1}
			\partial y^{\epsilon,\epsilon'}_q = \sum_{b\in S^{\epsilon,\epsilon'}_q}(\partial^+b - \partial^- b),
		\end{equation}
		on the other hand  we have
		\begin{equation}\label{eq:sum2}
			(y^{\epsilon,+}_{q-1} - y^{\epsilon,-}_{q-1}) =\sum_{b_0\in S^{\epsilon,+}_{q-1}}b_0 - \sum_{b_1\in S^{\epsilon,-}_{q-1}}b_1 =\sum_{b_0\in S^{\epsilon,+}_{q-1}\setminus S^{\epsilon,-}_{q-1}}b_0 - \sum_{b_1\in S^{\epsilon,-}_{q-1}\setminus S^{\epsilon,+}_{q-1}}b_1.
		\end{equation}
		Observe that the only terms that do not cancel out in \eqref{eq:sum1} are the ones that lie in $S^{\epsilon,\pm}_{q-1}$ while the terms in \eqref{eq:sum2} correspond to those elements of $S^{\epsilon,\pm}_{q-1}$ that lie in $\partial^\pm_{q-1}b$ for some $b\in x_q$, we claim that we may assume that $b$ is either maximal or minimal with respect to $\leq_{q}$, so that it lies in $S^{\epsilon,\pm}_q$. Indeed, assume that we have (say) $t\leq \partial^-_{q-1}b$ for $t\in S^{\epsilon,-}_{q-1}$ and $b\leq_q b'$ for some basis element $b'$, we may assume that $b\leq \partial^-_{q}b'$, then since $t$ was assumed to be minimal with respect to $(q-1)$ it follows that $t\leq \partial^-_{q-1}b'$, hence that there exists some $b''\leq \partial^+_q b'$ such that $t\leq \partial^-_{q-1}b'$, iterating this argument we can make it so that $b'\in S^{\epsilon,+}_q$. An obvious variation of the same argument also proves that we may pick $b\in S^{\epsilon,-}_q$ and that the claim also holds for $t\in S^{\epsilon,+}_{q-1}$. Finally, it remains to prove \eqref{eq:boundary} for $q=k+1$, note that we can assume $k\geq 0$. We will in fact show that $y^{\epsilon,-\epsilon}_{k+1} = (d^\epsilon_{k+1} x')_{k+1}$ and $y^{\epsilon,\epsilon}_{k+1} = (d^\epsilon_{k+1} x)_{k+1}$, this will prove the claim since $f$ was assumed to be $k$-active. Note that the first of those claims follows directly from the definition of $S^{\epsilon,\epsilon'}_{k+1}$, so it suffices to prove $d^\pm_{k+1}y^{\pm} = d^\epsilon_{k+1} x$. For this note that $d_{k+1}^\pm y^\pm$ defines a composable subset of $d_{k+1}^\pm x$ that contains $d_k^\sigma d_{k+1}^\pm x' = d^\sigma_kx$ since $f$ was assumed to be $k$-active, hence it must coincide with $d^\pm_{k+1}x$ by \Cref{lem:n-1act}.
	\end{proof}
	\begin{cor}\label{cor:D_ext}
		Any morphism $f:c_n\rightarrow x$ for $x\in\stn_n$ extends to an active morphism $a':D_n(1,0,...,0)\xactive{} x$ in the notation of \Cref{constr:D_obj} such that the composition 
		\[c_n\xinert{D([0]\xinert{\{1\}}[1],\id_{[0]},...,\id_{[0]})}D_n(1,...,0)\xactive{a'}x\]
		is equal to $f$.
	\end{cor}
	\begin{proof}
		Consider the functor $G:\Delta^\el_{/[n]}\rightarrow\cat$ corresponding to the diagram
		\[\begin{tikzcd}[sep=huge]
			& {C_n^0(1)} && {C_n^1(1)} && {C_n^{n-1}(1)} \\
			{c_n} && {c_n} && {c_n} && {c_n}
			\arrow["{i_0^*}"', tail, from=2-1, to=1-2]
			\arrow[two heads, from=2-3, to=1-2]
			\arrow["{i_1^*}"', tail, from=2-3, to=1-4]
			\arrow["{...}"{description}, two heads, from=2-5, to=1-4]
			\arrow["{i_{n-1}^*}"', tail, from=2-5, to=1-6]
			\arrow[two heads, from=2-7, to=1-6]
		\end{tikzcd},\]
		then inductively using \Cref{lem:jact} for $-1\leq k\leq n-2$ we can construct morphisms $f_k:C^{k+1}_n(1)\rightarrow x$ such that each $f_k$ is $(k+1)$-active and together they define a cocone on $G$ with vertex $x$. The claim now follows since the colimit of $G$ is by construction isomorphic to $D_n(1,0,...,0)$, so the cocone defines a morphism $a':D_n(1,...,0)\rightarrow x$, the fact that it is active follows since both $f_{n-1}:C^{n-1}_n(1)\rightarrow x$ and $C^{n-1}_n\xactive{}D_n(1,...,0)$ are active (the former by \Cref{lem:n-1act}, the latter by construction).
	\end{proof}
	\begin{prop}\label{prop:stn_cont}
		For $\cC\in\cat$ denote 
		\[\psh_\omega(\cC)\bydef \mor_\cat(\cC^\op,\{\varnothing,*\})\hookrightarrow\psh_\cS(\cC),\]
		where the inclusion identifies $\psh_\Omega(\cC)$ with the subcategory of (-1)-truncated objects. Every $(y\xinert{i}x)\in\stn^\inrt_{n,/x}$ for $x\in\Stn_n$ defines an element $h_y\in\psh_\Omega(\stn^{\el}_{n,/x})$ sending $e\xinert{i_e}x$ to $\mor_{\stn^\inrt_{n,/x}}(i_e,i)$, then $h_y$ belongs to the full subcategory $\pfib(\stn^\el_{n,/x})\hookrightarrow \psh_\Omega(\stn^\el_{n,/x})$ generated by pushout of representables. In particular, if $x\in\stn_n$, then $\stn^\el_{n,/x}$ is contractible.
	\end{prop}
	\begin{proof}
		We will prove this by induction on $\dim(y)$: if $\dim(y) = 0$, then by \eqref{eq:unital} we have $y\cong*$, so that $h_y$ is representable. Assume we have proved the claim in dimensions $<k$ and $y_k = \{b_1,...,b_l\}$, we will complete the proof by induction on $l$. If $l=0$, then $\dim(y) <k$ and the claim follows by induction, so assume we have proved the claim for $(l-1)$. We may assume that $b_l$ is maximal among $b_i$ with respect to $\leq_{\bN}$, denote by $y'$ the subcomplex of $y$ generated by all the basis elements except $b_l$ and by $[b_l]$ the subcomplex generated by $b_l$, then it follows from \cite[Proposition 5.1.]{steiner2004omega} that $y'\in \stn_n$ and moreover that we have
		\[h_y\cong h_{y'}\coprod_{h_{d^-_{r}[b_l]}} h_{[b_l]}\]
		for some $r<k$, which concludes the proof since $h_{y'}\in \pfib(\stn^\el_{n,/x})$ by induction. The last claim now follows since 
		\[|\stn^\el_{n,/x}|\cong p_!(*),\]
		where $p:\stn^\el_{n,/x}\rightarrow\pt$ is the unique morphism to the final object of $\cat$ and $*\in\psh_\cS(\stn^\el_{n,/x})$ denotes the terminal presheaf. By definition $x\in\stn_n$ if and only if $*\in\stn^\inrt_{/x}$, which implies that it is an iterated pushout of representables, which in particular means that $p_!(*)\cong *$ since $p_!(h_y)\cong *$ by construction and pushout is a contractible colimit. 
	\end{proof}
	\begin{lemma}\label{lem:delta_push}
		Assume we have a cospan $[n]\overset{a_0}{\twoheadleftarrow}[l]\xactive{a_1}[m]$ of active injective morphisms in $\Delta$ such that for every elementary interval $(k<k+1)$ in $[l]$ either the restriction $a_i|_{(k<k+1)}:[1]\xactive{}[q_k]$ is given by the identity morphism for some $i\in \{0,1\}$, then we have a pushout diagram
		\begin{equation}\label{eq:delta_push}
			\begin{tikzcd}[sep=huge]
				{[l]} & {[n]} \\
				{[m]} & {[n+m-l]}
				\arrow["{a_0}", two heads, from=1-1, to=1-2]
				\arrow["{a_1}"', two heads, from=1-1, to=2-1]
				\arrow[two heads, from=1-2, to=2-2]
				\arrow[two heads, from=2-1, to=2-2]
				\arrow["\ulcorner"{anchor=center, pos=0.125, rotate=180}, draw=none, from=2-2, to=1-1]
			\end{tikzcd}
		\end{equation}
		such that additionally $[n+m-l]$ can be identified with the set 
		\[X\bydef (\obj([n])\coprod \obj([m]))/\sim,\]
		where $\sim$ denotes the equivalence relation that identifies $a_0(i)$ with $a_1(i)$ for $i\in [l]$ such that $x_0\leq x_1$ for $(x_0,x_1)\in X^{\times 2}$ if and only if either both $x_i$ lie in $\obj([n])$ or $\obj([m])$ and $x_0\leq x_1$ or there exists $q\in [l]$ such that $x_0\leq a_k(q)$ and $a_t(q)\leq x_1$ for $k\neq t$.
	\end{lemma}
	\begin{proof}
		We will first prove that \eqref{eq:delta_push} is a pushout, we will do so by induction on the number of elementary segments $(k<k+1)$ for which $a_0|_{(k<k+1)}$ does not equal identity. If this number is 0, then $a_0=\id$ and the claim is vacuous. In the general case denote by $(t<t+1)$ the maximal (with respect to the linear order on $[l]$) subinterval for which the restriction of $a_0$ is non-trivial, denote this restriction by $a_t:[1]\xactive{}[v]$ note that we can decompose $a_0\cong a_0'\circ a_0''$ such that $a_0''|_{(k<k+1)} = a_0|_{(k<k+1)}$ if $k\neq t$ and $a_0'|_{(k<k+1)} = a_0|_{(k<k+1)}$ if $k=t$ and is identity otherwise. We can then construct the following diagram
		\begin{equation}\label{eq:push_diag}
			\begin{tikzcd}[sep=huge]
				& {[1]} & {[v]} \\
				{[l]} & {[n']} & {[n]} \\
				{[m]} & {[n'+m-l]} & {[n+m-l]}
				\arrow["{a_t}", two heads, from=1-2, to=1-3]
				\arrow["{i_t}"', tail, from=1-2, to=2-2]
				\arrow["{i'_t}", tail, from=1-3, to=2-3]
				\arrow["{a''_0}", two heads, from=2-1, to=2-2]
				\arrow["{a_1}"', two heads, from=2-1, to=3-1]
				\arrow["{a_0'}", two heads, from=2-2, to=2-3]
				\arrow["{a'_1}", two heads, from=2-2, to=3-2]
				\arrow["{a''_1}", two heads, from=2-3, to=3-3]
				\arrow[two heads, from=3-1, to=3-2]
				\arrow["\ulcorner"{anchor=center, pos=0.125, rotate=180}, draw=none, from=3-2, to=2-1]
				\arrow[two heads, from=3-2, to=3-3]
			\end{tikzcd},
		\end{equation}  
		where $i_t$ is the inclusion of the subinterval $(t<t+1)$. Note that the leftmost square in \eqref{eq:push_diag} is a pushout by induction, we need to prove that the bottom right square is a pushout. By \cite[Lemma 4.4.2.1.]{lurie2009higher} it suffices to prove that the top right square and the outer right rectangle are pushouts. Note that both $a'_1\circ i_t$ and $a''_1\circ i'_t$ are inert by our condition on $a_1$, hence both squares in question are of the form
		\[\begin{tikzcd}[sep=huge]
			{[1]} & {[v]} \\
			{[p]} & {[p+v-1]}
			\arrow[two heads, from=1-1, to=1-2]
			\arrow[tail, from=1-1, to=2-1]
			\arrow[tail, from=1-2, to=2-2]
			\arrow[two heads, from=2-1, to=2-2]
		\end{tikzcd},\]
		so it remains to prove that all such squares are pushouts. First, note that by Segal condition we have
		\begin{equation}\label{eq:col_delta}
			[n]\cong \underset{([e]\xinert{i}[n])}{\colim}[e]\cong [1]\coprod_{[0]}[1]\coprod_{[0]}...\coprod_{[0]}[1].
		\end{equation}
		Applying \eqref{eq:col_delta} we see that both $[p+v-1]$ and
		\[[t-1]\coprod_{[0]}[v]\coprod_{[0]}[p-t]\]
		are isomorphic. This concludes the proof of the first part of the claim, the second easily follows by inspection.
	\end{proof}
	\begin{lemma}\label{lem:lin}
		The partial order $\leq_\bN$ on $\stn^\el_{n,/x}$ is linear if $x\in\stn_n$.
	\end{lemma}
	\begin{proof}
		We will prove the claim by induction on $\dim(x)$: if $\dim(x) = 0$, then $x$ is a singleton by \eqref{eq:unital}, so in particular the claim is true. Assume we have proved the claim for dimensions $<k$ and $x_k = \{b_1,...,b_l\}$, we will prove the claim by induction on $l$: of $l=0$, then $\dim(x)<k$ and it holds by induction. Assume now that $x\in\stn^\el_n$ and denote by $b$ its unique basis element. Denote by $A$ the subset of $\stn^\el_{/x}$ containing all basis elements in $d^-_{n-1}x$ that do not appear in $d^+_{n-1}x$, set $B\bydef \{b\}$ and denote by $C$ the subset containing all basis elements in $d^+_{n-1}x$ that do not appear in $d^-_{n-1}x$, it follows from \Cref{rem:k_bound} that the underlying set of $\stn^\el_{n,/x}$ decomposes as a disjoint union:
		\[|\stn^\el_{n,/x}| = A\coprod B\coprod C.\]
		By construction we have $a\leq_{\bN}b$ for all $a\in A$ and $b\leq_{\bN}c$ for all $c\in C$, so it remains to prove that all the sets $A$, $B$ and $C$ are linearly ordered. This claim is trivial for the singleton $B$, to prove it for $A$ note that it is a subset of $d^-_{n-1}x$, which is linearly ordered by $\leq_{\bN}$ by induction on dimension, and that $a\leq_\bN a'$ as elements of $\stn^\el_{n,/d^-_{n-1}x}$ implies that also $a\leq_\bN a'$ in $\stn^\el_{n,/x}$. The same argument also applies for $C$, which concludes the proof.\par
		In the general case we again assume that $b_l$ is maximal with respect to $<_\bN$, denote by $x'$ the subcomplex generated by $\{b_1,...,b_{l-1}\}$ and by $[b_l]\in\stn^\el_{n}$ the subcomplex generated by $b_l$, so that we have
		\[x\cong x'\coprod_{d^-_r [b_l]}[b_l]\]
		for some $r<n$. We first claim that it suffices to treat the case $r=(n-1)$. Indeed, using \Cref{cor:D_ext} we can extend $c_{r}\xactive{}d^-_r [b_l]\xinert{}d^+_rx$ to an active morphism morphism $D_r(1,...,0)\xactive{a'} d^+_rx$. Denote by $X$ the pushout 
		\[\begin{tikzcd}[sep=huge]
			{c_r} & {c_n} \\
			{D_r(1,0,...,0)} & X
			\arrow["{i^+_r}", tail, from=1-1, to=1-2]
			\arrow[two heads, from=1-1, to=2-1]
			\arrow[two heads, from=1-2, to=2-2]
			\arrow["{i_D}",tail, from=2-1, to=2-2]
			\arrow["\ulcorner"{anchor=center, pos=0.125, rotate=180}, draw=none, from=2-2, to=1-1]
		\end{tikzcd}\]
		in $\stn_n$, so that the unique active morphism $c_n\xactive{}x'$ and the morphism $a'$ described above now define a morphism $X\xactive{a''}x'$. Next, define $X'$ to be the pushout
		\[\begin{tikzcd}[sep=huge]
			{c_r} & {D_r(1,...,0)} & X \\
			{c_{n-1}} && {X'}
			\arrow["{D([0]\xinert{\{1\}}[1],...,\id_{[0]})}"', tail, from=1-1, to=1-2]
			\arrow["{i^-_r}"', tail, from=1-1, to=2-1]
			\arrow["{i_D}"', tail, from=1-2, to=1-3]
			\arrow[from=1-3, to=2-3]
			\arrow[from=2-1, to=2-3]
		\end{tikzcd},\]
		then there is a morphism $X'\xrightarrow{g}x$ whose restriction to $X$ is given by $X\xactive{a''}x'\xinert{}x$ and to $c_{n-1}$ by the composition $c_{n-1}\xactive{a_0}d^-_{n-1}[b_l]\xinert{}x$, where $a_0$ exists by induction. Denote by $x'_0$ the image of $g$, then it $x'_0\in\stn_n$ since $X'$ admits an active morphism from $c_n$ and $x'_0$ shares a common $(n-1)$-boundary with $[b_l]$. It follows from these considerations that we may assume $x'=x'_0$ and $r=(n-1)$.\par
		Note that all of the categories $\stn^\el_{n,/x'}$, $\stn^\el_{n,/d^-_r[b_l]}$ and $\stn^\el_{n,/[b_l]}$ are linearly ordered either by induction on $l$ or $k$ or by previous considerations. We will first treat the case $n=1$: note that $*\bydef d^+_0x'\cong d^-_0[b_l]$ is a singleton and moreover the image of the unique element of $*_0$ in $\stn^\el_{n,/x'}$ (resp. $\stn^\el_{n,/[b_l]}$) is a maximal (resp. minimal) element, from which it follows that $\stn^\el_{n,/x}$ is a linearly ordered set in which each element $z\neq *$ in the image of $\stn^\el_{n,/[b_l]}$ is greater that each element $y$ in the image of $\stn^\el_{n,/x'}$. Assume now that $n>1$, in that case both $*_-\bydef d^-_0[b_l]$ and $*_+\bydef d^+_0[b_l]$ lie in $d^-_{n-1}[b_l]$, note that $*_-$ (resp. $*_+$) is the minimal (resp. maximal) element of $\stn^\el_{n,/[b_l]}$. It follows that if $z< *_-$ (resp. $z> *_+$) for $z\in\stn^\el_{n,/x}$, then $z\leq y$ (resp. $z\geq y$) in $\stn^\el_{n,/x}$ for all $y$ in the image of $\stn^\el_{n,/[b_l]}$, so we may assume $z\in[*_-, *_+]$. Denote $[l]\bydef \stn^\el_{n,/ d^-_{n-1}[b_l]}$, $[m]\bydef \stn^\el_{n,/[b_l]}$ and by $[n]$ the subinterval $[*_-, *_+]$ of $\stn^\el_{n,/x'}$, in that case we may identify $[l]$ with an ordered subset of both $[m]$ and $[l]$ that contains the endpoints, in other words we have a cospan $[n]\overset{a_0}{\twoheadleftarrow}[l]\xactive{a_1}[m]$ of active injective morphisms. We will prove that it satisfies the conditions of \Cref{lem:delta_push}, before doing that we will demonstrate that this suffices to conclude the proof. Note that for $z\in\stn^\el_{n,/x'}$ and $y\in\stn^\el_{n,/[b_l]}$ we have $z\leq_\bN y$ in $\stn^\el_{n,/x}$ if and only if there exists an $w\in \stn^\el_{n,/d^-_r[b_l]}$ such that $z\leq_\bN w\leq_\bN y$ and similarly for $y\leq_\bN z$. It follows from this and the second claim of \Cref{lem:delta_push} that (if the conditions of the lemma is satisfies) the order $\leq_\bN$ on $\stn^\el_{n,/x}$ coincides with the linear order on $[n+m-l]$.\par
		Finally, we will prove that the conditions of \Cref{lem:delta_push} are satisfied. Note that for any pair of consecutive elements $(y<y')$ for $\leq_\bN$ we have $\dim(y')=\dim(y)\pm1$. It follows from our previous description of the order on $[b_l]$ that a pair of consecutive elements $(z<z')\in[l]$ remain consecutive in $[m]$ unless $z\in A$ and $z'\in C$, where 
		\[A\bydef \coprod_{0\leq k\leq n-1} A_k,\; A_{n-1}\bydef (d^-_{n-1}b_l)_{n-1},\;A_k\bydef (\bigcup_{a\in A_{k+1}} (d^\pm_{k}a)_k\setminus (d^+_k(b_l))_k)\]
		and similarly
		\[C\bydef \coprod_{0\leq k\leq n-2} C_k,\; C_{n-1}\bydef (d^+_{n-2}b_l)_{n-2},\;C_k\bydef (\bigcup_{c\in C_{k+1}} (d^\pm_{k}a)_k\setminus (d^-_k(b_l))_k).\]
		We claim that this is only possible if $\dim(z) = n-1$ and $\dim(z')=n-2$: assume that $\dim(z)=k<n-1$, we then claim that there exists $y\in (d^-_{n-1}[b_l])_{n-1}$ such that $z\leq_\bN y$. Indeed, it follows from \Cref{rem:k_bound} that $z$ belongs to the boundary of some $y'\in(d^-_{n-1}[b_l])_{n-1}$, assume first that $z$ belongs to $d^+_{q}y$ for some $q$, so that in particular $z\geq_\bN y'$, in that case since $z$ does not belong to $d^+_q[b_l]$ by construction there must be another element $y'\leq_qy''$ such that $z$ appears in $d^-_q y''$. Iterating this procedure for all values of $q$ gives us a basis element $y$ containing $z$ in its boundary and such that $z$ does not appear in $d^+_q y$ for any $q$, meaning that $z\leq_\bN y$. But then note that since $z'\in C$ it must belong to $d^+_{n-2}y$, meaning that $y\leq_\bN z'$, which contradicts our assumption that $z$ and $z'$ are consecutive.\par
		So assume now that we have consecutive $z\leq_\bN z'$ with $z\in C$, $z'\in D$, $\dim(z) = n-1$ and $\dim(z')=n-2$, to finish the proof it remain to show that they remain consecutive in $\stn^\el_{n,/x'}$. Assume the contrary, pick the maximal chain $z\leq_\bN z_1\leq_\bN...\leq_\bN z_p\leq_\bN z'$, then we must have $\dim(z_p) = n-1>\dim(z')$ since otherwise it would lie in $d^-_{n-1}[b_l]$. It follows that $z$ and $z_p$ are a pair of basis elements of dimension $(n-1)$ for which $d^+_{n-2}z\cap d^+_{n-2}z_p\neq \varnothing$ (as it must contain $z'$). We claim that this is only possible if $z_p\leq_{n-1}z$. Note that if this holds, then in particular we have $z\leq_\bN z_p\leq_\bN z$, where the first inequality holds by assumption and the second since $\leq_\bN$ refines all $\leq_q$, hence $z=z_p$ because the basis of $x$ is strongly loop-free -- a contradiction. It now remains to show that $z_p\leq_{n-1}z$. Since it does not belong to $d^+_{n-1}x'$, there must be $y\in x'_n$ such that $z_p\in (d^-_{n-1}y)_{n-1}$, we claim that $z'\geq_{n-2} y$: indeed, since $z'\in (d^+_{n-2}z_p)_{n-2}$ and $z_p\in (d^-_{n-1}y)_{n-1}$ we must have $z'\in (d^-_{n-1}y)_{n-2}$, but also $z'\in(d^+_{n-1}y)_{n-2}$ since it belongs to $d^+_{n-1}x'$ by construction, which together implies that $z'\in (d^+_{n-2}y)_{n-2}$. Now we may choose some $z^{(0)}_p\in (d^+_{n-1}y)_{n-1}$ such that $z'\in (d^+_{n-2}z^{(0)}_p)_{n-2}$, applying the same reasoning to $z^{(0)}_p$ we may construct a maximal string $z_p\leq_{n-1}z^{(0)}_p\leq_{n-1}...\leq_{n-1}z^{(s)}_p$ such that $z'\in (d^+_{n-2}z^{(i)}_p)_{n-2}$ for all $i$. In that case $z^{(s)}_p$ must lie in $(d^+_{n-1}x')_{n-1}$, so that $z$ and $z^{(s)}_p$ are both the basis elements in $d^+_{n-1}x'$ for which $z'\in (d^+_{n-2}z^{(s)}_p)_{n-2}\cap(d^+_{n-2}z)_{n-2}$, which contradicts \Cref{prop:unit}.
	\end{proof}
	\begin{lemma}\label{lem:seq}
		Given a morphism $f:x\rightarrow y$ in $\stn_n$ and a basis element $b\in y_l$, the set $f^{-1}(b)$ containing all basis elements $c$ for which $f(c)=b$ is either empty or a segment for $\leq_\bN$.
	\end{lemma}
	\begin{proof}
		Assume that $f^{-1}(b)$ is non-empty and take a pair of elements $(c,d)$ such that $f(c) = f(d)=b$. In that case by \Cref{lem:lin} we may assume $c\leq_\bN d$, denote by $c\leq_\bN c_0\leq_{\bN}c_1\leq_\bN...\leq_\bN c_s\leq_\bN d$ the segment containing all the elements $t$ such that $c\leq_\bN t\leq_\bN d$, we claim that $f(c_i)=b$ for all $i$. Assume the contrary, denote $b_i\bydef \ima(f|_{[c_i]})$, so that $b_i$ is the subcomplex of $y$ generated by the image of $c_i$. Note that since $c_i\leq_\bN c_{i+1}$ we either have $b_i\xinert{} d_l^- b_{i+1}$ or $b_{i+1}\xinert{} d^+_l b_i$ for some $l$. Note also that all $b_i\in\stn^\el_n$, hence are also linearly ordered by $\leq_\bN$, denote by $m_i$ the minima element (with respect to $\leq_\bN$) of $b_i\cap b_{i-1}$ (where we set $b_{-1}\bydef b$), then we have a string $b\leq_\bN m_0\leq_\bN m_1\leq_\bN...\leq_{\bN} m_s\leq_\bN b$ in $\stn^\el_{n,/y}$, this defines a loop, hence since $y$ was assumed to be a strong Steiner complex this string must be trivia, meaning that $f(c_i)=b_i=b$ for all $i$.
	\end{proof}
	\begin{prop}\label{prop:stein_fact}
		$\stn_n$ admits an active/inert factorization.
	\end{prop}
	\begin{proof}
		Assume we have a morphism $f:x\rightarrow y$, denote by $y'$ the subcomplex of $y$ containing only the basis elements appearing in the decomposition of the images of the basis elements of $x$, then we have a natural inert morphism $y'\xinert{i}y$ and $f$ factors as $i\circ a$ for some active morphism $a:x\xactive{}y'$. We claim that $y'\in\stn_n$: it admits a unital basis by construction, and it is strongly loop-free since any loop for $\leq_\bN$ in $y'$ induces a loop in $y$, and hence must be trivial since $y\in \Stn_n$. Finally, if there is an active morphism $c_n\xactive{a_0}x$, then there is an active morphism $c_n\xactive{a\circ a_0}y'$.\par
		It is easy to see that such a factorization is also functorial in the sense that the following diagrams commute
		\[\begin{tikzcd}[sep=huge]
			{x''} & {y''} \\
			{x'} & {y'} & w & {w'} & {w''} \\
			x & y & z & {z'} & {z''}
			\arrow[two heads, from=1-1, to=1-2]
			\arrow["i"', tail, from=1-1, to=2-1]
			\arrow[tail, from=1-2, to=2-2]
			\arrow[two heads, from=2-1, to=2-2]
			\arrow["{i'}"', tail, from=2-1, to=3-1]
			\arrow[tail, from=2-2, to=3-2]
			\arrow[two heads, from=2-3, to=2-4]
			\arrow["j"', tail, from=2-3, to=3-3]
			\arrow[two heads, from=2-4, to=2-5]
			\arrow[tail, from=2-4, to=3-4]
			\arrow[tail, from=2-5, to=3-5]
			\arrow["a"', two heads, from=3-1, to=3-2]
			\arrow["{a'}"', two heads, from=3-3, to=3-4]
			\arrow["{a''}"', two heads, from=3-4, to=3-5]
		\end{tikzcd}\]
		for any composable strings $x''\xinert{i}x'\xinert{i'}x''\xactive{a}y$ and $w\xinert{j}z\xactive{a'}z'\xactive{a''}z''$, where each square in the diagrams is obtained by factorization described above. To prove the claim it remains to show that the only morphisms that are both inert and active are isomorphisms, so assume that $f:x\rightarrow y$ is such a morphism. Since $f$ is inert it sends basis elements to basis elements and since it is active every basis element in $y$ is in the image of at least one basis element of $x$. It follows that for any $b\in y_l$ the preimage $f^{-1}(b)$ is a finite set of basis elements of dimension $\dim(b)$, however by \Cref{lem:seq} it must be a segment for $\leq_\bN$, which is only possible if $f^{-1}(b)$ is a singleton. Since this holds for all such $b$, $f$ must be an isomorphism.
	\end{proof}
	\begin{remark}
		Note that the conclusion of \Cref{prop:stein_fact} does not hold in the bigger category $\Stn_n$: we can still define the functorial factorization $x\xactive{a}y'\xinert{i}y$ for any $f:x\rightarrow y$, but there could be morphisms that are both inert and active, for example the functor $F$ between categories shown below
		\[\begin{tikzcd}[sep=huge]
			& a \\
			0 & {} && 0 & 1 \\
			& b
			\arrow[from=2-1, to=1-2]
			\arrow[from=2-1, to=3-2]
			\arrow["F", shorten <=10pt, shorten >=14pt, from=2-2, to=2-4]
			\arrow[from=2-4, to=2-5]
		\end{tikzcd}\]
		that takes 0 to 0 and both $a$ and $b$ to 1 is both active and inert (naturally, the source of $F$ does not belong to $\stn_1$).
	\end{remark}
	\begin{lemma}\label{lem:lin_cont}
		Every subcategory of $\stn^\el_{n,/x}$ for $x\in\stn_n$ which is a segment for $\le_\bN$ is contractible.
	\end{lemma}
	\begin{proof}
		We will prove the claim by induction on the length of the segment: segments of length 1 are singletons, so the claim is trivial for them. Given a segment $b_0\leq_\bN b_1\leq_\bN...\leq_\bN b_m$ denote by $S$ the corresponding full subcategory of $\stn^\el_{n,/x}$. Pick some element $b_r$ of maximal dimension $l$ among $b_i$, so that $\dim(b_{r-1})< l=\dim(b_r)$ and $\dim(b_{r+1})<l$, denote by $S_-$ (resp. $S_+$) the subcategory of $S$ containing $b_i$ with $i<r$ (resp. $i>r$) and denote $S'_\pm\bydef S_\pm\cap d^\pm_l[b_r]$, in that case $S$ is isomorphic to the colimit of
		\begin{equation}\label{eq:S_colim}
			\begin{tikzcd}[sep=huge]
				& {S'_-} && {*} && {S'_+} \\
				{S_-} && {(S'_-)^\triangleright} && {(S'_+)^\triangleleft} && {S_+}
				\arrow[from=1-2, to=2-1]
				\arrow[from=1-2, to=2-3]
				\arrow[from=1-4, to=2-3]
				\arrow[from=1-4, to=2-5]
				\arrow[from=1-6, to=2-5]
				\arrow[from=1-6, to=2-7]
			\end{tikzcd},
		\end{equation}
		so it remains to show that all the terms in \eqref{eq:S_colim} are contractible. This is obvious for $*$, $(S'_-)^\triangleright$ and $(S'_+)^\triangleleft$ and follows by induction on length for $S_\pm$, so it remains to show that $S'_\pm$ are contractible, however note that they define segments in $d^\pm_l[b_r]\in\stn_n$, so they are again contractible by induction.
	\end{proof}
	\begin{prop}\label{prop:stn_sat}
		$\stn_n$ is a saturated algebraic pattern in the sense of \cite{chu2019homotopy}.
	\end{prop}
	\begin{proof}
		$\stn_n$ admits a factorization system by \Cref{prop:stein_fact} and the notion of elementary objects, so it suffices to prove the following claims:
		\begin{enumerate}
			\item\label{it:sat1} denote by $\Act_{\stn_n}(x)$ the set of active morphisms $x\xactive{a}y$, then 
			\[\Act_{\stn_n}(x)\cong\underset{(e\xinert{i}x)\in\stn^\el_{n,/x}}{\lim}\Act_{\stn_n}(e);\]
			\item\label{it:sat2} given an active morphism $a:x\xactive{}y$ denote by $C$ the category of pairs $(e_0\xinert{i_0}x,e\xinert{i}a_!e_0)$, where $a_!e_0$ appears in the factorization square
			\begin{equation}\label{eq:fact_sq}
				\begin{tikzcd}[sep=huge]
					{e_0} & {a_!e_0} \\
					x & y
					\arrow["{i^*a}",two heads, from=1-1, to=1-2]
					\arrow["{i_0}"', tail, from=1-1, to=2-1]
					\arrow["{a_! i_0}", tail, from=1-2, to=2-2]
					\arrow["a", two heads, from=2-1, to=2-2]
				\end{tikzcd},
			\end{equation}
			then the fibers of the natural morphism $m:C\rightarrow \stn^\el_{n,/y}$ sending the pair above to the composition $e\xinert{i}a_!e_0\xinert{a_!i_0}y$ are contractible;
			\item\label{it:sat3} every $x\in\stn_n$ admits an active morphism $a:e\xactive{}x$ with source in $\stn^\el_n$;
			\item\label{it:sat4} for every $x\in\stn_n$ we have
			\[x\cong \underset{(e\xinert{i}x)\in\stn^\el_{n,/x}}{\colim}e.\]
		\end{enumerate}
		We will now prove the claim in order:
		\begin{enumerate}
			\item Note that there is a morphism 
			\[F:\Act_{\stn_n}(x)\rightarrow\underset{(e\xinert{i}x)\in\stn^\el_{n,/x}}{\lim}\Act_{\stn_n}(e)\]
			sending $x\xactive{a}y$ to the system of $i^*a: e\xactive{}a_!e$ for $i:e\xinert{}x$ in the notation of \eqref{eq:fact_sq}, we need to define an inverse to $F$. Assume that we have a compatible system $a_e:e\xactive{}y_e$ over all $i_e:e\xinert{}x$, then define the complex $y$ such that the basis of $y_k$ consists of pairs $(c,b)$, where $c$ is a basis element of $x$ corresponding to some $i:[c]\xinert{}x$ and $b\in (y_{[c]})_k$, modulo an equivalence relation that identifies $(c',b)$ with $(c'', i_0(b))$ for any $i:[c']\xinert{}[c'']$, where $i_0:y_{[c']}\xinert{}y_{[c'']}$ is the morphism induced by $i$. We define the structure of an ADC on $y$ by
			\[\partial(c,b)\bydef (c,\partial b).\]
			It is clear that this defines an ADC with unital basis, moreover it clearly admits an active morphism from $x$ that sends $c$ as above to $\sum_{b\in (y_{[c]})_{\dim(c)}} (c,b)$, so it remains to prove the basis is strongly loop-free. Assume now that we have a loop $(c_0,n_0)\leq_\bN(c_1,b_1)\leq_\bN...\leq_\bN (c_0,b_0)$, then by construction $c_0\leq_\bN c_1\leq_\bN...\leq_\bN c_0$ is a loop in $x$, hence it must be trivial since $x\in\stn_n$, which means that the original loop is entirely contained within some $y_{[c]}$, so it too must be trivial since $y_{[c]}\in\stn_n$.
			\item Given $i:e\xinert{}y$ corresponding to the basis element $b\in y_l$ denote by $C_b$ the fiber of $m$ over it, by definition it consists of $e_0\xinert{i_0}x$ such that $i$ factors through $a_! e_0$. Note that $C_b$ is non-empty since $a$ was assumed to be active, assume first that $a^{-1}=\varnothing$, this means that there is a basis element $c\in x_k$ such that $b$ belongs to the subcomplex $[a(c)]$ generated by elements in the image of $c$. There exists a minimal such $c$ in $x$, which defines the initial object of $C_b$, which proves that it is contractible. Assume now that $a^{-1}(b)$ is non-empty, then $a^{-1}(b)$ is coinitial in $C_b$, hence $|C_b|\cong |a^{-1}(b)|$, but the latter is a segment for $\leq_\bN$ by \Cref{lem:seq} and those are contractible by \Cref{lem:lin_cont}.
			\item This follows immediately from the definition since every $x\in\stn_n$ admits an active morphism $c_n\xactive{a}x$.
			\item This can either be seen directly or deduced from the freeness property of \cite{ara2023categorical}.
		\end{enumerate}
	\end{proof}
	\begin{cor}\label{cor:rel_seg}
		Given an active morphism $a:x\xactive{}y$ in $\stn_n$ and $i:e\xinert{}x$ in $\Theta^\el_{n,/x}$, denote by $a_!e\xinert{}y$ its image under $a$, then we have an equivalence
		\begin{equation}\label{eq:rel_seg}
			y\cong \underset{(e\xinert{i}x)\in\stn^\el_{n,/x}}{\colim}a_!e
		\end{equation}
		in $\stn_n$.
	\end{cor}
	\begin{proof}
		It follows from \eqref{it:sat4} that for every $z\in\stn_n$ we have
		\begin{equation}\label{eq:theta_seg}
			z\cong \underset{(e\xinert{i}z)\in\stn^\el_{n,/z}}{\colim}e.
		\end{equation}
		Applying this to the right-hand side of \eqref{eq:rel_seg} we get
		\begin{equation}\label{eq:seg_string}
			\underset{e\xinert{i}x}{\colim}\;a_!e\cong \underset{e\xinert{i}x}{\colim}\underset{e'\xinert{i'}a_!e}{\colim}\;e'\cong \underset{e'\xinert{i'}a_!e\xinert{}y}{\colim}\;e',
		\end{equation}
		where the colimit on the right is taken over the category $C$ of pairs $(e\xinert{i}x, e'\xinert{i'}a_!e)$. It then follows from \eqref{it:sat2} that the colimit in \eqref{eq:seg_string} can be identified with $\underset{e\xinert{}y}{\colim}\;y$, which is isomorphic to $y$ by \eqref{it:sat4}.
	\end{proof}
	\begin{cor}\label{cor:rel_seg_theta}
		Given an active morphism $a:\theta\xactive{}\theta'$ and $i:e\xinert{}\theta$ in $\Theta^\el_{n,/\theta}$, denote by $\ima_a(i)\xinert{}\theta'$ its image under $a$, then we have an equivalence
		\begin{equation}\label{eq:rel_seg_theta}
			\theta'\cong \underset{e\xinert{i}\theta}{\colim}\;\ima_a(i)
		\end{equation}
		in $\cat_n$.
	\end{cor}
	\begin{proof}
		That \eqref{eq:rel_seg_theta} is a colimit in $\Theta_n$ follows immediately from \Cref{cor:rel_seg} since $\Theta$ is a full subcategory of $\stn_n$ by \Cref{rem:theta}, that it remains a colimit in $\cat_n$ follows since the isomorphism
		\[\theta\cong\underset{e\xinert{}\theta}{\colim}\;e\]
		is preserved by the inclusion $\Theta_n\hookrightarrow\cat_n$ -- this is essentially a reformulation of the Segal condition.
	\end{proof}
	\section{An alternative model for $\stab(\cat_{n,/\cE})$}\label{sect:twar_theta}
	In \Cref{prop:seg} we have identified the stabilization of $\cat_{n,/\cE}$ for $\cE\in\cat_n$ with the category of spectrum-valued Segal objects for a certain algebraic pattern, however later in the same section we have proved in \Cref{thm:main_theta} that $\stab(\cat_{n,/\theta})$ is isomorphic to the category of spectrum-valued presheaves. This fact is not exclusive to the objects of $\Theta_n$ -- in this section we will construct a category $\twar_\theta(\cE)$ such that
	\[\mor_\cat(\twar_\theta(\cE),\spc)\cong \stab(\cat_{n,/\cE}),\]
	more specifically this result is contained in \Cref{prop:main_theta_stab}.\par 
	Unfortunately, the model of $\twar_\theta(\cE)$ presented here does not allow for a direct computation of its space of objects and morphisms since it is presented as a localization $\cL_\cat(\twar^C_\theta(\cE))$ of a certain simplicial space $\twar^C_\theta(\cE)$, where $\cL_\cat:\psh_\cS(\Theta_n)\rightarrow \cat_n$ is the left adjoint to inclusion. The reason we ultimately need this intermediate model is \Cref{lem:twar_theta1} which allows us to express it in terms of $\twar_\theta(\theta')$ for $\theta'\in\Theta_n$ -- the equivalent result of course holds for the category $\twar(\cE)$ defined in \Cref{sect:twar}, however we do not know how to prove it directly except by constructing an isomorphism $\twar_\theta(\cE)\cong \twar(\cE)$.\par 
	In the first part of this section we define $\twar_\theta(\cE)$ and prove some auxiliary results culminating in the isomorphism \Cref{prop:main_theta_stab}, the second part is dedicated to the computation of $\twar_\theta(x)$ for $x\in\Stn_n$ in \Cref{prop:stn_inrt}, which will later be used in \Cref{sect:twar}. Finally, we use the results obtained so far to prove the first major theorem of this paper -- \Cref{thm:stein}. We also note in passing that a generalization of \Cref{cor:stein_push} to torsion-free complexes has been obtained in \cite{campion2023infty} using different methods.
	\begin{construction}\label{constr:theta_twar}
		Note that for any morphism $f:[q]\rightarrow[p]$ we have an induced functor $f_*:\Delta^\inrt_{[q]}\rightarrow\Delta^\inrt_{/[p]}$ sending $[l]\xinert{i}[q]$ to the inert part of $f\circ i$, this makes $[q]\mapsto\Delta^\inrt_{/[q]}$ into a functor $\Delta\xrightarrow{\Delta^\inrt_{/-}}\cat$. For any category $\cC$ we then get a functor $\mor_\cat(\Delta^\inrt_{/-},\cC):\Delta^\op\rightarrow\cat$.\par
		Given $\cE\in\cat_n$ denote by $\twar'_\theta(\cE)$ the subfunctor of 
		\[\Delta^\op\xrightarrow{\mor_\cat(\Delta^\inrt_{/-},\Theta_{n,/\cE})} \cat\]
		sending $[q]$ to the subcategory of $\mor_\cat(\Delta^\inrt_{/[q]},\Theta_{n,/\cE})$ such that:
		\begin{enumerate}
			\item for $i:[l]\xinert{}[q]$ denote by $\theta_i$ the source of $F(i)\in\Theta_{n,/\cE}$, then for $k\in[q]$ we have $\theta_{\{k\}}\cong c_n$, where $[0]\xinert{\{k\}}[q]$ is the inclusion of the element $\{k\}$;
			\item given any morphism $[l]\xinert{i'}[m]$ between $i_0:[l]\xinert{}[q]$ and $i_1:[m]\xinert{}[q]$, if $i'$ preserves the minimal element, the corresponding morphism $h_{i'}:\theta_{i_0}\rightarrow \theta_{i_1}$ is inert, and if $i'$ preserves the maximal element, then the morphism $h_{i'}$ is an active morphism;
			\item any natural transformation $\alpha:F\rightarrow G$ in $\mor_\cat([1]\times \Delta^\inrt_{/[q]},\Theta_{n,/\cE})$ that lies in $\twar'_\theta(\cE)$ satisfies $\alpha_{[0]\xinert{\{k\}}[q]}\cong \id$ for all $k\in[q]$.
		\end{enumerate}
		Note that this is indeed a subfunctor since both inert and active morphisms are closed under composition. Finally, denote
		\[\twar_\theta(\cE)\bydef \cL_\cat(|\twar'_\theta(\cE)|),\]
		where $|-|:\cat\rightarrow\cS$ is the functor of geometric realization and $\cL_\cat:\psh_\cS(\Delta)\rightarrow \cat$ is the left adjoint to the natural inclusion.
	\end{construction}
	\begin{lemma}\label{lem:twar_theta1}
		We have 
		\[\twar_\theta(\cE)\cong \underset{\theta\xrightarrow{f}\cE}{\colim}\twar_\theta(\theta).\]
	\end{lemma}
	\begin{proof}
		Since $|-|$ and $\cL_\cat$ preserve colimits it suffices to show that 
		\[\twar'_\theta(\cE)\cong \underset{\theta\xrightarrow{f}\cE}{\colim}\twar'_\theta(\theta).\]
		To prove this it suffices to show that for $[m]\times\Delta^\inrt_{/[q]}\xrightarrow{F}\Theta_{n,/\cE}$ in $\twar'_\theta(\cE)([q])([m])$ the space of factorizations 
		\[\Delta^\inrt_{/[q]}\times[m]\xrightarrow{F'}\Theta_{n,/\theta}\xrightarrow{f_!}\Theta_{n,/\cE}\]
		of $F$ for $x\in\stn_n$ is contractible, however note that it in fact has an initial object given by 
		\[\Delta^\inrt_{/[q]}\times[m]\xrightarrow{F'}\Theta_{n,/s(F([q]\eq[q],m))}\xrightarrow{F([q]\eq[q],m)_!}\Theta_{n,/\cE},\]
		where $s(F([q]\eq[q]))$ denotes the source of the corresponding object of $\Theta_{n,/\cE}$ viewed as a morphism in $\cat_n$.
	\end{proof}
	\begin{construction}\label{constr:twar_C}
		We will define certain objects $C_n(k)\in \Theta_n$ by induction on $n$ and $k$: we set $C_n(0)\bydef c_n$ for all $n$, for $k=n=1$ we set $C_1(1)\bydef [3]\in \Delta$, also denote by $i_1(1): c_1\xinert{}C_1(1)$ the inclusion of thew subinterval $[1,2]$. Assume we have defined $i_{n-1}(1): c_{n-1}\xinert{} C_{n-1}(1)$, define $C_n(1)$ to be the object of $\Theta_n$ with $\obj(C_n(1))\bydef \{0,1,2,3\}$ such that $\mor_{C_n(1)}(1,2)\bydef C_{n-1}(1)$ and $\mor_{C_n(1)}(i,i+1)\bydef c_n$ for $i\in \{0,2\}$, define $i_n(1):c_n\xinert{} C_n(1)$ to be the morphism sending the object $i$ to $(i+1)$ for $i\in \{0,1\}$ and restricting to $i_{n-1}(1):c_{n-1}\xinert{} C_{n-1}(1)$ on $c_{n-1}\cong \mor_{c_n}(0,1)$. Finally, assume we have defined $i_{n}(k-1): c_n\xinert{} C_n(k-1)$, we define $C_n(k)$ to be the pushout
		\begin{equation}\label{eq:C_n}
			\begin{tikzcd}[sep=huge]
				& {c_n} & {C_n(k-1)} \\
				{c_n} & {C_n(1)} & {C_n(k)}
				\arrow["{i_{n}(k-1)}", tail, from=1-2, to=1-3]
				\arrow[two heads, from=1-2, to=2-2]
				\arrow[two heads, from=1-3, to=2-3]
				\arrow["{i_n(1)}", tail, from=2-1, to=2-2]
				\arrow["{i'}", tail, from=2-2, to=2-3]
				\arrow["\ulcorner"{anchor=center, pos=0.125, rotate=180}, draw=none, from=2-3, to=1-2]
			\end{tikzcd}
		\end{equation}
		taken in the category $\Theta_n$ (which exists by \Cref{lem:theta_push}), define $i_n(k)$ to be the composition $i'\circ i_n(1)$ in the notation of \eqref{eq:C_n}.\par
		Denote by $C^\el_q:\Delta^\el_{/[q]}\rightarrow \Theta_n$ the morphism sending $[0]\xinert{\{i\}}[q]$ to $c_n$, $[1]\xinert{i<i+1}[q]$ to $C_n(1)$, the inclusion $[0]\xinert{\{0\}}[1]$ over $[q]$ to the inert morphism $c_n\xinert{i_n(1)} C_n(1)$ and $[0]\xinert{\{1\}}[1]$ to the unique active morphism $c_n\xactive{} C_n(1)$. By construction there exists a left Kan extension $C^\inrt_q\bydef i^\el_! C_q^\el$ along $i^\el: \Delta^\el_{/[q]}\hookrightarrow \Delta^\inrt_{/[q]}$ that sends $[l]\xinert{i'}[q]$ to $C_n(l)$.
	\end{construction}
	\begin{remark}\label{rem:C_n}
		By untangling the definitions we see that $C_1(k) = [2k+1]$ and inductively $C_n(k)$ has objects given by the set $\{0,1,...2k+1\}$ such that $\mor_{C_n(k)}(i,i+1) = C_{n-1}(i)$ for $i\leq k$ and $\mor_{C_n(k)}(i',i'+1) = C_{n-1}(2k+1-i)$ for $i\geq k+1$. In particular, we see that the natural inclusion $c_n\xinert{i_n^k} C_n(k)$ admits a section $s_n^k: C_n(k)\xactive{} c_n$ that we also define by induction: for $n=1$ we set $s_1^k(i)\bydef 0$ if $i\leq k$ and $s_1^k(i')\bydef 1$ otherwise, for general $n$ we define $s_n^k$ to act by the same formula on objects and define the morphism
		\[C_{n-1}(k) = \mor_{C_n(k)}(k,k+1)\rightarrow c_{n-1} = \mor_{c_n}(0,1)\]
		to be $s^k_{n-1}$.
	\end{remark}
	\begin{prop}\label{prop:twar_C}
		$|\twar_\theta'(\cE)|([q])$ is isomorphic to the space of morphisms $\Delta^\inrt_{/[q]}\xrightarrow{F}\Theta_{n,/\cE}$ for which the composition $\Delta^\inrt_{/[q]}\xrightarrow{F}\Theta_{n,/\cE}\xrightarrow{p}\Theta_n$ with the forgetful functor is isomorphic to $C_q^\inrt$ in the notation of \Cref{constr:twar_C}.
	\end{prop}
	\begin{proof}
		Denote by $X([q])$ the space described in the statement of the proposition, note that there is a natural morphism $j_q:X([q])\rightarrow \twar'_\theta(\cE)([q])$, we need to prove that $X_q$ induces an isomorphism upon taking geometric realizations. We will in fact show that $j_q$ is coinitial, we will do so by induction on $q$ starting with the case $q=1$. The objects of $\twar'_\theta(\cE)([1])$ are given by cospans
		\begin{equation}\label{eq:theta_obj}
			\begin{tikzcd}[sep=huge]
				{c_n} & \theta & {c_n} \\
				& \cE
				\arrow["i", from=1-1, to=1-2]
				\arrow["f"{description}, from=1-1, to=2-2]
				\arrow["h"{description}, from=1-2, to=2-2]
				\arrow["a"', two heads, from=1-3, to=1-2]
				\arrow["g"{description}, from=1-3, to=2-2]
			\end{tikzcd}
		\end{equation}
		with morphisms given by 
		\begin{equation}\label{eq:theta_mor}
			\begin{tikzcd}[sep=huge]
				& {\theta'} \\
				{c_n} & \theta & {c_n} \\
				& \cE
				\arrow["b"{description}, two heads, from=1-2, to=2-2]
				\arrow["{i'}", tail, from=2-1, to=1-2]
				\arrow["i", tail, from=2-1, to=2-2]
				\arrow["f"{description}, from=2-1, to=3-2]
				\arrow["h"{description}, from=2-2, to=3-2]
				\arrow["{a'}"', two heads, from=2-3, to=1-2]
				\arrow["a"', two heads, from=2-3, to=2-2]
				\arrow["g"{description}, from=2-3, to=3-2]
			\end{tikzcd},
		\end{equation}
		note that the morphism $b$ in \eqref{eq:theta_mor} is automatically active. We need to show that for any cospan
		\begin{equation}\label{eq:theta_cosp}
			c_n\xinert{i}\theta\overset{a}{\twoheadleftarrow}c_n
		\end{equation}
		over $\cE$ there is a unique morphism of the form 
		\[\begin{tikzcd}[sep=huge]
			& {C_n(1)} \\
			{c_n} & \theta & {c_n}
			\arrow["{b_0}"{description}, two heads, from=1-2, to=2-2]
			\arrow["{i_n(1)}", tail, from=2-1, to=1-2]
			\arrow["i", tail, from=2-1, to=2-2]
			\arrow[two heads, from=2-3, to=1-2]
			\arrow["a"', two heads, from=2-3, to=2-2]
		\end{tikzcd},\]
		we will prove this by induction on $n$. For $n=1$ we are given a cospan $[1]\xinert{i}[m]\overset{a}{\twoheadleftarrow}[1]$ and we need to define an active morphism $a_0:[3]\xactive{}[m]$ such that $a_0\circ i_1(1)\cong i$, clearly there is a unique such morphism given by $a_0(1)\bydef i(0)$, $a_0(2)\bydef i(1)$. Assume we have constructed it in dimensions $\leq (n-1)$ and we are given \eqref{eq:theta_cosp}, we need to construct a unique morphism $C_n(1)\xactive{a_0}\theta$ such that $a_0\circ i_n(1)\cong i$. Suppose that $\obj(\theta) = \{0,1,...,m\}$ and denote $\theta'\bydef \mor_\theta(i(0), i(1))\in\Theta_{n-1}$, note that $i$ induces an inert morphism $i':c_{n-1}\xinert {}\theta'$. Using $a_0\circ i_n(1)\cong i$ and the fact that $a_0$ is active we see that 
		\begin{align*}
			&a_0(0) = 0,\\
			&a_0(1) = i(0),\\
			&a_0(2) = i(1),\\
			&a_0(3) = m,
		\end{align*}
		the morphisms $c_{n-1}\cong \mor_{C_n(1)}(0,1)\xactive{}\mor_\theta(0,i(0))$ and $c_{n-1}\cong \mor_{C_n(1)}(2,3)\xactive{}\mor_\theta(i(1),m)$ are the unique active morphisms and $C_{n-1}(1)\cong \mor_{C_n(1)}(1,2)\xactive{}\theta'\cong \mor_\theta(i(0),i(1))$ is the unique morphism corresponding to the cospan $c_{n-1}\xinert{i'}\theta'\overset{a}{\twoheadleftarrow}c_{n-1}$, we see that these conditions uniquely define $a_0$.\par
		Now, assume we have proved the claim for $(q-1)$ and we are given a morphism $F:\Delta^\inrt_{/[q]}\rightarrow\Theta_{n,/\cE}$, by induction we may assume that we have unique morphisms $\alpha_i: C_n(l)\xactive{}F([l]\xinert{i}[q])$ for all $i\neq \id_{[q]}$ defining a natural transformation between functors from $\Delta^\inrt_{/[q]}\setminus\{\id_{[q]}\}$, we need to define a morphism $\alpha_\id: C_n(q)\xactive{} F(\id_{[q]})$. Denote $\theta_0\bydef F([q-1]\xinert{\delta_q}[q])$, $\theta_1\bydef F([1]\xinert{q-1<1}[q])$ and $\theta\bydef F(\id_{[q]})$ so that we have a commutative square 
		\[\begin{tikzcd}[sep=huge]
			{c_n} & {\theta_0} \\
			{\theta_1} & \theta
			\arrow[two heads, from=1-1, to=1-2]
			\arrow["{i'_0}"', tail, from=1-1, to=2-1]
			\arrow["{i_0}", tail, from=1-2, to=2-2]
			\arrow["{a_1}"', two heads, from=2-1, to=2-2]
		\end{tikzcd},\]
		then by induction we have all solid arrows in the commutative diagram 
		\[\begin{tikzcd}[sep=huge]
			{c_n} & {C_n(q-1)} & {\theta_0} \\
			{C_n(1)} & {C_n(q)} \\
			{\theta_1} && \theta
			\arrow[two heads, from=1-1, to=1-2]
			\arrow["{i_{n}(q-1)}"', tail, from=1-1, to=2-1]
			\arrow["{a''_0}", two heads, from=1-2, to=1-3]
			\arrow[from=1-2, to=2-2]
			\arrow["{i_0}", tail, from=1-3, to=3-3]
			\arrow[from=2-1, to=2-2]
			\arrow["{a'_0}"', from=2-1, to=3-1]
			\arrow["\ulcorner"{anchor=center, pos=0.125, rotate=180}, draw=none, from=2-2, to=1-1]
			\arrow[dashed, from=2-2, to=3-3]
			\arrow["{a_1}"', two heads, from=3-1, to=3-3]
		\end{tikzcd},\]
		we can then uniquely fill in the dashed arrow by the universal property of the pushout \eqref{eq:C_n}.
	\end{proof}
	\begin{notation}
		Denote by $\twar^C_\theta(\cE)$ the functor $\Delta^\op\rightarrow\cS$ described in \Cref{prop:twar_C}, so that we have $\twar_\theta(\cE)\cong \cL_\cat(\twar^C_\theta(\cE))$. 
	\end{notation}
	\begin{construction}\label{constr:C_delta}
		We will define a functor $(k_0,...,k_{n-1})\mapsto C_n(k_0,...,k_{n-1})$ from $\Delta^n$ to $\Theta_n$, we will do so by induction on $n$: for $n=1$ we define $C_1(k)$ to be the object $C_1(k)\cong [2k+1]\in \Delta$ of \Cref{constr:twar_C}, given a morphism $f:[k]\rightarrow[l]$ we set $C_1(f)(i) \bydef f(i)$ for $i\leq k$ and $C_1(f)(i')\bydef 2l+1-f(i')$ for $i'\geq k+1$. Note that $C_1(0) = c_1$.\par
		Assume we have defined $C_{n-1}(k_0,...,k_{n-2})$ and moreover that $C_{n-1}(0,...,0) = c_{n-1}$, we first describe the value of $C_n(-)$ on objects: define $C_n(k_0,...,k_{n-1})$ to be the category with objects $\{0,1,...,2k_0+1\}$ such that $\mor_{C_n(k_0,...,k_{n-1})}(i,i+1) \bydef c_n$ if $i\neq k_0$ and $\mor_{C_n(k_0,...,k_{n-1})}(k,k+1) \bydef C_{n-1}(k_1,...,k_{n-1})$. Given a morphism $f = (f_0,...,f_{n-1})$ such that $f_i:[k_i]\rightarrow [l_i]$ we define $C_n(f)$ on objects by $C_n(f)(i) \bydef f_0(i)$ for $i\leq k$ and $C_n(f)(i')\bydef 2l+1-f_0(i')$ for $i'\geq k+1$, on morphism $(n-1)$-categories we define the morphism
		\[\mor_{C_n(k_0,...,k_{n-1})}(i,i+1) = c_n\rightarrow \prod_{f_0(i)<j\leq f_0(i+1)}c_n = \mor_{C_n(l_0,...,l_{n-1})}(f(i),f(i+1))\]
		for $i\neq k$ to be the diagonal morphism (where the indexing set for the product might be empty) and 
		\begin{align*}
			&\mor_{C_n(k_0,...,k_{n-1})}(i,i+1) = C_{n-1}(k_1,...,k_{n-1})\rightarrow\\ \rightarrow&\prod_{f_0(k)<j\leq l}c_n\times  C_{n-1}(l_1,...,l_{n-1})\times \prod_{f_0(k)<j\leq l}c_n = \mor_{C_n(l_0,...,l_{n-1})}(f(k),f(k+1))
		\end{align*}
		\[\]
		to have components $C_{n-1}(f_1,...,f_{n-1}): C_{n-1}(k_1,...,k_{n-1})\rightarrow C_{n-1}(l_1,...,l_{n-1})$ and 
		\[C_{n-1}(s,...,s):C_{n-1}(k_1,...,k_{n-1})\rightarrow C_{n-1}(0,...,0) = c_{n-1}\]
		(using the inductive assumption on $C_{n-1}(0,...,0)$), where $s:[k_i]\rightarrow[0]$ is the unique morphism. The functoriality of this construction easily follows from the functoriality of $C_{n-1}(f_1,...,f_{n-1})$. Finally, note that by construction $C_n(0,...,0)$ is the category with objects $\{0,1\}$ such that $\mor_{C_n(0,...,0)}(0,1) = c_{n-1}$, i.e. $C_n(0,...,0) = c_n$, meaning that all the inductive assumptions are verified for $n$.\par
		Finally, we will define certain morphisms $S_n^i(k):C_n(k)\rightarrow C_n^i(k)\bydef C_n(0,...,k,...,0)$, where on the right the only nontrivial value $k$ is in position $i$ for $0\leq i\leq n-1$. We prove the claim by induction on $n$: for the case $n=1$ we define $S_1^0(k)\bydef \id_{[2k+1]}$. Assume we have defined the required morphism up to dimension $(n-1)$: if $i=0$ then it is easy to see from the definitions that $C^0_n(k)$ is a category with objects $\{0,1,...,2k+1\}$ such that $\mor_{C^0_n(k)}(i,i+1) = c_{n-1}$, using the explicit description of $C_n(1)$ in \Cref{rem:C_n} we define $S_n^0(1)$ to be identity on objects and restrict to $s^j_{n-1}: C_{n-1}(j)\rightarrow c_{n-1}$ on morphism categories, where the morphism $s^j_{n-1}$ is also defined in \Cref{rem:C_n}. For general $i>0$ the category $C_n^i(k)$ has object $\{0,1\}$ such that $\mor_{C_n^i(k)}(0,1) = C_{n-1}^{i-1}(k)$ and we define $S_n^i(k)$ to sends objects $j\leq k$ to 0, $j\geq (k+1)$ to 1 and to restrict to act by
		\[\mor_{C_n(k)}(k,k+1)\cong C_{n-1}(k)\xrightarrow{S^{i-1}_{n-1}(k)}C^{i-1}_{n-1}(k)\cong \mor_{C_n^i(k)}(k,k+1)\]
		on the morphism category.\par
		Finally, we can define $C_q^{\inrt,i}:\Delta^\inrt_{/[q]}\rightarrow\Theta_n$ by sending $[l]\xinert{i}[q]$ to $C_n^i(l)$ and a morphism $[l]\xinert{j'}[m]$ over $[q]$ to the induced morphism $C_n^i(j'): C_n^i(l)\xinert{} C_n^i(m)$. It is easy to see with these notations that the morphisms $S_n^i(k)$ defined above define a natural transformation $S^i_n: C_q^\inrt\rightarrow C_q^{\inrt,i}$ of functors in $\mor_\cat(\Delta^\inrt_{/[q]},\Theta_n)$. Denote by $\twar^i_\theta(\cE)$ for $\cE\in\cat_n$ the functor $\Delta^\op\rightarrow \cS$ sending $[q]$ to the space of functors $\Delta^\inrt_{/[q]}\xrightarrow{F}\Theta_{n,/\cE}$ for which the composition $\Delta^\inrt_{/[q]}\xrightarrow{F}\Theta_{n,/\cE}\xrightarrow{p}\Theta_n$ with the forgetful functor is isomorphic to $C_q^{\inrt,i}$, then precomposing with $S^i_n$ defines a natural transformation $\twar^i_\theta(\cE)\rightarrow \twar^C_\theta(\cE)$, we will often implicitly use it to identify $\twar^i_\theta(\cE)$ with a simplicial subspace of $\twar^C_\theta(\cE)$.
	\end{construction}
	\begin{lemma}\label{lem:C_push}
		Assume we have a collection $(k_0,...,k_l,0,...,0)$ of non-negative integers such that $k_j = 0$ for $j>l$ and a collection $(0,...,0,h_l,...,h_{n-1})$ such that $h_i = 0$ for $i<l$, then we have a pushout diagram 
		\begin{equation}\label{eq:C_push}
			\begin{tikzcd}[sep=huge]
				{c_n\cong C_n(0,...,0)} & {C_n(0,...,h_l,...,h_{n-1})} \\
				{C_n(k_0,...,k_l,...,0)} & {C_n(k_0,...,k_{l-1},k_l+h_l,h_{l+1},...,h_{n-1})}
				\arrow["{C_n(i'_1)}", two heads, from=1-1, to=1-2]
				\arrow["{C_n(i_0)}"', tail, from=1-1, to=2-1]
				\arrow["{C_n(i'_0)}", tail, from=1-2, to=2-2]
				\arrow["{C_n(i_1)}"', two heads, from=2-1, to=2-2]
				\arrow["\ulcorner"{anchor=center, pos=0.125, rotate=225}, draw=none, from=2-2, to=1-1]
			\end{tikzcd}
		\end{equation}
		in $\cat_n$, where $i_0$ and $i'_0$ correspond to the unique collections of inert maximal-element-preserving morphisms while $i_1$ and $i'_1$ to the collections of inert minimal-element-preserving morphisms.
	\end{lemma}
	\begin{proof}
		Denote by $F_1:\Theta^\el_{n,/C(k_0,...,0)}\rightarrow \cat_n$ the functor sending $c_s\xinert{i}C(k_0,...,0)$ to $c_s$, by $F_0:\Theta^\el_{n,/C(k_0,...,0)}\rightarrow \cat_n$ the functor sending all $c_s\xinert{i}C_n(k_0,...,0)$ to $\varnothing$ except for $c_n\xinert{C_n(i_0)}C_n(k_0,...,0)$ which it sends to $c_n$ and by $F'_1:\Theta^\el_{n,/C(k_0,...,0)}\rightarrow \cat_n$ the functor sending all $c_s\xinert{i}C_n(k_0,...,0)$ to $\varnothing$ except for $c_n\xinert{C_n(i_0)}C_n(k_0,...,0)$ which it sends to $C_n(0,...,h_{n-1})$. In that case we have a cospan $F_1\xleftarrow{f_0}F_0\xrightarrow{f_1}F_1'$ of natural transformations, where the only non-trivial component of $f_0$ is $C_n(i_0)$ and the only non-trivial component of $f_1$ is $C_n(i'_1)$, moreover it is easy to see that upon taking colimits this cospan induces the upper left corner of \eqref{eq:C_push}. Using the commutativity of colimits, we see that the pushout of \eqref{eq:C_push} is given by 
		\[\underset{(c_s\xinert{i}C_n(k_0,...,0))\in\Theta^\el_{n,/C_n(k_0,...,0)}}{\colim}G(i),\] 
		where $G:\Theta^\el_{n,/C_n(k_0,...,0)}\rightarrow \cat_n$ is the functor sending all $c_s\xinert{i}C_n(k_0,...,0)$ to $c_s$ except for $c_n\xinert{C_n(i_0)}C_n(k_0,...,0)$ which it sends to $C_n(0,...,h_{n-1})$. However by direct inspection we see that $G$ coincides with the functor $i\mapsto \ima_{C_n(i_1)}(i)$, so the claim now follows from \Cref{cor:rel_seg_theta}.
	\end{proof}
	\begin{cor}\label{cor:i_seg}
		The simplicial spaces $\twar^i_\theta(\cE)$ for $0\leq i<n$ described in \Cref{constr:C_delta} satisfy the Segal condition.
	\end{cor}
	\begin{proof}
		By construction we may identify $\twar^i_\theta(\cE)([q])$ with the space of natural transformations $\alpha:F_0\rightarrow F_1$ in $\mor_\cat(\Delta^\inrt_{/[q]},\cat_n)$, where $F_0$ is the functor sending $[l]\xinert{}[q]$ to $C_n^i(l)$ and $F_1$ is the constant functor with value $\cE$. It follows from \Cref{lem:C_push} that both of those functors are left Kan extensions of their restriction along $i^\el:\Delta^\el_{/[q]}\hookrightarrow\Delta^\inrt_{/[q]}$, so we have
		\[\mor_{\mor_\cat(\Delta^\inrt_{/[q]},\cat_n)}(F_0,F_1)\cong \mor_{\mor_\cat(\Delta^\inrt_{/[q]},\cat_n)}(i^\el_!F'_0,i^\el_!F'_1)\cong\mor_{\mor_\cat(\Delta^\el_{/[q]},\cat_n)}(F'_0,F'_1) ,\]
		where $F'_i$ denotes the restriction of $F_i$ along $i^\el$. The claim now follows since
		\begin{equation}\label{eq:delta_col}
			\Delta^\el_{/[q]}\cong \underset{([e]\xinert{i}[q])\in\Delta^\el_{/[q]}}{\colim}\Delta^\el_{/[e]}.
		\end{equation}
	\end{proof}
	\begin{remark}
		Given $(\cE:\Theta_n^\op\rightarrow\cS)\in\cat_n$ we can consider a category $\cE_\cat$ given by pullback along $P:\Delta\rightarrow\Theta_n$ which sends $[n]$ to $P([n])$ such that $\obj(P([n])) = \{0,1,...,n\}$ and $\mor_{P([n])}(i,i+1) = c_{n-1}$. More explicitly, $\cE_\cat$ can be described as the category with the same space of objects as $\cE$ and with morphisms given by $n$-morphisms of $\cE$ with composition given by 0-composition of $n$-morphisms. With this notation it is easy to see that $\twar^0_\theta(\cE)\cong \twar(\cE_\cat)$, where on the right $\twar(-)$ denotes the ordinary twisted arrows category. More generally, given $f:c_n\rightarrow \cE$ and $0<l<n$ we can view $f$ as an $(n-l)$-morphism in an $(n-l)$-category 
		\[\cE(d^-_{l-1}f, d^+_{l-1}f)\bydef \mor_\cE(d^-_{l-1}f, d^+_{l-1}f),\]
		then $\twar^l_\theta(\cE)$ is the category with objects given by $n$-morphisms $f:c_n\rightarrow\cE$ and morphisms given by morphisms in $\twar(\cE(d^-_{l-1}f, d^+_{l-1}f)_\cat)$ (so that the space of morphisms between $f$ and $g$ is non-empty only if $d^\pm_{l-1}f\cong d^\pm_{l-1}g$).
	\end{remark}
	\begin{lemma}\label{lem:theta_twar}
		For $\theta_0\in\Theta$ we have 
		\begin{equation}\label{eq:theta_twar}
			\twar_\theta^C(\theta_0)\cong \Theta^\inrt_{n,/\theta_0}.
		\end{equation}
	\end{lemma}
	\begin{proof}
		We first note that $\twar^C_\theta(\theta_0)$ satisfies the Segal condition -- this follows by the same argument as in the proof of \Cref{cor:i_seg}: $\twar_\theta^C(\theta_0)([q])$ is the space of natural transformations $\alpha:F_0\rightarrow F_1$ between functors $F_i:\Delta^\inrt_{/[q]}\rightarrow\Theta_n$, where $F_1$ is the constant functor with value $\theta_0$ and $F_0:([l]\xinert{i}[q])\mapsto C_n(l)$, note that we have $F_i \cong i^\el_! i^{\el,*}F_i$ for $i\in\{0,1\}$ (where the left Kan extension is taken with respect to the category $\Theta_n$), so the claim again follows by \eqref{eq:delta_col}. To complete the proof it now suffices to show that the spaces of morphisms in $\twar^C_\theta(\theta_0)$ are contractible, i.e. that for any $i':\theta'\xinert{}\theta_0$ and $i'':\theta''\xinert{}\theta_0$ such that $i'$ factors through $i''$ there is a unique diagram of the form 
		\[\begin{tikzcd}[sep=huge]
			{c_n} & {C_n(1)} & {c_n} \\
			{\theta'} & {\theta_0} & {\theta''}
			\arrow["{i_n(1)}", tail, from=1-1, to=1-2]
			\arrow[two heads, from=1-1, to=2-1]
			\arrow[from=1-2, to=2-2]
			\arrow[two heads, from=1-3, to=1-2]
			\arrow[two heads, from=1-3, to=2-3]
			\arrow["{i'}"', tail, from=2-1, to=2-2]
			\arrow["{i''}", tail, from=2-3, to=2-2]
		\end{tikzcd},\]
		which follows from \Cref{lem:C_cont}.
	\end{proof}
	\begin{prop}\label{prop:main_theta_stab}
		There is an isomorphism
		\[\mor_\cat(\twar_\theta(\cE),\spc)\cong \stab(\cat_{n,/\cE}).\]
	\end{prop}
	\begin{proof}
		Denote by $F:\cat^\op_n\rightarrow\cat$ the functor $\cE\mapsto \stab(\cat_{n,/\cE})$ and by $G:\cat_n^\op\rightarrow\cat$ the functor $\cE\mapsto \mor_\cat(\twar_\theta(\cE)\spc)$, it follows from \Cref{lem:globn} for $F$ and \Cref{lem:twar_theta1} for $G$ that both of those functors are the right Kan extensions of their restriction to $\Theta_n$, so it suffices to construct a natural transformation $\beta:G|_{\Theta_n}\rightarrow F|_{\Theta_n}$. It follows from \Cref{lem:theta_twar} that the functor $G|_\theta$ is equivalent to the functor $\theta\mapsto \mor_\cat(\Theta^\inrt_{n,/\theta},\spc)$ and from \Cref{prop:theta_comp} that $F|_\theta$ is isomorphic to the functor $\theta\mapsto\seg_\spc(\Theta^\inj_{n,/\theta})$, so it suffices to extend the isomorphism
		\begin{equation}\label{eq:theta_inj}
			\mor_\cat(\Theta^\inrt_{n,/\theta},\spc)\cong \seg_\spc(\Theta^\inj_{n,/\theta})
		\end{equation}
		of \Cref{prop:seg_inj} to a natural transformation of functors. Denote by $\beta_\theta:\Theta^\inj_{n,/\theta}\rightarrow \psh_\cS(\Theta^{\inrt,\op}_{n,/\theta})$ the functor sending $j:\theta_j\rightarrow\theta$ to $\underset{(e\xinert{i}\theta_j)\in\Theta^\el_{n,/\theta_j}}{\colim}\ima(j\circ i)$ and a morphism $g:\theta_j\rightarrow\theta_t$ over $\theta$ to the morphism $\beta(\theta_t)\rightarrow\beta(\theta_j)$ induced by morphisms $\ima(i_0):\ima(t\circ i')\xinert{}\ima(t\circ g\circ i) = \ima(j\circ i)$ indexed by diagrams 
		\[\begin{tikzcd}[sep=huge]
			{c_k} && {\ima(g\circ i)} & {c_l} \\
			{\theta_j} && {\theta_t} \\
			& \theta
			\arrow[two heads, from=1-1, to=1-3]
			\arrow["i"', tail, from=1-1, to=2-1]
			\arrow[tail, from=1-3, to=2-3]
			\arrow["{i_0}"', tail, from=1-4, to=1-3]
			\arrow["{i'}", tail, from=1-4, to=2-3]
			\arrow["g", from=2-1, to=2-3]
			\arrow["j"', from=2-1, to=3-2]
			\arrow["t", from=2-3, to=3-2]
		\end{tikzcd}.\]
		Denote $\beta_\theta^\spc\bydef \Sigma^\infty\beta$, then the isomorphism \eqref{eq:theta_inj} is induced by the functor \[\beta_\theta^{\spc,*}:\psh_\spc(\Theta^{\inrt,\op}_{n,/\theta})\rightarrow\seg_\spc(\Theta^\inj_{n,/\theta})\]
		sending $\cF:\Theta^\inrt_{n,/\theta}\rightarrow\spc$ to $\cF(\beta_\theta^\spc(\theta_j)):\Theta^{\inj,\op}_{n,/\theta}\rightarrow\spc$. Note that the functor $F|_{\Theta_n}$ is the stabilization of the functor $F_0:\Theta^\op_n\rightarrow\cat$ given by $F_0(\theta)\cong \seg_\cS(\Theta^\inj_{n,/\theta})$ and similarly $G|_{\Theta_n}$ is the stabilization of $G_0:\Theta^\op_n\rightarrow\cat$ given by $G_0(\theta)\bydef \psh_\cS(\Theta^{\inrt,\op}_{n,/\theta})$, so it suffices to show that $\beta^*_\theta:G_0(\theta)\rightarrow F_0(\theta)$ defines a natural transformation of functors. Denote by $F'_0:\Theta_n^\op\rightarrow \cat$ the functor $\theta\mapsto\psh_\cS(\Theta^\inj_{n,/\theta})$, then $F_0$ is a subfunctor of $F'_0$, so it suffices to show that $\beta^*_\theta$ defines a natural transformation $G_0\rightarrow F'_0$. Note that all the morphisms $\beta^*_\theta$ as well as all the functors $F'_0(f)$ and $G_0(f)$ for $f:\theta\rightarrow\theta'$ admits left adjoints, denote by $F_0^{',L}$ and $G_0^L$ the functors $\Theta_n\rightarrow\cat$ the functors obtained by taking $\theta$ to $F_0'(\theta)$ and $G_0(\theta)$ respectively and a morphism $f:\theta\rightarrow\theta'$ to the left adjoints $F'_0(f)^L$ and $G_0(f)^L$ respectively and by $\beta_{\theta,!}$ the left adjoints of $\beta^*_\theta$, it suffices to prove that $\beta_{\theta,!}$ defines a natural transformation between $F_0^{',L}$ and $G_0^L$. Finally, note that all the categories $\Theta^\inj_{n,/\theta}$ and $\Theta^{\inrt,\op}_{n,/\theta}$ are discrete and all the objects $\beta_\theta(\theta_j)\in\psh_\cS(\Theta^{\inrt,\op}_{n,/\theta})$ factor through the discrete subcategory $\psh_\mathrm{Set}(\Theta^{\inrt,\op}_{n,/\theta})$, so in order to prove that $\beta_{\theta,!}$ defines a natural transformation it suffices to prove that for any morphism $f:\theta\rightarrow\theta'$ we have a commutative diagram
		\[\begin{tikzcd}[sep=huge]
			{\psh_{\mathrm{Set}}(\Theta^\inj_{n,/\theta})} & {\psh_{\mathrm{Set}}(\Theta^\inj_{n,/\theta'})} \\
			{\psh_{\mathrm{Set}}(\Theta^{\inrt,\op}_{n,/\theta})} & {\psh_{\mathrm{Set}}(\Theta^{\inrt,\op}_{n,/\theta'})}
			\arrow["{f_!}", from=1-1, to=1-2]
			\arrow["{\beta_{\theta,!}}"', from=1-1, to=2-1]
			\arrow["{\beta_{\theta',!}}", from=1-2, to=2-2]
			\arrow["{f_!}"', from=2-1, to=2-2]
		\end{tikzcd}.\]
		In other words, we need to show that for $\theta_j\xrightarrow{j}\theta$ we have
		\[\underset{(e\xinert{i}\theta_j)\in\Theta^\el_{n,/\theta_j}}{\colim}\ima((f\circ j)\circ i)\cong \underset{(e\xinert{i}\theta_j)\in\Theta^\el_{n,/\theta_j}}{\colim}f_!\ima(j\circ i),\]
		which follows since $f_!\ima(j\circ i)\cong \ima(f\circ j\circ i)$ by construction of $f_!$.
	\end{proof}
	The rest of this section will be dedicated to the proof of \Cref{prop:stn_inrt}, before that we will need some auxiliary results and constructions.
	\begin{construction}\label{constr:sqs}
		Assume we have a string $S = l_{i_N}l_{i_{N-1}}...l_{i_1}$ containing terms $l_i$ with $0\leq i\leq n-1$. Denote by $m_q$ the total number of terms $l_q$ for $0\leq q\leq n-1$ in $S$, for $1\leq t\leq m_q$ denote by $1\leq s_t^q\leq N$ the index such that $i_{s^q_t} = q$ and there are $(t-1)$ terms $l_q$ among $l_{i_k}$ with $k<s^q_t$ and set $s_0^q=0$, for $0\leq k\leq n-1$ denote by $p_q^k(t)$ the number of $i_s = k$ with $s\leq s_t^q$ (so that $p_q^q(t) = t$). Finally, define $\sq(S)$ to be the full subcategory of $\prod_{q=0}^{n-1} [m_q]$ containing $(x_0,...,x_{n-1})$ such that for all $k$ we have 
		\[x_k\leq \max_{q\leq k}(p_q^k(x_q)).\]
		For $0\leq i\leq (n-1)$ Call a substring of $S$ an $i$-segment if it only contains the term $l_i$ and is maximal by inclusion. Denote by $\widehat{S}$ the string obtained from $S$ by replacing every $i$-segment for every $i$ as above with a single appearance of the term $l_i$, denote by $\widehat{m}_i$ the number of terms $l_i$ in $\widehat{S}$. Note that we have an active morphism $a_i:[\widehat{m}_i]\xactive{}[m_i]$ which sends $k\in [\widehat{m}_i]$ to the maximal index of the term $l_i$ appearing in the $k$th segment, taken together they define a morphism $a:\prod_i[\widehat{m}_i]\xactive{}\prod_i[m_i]$, it is easy to see that $a$ takes $\sq(\widehat{S})$ to $\sq(S)$. Finally, for some subset $\Sigma\in\{0,1,...,n-1\}$ we define a $\Sigma$\textit{-quadrant} of $\sq(S)$ to be the subcategory of $\sq(S)$ containing $y\bydef (y_0,...,y_{n-1})\in \sq(S)$ such that for some $x\bydef (x_0,...,x_{n-1})\in \sq(\widehat{S})$ we have $y_k = a_k(x_k)$ for $k\notin \Sigma$ and $a_s(x_s)\leq y_s\leq a_s(x_s+1)$ for $s\in\Sigma$. We will denote by $\sq_0(S)\hookrightarrow\sq(S)$ the simplicial subset (which is generally not a category) containing simplices $[m]\xrightarrow{f}\sq(S)$ whose image lies entirely in some quadrant. Finally, denote by $[N]^0$ the \textit{spine} of $[N]$, i.e. the simplicial subset containing all 0-simplices and the 1-simplices of the form $(i<i+1)$, then there is a canonical morphism $i_S:[N]^0\rightarrow\sq(S)$ that factors through $\sq^0(S)$ taking $k\leq N$ to $(i_0^k(S),i_1^k(S),...,i_{n-1}^k(S))$, where $i_j^k(S)$ denotes the number of terms $l_{i_s}$ with $i_s=j$ and $s\leq k$.
	\end{construction}
	\begin{lemma}\label{lem:bruh}
		Given a string $S = l_{i^S_N}l_{i^S_{N-1}}...l_{i^S_1}$ as in \Cref{constr:sqs} denote by $P$ the set of injective morphisms $[N]\xrightarrow{j} \prod_{q=0}^{n-1} [m_q]$ such that $j(0) = (0,...,0)$ and $j(N) = (m_0,...,m_{n-1})$, then there is a partial order on $P$ such that $j$ factors through $\sq(S)\hookrightarrow\prod_{q=0}^{n-1} [m_q]$ if and only if $j\leq i_S$.
	\end{lemma}
	\begin{proof}
		Note that the set $P$ can be identified with the set of strings $S' = l_{i^{S'}_N}...l_{i^{S'}_0}$ with $0\leq i'_k\leq n-1$, where to such a string corresponds a morphism $[N]\rightarrow \prod_{q=0}^{n-1}[m_q]$ sending $k$ to $(i_0^k(S'),...,i^{k}_{n-1}(S'))$ , where $i_j^k(S)$ denotes the number of terms $l_{i^{S'}_s}$ with $i^{S'}_s=j$ and $s\leq k$. Assume we have $k\in [N]$ such that $p=i^{S'}_{k+1}<i^{S'}_k=q$, define $\sigma_k S'$ to be the string such that $i^{\sigma_k S'}_u = i^{S'}_u$ for $u\notin\{k,k+1\}$, $i^{\sigma_k S'}_{k+1}=q$ and $i^{\sigma_k S'}_{k}=p$, we declare $\sigma_k S' <^0_P S'$ and in general define $<_P$ to be the transitive closure of $<_P^0$, we claim that this satisfies the conditions of the lemma.\par
		Assume that $S'<_P S$, we need to prove that the corresponding morphism $j_S:[N]\rightarrow \prod_{q=0}^{n-1}[m_q]$ factors through $\sq(S)$. It follows from the definition of $<_P$ that it suffices to prove the following claim: assume $S'=l_{i^{S'}_N}...l_{i^{S'}_1}$ is such that $j_{S'}$ factors through $\sq(S)$ and $k\in[N]$ is such that $p=i^{S'}_{k+1}<i^{S'}_k=q$, then $j_{\sigma_k S'}$ also factors through $\sq(S)$. To prove this we will first introduce some notation: for $u\in[n-1]$ and $t\in [m_u]$ denote by $\tau^u_{\leq t}S$ the maximal substring $S_0$ of $S$ of the form $l_{i^S_v}l_{i^S_{v-1}}...l_{i^S_0}$ containing $\leq t$ terms $l_u$, then it is easy to see that $(x_0,...,x_{n-1})$ belongs to $\sq(S)$ if and only if for every $u$ as above and all $s>u$ the string $\tau_{\leq x_u}^u S$ contains $\geq x_s$ terms $l_s$. It follows easily from the definitions that for $x\in[N]$ we have $j_{\sigma_k S'}(x) = j_{S'}(x)$ for $x<k$, for $x=k$ we have
		\[j_{S'}(k)_s=\begin{cases}
			j_{S'}(k-1)_s\text{ if $s\notin\{p,q\}$}\\
			j_{S'}(k-1)_p\text{ if $s=p$}\\
			j_{S'}(k-1)_q+1\text{ if $s=q$}
		\end{cases}, 
		j_{\sigma_k S'}(k)_s = \begin{cases}
			j_{S'}(k-1)_s\text{ if $s\notin\{p,q\}$}\\
			j_{S'}(k-1)_p+1\text{ if $s=p$}\\
			j_{S'}(k-1)_q\text{ if $s=q$}
		\end{cases}\]
		and for $x>k$ we again have $j_{S'}(x) = j_{\sigma_k S'}(x)$. We need to show that $(x_0,...,x_{n-1})\bydef (j_{S'}(k-1)_0,...,j_{S'}(k-1)_p+1,...,j_{S'}(k-1)_{n-1})$ lies in $\sq(S)$, note that the conditions on the coordinates $x_s$ with $s>p$ are vacuous since for $t\geq s$ the coordinates $x_s$ coincide with $j_{S'}(k-1)_s$ and $j_{S'}(k-1)\in\sq(S)$, for $s=p$ we see that $\tau^p_{\leq j_{S'}(k-1)_p}S$ contains at least $x_t$ instances of $l_t$ for $t>p$ since $j_{S'}(k-1)\in\sq(S)$, hence so does $\tau^p_{\leq j_{S'}(k-1)_p+1}S$, and for $s<p$ we see that $\tau^s_{\leq x_s}S'$ contains at least $x_t$ instances of $l_t$ for $t>s$ since $x_s = j_{S'}(k+1)_s$, $j_{S'}(k+1)\in \sq(S)$ and $j_{S'}(k+1)_t\geq x_t$, which concludes the proof of the claim.\par
		Now, assume that $S'$ is such that $j_{S'}$ factors through $\sq(S)$, we need to prove that $S' \leq_P S$, for that we will construct a sequence of strings $S'_l$ for $l\in [N]$ such that $S'_l\leq_p S$, $j_{S'}$ factors through $\sq(S'_l)$ and the first $l$ terms of $S'_l$ and $S'$ coincide (so that $S'_N =S'$), starting with $S'_0\bydef S$. Assume we have constructed $S'_{l-1}$, we first claim that $p = i^{S'}_l \leq i^{S'_{l-1}}_l = q$: indeed, if we have $q<p$ then since $j_{S'}(l)\in \sq(S'_l)$ that would mean (since $j_{S'}(l-1)_q = j_{S'}(l)_q$) that $\tau^q_{\leq j_{S'}(l-1)_q}S_{l-1}'$ contains at least $j_{S'}(l)_p = j_{S'_{l-1}}(l-1)_p +1$ terms $l_p$, however by construction $\tau^q_{\leq j_{S'}(l-1)_q}S_{l-1}' = \tau^q_{\leq j_{S_{l-1}'}(l-1)_q}S_{l-1}'$ consists of the first $(l-1)$ terms of $S'_{l-1}$ and contains exactly $j_{S'}(l-1)_p = j_{S'_{l-1}}(l-1)_p$ terms $l_p$. Denote by $x>l$ the first index for which $i^{S'_l}_x = p$, we claim that for all $l<y<x$ we have $i^{S'_l}_y > p$. Indeed, we cannot have $i^{S'_l}_y = p$ by definition, so assume that some $i^{S'_l}_y = z < p$, then it follows that $\tau^z_{\leq j_{S'_{l-1}}(l-1)_z}S'_{l-1}$ contains exactly $j_{S'_{l-1}}(l-1)_p$ terms $l_p$, however since $j_{S'}(l)\in\sq(S'_l)$, $j_{S'}(l)_z = j_{S'_{l-1}}(l-1)_z$ and $z<p$ it must contain at least $j_{S'}(l)_p = j_{S'_{l-1}}(l-1)_p +1$ terms $l_p$. It follows that we can define $S'_l\bydef \sigma_l\sigma_{l+1}...\sigma_x S'_{l-1}$, so that $S'_l \leq_P S'_{l-1}$ and $S'_l$ coincides with $S'$ up to the $l$th term, it remains to show that $j_{S'}$ also factors through $\sq(S'_l)$. Observe that it suffices to prove that $j_{S'}(k)\in \sq(S'_l)$ for $k>l$ since $j_{S'}(k\leq l)$ lie in $\sq(S'_l)$ by construction. Note that for $z<p$ the substrings $\tau^z_{\leq t}S'_l$ contain the same number of terms $l_i$ as $\tau^z_{\leq t}S'_{l-1}$ for all $i$ by construction, for $z>p$ the substrings $\tau^z_{\leq t}S'_l$ contain the same number of terms $l_{i'}$ as $\tau^z_{\leq t}S'_{l-1}$ for $i'>z$ (since $z>p$), and for $z=p$ we see that $\tau^p_{\leq t}S'_l$ contain the same number of terms $l_{i}$ as $\tau^p_{\leq t}S'_{l-1}$ for all $i$ and $t>j_{S'_{l-1}}(l-1)_p$. The claim follows immediately from these observations since for all $j_{S'}(k>l)$ we have $j_{S'}(k)_p \geq  j_{S'}(l)_p > j_{S'_{l-1}}(l-1)_p$.
	\end{proof}
	\begin{lemma}\label{lem:sq_ext}
		Assume we have a string $S = l_{i_N}l_{i_{N-1}}...l_{i_1}$ as in \Cref{constr:sqs} and a morphism $F:[N]^0\rightarrow \twar^C_\theta(\cE)$ such that the 1-simplex $F(j<j+1)$ factors through $\twar_\theta^{i_{j+1}}(\cE)\rightarrow \twar^C_\theta(\cE)$ of \Cref{constr:C_delta}, then we can extend $F$ to a morphism $\widetilde{F}:\sq_0(S)\rightarrow\twar^C_\theta(\cE)$ such that $\widetilde{F}\circ i_S\cong F$.
	\end{lemma}
	\begin{proof}
		Assume first that $S$ has the form $l^{q_0}_{k_0}l^{q_1}_{k_1}...l^{q_m}_{k_m}$ for some $0\leq k_0<k_1<...<k_m\leq n-1$ and $i_j>0$, in that case we need to extend $F:[\sum q_j]^0\rightarrow \twar^C_\theta(\cE)$ to $\widetilde{F}:\prod_{0\leq j\leq m} [q_j]\rightarrow \twar_\theta^C(\cE)$. By possibly reindexing the terms we may assume that $k_i = i$, denote $N\bydef \sum_{j=0}^m q_j$ and for $0\leq i\leq m$ denote $s_i\bydef \sum_{j=i}^m q_j$ and set $s_{m+1}=0$, so that we have $0=s_{m+1} < s_m<...< s_0=N$. First, we will define a morphism $F_0:[m]\rightarrow \twar_\theta^C(\cE)$ extending $F$: note that we can identify $F$ with a natural transformation $\alpha:F'\rightarrow F''$ between functors $\Delta^\el_{/[N]}\rightarrow \cat_n$, where $F''$ is the constant functor with value $\cE$ and $F'$ sends $[0]\xinert{\{i\}}[n]$ to $c_n\cong C_n(0,...,0)$ and $(v<v+1)$ for $0\leq v<N$ to $C_n^{i(v)}(1)$ in the notation of \Cref{constr:C_delta}, where $i(v)$ denotes the index $0\leq k\leq m$ such that $s_{k+1}\leq v<v+1\leq s_{k}$. By \Cref{lem:C_push} both $F'$ and $F''$ admit left Kan extensions $i^\el_!F'$ and $i^\el_!F''$ to $\Delta^\inrt_{/[N]}$ such that $i^\el_! F''$ is still constant, while $i^\el_! F'$ sends $[l]\xinert{i}[N]$ to $C_n(j_0,...,j_{m},0,...,0)$, where $j_k\bydef \max(0,|i([l])\bigcap[s_{k-1}, s_k]|-1)$. It follows that $i^\el_!\alpha$ defines an object of $|\twar'_\theta(\cE)|([N])\cong \twar^C_\theta(\cE)([N])$. Denote by $\alpha_\id$ the component $C_n(i_0,...,i_m,0,...,0)\rightarrow\cE$, by construction we see that the components $\alpha_i:C_n(j_0,...,j_{m},...,0)\rightarrow\cE$ for $i:[l]\xinert{}[N]$ as above are given by $C_n(j_0,...,j_{m},...,0)\xrightarrow{C(u)}C_n(i_0,...,i_m,...,0)\xrightarrow{\alpha_\id}\cE$, where $u$ has components $u_k:i([l])\bigcap[s_{k+1},s_k]\xinert{}[s_{k+1},s_k]$.\par    
		Now assume that we have an injective morphism $f:[N]\rightarrow\prod_{j=0}^m [i_j]$, note that those correspond to strings $S_\sigma = l_{i_N}...l_{i_1}$ with $0\leq i_k\leq m$ where the total number of terms $l_j$ equals $q_j$. Given an inert morphism $i:[l]\xinert{}[N]$ and $0\leq x\leq m$, denote by $j^0_x(i)$ (resp. $j^1_x(i)$) the minimal (resp. maximal) $t$ such that $s_t^x$ lies in the image of $i$. We need to construct a natural transformation $\alpha_\sigma:F'_\sigma\rightarrow i^\el_!F''$ of functors from $\Delta^\inrt_{/[N]}$ to $\cat_n$ such that $F'_\sigma$ lands in the full subcategory $\Theta_n$ and $i^\el_!F''$ is the constant functor with value $\cE$ as above. Denote by $F_0:\Delta^\inrt_{/[N]}\rightarrow\Theta_n$ the constant functor with value $C_n(i_0,...,i_m,...,0)$, we will define $\alpha_\sigma$ to be given by the composition \[F'_\sigma\xrightarrow{\alpha'}F_0\xrightarrow{\alpha_\id}i^\el_!F'',\]
		where $\alpha_\id:C_n(i_0,...,i_m,...,0)\rightarrow\cE$ is defined above. Given $i:[l]\xinert{}[N]$ we define $F'_\sigma(i)$ to be the left Kan extension of $F_\sigma:\Delta^\el_{/[N]}$ sending $i_0:[e]\xinert{}[N]$ for $e\in\{0,1\}$ to $C_n(j^1_0(i_0) - j^0_0(i_0),..., j^1_m(i_0) - j^0_m(i_0),0,...,0)$: it follows from \Cref{lem:C_push} that this left Kan extension exists and is given by the functor sending $[l]\xinert{i}[N]$ to $C_n(j^1_0(i) - j^0_0(i),..., j^1_m(i) - j^0_m(i),0,...,0)$. Finally, to define the natural transformation $\alpha'$ as above it suffices (since $F_0$ is constant) to define its component $\alpha'_{\id_{[N]}}$, note that $F'_\sigma(\id_{[N]}) = C_n(i_0,...,i_m,...,0) = F_0(\id_{[N]})$, so we define $\alpha'_{\id_{[N]}}$ to be the identity morphism.\par
		Assume that we are now given a morphism $f:[m]\rightarrow [1]^n$, to define $\widetilde{F}(f)$ pick some $f_\sigma:[N]\rightarrow \prod[q_j]$ such that $f$ factors as $[m]\xrightarrow{f'}[n]\xrightarrow{f_\sigma}[1]^n$ and define $\widetilde{F}(f):\Delta^\inrt_{/[m]}\rightarrow\Theta_{n,/\cE}$ to be the composition 
		\[\Delta^\inrt_{/[m]}\xrightarrow{f'_*}\Delta^\inrt_{/[N]}\xrightarrow{\widetilde{F}(f_\sigma)}\Theta_{n,/\cE}.\]
		We now need to prove that this is correctly defined, i.e. that it does not depend on $f_\sigma$, for this note that by definition all $\widetilde{F}(f_\sigma)$ factor through $\Theta_{n,/C_n(i_0,...,i_m,...,0)}\xrightarrow{\alpha_{\id,!}}\Theta_{n,/\cE}$, so we may assume $\cE\cong C_n(i_0,...,i_m,...,0)\in\Theta_n$, in which case the claim follows from \Cref{lem:theta_twar} since a morphism in $\Theta^\inrt_{n,/C(i_0,...,i_m,...,0)}$ is uniquely determined by its endpoints.\par
		Assume now that we have a morphism $j':[m]\rightarrow [q_j]$ taking $(i<i+1)$ to a morphism of the form $(a_0,...,a_i,...,a_m)< (a_0,..., a_i+b_i,...,a_m)$. Factor it as $[n]\xactive{b}[N']\xrightarrow{j}\prod [q_j]$, where $j$ is an injective morphism such that $\sum_{i=0}^m(j_i(i+1) - j_i(i))=1$, $N'\bydef \sum b_i$ and $b(i)\bydef \sum_{k=0}^i b_i$, it is clear that such a factorization is unique, then the composition $F':[N']^0\hookrightarrow[N']\xrightarrow{j}\prod[q_i]\xrightarrow{\widetilde{F}}\twar^C_\theta(\cE)$ corresponds to the string $S' = l_{0}^{q'_0}...l_m^{q'_m}$ of the same form as $S$, hence by our previous construction it defines a morphism $\widetilde{F}':\prod [q'_i]\rightarrow \twar_\theta^C(\cE)$. Using the construction of $\widetilde{F}'$ and the fact that $F'$ factors through $\prod[q'_i]$ we see that $\widetilde{F}'$ is isomorphic to the composition $\prod [q'_i]\rightarrow\prod [q_i]\xrightarrow{\widetilde{F}}\twar_\theta^C(\cE)$.\par
		Assume now that we have an arbitrary string $S = l_{i_N}...l_{i_0}$, we can subdivide as $S = S_kS_{k-1}...S_1$ such that $S_j = l_{k^j_0}^{q^j_0}...l_{k^j_{m_j}}^{q_{m_j}^j}$ for $1\leq j\leq k$ is of the form considered at the start of the proof and is a maximal by inclusion substring of $S$ with that property, then we can extend $F:[N]^0\rightarrow\twar^C_\theta(\cE)$ to 
		\[\widetilde{F}: X_1(S)\bydef(\prod [q^1_i])\coprod_{[0]}(\prod [q^2_i])\coprod_{[0]}...\coprod_{[0]}(\prod [q^k_i])\rightarrow\twar^C_\theta(\cE).\]
		Note that every term $\prod [q^k_i]$ above is in the image of a quadrant in $\sq^0(S)$, so $X_1(S)$ is a simplicial subset of $\sq^0(S)$, moreover note that an injective morphism $[N]\xrightarrow{j}\sq(S)$ corresponding to a string $S'$ factors through $X_1(S)$ if and only if $S' = S'_k S'_{k-1}...S'_{1}$ for some strings $S'_j$ such that $S'_j <_P S_j$ in the notation of \Cref{lem:bruh}. Given any such string $S'$ we can apply the same construction to it in order to extend $\widetilde{F}$ to $X_1(S')\hookrightarrow\sq^0(S')\hookrightarrow\sq^0(S)$, note that if we have another $S''$ such that $X_1(S')\bigcap X_1(S'')\neq \varnothing$, then by the observations from the previous paragraph the extensions $\widetilde{F}|_{X_1(S')}$ and $\widetilde{F}|_{X_1(S'')}$ agree on the intersection, so we can extend $\widetilde{F}$ to 
		\[X_2(S)\bydef \bigcup_{j:[N]^0\xrightarrow{}{X_1(S)}} X_1(S_j)\hookrightarrow\sq^0(S),\]
		where $S_j$ denotes the string corresponding to the morphism $j$, iterating this construction we may extend $\widetilde{F}$ to $X_{m}(S)\bydef \bigcup_{j:[N]^0\xrightarrow{}{X_{m-1}(S)}} X_1(S_j)$ for any $m$, so it suffices to prove that for some $M$ we have $X_M(S) = \sq^0(S)$. Observe that if some string $S'$ admits an index $k$ such that $i^{S'}_{k+1}< i^{S'}_k$, then $\sigma_kS'$ factors through $X_1(S')$, so it follows from \Cref{lem:bruh} every $j_S$ factors through some $X_M(S)$ for large enough $M$, meaning that we may extend $\widetilde{F}$ to the entirety of $\sq^0(S)$.
	\end{proof}
	\begin{lemma}\label{lem:sqs_cat}
		The natural inclusion $\sq^0(S)\hookrightarrow\sq(S)$ induces an isomorphism
		\[\cL_\cat\sq^0(S)\cong \sq(S).\]
	\end{lemma}
	\begin{proof}
		Consider $\sq(S')$ as a subcategory of some $\prod_{i=0}^{m'}[q'_i]$, then denote by $p^S_i:\sq(S')\rightarrow[q'_i]$ the composition $\sq(S)\hookrightarrow\prod_{k=0}^{m'} [q'_k]\xrightarrow{p_i}[q'_i]$, where the last morphism is the projection to the $i$th coordinate. Call a morphism $\sq(S')\xrightarrow{j}\sq(S)$ with target $\sq(S)\hookrightarrow\prod_{i=0}^m [q_i]$ \textit{inert} if for any morphism $f:[1]\rightarrow\sq(S')$ such that $p_k\circ f:[1]\rightarrow [q'_k]$ are constant for $k\neq i'$ and inert for $k=i'$ the composition $j\circ f:[1]\rightarrow\sq(S)$ is such that $p_k\circ j\circ f:[1]\rightarrow[q_k]$ is constant for $k\neq i$ and inert for $k=i$ (with $i\neq i'$ in general), it is easy to see that inert morphisms are closed under composition and include identities, hence they form a subcategory of $\cat$. Call an object $\sq(S)$ \textit{elementary} if it is isomorphic to $[1]^l$ for some $l\leq (n-1)$, denote by $\Delta^\el_{/\sq(S)}$ the full subcategory of $\cat_{/\sq(S)}$ on inert morphisms from elementary objects $[1]^l\xinert{i}\sq(S)$. Denote by $\widehat{S}$ the reduced string of $S$ defined in \Cref{constr:sqs}, denote $a:\sq(\widehat{S})\rightarrow \sq(S)$ the morphism defined therein. Given $[1]^l\xinert{i}\sq(\widehat{S})$ denote by $\ima(i)\xinert{}\sq(S)$ the full subcategory containing all objects $(x_0,...,x_m)$ such that $a\circ i(0,...,0)\leq (x_0,...,x_m)\leq a\circ i(1,...,1)$, note that the subcategories of the form $\ima(i)$ are the same thing as quadrants described in \Cref{constr:sqs}. It follows from this and the definition of $\sq^0(S)$ that we have an isomorphism
		\begin{equation}\label{eq:S0_inrt}
			\sq^0(S)\cong \underset{[1]^l\xinert{i}\sq(\widehat{S})}{\colim}\ima(i),
		\end{equation}
		as simplicial sets, since $\cL_\cat$ preserves colimits to prove the claim it suffices to show that
		\begin{equation}\label{eq:S_inrt}
			\sq(S)\cong \underset{[1]^l\xinert{i}\sq(\widehat{S})}{\colim}\ima(i).
		\end{equation}
		We will prove \eqref{eq:S_inrt} by induction on $n$, the case $n=1$ being trivial. Assume we have proved the claim for $(n-1)$, we may assume that $S=l_{i'_{N'}}...l_{i'_1}$ contains all the terms $l_k$ for $0\leq k\leq n-1$, since otherwise the claim follows by induction, denote $\widehat{S}\bydef l_{i_N}...l_{i_1}$. Denote by $m_q$ the total number of terms $l_q$ in $\widehat{S}$ for $0\leq q\leq n-1$ and $s^{q}_t$ for $1\leq t\leq m_s$ the index of the $t$th appearance of $l_q$. For $i:[e]\xinert{}[m_{n-1}]$ in $\Delta^\el_{/[m_{n-1}]}$ denote by $\sq(\widehat{S})_i\xinert{}\sq(\widehat{S})$ the full subcategory of $\sq(\widehat{S})$ containing $(x_0,...,x_{n-1})$ with $x_{n-1}\in[m_{n-1}]$ in the image of $i$. For $t\in [m_{n-1}]$ denote by $\widehat{S}_t\bydef l_{i_{j_k}}l_{i_{j_{k-1}}}...l_{i_{j_0}}$ the substring of $\widehat{S}$ containing terms $l_{i_s}$ with $s\geq s^{n-1}_t$ and $i_s\neq n-1$, then it is easy to see that $\sq(\widehat{S})_{\{t\}}\cong \sq(\widehat{S}_t)$, where $\{t\}:[0]\xinert{}[m_{n-1}]$ is the inclusion of the object $t$, and $\sq(\widehat{S})_{t<t+1}\cong [1]\times \sq(\widehat{S}_t)$. Note that for $t_0 < t_1$ the string $\widehat{S}_{t_1}$ is a substring of $\widehat{S}_{t_0}$, hence we have an inert inclusion $\sq(\widehat{S}_{t_1})\xinert{} \sq(\widehat{S}_{t_0})$. From these considerations we see that there is a functor $\sq(\widehat{S})_{(-)}:\Delta^\el_{/[m_{n-1}]}\rightarrow \cat$ sending $\{t\}$ to $\sq(\widehat{S}_t)$, $(t<t+1)$ to $[1]\times \sq(\widehat{S}_t)$, the inclusion $\{t\}\xinert{}(t<t+1)$ to $\sq(\widehat{S}_t)\xinert{\{0\}\times\id}[1]\times \sq(\widehat{S}_t)$ and $\{t+1\}\xinert{}(t<t+1)$ to the composition $\sq(\widehat{S}_{t+1})\xinert{}\sq(\widehat{S}_t)\xinert{\{1\}\times\id}[1]\times \sq(\widehat{S}_t)$. Since all morphisms in $\Delta^\el_{/[m_{n-1}]}$ are sent to inert morphisms, it follows that we have an induced functor $\Delta^\el_{/\sq(\widehat{S})_{(-)}}$ sending $i:[e]\xinert{}[m_{n-1}]$ to $\Delta^\el_{/\sq(\widehat{S})_{i}}$, we claim that we have an isomorphism:
		\begin{equation}\label{eq:inrt_sq}
			\underset{([e]\xinert{i}[m_{n-1}])\in\Delta^\el_{/[m_{n-1}]}}{\colim}\Delta^\el_{/\sq(\widehat{S})_i}\cong \Delta^\el_{/\sq(\widehat{S})}.
		\end{equation}
		Indeed, we will show that \eqref{eq:inrt_sq} holds already at the level of simplicial sets: given $[1]^m\xinert{i_0}\sq(\widehat{S})$ we need to show that the space of factorizations $[1]^m\xinert{i'_0}\sq(\widehat{S})_{i}\xinert{}\sq(\widehat{S})$ is contractible, however note that it is a singleton if $i_0$ does not factor through any $\sq(\widehat{S})_{\{t\}}$ for $t\in[m_{n-1}]$ and is isomorphic to the walking cospan $\Lambda^0_2$ otherwise, both os those simplicial sets are contractible.\par
		Using \eqref{eq:inrt_sq} and commutativity of colimits, we see that the right hand side of \eqref{eq:S_inrt} can be rewritten as
		\begin{equation}\label{eq:colim_iter}
			\underset{[e]\xinert{i_0}[m_{n-1}]}{\colim}\underset{[1]^l\xinert{i}\sq(\widehat{S})_i}{\colim}\ima(\sq(i_0)\circ i).
		\end{equation}
		For $i_0:[e]\xinert{}[m_{n-1}]$ as above denote by $\sq(S)_{i_0}$ the image of $\sq(\widehat{S})_{i_0}$ in $\sq(S)$ under $a$, we claim that
		\begin{equation}\label{eq:colim_iter1}
			\sq(s)_{i_0}\cong \underset{([1]^l\xinert{i}\sq(\widehat{S})_{i_0})\in\Delta^\el_{/\sq(\widehat{S})_{i_0}}}{\colim}\ima(i).
		\end{equation}
		Indeed, if $i_0$ is of the form $[0]\xinert{\{t\}}[m_{n-1}]$, then \eqref{eq:colim_iter1} follows immediately from the inductive assumption. If $i_0$ has the form $[1]\xinert{t<t+1}[m_{n-1}]$, then 
		\begin{equation}\label{eq:S_aux1}
			\sq(S)_{t<t+1}\cong [k_t]\times \sq(S)_{\{t\}}
		\end{equation}
		for some $k_t>0$ and 
		\begin{equation}\label{eq:S_aux}
			\Delta^\el_{/\sq(\widehat{S})_{t<t+1}}\cong \Delta^\el_{/[1]}\times \Delta^\el_{/\sq(\widehat{S})_t},
		\end{equation}
		using this and the inductive assumption we get
		\begin{align*}
			\underset{([1]^l\xinert{i}\sq(\widehat{S})_{t<t+1})\in\Delta^\el_{/\sq(\widehat{S})_{t<t+1}}}{\colim}\ima(i)&\cong \underset{([e]\xinert{i_0}[1])\in\Delta^\el_{/[1]}}{\colim}\underset{([1]^{l'}\xinert{i_1}\sq(\widehat{S})_{t})\in\Delta^\el_{/\sq(\widehat{S})_t}}{\colim}\ima(i_0)\times \ima(i_1)\\
			&\cong \underset{([1]^{l'}\xinert{i_1}\sq(\widehat{S})_{t})\in\Delta^\el_{/\sq(\widehat{S})_t}}{\colim}[k_t]\times \ima(i_1)\\
			&\cong [k_t]\times \underset{([1]^{l'}\xinert{i_1}\sq(\widehat{S})_{t})\in\Delta^\el_{/\sq(\widehat{S})_t}}{\colim}\ima(i_1)\\
			&\cong [k_t]\times \sq(S)_t\cong \sq(S)_{t<t+1},
		\end{align*}
		where the first isomorphism follows from \eqref{eq:S_aux}, he second since $\Delta^\el_{/[1]}$ has a final object, the third since $-\times-$ commutes with colimits, the fourth by the inductive assumption and the fifth by \eqref{eq:S_aux1}. \par
		It now follows from \eqref{eq:colim_iter} and \eqref{eq:colim_iter1} that it suffices to prove 
		\begin{equation}\label{eq:colim_simp}
			\sq(S)\cong \underset{([e]\xinert{i}[m_{n-1}])\in\Delta^\el_{/[m_{n-1}]}}{\colim}\sq(S)_i,
		\end{equation}
		we will prove it using deformation theory: denote by $X$ the colimit in the right hand side of \eqref{eq:colim_simp}, we need to show that the natural morphism $f:X\rightarrow\sq(S)$ is an isomorphism. By \cite[Corollary 2.6.2.]{harpaz2018abstract} and \cite[Theorem 5.2.]{harpaz2020k} it suffices to prove that \eqref{eq:colim_simp} holds in the subcategory $\cat_{(2,1)}$ of $(2,1)$-categories, which is immediate using the explicit description of colimits of discrete categories, and that the relative cotangent complex $L_f\cong 0$. By the results of \cite{nuiten2019quillen} we see that for $\cC\in\cat$ $\stab(\cat_{/\cC})\cong \mor_\cat(\twar(\cC),\spc)$, the absolute cotangent complex $L_\cC$ is isomorphic (up to a shift) to the constant functor with value $\bS$ and for $f:\cC\rightarrow\cD$ we have
		\[L_f\cong \cok(f_!\bS\rightarrow\bS).\]
		Applying this to our situation, we need to show that in $\mor_\cat(\twar(\sq(S)),\spc)$ we have
		\begin{equation}\label{eq:rel_cot}
			\bS\cong  \underset{([e]\xinert{i}[m_{n-1}])\in\Delta^\el_{/[m_{n-1}]}}{\colim}\sq(i)_!\bS,
		\end{equation}
		where $\sq(i):\twar(\sq(S)_i)\rightarrow\twar(\sq(S))$ denotes the functor induced by the natural inclusion $\sq(S)_i\xinert{}\sq(S)$. We now need to calculate $\sq(i)_!\bS$, note that for $f\in\twar(\sq(S))$ the value $\sq(i)_!\bS(f)$ it is isomorphic to $|\sq(i)/f|\otimes \bS$, we will first show that $\sq(i)/f$ are either empty or contractible. Fix some $i:[e]\xinert{}[m_{n-1}]$ and $f$ as above, note that $f$ corresponds to a pair $((x_0,...,x_{n-1}),(y_0,...,y_{n-1}))$ with $y_k\geq x_k$. First, note that if the image of $i$ does not intersect $[x_{n-1}, y_{n-1}]$, then $\sq(i)/f$ is empty. Assume now that $t$ is the minimal element of that intersection, in that case denote by $i_t:\sq(S)_{\{t\}}\hookrightarrow\sq(S)$ the corresponding inclusion, it is easy to see that we either have $\sq(i)/f\cong \sq(i_t)/f$ (if $[x_{n-1},y_{n-1}]\cap\ima(i) = \{t\}$) or $\sq(i)/f\cong \twar([1])\times \sq(i_t)/f$ (if $[x_{n-1}, y_{n-1}]$ contains $(t<t+1)$), in either case it suffices to prove that $\sq(i_t)/f$ is contractible, so we assume $i=i_t$ from now on.\par
		Note that an object of $\sq(i_t)/f$ is given by a quadruple $((x_0,...,x_{n-1}),(a_0,...,t),(b_0,...,t),(y_0,...,y_{n-1}))$ such that $x_i\leq a_i\leq b_i\leq y_i$, with a morphism from $(x,a,b,y)$ to $(x,a',b',y)$ given by $a'_i\leq a_i\leq b_i\leq b'_i$. Note that it contains a final object given by $((x_0,...,x_{n-2},x_{n-1}),(x_0,...,x_{n-2},t),(y_0,...,y_{n-2},t),(y_0,...,y_{n-2},y_{n-1}))$, meaning that it is indeed contractible. Applying these observations we get
		\begin{align*}
			\underset{([e]\xinert{i}[m_{n-1}])\in\Delta^\el_{/[m_{n-1}]}}{\colim}\sq(i)_!\bS(f)&\cong \underset{([e]\xinert{i}[m_{n-1}])\in\Delta^\el_{/[m_{n-1}]}}{\colim}|\sq(i)/f|\otimes\bS\\
			&\cong \underset{([e]\xinert{i}[y_{n-1}-x_{n-1}])\in\Delta^\el_{/[y_{n-1}-x_{n-1}]}}{\colim}\bS\\
			&\cong |\Delta^\el_{/[y_{n-1}-x_{n-1}]}|\otimes \bS\cong \bS,
		\end{align*}
		which proves \eqref{eq:rel_cot} and concludes the proof.
	\end{proof}
	\begin{prop}\label{prop:stn_inrt}
		For $x\in\stn_n$ we have
		\[\twar_\theta(x)\cong \stn^\inrt_{n,/x}.\]
	\end{prop}
	\begin{proof}
		We will define functors $F:\stn^\inrt_{n,/x}\rightarrow \twar_\theta(X)$ and $G:\twar^C_\theta(x)\rightarrow\stn^\inrt_{n,/x}$ (which necessarily factors through $\cL_\cat\twar^C_\theta(x)\cong \twar_\theta(x)$) and prove $G\circ F\cong \id$ and $F\circ G\cong \id$. We will start with $G$: note that any morphism $f:\Delta^\inrt_{/[q]}\rightarrow\Theta_{n,/x}$ in $\twar^C_\theta(X)([q])$ such that the values at $([0]\xinert{\{i\}}[q])$ are given by $\theta_{f_i}\xrightarrow{f_i}x$ for $i\in[q]$ induces a string of inclusions $\ima(f_0)\xinert{}\ima(f_1)\xinert{}...\xinert{}\ima(f_q)$ in $\stn^\inrt_{n,/x}$, we define $G(f)$ to be this string of inert morphisms, it is easy to see that this indeed defines a functor.\par
		We will now define $F$, we start by describing its value on objects and morphisms.
		Given $y_i\xinert{i}x$ define its image to be the composition $c_n\xactive{}y_i\xinert{i}x$ (where the first morphism exists since by definition $y\in \stn^c_n$). Assume now that we are given an inclusion $y_{i'_0}\xinert{i'}y_{i'_1}$ in $\stn^\inrt_{n,/x}$, define $g:c_n\xactive{}y_{i'_0}\xinert{i'}y_{i'_1}$, then by an iterated application of \Cref{lem:jact} for $k=0,1,...,n-1$ we get a diagram
		\begin{equation}\label{eq:el_coc}
			\begin{tikzcd}[sep=huge]
				&&& x \\
				&&& {y_{i'_1}} \\
				& {C_n^0(1)} && {C_n^1(1)} && {C_n^{n-1}(1)} \\
				{c_n} && {c_n} && {c_n} && {c_n}
				\arrow["{i'_1}"{description}, from=2-4, to=1-4]
				\arrow["{g_0}"{description}, from=3-2, to=2-4]
				\arrow["{g_1}"{description}, from=3-4, to=2-4]
				\arrow["{g_{n-1}}"{description}, from=3-6, to=2-4]
				\arrow["{i_0}"', tail, from=4-1, to=3-2]
				\arrow["{a_0}", two heads, from=4-3, to=3-2]
				\arrow["{i_1}"', tail, from=4-3, to=3-4]
				\arrow["{...}"{description}, two heads, from=4-5, to=3-4]
				\arrow["{i_{n-1}}"', tail, from=4-5, to=3-6]
				\arrow["{a_{n-1}}", two heads, from=4-7, to=3-6]
			\end{tikzcd}
		\end{equation}
		such that $g_k\circ a_k$ is $k$-active in the sense of \Cref{def:jact} and $\ima(g_0\circ i_0)\cong y_{i'_0}$. This defines a morphism $g_{i'}:[n]^0\rightarrow\twar^C_\theta(x)$ such that moreover $g_{i'}(k<k+1)$ lies in $\twar_\theta^k(x)$, since $\cL_\cat[n]^0\cong [n]$ we can extend $g_{i'}$ to $g^0_{i'}:[n]\rightarrow\twar_\theta(x)$ and we define $F(i')$ to be the composition $[1]\xactive{}[n]\xrightarrow{g^0_{i'}}\twar_\theta(x)$. More generally, given a morphism $f:[m]\rightarrow\stn^\inrt_{n,/x}$ we can use the same reasoning to define a morphism $g_f:[m*n]^0\xrightarrow{}\twar^C_\theta(x)$ and again define $F(f)$ to be the composition $[m]\xactive{a_n}[m*n]\xrightarrow{g_f^0}\twar_\theta(x)$, where $a_n(j)\bydef n*j$. We now need to prove that this actually defines a natural transformation of functors, i.e. that for any $h:[l]\rightarrow [m]$ we have a commutative diagram
		\begin{equation}\label{eq:comm_diag}
			\begin{tikzcd}[sep=huge]
				{[l]} & {[m]} \\
				& {\twar_\theta(x)}
				\arrow["h", from=1-1, to=1-2]
				\arrow["{F(f\circ h)}"', from=1-1, to=2-2]
				\arrow["{F(f)}", from=1-2, to=2-2]
			\end{tikzcd}.
		\end{equation}
		Note that \eqref{eq:comm_diag} clearly commutes for inert $h$, so we may assume $h=a$ is active, moreover by construction it suffices to check that for any inert morphism $i:[1]\xinert{}[l]$ we have $F(f|_{\ima(a\circ i)})\circ a\circ i\cong F(f\circ a\circ i)$, so we may assume $l=1$. Denote by $S_m\bydef l_{n-1}l_{n-2}...l_0 l_{n-1}...l_0$ of length $n*m$ that contains each term $l_i$ $m$ times, then by construction, \Cref{lem:sq_ext} and \Cref{lem:sqs_cat} we can extend $g^0_f:[m*n]\rightarrow \twar_\theta(x)$ to $g^S_f:\sq(S_m)\rightarrow\twar_\theta(x)$ such that $g^0_f$ equals the composition $[m]\xrightarrow{\Delta_m}\sq(S_m)\xrightarrow{g^S_f}\twar_\theta(x)$, where $\Delta_m(i) = (i,...,i)\in\sq(S_m)$. Note that for the active morphism $a:[1]\xactive{}[m]$ the induced morphism $a^{\times n}:[1]^{\times n}\rightarrow[m]^{\times n}$ restricts to a morphism $a^S:\sq(S_l)\rightarrow\sq(S_m)$ between the respective subcategories, also note that $\sq(S_1)\cong [n]$. Denote by $S^0_m$ the string $l^n_{n-1}l^n_{n-2}...l^n_0$ and by $i_{S^0_m}:[n*m]\rightarrow\sq(S_m)$ the corresponding injective morphism defined in the proof of \Cref{lem:bruh}, then the morphism $a^S:[n]\cong \sq(S_1)\rightarrow\sq(S_m)$ can be identified with $[n]\xactive{a_m}[n*m]\xrightarrow{i_{S_m^0}}\sq(S_m)$, where the restriction of $a_m$ to every elementary interval $[1]\xinert{(i<i+1)}[n]$ is isomorphic to $[1]\xactive{}[m]$, so it suffices to prove that for each $i$ the following diagram 
		\begin{equation}\label{eq:comm_j}
			\begin{tikzcd}[sep=huge]
				{\sq(S_1)\cong [n]} & {\twar_\theta(x)} \\
				{[1]} \\
				{\sq(S_1)\cong [n]} & {\sq(S_m)}
				\arrow["{g^S_{f\circ a}}", from=1-1, to=1-2]
				\arrow["{j_i}", tail, from=2-1, to=1-1]
				\arrow["{j_i}"', tail, from=2-1, to=3-1]
				\arrow["{a^S}"', from=3-1, to=3-2]
				\arrow["{g^S_f}"', from=3-2, to=1-2]
			\end{tikzcd}
		\end{equation}
		\[\]
		commutes, where $j_i:[1]\xinert{}[n]$ denotes the inclusion of $(i<i+1)$.\par
		Given $f:[m]\rightarrow \stn^\inrt_{n,/\theta}$ the corresponding morphism $g_f:[m*n]^0\rightarrow \twar^C_\theta(x)$ can be identified (as in the diagram \eqref{eq:el_coc}) with a natural transformation $\alpha:H_f\rightarrow H_x$ between functors $\Delta^\el_{/[n*m]}\rightarrow \stn_n$, where $H_x$ is the constant functor with value $x$ and $H_f$ is a functor sending $[0]\xinert{}[n*m]$ to $c_n$ and $[1]\xinert{k,k+1}[n*m]$ to $C^{k\mod m}_n(1)$, note that both of those functors admit left Kan extensions along $i^\el:\Delta^\el_{/[n*m]}\hookrightarrow\Delta^\inrt_{/[n*m]}$ such that $i^\el_! H_x$ is a constant functor with value $x$ and $i^\el_! H_f([l]\xinert{i'}[n*m])$ is the object of $\ctree_n^h$ corresponding to the linear tree of length $l$ where the $k$th node is marked with $H_f([1]\xinert{i'(k)<i'(k+1)}[n*m])$, in particular $i^\el_! H_f(\id_{[n*m]})\cong D_n(m,0,...,0)$ in the notation of \Cref{constr:D_obj}. It then follows by construction that the morphism $g^S_f:\sq(S_m)\rightarrow \twar_\theta(x)$ factors as 
		\[\sq(S_m)\xrightarrow{g^{S,'}_f}\twar_\theta(D_n(m,...,0))\xrightarrow{\alpha_{\id_{[n*m]},!}}\twar_\theta(x),\]
		hence to prove the commutativity of \eqref{eq:comm_j} we may assume $x=D_n(m,...,0)$ and $f:[m]\rightarrow \stn^\inrt_{n,/D_n(m,0,...,0)}$ is given by the string
		\begin{equation}\label{eq:D_string}
			c_n\cong D_n(0,...,0)\xinert{D(i_+^0)}D_n(1,...,0)\xinert{D(i_+^1)}...\xinert{D(i_+^m)}D_n(m,...,0),
		\end{equation}
		where $i^k_+:[k]\xinert{}[k+1]$ is the unique inert morphism that preserves the maximal element (see \Cref{prop:D_func} for the functoriality of $D_n$).\par
		Denote by $g_D:[m]\rightarrow \stn^\inrt_{n,/D_n(m,...,0)}$ the morphism corresponding to the string \eqref{eq:D_string}. We can identify $e\in[n*m]$ with a pair $(i,j)$ with $i\in[n]$ and $j\in[m]$, in that case the value of $g_D^0:\Delta^\el_{/[n*m]}\rightarrow \Theta_{n,/D_n(m,...,0)}$ on $(k,j)<(k+1,j)$ is given by the morphism $C_n^k(1)\rightarrow D_n(m,...,0)$ sending the elementary cells $i^\pm_k:c_n\xinert{} C_n^k(1)$ to $((\delta_1,...,\delta_n),(j,0))$, where $\delta_{k+1}= \pm$ and $\delta_i=*$ for $i\neq k+1$. Denote by $g^0_{D_0}:\Delta^\el_{/[n*m]}\rightarrow \Theta_{n,/D_n(m,...,0)}$ the morphism corresponding to the composition $[n*m]^0\xrightarrow{i_{S_m^0}}\sq(S_m)\xrightarrow{g^S_D}\twar_\theta(D_n(m,...,0))$, then we need to prove that the value of $g^0_{D_0}$ on $(k,j)<(k+1,j)$ is given by $C_n^j(1)\rightarrow D_n(m,...,0)$ sending the elementary cells $i^\pm_j:c_n\xinert{} C_n^j(1)$ to $((\delta_1,...,\delta_n),(k,0))$, where $\delta_{n-j}= \pm$ and $\delta_i=*$ for $i\neq n-j$. Finally, note that by induction on $m$ we may assume $m=2$.\par
		Observe that $\sq(S_2)$ is isomorphic to the full subcategory of $[n]\times[n]$ containing $(i,j)$ where $i\geq j$. We can represent all the morphisms involved by a diagram
		\begin{equation}\label{eq:sq_twar}
			\begin{tikzcd}[sep=huge]
				&&&& {f_{n,n}} \\
				&& {} & {f_{n-1,n-1}} & {f_{n,n-1}} \\
				& {f_{1,1}} & {f_{2,1}} & {f_{n-1,1}} & {f_{n,1}} \\
				{f_{0,0}} & {f_{1,0}} & {f_{2,0}} & {f_{n-1,0}} & {f_{n,0}}
				\arrow["{g_{n,n-1}}", from=2-4, to=2-5]
				\arrow["{h_{n,n}}"', from=2-5, to=1-5]
				\arrow["{g_{2,1}}", from=3-2, to=3-3]
				\arrow["{...}"{description}, draw=none, from=3-3, to=2-3]
				\arrow["{...}"{description}, from=3-3, to=3-4]
				\arrow["{...}"{description}, from=3-4, to=2-4]
				\arrow["{g_{n,1}}", from=3-4, to=3-5]
				\arrow["{...}"{description}, from=3-5, to=2-5]
				\arrow["{g_{1,0}}", from=4-1, to=4-2]
				\arrow["{h_{1,1}}", from=4-2, to=3-2]
				\arrow["{g_{2,0}}", from=4-2, to=4-3]
				\arrow["{h_{2,1}}"{description}, from=4-3, to=3-3]
				\arrow["{...}"{description}, from=4-3, to=4-4]
				\arrow["{h_{n-1,1}}"{description}, from=4-4, to=3-4]
				\arrow["{g_{n,0}}", from=4-4, to=4-5]
				\arrow["{h_{n,1}}"', from=4-5, to=3-5]
			\end{tikzcd},
		\end{equation}
		where the original morphism $g_D^0:[2n]\rightarrow \twar^C_\theta(x)$ gives us the morphisms $g_{s,0}$ and $h_{n,k}$ such that $g_{s,0}$ is a morphism in $\twar^{s-1}_\theta(x)$ that corresponds to the data of a cospan $c_n\xinert{i^*_{s-1}} C^{s-1}_n(1)\twoheadleftarrow c_n$ over $D_n(2,...,0)$ and a morphism $C^{s-1}_n(1)\xrightarrow{g'_{i,0}}D_n(2,...,0)$ that takes the cells $i^\pm_{s-1}$ to $((\delta_1,...,\delta_n),(1,0))$ with $\delta_{n+1-s}=\pm$ and $\delta_i=*$ for $i\neq n+1-s$ and the cell $i^*_{s-1}$ to the composition of cells $((\delta'_1,...,\delta'_n),(t,0))$ with $t\in \{1,2\}$ and such that for $t=1$ all $\delta'_i=*$ for $i\neq v$ and $\delta_v=\pm$ for all values of $v> n+1-s$. Similarly, $h_{n,s}$ corresponds to the data of a cospan $c_n\xinert{i^*_{s-1}} C^{s-1}_n(1)\twoheadleftarrow c_n$ over $D_n(2,...,0)$ and a morphism $C^{s-1}_n(1)\xrightarrow{g'_{i,0}}D_n(2,...,0)$ that takes the cells $i^\pm_{s-1}$ to $((\delta_1,...,\delta_n),(0,0))$ with $\delta_{n+1-s}=\pm$ and $\delta_i=*$ for $i\neq n+1-s$ and the cell $i^*_{s-1}$ to the composition of cells $((\delta'_1,...,\delta'_n),(t,0))$ with $0\leq t\leq 2$ and such that $\delta'_i = *$ for $i\neq v$ that can take any values if $t=1$ or $v> n+1-s$ if $t=0$. We need to prove that the morphism $g'_{k,k-1}: C^{k-1}_n(1)\rightarrow D_n(2,...,0)$ corresponding to $g_{k,k-1}$ is such that it takes the cells $i^\pm_{k-1}$ to $((\delta_1,...,\delta_n),(1,0))$ with $\delta_{n+1-k}=\pm$ and $\delta_i=*$ for $i\neq n+1-k$ and the cell $i^*_{k-1}$ to the composition of cells $((\delta'_1,...,\delta'_n),(t,0))$ with $t\leq 2$ such that $\delta'_{i\neq v}=*$, $\delta_v=\pm$ with $v> n+1 -k$ and the morphism $h_{k,k}: C^{k-1}_n(1)\rightarrow D_n(2,...,0)$ corresponding to $h_{k,k}$ is such that it takes the cells $i^\pm_{k-1}$ to $((\delta_1,...,\delta_n),(0,0))$ with $\delta_{n+1-k}=\pm$ and $\delta_i=*$ for $i\neq n+1-k$ and the cell $i^*_{k-1}$ to the composition of cells $((\delta'_1,...,\delta'_n),(t,0))$ with $t\leq 2$ such that $\delta'_{i\neq v} = \pm$, $\delta_v=\pm$ for $v> n+1 -k$ or $v=n+1-k$ and $t=1$.\par
		We first prove the following claim: $h'_{i,1}:C^{0}_n(1)\rightarrow D_n(2,...,0)$ sends the cells $i^\pm_0$ to $((*,*,...,\pm),(0,0))$ and $i^*_0$ to the composition of cells $((\delta'_1,...,\delta'_n),(t,0))$ with $0<t\leq 2$ such that $\delta'_{i\neq v}=*$, $\delta_v=\pm$ with $v> n+1-i$ while $g'_{s,1}:C_n^{s-1}\rightarrow D_n(2,...,0)$ sends the cells $i^\pm_s$ to $((\delta_1,...,\delta_n),(1,0))$ with $\delta_{n+1-s}=\pm$ and $\delta_i=*$ otherwise and the cell $i^*_s$ to the composition of cells $((\delta'_1,...,\delta'_n),(t,0))$ with $0\leq t\leq 2$ such that $\delta'_{i\neq v}=*$, $\delta_v=\pm$ for $v>n+1-s$ if $t=1$ or $v=n$ for $t=0$. Assume first that $i=n$, in that case the composable pair $(h_{n,1}, g_{n,0})$ gives rise to a morphism $C_n(1,0,...,0,1)\rightarrow D_n(2,...,0)$ such that it takes $i^\pm_0$ to $((*,...,\pm),(0,0))$, $i^\pm_{n-1}$ to $((\pm,...,*),(1,0))$ and $c_n\cong C_n(0,...,0)\xinert{} C_n(1,0,...,0,1)$ to the composition $c_n\xinert{i_{n-1}^*} C_n^{n-1}(1)\xrightarrow{g'_{n,0}}D_n(2,...,0)$, by construction $h'_{n-1,1}: C^0_n(1)\rightarrow D_n(2,...,0)$ is given by the composition $C_{n}(1,0,....,0)\xinert{(\id,...,\id,\{1\})}C_{n}(1,0,....,1)$ and $g'_{n,1}$ is given by the composition $C_n(0,...,0,1)\xactive{(\{0\},\id,...,\id)}C_n(1,...,0,1)$, from this the required description follows. Assume we have now proved the required description for $h_{i,1}$ with $i\geq u$ and $g_{i,1}$ with $i\geq (u+1)$, then the composable pair $(h_{u-1,1}, g_{u,1})$ again defines a morphism $C_n(1,...,1,...,0)\rightarrow D_n(2,...,0)$ (where the only non-zero terms are in positions 1 and $u$) such that it takes $i^\pm_0$ to $((*,...,\pm),(0,0))$, $i^\pm_{u-1}$ to $((*,...,\pm,...,*),(1,0))$ with the only non-trivial term in position $(n-u)$ and $c_n\cong C_n(0,...,0)\xinert{} C_n(1,0,...,0,1)$ to the composition $c_n\xinert{i_{n-1}^*} C_n^{n-1}(1)\xrightarrow{g'_{u,0}}D_n(2,...,0)$, by construction $h'_{u-1,1}: C^0_n(1)\rightarrow D_n(2,...,0)$ is given by the composition $C_{n}(1,0,....,0)\xinert{(\id,...,\{1\},\id,...,\id)}C_n(1,...,1,...,0)$ and $g'_{u,1}$ is given by the composition $C_n(0,...,1,...,0)\xactive{(\id,...,\{0\},...,\id)}C_n(1,...,1,...,0)$, from this the required description follows.\par
		We now claim that $g'_{q,k};C^{q-1}_n(1)\rightarrow D_n(2,...,0)$ with $q>k$ takes the cells $i^\pm_q$ to $((\delta_1,...,\delta_n),(1,0))$ with $\delta_{n+1-q}=\pm$ and $\delta_i=*$ otherwise and the cell $i^*_q$ to the composition of cells $((\delta'_1,...,\delta'_n),(t,0))$ with $0\leq t\leq 2$ such that $\delta'_{i\neq v}=*$, $\delta_v=\pm$ for $v>n+1-q$ if $t=1$ or $v\geq n+1-k$ for $t=0$ and that $h'_{k+1,q}:C^{k}_n(1)\rightarrow D_n(2,...,0)$ with $q\geq k+1$ sends the cells $i^\pm_k$ to $((*,...,\pm,...,*),(0,0))$ with the only non-trivial $\delta$ in position $(n-k)$ and $i^*_0$ to the composition of cells $((\delta'_1,...,\delta'_n),(t,0))$ with $t\leq 2$ such that $\delta'_{i\neq v}=\pm$, $\delta_v=\pm$ with either $t=1$ and $v> n+1-q$ or $t=0$ and $v>k+1$. Indeed, since we have already proved the claim for $k=0$, we may assume that it has been proved for $(k-1)$, so that we have the description of morphisms $g'_{q,k-1}$ and $h_{n,k}$, the proof then follows by the same argument described in the previous paragraph by changing the indices. This concludes the proof that $F$ is well-defined.\par
		Now, it is immediate that $G\circ F= \id$, we need to show that $F\circ G\cong \id$, so assume we have $f:\Delta^\inrt_{/[q]}\rightarrow \Theta_{n,/x}$ in $\twar^C_\theta(x)([q])$, denote $\alpha:C_n(q)\rightarrow x$ the value $f(\id_{[q]})$, then it follows from the construction that the morphism $g_{G(f)}^S:\sq(S_q)\rightarrow \twar_\theta(x)$ appearing in the construction of $F$ factors as $\sq(S_q)\xrightarrow{g'} \twar_\theta(C_n(q))\xrightarrow{\alpha_!}\twar_\theta(x)$, hence we may assume $x= C_n(q)\in \Theta_n$, in which case the result follows immediately from \Cref{lem:theta_twar}.
	\end{proof}
	\begin{cor}\label{cor:stein_inrt}
		We have an isomorphism
		\[\stn^\inrt_{n,/x}\cong \underset{(\theta\xrightarrow{f} x)\in \Theta_{n,/x}}{\colim}\Theta^\inrt_{n,/\theta}.\]
	\end{cor}
	\begin{proof}
		Combine \Cref{lem:twar_theta1}, \Cref{lem:theta_twar} and \Cref{prop:stn_inrt}.
	\end{proof}
	\begin{prop}\label{prop:def}
		Given $\cE\in\cat_n$ denote by $L_\cE\bydef \Sigma^\infty (*)\in\stab(\cat_{n,/\cE})$ and for $f:\cE\rightarrow \cD$ denote 
		\[L_{\cE/\cD}\bydef \cok(f_! L_{\cE}\rightarrow L_\cD),\]
		then $f$ is an isomorphism if and only if the following conditions hold:
		\begin{enumerate}
			\item\label{it:tau} $\tau_{\leq n+1}f:\tau_{\leq n+1}\cE\rightarrow\tau_{\leq n+1}\cD$ is an isomorphism, where $\tau_{\leq n+1}:\cat_n\rightarrow\cat_{(n+1,n)}$ is the left adjoint to the inclusion of $(n+1,n)$-categories;
			\item\label{it:cot} $L_{\cE/\cD}\cong 0$.
		\end{enumerate}
	\end{prop}
	\begin{proof}
		The conditions are clearly necessary, so we need to prove they are sufficient. It suffices to prove that $f^*: h_\cD\rightarrow h_\cE$ is an isomorphism, where $h_\cE:\cat_n\rightarrow\cS$ is a functor corepresented by $\cE$ and similarly for $h_{\cD}$. For any $C\in\cat_{(n+1,n)}$ and $\cA\in\cat_n$ any morphism $\cA\rightarrow C$ factors through $\tau_{\leq n+1}\cA$, so condition \eqref{it:tau} implies that the restrictions of $h_\cD$ and $h_{\cE}$ to $\cat_{(n+1,n)}$ are isomorphic. Assume we have proved that $f^*$ restricts to an isomorphism on $\cat_{(m,n)}$ for some $m>n+1$, we will prove that $f^*$ also defines an isomorphism on $\cat_{(m+1,n)}$; since Postnikov towers converge for $\cat_n$ this will conclude the proof. \par
		Given $A\in\cat_{(m,n)}$ using \cite[Theorem 5.2.]{harpaz2020k} we can form a pullback square
		\[\begin{tikzcd}[sep=huge]
			A & {\tau_{\leq n+1}A} \\
			{\tau_{\leq m}A} & {\Omega^\infty(\Sigma^{m+1} \mathrm{H}\pi_m(A))}
			\arrow[from=1-1, to=1-2]
			\arrow[from=1-1, to=2-1]
			\arrow["\ulcorner"{anchor=center, pos=0.125}, draw=none, from=1-1, to=2-2]
			\arrow[from=1-2, to=2-2]
			\arrow[from=2-1, to=2-2]
		\end{tikzcd}.\]
		using our inductive assumption (and the adjunction $\Sigma^\infty\dashv \Omega^\infty$) we see that it remains to show that for any morphism $p:\cD\rightarrow\tau_{\leq n+1}A$ we have an isomorphism
		\[\mor_{\stab(\cat_{n,/\tau_{\leq n+1}A})}(p_! L_\cE,\Sigma^{m+1} \mathrm{H}\pi_m(A))\cong \mor_{\stab(\cat_{n,/\tau_{\leq n+1}A})}(p_! L_\cD,\Sigma^{m+1} \mathrm{H}\pi_m(A)),\]
		which clearly follows from \eqref{it:cot}.
	\end{proof}
	\begin{theorem}\label{thm:stein}
		For $x\in\stn_n$ we have an isomorphism
		\begin{equation}\label{eq:stein_free}
			x\cong \underset{(e\xinert{i}x)\in\stn^\el_{n,/x}}{\colim}e.
		\end{equation}
	\end{theorem}
	\begin{proof}
		Denote by $X$ the colimit in the right-hand-side of \eqref{eq:stein_free}, we have a natural morphism $f:X\rightarrow x$, by \Cref{prop:def} it suffices to prove that $L_f\cong 0$ and that $\tau_{\leq n+1}f$ is an isomorphism. By definition the value of $L_f$ at $y\xinert{i}x$ is given by
		\[\cok(|(\stn^\el_{n,/x})_{/y}|\otimes\bS\rightarrow\bS),\]
		so it suffices to prove that $(\stn^\el_{n,/x})_{/y}\cong \stn^\el_{n,/x}$ is contractible, however this follows from \Cref{prop:stn_cont} since $y\in\stn_n$. That $\tau_{\leq n+1}f$ is an isomorphism follows from the freeness property of \cite{ara2023categorical}.
	\end{proof}
	\begin{cor}\label{cor:stein_push}
		A pushout diagram in $\stn_n$ of the form 
		\[\begin{tikzcd}[sep=huge]
			y & x \\
			z & w
			\arrow["i", tail, from=1-1, to=1-2]
			\arrow["a"', two heads, from=1-1, to=2-1]
			\arrow[two heads, from=1-2, to=2-2]
			\arrow[tail, from=2-1, to=2-2]
			\arrow["\ulcorner"{anchor=center, pos=0.125, rotate=180}, draw=none, from=2-2, to=1-1]
		\end{tikzcd}\]
		remains a pushout in $\cat_n$.
	\end{cor}
	\begin{proof}
		This follows by the same argument as in \Cref{lem:C_push} using \Cref{cor:rel_seg} and the fact that the colimit \eqref{eq:stein_free} is preserved by $\stn_n\hookrightarrow\cat_n$.
	\end{proof}
	\section{$\Theta_n$-trees}
	In \Cref{sect:stein} we have introduced the formalism of strong Steiner complexes and proved a number of results regarding them. In practice, however, it is not easy to prove that a particular ADC defines a strong Steiner complex since it is difficult to demonstrate that its basis is loop-free. In the present section we introduce $\Theta_n$-trees -- essentially they are given by iterated active/inert pushouts of objects of $\Theta_n$. Not all such pushouts lie in $\stn_n$, so we need additional constraints to ensure this is the case. We accomplish this by introducing a subcategory of \textit{healthy trees} in \Cref{def:tree_health}, the notion that was inspired by the healthy objects of $\Theta_n$ defined in \cite{ayala2014configuration}. The main results of this section are \Cref{prop:strong_tree} proving that healthy trees indeed define objects of $\stn_n$ and \Cref{prop:health_push} which shows that they are closed under certain pushouts in $\cat_n$ -- the fact that will be used in the next section in construction of $\twar(\cE)$
	\begin{construction}\label{constr:tree}
		By a \textit{tree} we mean a finite rooted tree, each node except for the root admits one incoming edge and a finite number of outgoing edges, nodes with no outgoing edges are called \textit{leaves}. We will denote nodes of the tree by $c$ and edges by $e$, for a given edge $e$ we will denote by $s(e)$ its source and by $t(e)$ its target (which are nodes of $t$). We will define a certain category $\ctree'_n$ using \Cref{constr:stein}, but first we define its underlying set as follows: the elements of $\ctree'_n$ are given by trees $t$ such that each node $c$ is marked with $\theta_c$ and to each edge $e$ with $s(e)=c$ and $t(e)=c'$ corresponds an inert morphism $c_j\xinert{i_e}\theta_c$ for some $j\leq n$ and moreover $\theta_{c'}\in \Theta_j$, we would refer to cells of the form $c_j\xinert{i_e}\theta_c$ for $c\in T_0(x)$ as \textit{marked} with $\theta_{t(e)}$. These markings are required to satisfy an additional condition, however to introduce it we will need some notation. Given a marked tree $x$ we will denote by $T(x)$ the associated tree, by $T(x)_0$ its set of nodes and by $L(x)\subset T_0(x)$ the subset of leaves, also given $c\in T_0(x)$ we will denote $\tau_{\geq c}$ the object whose underlying tree is $T(x)_{c/}$ and whose decorations are induced from $x$ and similarly we will denote by $\tau_{<c}x$ the object corresponding to the subtree of $T(x)_{/c}$ obtained by removing the maximal element with induced decorations. With these notations we also require that the following holds:
		\begin{enumerate}
			\item[(*)]\label{it:tree_cond} if $c_j\xinert{i_e}\theta_c$ is marked with $\theta_{t(e)}$ and $d^\pm_q \tau_{\geq t(e)}x\neq c_q$, then, if we denote by $i_{e'}$ the composition $c_q\xinert{i_q^\pm} c_j\xinert{i_e}\theta_c$, $i_{e'}$ is marked with $d^\pm_q \theta_{t(e)}$ and moreover $\tau_{\geq t(e')}x = d^\pm_q\tau_{\geq t(e)}x$.
		\end{enumerate}
		Define $T_m(x)$ to be the subtree of $T(x)$ that only contains nodes $c$ such that if $c\in \tau_{\geq t(e')}x$ for some marked edge $e'$ with $i_{e'}:c_k\xinert{}\theta_{c'}$, then there does not exist a marked edge $e$ such that $i_{e'}$ factors through $i_e:c_l\xinert{}\theta_{c'}$. We associate to $x\in\ctree'_n$ an ADC $C(x)$ with basis as follows: the basis elements of dimension $j\leq n$ are given by \textit{unmarked} cells $c_j\xinert{i_0}\theta_c$ for $c\in T_m(x)$, we will denote an element corresponding to such an inclusion by $[i_0, c]$. Given such an element we set 
		\begin{equation}\label{eq:unmarked}
			d^\pm [i_0,c]\bydef [i_0\circ i^\pm_{j-1},c]
		\end{equation}
		if $i_0\circ i_{j-1}^\pm$ is also unmarked. If it is marked denote by $e'$ the corresponding edge and pick some $i_e:c_k\xinert{}\theta_{c'}$ such that $t(e')$ lies in $d^\pm_{j-1}\tau_{\geq t(e)}x$, then by \eqref{it:tree_cond} we can view $\tau_{\geq t(e')}x$ as a subobject of $\tau_{\geq t(e)}x$ which does not depend on the choice of $e$, so in particular we can view unmarked edges of $\tau_{\geq t(e')}x$ as unmarked edges of $\tau_{\geq t(e)}x$, hence as basis elements of $C(x)$. With these considerations in mind, we set
		\begin{equation}\label{eq:marked}
			d^\pm [i_0,c]\bydef \sum_{c_0\in T_0(\tau_{\geq t(e')} x)}\sum_{c_{j-1}\xinert{i'}\theta_{c_0}} [i', c_0],
		\end{equation}
		where the first sum is taken over all nodes of $T(\tau_{\geq t(e')}x)$ and the second over all \textit{unmarked} $(j-1)$-cells. We will define augmentation by setting $e(c_0\xinert{}\theta_c)\bydef 1$. We will prove in \Cref{prop:adc_tree} below that this defines an ADC with basis, denote it by $C(x)$, we define morphisms from $x$ to $x'$ in $\ctree'_n$ to be the morphisms of $ADCs$ between $C(x)$ and $C(x')$. 
	\end{construction}
	\begin{notation}\label{not:tree_basis}
		Given $x\in\ctree'_n$ we will denote by $B(x)_\bullet$ the basis of the corresponding ADC $C(x)$, i.e. $B(x)_j$ for $0\leq j\leq n$ is the set of elementary unmarked $j$-cells in some $\theta_c$ for $c\in T_0(x)$.
	\end{notation}
	\begin{prop}\label{prop:adc_tree}
		The objects $C(x)$ defined in \Cref{constr:tree} for $x\in \ctree'_n$ are ADCs with basis.
	\end{prop}
	\begin{proof}
		The components $C(x)_i$ are obviously free and the basis is unital, so all we need to show is that $\partial\circ \partial\cong 0$. Fix some unmarked $c_j\xinert{i}\theta_c$ in $T_m(x)$, we will assume that both $i\circ i^-_{j-1}$ and $i\circ i^+_{j-1}$ are marked since all other cases are similar but easier, denote the corresponding edges by $e'$ and $e''$ respectively. By definition
		\begin{equation}
			\partial[i,c] = \sum_{c_0\in T_0(\tau_{\geq t(e')} x)}\sum_{c_{j-1}\xinert{i'}\theta_{c_0}} [i', c_0] - \sum_{c_1\in T_0(\tau_{\geq t(e'')} x)}\sum_{c_{j-1}\xinert{i'}\theta_{c_1}} [i', c_1],
		\end{equation}
		where the inner sums are taken over unmarked cells and as in \Cref{constr:tree} we have identified $\tau_{\geq t(e')}x$ and $\tau_{\geq t(e'')}x$ with subobjects of some $\tau_{\geq t(e)}x$ with $t(e)\in T_m(x)$. Also by definition we have
		\begin{equation}\label{eq:d2}
			\partial(\sum_{c_0\in T_0(\tau_{\geq t(e')} x)}\sum_{c_{j-1}\xinert{i'}\theta_{c_0}} [i', c_0]) = \sum_{c_0\in T_0(\tau_{\geq t(e')} x)}\sum_{c_{j-1}\xinert{i'}\theta_{c_0}} (d^+[i', c_0] - d^-[i', c_0]).
		\end{equation}
		Note that all the terms corresponding to $(j-2)$-cells that are both the positive and the negative boundary of some $(j-1)$-cells cancel out, so we can rewrite \eqref{eq:d2} as 
		\begin{equation}
			\sum_{c_0\in T_0(\tau_{\geq t(e^{',+}_{j-2})} x)}\sum_{c_{j-2}\xinert{i'}\theta_{c_0}}[i',c_0] - \sum_{c_1\in T_0(\tau_{\geq t(e^{',-}_{j-2})} x)}\sum_{c_{j-2}\xinert{i''}\theta_{c_1}}[i'',c_1].
		\end{equation}
		The same reasoning also applies to $e''$, the claim now follows from the fact that $e^{',\pm}_{j-2} = e^{'',\pm}_{j-2}$ and the markings on this cell induced from $e'$ and $e''$ agree by definition of $\ctree'_n$.
	\end{proof}
	\begin{prop}\label{prop:down_tree}
		Given $x\in\ctree'_n$ and $c\in T_m(x)$, denote by $\tau_{<c}x$ the decorated tree whose underlying tree is $T(x)_{<c}$ with decorations induced from $x$, then $\tau_{<c}x\in \ctree'_n$.
	\end{prop}
	\begin{proof}
		We need is to prove that condition \eqref{it:tree_cond} holds for $\tau_{<c}x$, however note that by definition for any marked cell $e$ with a $q$-boundary $e'$ we have $\tau_{\geq t(e')}x\cong \tau_{\geq t(e')}\tau_{<c}x$ since by construction $c$ does not belong to $\tau_{\geq t(e')}x$ and similarly $d^\pm_q\tau_{\geq t(e)}x = d^\pm_q\tau_{\geq t(e)}\tau_{<c}x$ since $c$ does not belong to $d^\pm_q\tau_{\geq t(e)}x$.
	\end{proof}
	\begin{remark}
		Note that it is \textit{not} true in general that $T_m(\tau_{<c}x)\cong T_m(x)_{<c}$.
	\end{remark}
	\begin{defn}\label{def:tree}
		Denote by $\ctree_n\hookrightarrow\ctree'_n$ the full subcategory of $\ctree'_n$ containing objects $x\in \ctree_n'$ as above such that the bases of all $C(\tau_{<c}x)$ for $c\in T_m(x)$ are strongly loop-free in the sense of \Cref{def:stein}, so that $\tau_{<c}x$ are strong Steiner complexes.
	\end{defn}
	\begin{ex}\label{eq:counter_tree}
		The condition that all $C(\tau_{< c}x)$ are strongly loop-free does not follows from only $C(x)$ being strongly loop-free: for example consider the following 2-composable pair of 3-morphisms:
		
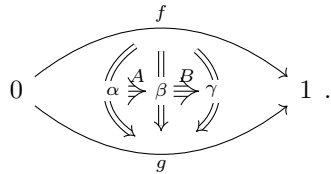
\[\begin{tikzcd}[sep=huge]
			0 && 1
			\arrow[""{name=0, anchor=center, inner sep=0}, "g"', curve={height=30pt}, from=1-1, to=1-3]
			\arrow[""{name=0p, anchor=center, inner sep=0}, phantom, from=1-1, to=1-3, start anchor=center, end anchor=center, curve={height=30pt}]
			\arrow[""{name=1, anchor=center, inner sep=0}, "f", curve={height=-30pt}, from=1-1, to=1-3]
			\arrow[""{name=1p, anchor=center, inner sep=0}, phantom, from=1-1, to=1-3, start anchor=center, end anchor=center, curve={height=-30pt}]
			\arrow[""{name=2, anchor=center, inner sep=0}, "\beta"{description}, shorten <=8pt, shorten >=8pt, Rightarrow, from=1, to=0]
			\arrow[""{name=3, anchor=center, inner sep=0}, "\gamma"{description}, shift left=2, curve={height=-18pt}, shorten <=10pt, shorten >=10pt, Rightarrow, from=1p, to=0p]
			\arrow[""{name=4, anchor=center, inner sep=0}, "\alpha"{description}, curve={height=24pt}, shorten <=11pt, shorten >=11pt, Rightarrow, from=1, to=0]
			\arrow["B", shorten <=4pt, shorten >=4pt, Rightarrow, scaling nfold=3, from=2, to=3]
			\arrow["A", shorten <=5pt, shorten >=5pt, Rightarrow, scaling nfold=3, from=4, to=2]
		\end{tikzcd}.\]
		We can view it as an object of $\ctree_3$, mark both $A$ and $B$ with the object
		\[\begin{tikzcd}
			0 & 1 && 2 & 3
			\arrow["a"', from=1-1, to=1-2]
			\arrow[""{name=0, anchor=center, inner sep=0}, "g"', curve={height=30pt}, from=1-2, to=1-4]
			\arrow[""{name=0p, anchor=center, inner sep=0}, phantom, from=1-2, to=1-4, start anchor=center, end anchor=center, curve={height=30pt}]
			\arrow[""{name=1, anchor=center, inner sep=0}, "f", curve={height=-30pt}, from=1-2, to=1-4]
			\arrow[""{name=1p, anchor=center, inner sep=0}, phantom, from=1-2, to=1-4, start anchor=center, end anchor=center, curve={height=-30pt}]
			\arrow["b"', from=1-4, to=1-5]
			\arrow[""{name=2, anchor=center, inner sep=0}, "\gamma"{description}, shift left=2, curve={height=-18pt}, shorten <=10pt, shorten >=10pt, Rightarrow, from=1p, to=0p]
			\arrow[""{name=3, anchor=center, inner sep=0}, "\alpha"{description}, curve={height=24pt}, shorten <=11pt, shorten >=11pt, Rightarrow, from=1, to=0]
			\arrow["X"{description}, shorten <=9pt, shorten >=9pt, Rightarrow, scaling nfold=3, from=3, to=2]
		\end{tikzcd}\]
		and denote the resulting object of $\ctree_3$ by $x$ and the corresponding nodes of $T(x)$ by $c_A$ and $c_B$. Then the 3-category corresponding to $x$ is an object of $\Theta_3$ which we may depict as 
		\[\begin{tikzcd}[sep=huge]
			0 & 1 && 2 & 3
			\arrow["a"', from=1-1, to=1-2]
			\arrow[""{name=0, anchor=center, inner sep=0}, "g"', curve={height=30pt}, from=1-2, to=1-4]
			\arrow[""{name=0p, anchor=center, inner sep=0}, phantom, from=1-2, to=1-4, start anchor=center, end anchor=center, curve={height=30pt}]
			\arrow[""{name=1, anchor=center, inner sep=0}, "f", curve={height=-30pt}, from=1-2, to=1-4]
			\arrow[""{name=1p, anchor=center, inner sep=0}, phantom, from=1-2, to=1-4, start anchor=center, end anchor=center, curve={height=-30pt}]
			\arrow["b"', from=1-4, to=1-5]
			\arrow[""{name=2, anchor=center, inner sep=0}, "\beta"{description}, shorten <=8pt, shorten >=8pt, Rightarrow, from=1, to=0]
			\arrow[""{name=3, anchor=center, inner sep=0}, "\gamma"{description}, shift left=2, curve={height=-18pt}, shorten <=10pt, shorten >=10pt, Rightarrow, from=1p, to=0p]
			\arrow[""{name=4, anchor=center, inner sep=0}, "\alpha"{description}, curve={height=24pt}, shorten <=11pt, shorten >=11pt, Rightarrow, from=1, to=0]
			\arrow["B", shorten <=4pt, shorten >=4pt, Rightarrow, scaling nfold=3, from=2, to=3]
			\arrow["A", shorten <=5pt, shorten >=5pt, Rightarrow, scaling nfold=3, from=4, to=2],
		\end{tikzcd}\]
		while $\tau_{< c_A}$ corepresents composable pair of 3-morphisms of the form
		\[\alpha\xrightarrow{A} b*\beta *a\xrightarrow{B}b*\gamma *a.\]
		In particular, $a$ and $b$ lie in both the positive and the negative boundary of $\alpha$, meaning that $C(\tau_{< c_A}x)$ cannot be strongly loop-free. 
	\end{ex}
	\begin{warning}
		An object $x\in \ctree_n$ may admit many presentations as a decorated tree as in \Cref{constr:tree}, for example if $x$ is such that all elementary cells in $\theta_c$ are marked, then such an object corresponds to a string of active morphisms $\theta^0\xactive{a^1}...\xactive{a^m}\theta^m$ in $\Theta$ and $C(x)$ is isomorphic to $C(\theta^m)$, where $\theta^m$ is identified with an object of $\ctree_n$ as in \Cref{rem:theta}. In particular, every object of $\ctree_1$ is isomorphic to one of this form, so we have $\ctree_1\cong \Delta$.
	\end{warning}
	\begin{notation}\label{not:elem}
		Assume we have a cell $i:c_j\xinert{}\theta_c$ in some $x\in \ctree_n$, if it is unmarked denote by $C(\{i,c\})$ the subcomplex of $C(x)$ containing the basis element $[i,c]$ and basis elements appearing in various $d^\pm ...d^\pm[i, c]$, if it is marked denote by $e$ the corresponding edge and by $\{i,c\}$ the subcomplex of $C(x)$ generated by the basis elements $[i', c']\in \tau_{\geq t(e)}x$. Note that in both cases it can be identified with an object of $\ctree_n$ corresponding to the subtree of $T(x)$ containing $c$ and the nodes in $\tau_{\geq t(e')}x$ for all marked $i_{e'}:c_k\xinert{}\theta_c$ that factor through $i$ such that $t(e')$ has markings induced from $x$ and $c$ is marked with $c_j$. It follows that the natural inclusion $C(\{i, c\})\hookrightarrow C(x)$ can be identified with an inert morphism $\{i_e,c\}\xinert{} x$ in $\ctree_n$. We will call such inert morphisms \textit{elementary} if the cell $i$ is unmarked.
	\end{notation}
	\begin{prop}\label{prop:up_tree}
		The objects $\{i,c\}$ of \Cref{not:elem} lie in $\ctree_n$.
	\end{prop}
	\begin{proof}
		Note that $\{i,c\}$ satisfies \eqref{it:tree_cond} since its marking is induced from the one on $x$, hence it suffices to show that for any $c_0\in T_m(\{i,c\})$ the object $C(\tau_{<c_0}\{i,c\})$ has a strongly loop-free basis. For this note that we can identify $C(\tau_{<c_0}\{i,c\})$ with a subcomplex of $C(\tau_{<c_0}x)$, it has a strongly loop-free basis by \Cref{def:tree}, hence the basis of $C(\tau_{<c_0}\{i,c\})$ must also be loop-free since any loop in it would also be a loop in $C(\tau_{<c}x)$, and hence must be trivial.
	\end{proof}
	\begin{lemma}\label{lem:push_tree}
		Given $x\in \ctree_n$ and an edge $e$ of $T(x)$ corresponding to $i_e:c_j\xinert{}\theta_c$ such that $t(e)\in T_m(x)$, then we can also identify $i_e$ with an unmarked cell in $\tau_{< t(e)}x$, denote it by $i_0$. With these notations we have a pushout diagram
		\begin{equation}\label{eq:tree_push}
			\begin{tikzcd}[sep=huge]
				{\{i_0,c\}} & {\tau_{<t(e)}x} \\
				{\{i,c\}} & x
				\arrow[tail, from=1-1, to=1-2]
				\arrow[two heads, from=1-1, to=2-1]
				\arrow[two heads, from=1-2, to=2-2]
				\arrow[tail, from=2-1, to=2-2]
				\arrow["\ulcorner"{anchor=center, pos=0.125, rotate=180}, draw=none, from=2-2, to=1-1]
			\end{tikzcd}
		\end{equation}
		in $\ctree_n$ which remains a pushout diagram in $\cat_n$.
	\end{lemma}
	\begin{proof}
		That all objects in the diagram \Cref{eq:tree_push} lie in $\ctree_n$ follows from \Cref{prop:up_tree} and \Cref{prop:down_tree}. We now need to define the active morphism $\{i_0,c\}\xactive{a}\{i,c\}$, we do so by sending each basis element in $d_j^\pm\{i_0,c\}$ to the corresponding element of $B_\bullet(\{i,c\})$ and the basis element $[i_0,c]$ to the sum of all unmarked $j$-cells in $\{i,c\}$. It is easy to see that \eqref{eq:tree_push} is then a pushout square of the corresponding complexes, hence defines a pushout square in $\ctree_n$, that it remains a pushout square in $\cat_n$ follows from \Cref{cor:stein_push}.
	\end{proof}
	\begin{cor}\label{cor:tree_stn}
		Any $x\in\ctree_n$ admits an active morphism $c_n\xactive{}x$, so that $x\in\stn_n$.
	\end{cor}
	\begin{proof}
		If $x\cong \theta\in\Theta_n$, then the claim follows from the elementary properties of $\Theta_n$, in the general case we may use induction on the size of $T(x)$ and represent $x$ as a pushout of the form \eqref{eq:tree_push}. By induction there is an active morphism $c_n\xactive{}\tau_{<t(e)}x$ and we can define the required morphism to be the composition $c_n\xactive{}\tau_{< t(e)}x\xactive{}x$. 
	\end{proof}
	\begin{defn}\label{def:health}
		We will call an object $\theta\in\Theta_n$ \textit{healthy} if every elementary $i$-cell in $\theta$ for $i<n$ lies in the boundary of a non-trivial elementary $n$-cell. We will call a morphism $f:\theta\rightarrow\theta'$ between healthy objects \textit{healthy} if the image of any non-trivial $n$-cell of $\theta$ in $\theta'$ is a healthy object.
	\end{defn}
	\begin{lemma}\label{lem:health}
		Assume $\theta$ is healthy in the sense of \Cref{def:health} with $n>1$, then there are no elementary cells $c_j\xinert{i}\theta$ that lie both in $d_j^-\theta$ and $d_j^+\theta$.
	\end{lemma}
	\begin{proof}
		We will prove the claim by induction on $n$: in the case $n=1$ healthy objects correspond to object $[m]$ with $m>0$ and their boundaries are just endpoints, which are distinct since $m\neq 0$. In the general case for $j>0$ observe that it necessarily factors through some $\theta_k\bydef \mor_\theta(k,k+1)$ (which are easily seen to also be healthy), from which the claim follows by induction, and for $j=0$ we need to show that the endpoints of $\theta$ are distinct, which follows since it contains at least one elementary $n$-cell.
	\end{proof}
	\begin{prop}\label{prop:strong_tree}
		Assume that $x\in \ctree_n'$ is such that:
		\begin{enumerate}
			\item\label{it:h1} the objects $\theta_*$ corresponding to the root node $*\in T_0(x)$ is healthy;
			\item\label{it:h2} for any edge $e$ corresponding to an inert morphism of the form $c_l\xinert{i_e}\theta_c$ we have $\theta_{t(e)}\in\Theta^h_l$,
		\end{enumerate}
		then $x\in \ctree_n$.
	\end{prop}
	\begin{proof}
		Note that all the subcomplexes $\tau_{<c}x$ also satisfy the conditions of the proposition, so it suffices to show that $x$ itself is loop-free. We will prove this by induction on the size of the underlying tree, starting with the case of a tree with a single root node. Such an objects is equivalent to an object of $\Theta_n$ and we again use induction, this time on $n$, note that the case $n=0$ is trivial. Assume we have shown that all $\theta'\in\Theta_{n-1}$ correspond to strong Steiner complexes and take $\theta\in\Theta_n$, then observe that for two elementary cells $a$ and $b$ in $\theta$ we have $a<_\bN b$ either if both $a$ and $b$ lie in $\theta'\bydef \mor_\theta(i,i+1)\in\Theta_{n-1}$ for some object $i$ of $\theta$ and $a<_\bN b$ as cells in $\theta'$ or if $a$ and $b$ are both 1-dimensional and $a$ belongs to $\mor_\theta(k-1,k)$ while $b$ belongs to $\mor_\theta(k,k+1)$ for some object $k$ of $\theta$. It follows that any potential loop $a_0<_\bN a_1<_\bN...<_\bN a_0$ should either lie entirely in some $\mor_\theta(i,i+1)$, which is impossible by induction, or induce a sequence $k_0<k_0+1<...<k_0$ of objects of $\theta$, which is also clearly impossible.\par
		Assume now that $x\in \ctree_n$ is obtained from $x'\bydef \tau_{\leq c}x$, for which we have shown that its basis is strongly loop-free, by marking a cell $c_j\xinert{i'}\theta_c$ for some $c\in L(x')$, adding an extra edge $e$ with source $c$ and decorating its target $t(e)$ with $\theta\in\Theta^h_j$ -- it is easy to see that all object of $\ctree'_n$ of the type described in the statement of the proposition can be inductively constructed like this starting with objects of $\Theta_n$, denote by $f:x'\xactive{} x$ the induced morphism. Assume we have a string $L\bydef (a_0<_\bN a_1<_\bN...<_\bN a_0)$ of basis elements of $x$, note that if all $a_j$ and $[a_j]^\pm_i$ lie in $x'$ or $\theta$, then the string can be identified with a string in the corresponding subcomplex and hence must be trivial by inductive assumption. Without loss of generality we may therefore assume that $a_0$ lies in the image of $x'$ and the string intersects the interior of $\theta$. In that case we may choose a minimal $k$ such that $a_{k+1}$ lies in $d^-_{j-1}\theta$, while $a_{k+1}$ does not, and similarly we can pick a minimal $l>k$ such that $a_{l+1}$ does not lie in the image of $d_{j-1}^+\theta$, while $a_l$ does. In this case we can write the segment between $a_k$ and $a_{l+1}$ as
		\begin{equation}\label{eq:str_h}
			a_k <_\bN a_{k+1} <_\bN... <_\bN a_s<_\bN... <_\bN a_t <_\bN... a_l <_\bN a_{l+1},
		\end{equation}
		where all elements in the segment between $a_{k+1}$ and $a_{s-1}$ lie in $d_{j-1}^-$, the segment from $a_s$ and $a_t$ lies in the interior of $\theta$ and the segment between $a_{s+1}$ and $a_l$ lies in $d^+_{j-1}\theta$. Note that for any $a_i$ with $k+1\leq i\leq s-1$ there is a unique basis element $a'_i$ in $x'$ such that $a_i$ appears in the decomposition of $f(a'_i)$ which moreover lies in $d^-_{j-1} \{i',c\}$. Indeed, if $i'\circ i^-_{j-1}$ is marked then we may identify $a_i$ with a basis element of $x'$, which would satisfy our requirements, and if it is unmarked then $a_i$ appears in the decomposition of $a([i'\circ i^-_{j-1}])$ by construction. In both cases there are clearly no other basis elements in $d^-_{j-1} \{i',c\}$ that satisfy these requirements, so the only way there may be another basis element $a''_i$ of $x$ satisfying them is if $a_i$ also lies in $d^+_{j-1}\theta$, however this is impossible by \Cref{lem:health} since $\theta$ was assumed to be healthy. Using a similar statement for $a_q$ with $t+1\leq q\leq l$ and $d^+_{j-1}\{i',c\}$ we can lift the string \eqref{eq:str_h} to a string
		\[a_k <_\bN a'_{k+1} <_\bN... <_\bN a'_s< [i',c] <_\bN... a'_l <_\bN a_{l+1}\]
		in $x'$, concatenating it with the remainder of the loop $a_0 <_\bN... <_\bN a_0$ we may produce a loop that intersects $\theta$ one less time. Using this process for all other intersections of the original loop with $\theta$ we can produce a loop $ L'\bydef (a_0 <_\bN a'_1 <_\bN... <_\bN a_0)$ that lies entirely in $x'$, however it must be then trivial by our inductive assumption. Note that by construction the restriction of $L'$ to the complement of $\{i', c\}$ coincides with the corresponding restriction of $L$, which means that it must be trivial, hence $L$ must lie entirely in the image of $\theta$, which once again violates out inductive assumption.
	\end{proof}
	\begin{defn}\label{def:tree_health}
		Call an object $x\in\ctree_n$ \textit{healthy} if it satisfies the conditions of \Cref{prop:strong_tree}, call a morphism $f:x\rightarrow y$ between healthy objects \textit{healthy} if for every elementary $n$-cell $i:\{i_e,c\}\xinert{}x$ (in the sense of \Cref{not:elem}) the object $x'$ appearing in the factorization square in $\stn_n$ (which exists by \Cref{prop:stein_fact}) 
		\[\begin{tikzcd}[sep=huge]
			{\{i,c\}} & {x'} \\
			x & y
			\arrow["a", two heads, from=1-1, to=1-2]
			\arrow["i"', tail, from=1-1, to=2-1]
			\arrow["{i'}", tail, from=1-2, to=2-2]
			\arrow["f"', from=2-1, to=2-2]
		\end{tikzcd}\]
		lies in $\ctree_n$ and is healthy.
	\end{defn}
	\begin{lemma}\label{lem:health_fact}
		Assume that we have a healthy morphism $f:x\rightarrow y$ and an inert morphism $\{i_l,c\}\xinert{}x$ for some $c\in T_0(x)$ and an unmarked cell $i_l:c_l\xinert{}\theta_c$, then the morphism $a_{i_l}$ appearing in the factorization diagram 
		\[\begin{tikzcd}[sep=huge]
			{\{i_l,c\}} & x \\
			{y_{i_l}} & y
			\arrow["i_l", tail, from=1-1, to=1-2]
			\arrow["{a_{i_l}}"', two heads, from=1-1, to=2-1]
			\arrow["f", from=1-2, to=2-2]
			\arrow["{i'_l}"', tail, from=2-1, to=2-2]
		\end{tikzcd}\]
		is also a healthy morphism in $\ctree_l^h$.
	\end{lemma}
	\begin{proof}
		We can assume without the loss of generality that $f$ is an active morphism $a:x\xactive{}y$. If $l=n$, then the claim follows immediately from the definition, so we may assume that $l<n$. In that case, since $x$ was assumed to be healthy, we can find some $i_n:c_n\xinert{}\theta_c$ such that $i_l = d^\pm_l i_n$. In that case we also have $y_{i_l}\cong d^\pm_l y_{i_n}$, so it suffices to show that a boundary $d_l^\pm y$ of a healthy object $y$ is a healthy object. Note that the subtree $T(d^\pm_l y)\subset T(y)$ contains the root node $*$ marked with $d^\pm_l \theta_*$, where $\theta_*$ is the marking of the root node in $y$, as well as all the nodes in $T(\tau_{\geq t(e)})$ for all edges corresponding to $i_e:c_q\xinert{}\theta_*$ that factor as $c_q\xinert{}d^\pm_l\theta_*\xinert\theta_*$ with markings induced from $y$. It follows immediately from this that $d^\pm_l$ satisfies condition \eqref{it:h2} of \Cref{prop:strong_tree}, so it suffices to show that a boundary of a healthy object $\theta_*\in\Theta^h_n$ lies in $\Theta^h_l$. So assume that for some $q<l$ we have $c_q\xinert{i_q}d^+_l\theta_*$, pick some $c_n\xinert{i_n}\theta_*$ such that $i_q=d^\pm i_n$. Note that if $i'_n:c_n\xinert{}\theta_*$ is such that $d^=_l i_n = d^-_l i'_n$, then we still have $i_q = d^\pm_q i'_n$, so we may assume that $i_l\bydef d^+_l i_n$ lies in $d^+_l\theta_*$, but then $i_q = d^\pm_q i_l$, meaning that $d^+_l\theta_*$ is healthy.
	\end{proof}
	\begin{notation}
		Given $x\in\ctree^h_n$ denote $\act^h_{\ctree_n}(x)$ the set of healthy active morphisms in the sense of \Cref{def:tree_health} with source $x$.
	\end{notation}
	\begin{prop}\label{prop:health_seg}
		For $x\in \ctree_n^h$ we have
		\begin{equation}\label{eq:health_seg}
			\act^h_{\ctree_n}(x)\cong \underset{(e\xinert{}x)\in\ctree^\el_{n,/x}}{\lim}\act_{\ctree_n}^h(e).
		\end{equation}
	\end{prop}
	\begin{proof}
		Note that we have a function
		\begin{equation}\label{eq:Fh}
			F_h:\act^h_{\ctree_n}(x)\rightarrow\underset{(e\xinert{}x)\in\ctree^\el_{n,/x}}{\lim}\act^h_{\ctree_n}(e)
		\end{equation}
		that sends an active morphism $x\xactive{a}z$ to the family of morphisms $a_i:e\xactive{}z_i$ appearing in the factorization square
		\[\begin{tikzcd}[sep=huge]
			e & x \\
			{z_i} & z
			\arrow["i", tail, from=1-1, to=1-2]
			\arrow["{a_i}"', two heads, from=1-1, to=2-1]
			\arrow["a", two heads, from=1-2, to=2-2]
			\arrow["{i'}"', tail, from=2-1, to=2-2]
		\end{tikzcd}\]
		indexed over all inert morphisms $i$ from elementary objects to $x$, the fact that $F$ lands in $\underset{(e\xinert{}x)\in\ctree^\el_{n,/x}}{\lim}\act^h_{\ctree_n}(e)$ follows directly from \Cref{lem:health_fact}.\par
		Our goal is to construct an inverse $G_h$ to the morphism $F_h$ of \eqref{eq:Fh}, so assume we have a compatible family of active morphisms $a_{i,c}:\{i,c\}\xactive{}z_{i,c}$ indexed over unmarked cells $c_l\xinert{i}\theta_c$ over all $c\in T_0(x)$, we first define a tree $T'$ to be a tree whose nodes are either nodes of $T(x)$ or nodes of some $z_{i,c}$, we set $c_1 <_{T'} c_2$ if either both belong to $T_0(x)$ and $c_0 <_{T_0(x)}c_1$, both lie in some $T(z_{i,c})$ and $c_0 <_{T(z_{i,c})} c_1$ or $c_0\in T_0(x)$, $c_1\in T(z_{i,c})$ and $c_0 <_{T_0(x)}c$, it is easy to see that this defines a rooted tree. It admits a decorations of nodes and edges induced from $x$ and various $z_{i,c}$, since all of them were assumed to be healthy it is immediate that $T'$ with these decorations also satisfies \eqref{it:h1} and \eqref{it:h2}, hence defines an object $z\in \ctree_n^h$. Moreover, we have a natural active morphism $a:x\xactive{}z$ taking an unmarked cell $\{i_l:c_l\xinert{}\theta_c,c\}$ to the composition of all $l$-morphisms in $z_{i_l,c}\cong T(\tau_{\geq t_{i_l,c}}z)$, where $t_{i_l,c}$ denotes the target of the edge in $T(z)$ corresponding to $c_l\xinert{i_l}\theta_c$, where $c$ is viewed as a node in $T(z)$. This defines the required morphism $G_h$. immediately from construction we have $F_h\circ G_h$, to prove $G_h\circ F_h\cong \id$ note that an active morphism $a:x\xactive{}y$ is uniquely determined by the images of the elementary cells $\{i,c\}$, hence $F_h$ is injective, meaning that it is an isomorphism since it admits a section.
	\end{proof}
	\begin{cor}\label{cor:health_fact}
		For any healthy active morphism $a:x\xactive{}y$ and any inert $x'\xinert{i}x$, the morphism $a'$ appearing in the factorization diagram
		\[\begin{tikzcd}[sep=huge]
			{x'} & x \\
			{y'} & y
			\arrow["i", tail, from=1-1, to=1-2]
			\arrow["{a'}"', two heads, from=1-1, to=2-1]
			\arrow["a", two heads, from=1-2, to=2-2]
			\arrow["{i'}"', tail, from=2-1, to=2-2]
		\end{tikzcd}\]
		is also healthy.
	\end{cor}
	\begin{proof}
		We have already proved a special case of this claim for $x'\cong \{i,c\}$ in \Cref{lem:health_fact}, so in particular for any $\{i',c'\}\xinert{}x'$ we have healthy active morphisms $a'_{i',c'}:\{i',c'\}\xactive{}y_{i',c'}$. These morphisms define an element of $\underset{\{i',c'\}\xinert{}x'}{\lim}\act^h_{\ctree_n}(\{i',c'\})$, hence we conclude by \Cref{prop:health_seg}.
	\end{proof}
	\begin{prop}\label{prop:health_cat}
		Healthy morphisms are closed under composition.
	\end{prop}
	\begin{proof}
		Observe that all inert morphisms are healthy and more generally a morphism $f\cong i\circ a$ is healthy if and only if its active part $a$ is healthy. So it suffices to prove that the healthy active morphisms are closed under composition and stable under active/inert factorization.\par
		Assume first that we have a composable pair $x\xactive{a}x'\xactive{a''}x''$ where both morphisms are healthy, we must prove that $a'\circ a$ is also healthy. For that we need to show that for any elementary cell $i:c_n\xinert{}x$ it image $x_i''$ in $x''$ is healthy. Denote by $x'_i$ the image of $i$ in $x'$, then $x''_i$ is isomorphic to the image of $x'_i$ under $a''$, which is healthy by \Cref{cor:health_fact}.\par
		Finally, we need to prove that for any factorization square
		\[\begin{tikzcd}[sep=huge]
			{x} & {x_1} \\
			{x_1'} & {x_{2}}
			\arrow["{i_0}", tail, from=1-1, to=1-2]
			\arrow["{a_0}"', two heads, from=1-1, to=2-1]
			\arrow["{a_1}", two heads, from=1-2, to=2-2]
			\arrow["{i_1}"', tail, from=2-1, to=2-2]
		\end{tikzcd}\]
		if $a_1$ is healthy, then so is $a_0$, however this is easy to see since the restriction of $a_0$ to any elementary $n$-cell of $x_0$ coincides with the restriction of $a_1$ to its image in $x_1$.
	\end{proof}
	\begin{notation}\label{not:health_cat}
		We will denote by $\ctree^h_n$ the subcategory (which is well-defined by \Cref{prop:health_cat}) of $\ctree_n$ on healthy objects and healthy morphisms. 
	\end{notation}
	\begin{prop}\label{prop:health_push}
		Assume that we have a span $y\lxinert{i}x\xactive{a}z$ in $\ctree_n^h$, then there is a pushout diagram
		\[\begin{tikzcd}[sep=huge]
			x & y \\
			z & w
			\arrow["i", tail, from=1-1, to=1-2]
			\arrow["a"', two heads, from=1-1, to=2-1]
			\arrow["{a'}", two heads, from=1-2, to=2-2]
			\arrow["{i'}"', tail, from=2-1, to=2-2]
			\arrow["\ulcorner"{anchor=center, pos=0.125, rotate=180}, draw=none, from=2-2, to=1-1]
		\end{tikzcd}\]
		in $\ctree_n$ which remains a pushout in $\cat_n$.
	\end{prop}
	\begin{proof}
		Note that we can always take the required pushout in the category $\ctree'_n$ by taking the pushout of the corresponding ADCs, we need to prove that the resulting object lies in $\ctree_n$, however under the conditions of the lemma this follows from \Cref{prop:strong_tree} since all nodes of $w$ are marked with healthy objects by construction. Finally, the fact that this remains a pushout in $\cat_n$ follows by iterated application of \Cref{lem:push_tree}.
	\end{proof}
	\section{Twisted arrow categories}\label{sect:twar}
	In this section we finally complete the definition of $\twar(\cE)$ for $\cE\in\cat_n$ and prove that it coincides with the model $\twar_\theta(\cE)$ defined in \Cref{sect:twar_theta}. Accordingly, the first part of the section is dedicated to the definition of $\twar(\cE)$ culminating in \Cref{constr:twar_final}, while the second half is dedicated to computing $\twar(x)$ for $x\in\stn_n$, which is the key computation in the proof of the comparison isomorphism $\twar(\cE)\cong \stab(\cat_{n,/\cE})$. Finally, in \Cref{thm:main_twar} and \Cref{prop:main_prop} we use the results of \cite{harpaz2020k} on Postnikov towers of $(\infty,n)$-categories to connect our theorem to the deformation theory of $(\infty,n)$-categories.
	\begin{construction}\label{constr:twar'}
		Given $\cE\in\cat_n$ denote by $\twar'(\cE)$ the subfunctor of 
		\[\Delta^\op\xrightarrow{\mor_\cat(\Delta^\inrt_{/-},\cat_{n,/\cE})} \cat\]
		sending $[q]$ to the subcategory of $\mor_\cat(\Delta^\inrt_{/[q]},\cat_{n,/\cE})$ such that:
		\begin{enumerate}
			\item For any morphism $F:\Delta^\inrt_{/[q]}\rightarrow \cat_{n,/\cE}$ in $\twar'(\cE)([q])$ the values $F([l]\xinert{i_0}[q])$ have the form $x_{i_0}\xrightarrow{f_{i_0}}\cE$ for some $x_{i_0}\in\ctree_n^h$ of \Cref{not:health_cat}, moreover for $i\in[q]$ we have $x_{\{i\}}\cong c_n$, where $[0]\xinert{\{i\}}[q]$ is the inclusion of the element $\{i\}$;
			\item given any morphism $[l]\xinert{i'}[m]$ between $i_0:[l]\xinert{}[q]$ and $i_1:[m]\xinert{}[q]$, if $i'$ preserves the minimal element, the corresponding morphism $h_{i'}:x_{i_0}\rightarrow x_{i_1}$ is inert, and if $i'$ preserves the maximal element, then the morphism $h_{i'}$ is an active healthy morphism in the sense of \Cref{def:tree_health};
			\item any natural transformation $\alpha:F\rightarrow G$ in $\mor_\cat([1]\times \Delta^\inrt_{/[q]},\cat_{n,/\cE})$ that lies in $\twar'(\cE)$ satisfies $\alpha_{[0]\xinert{\{i\}}[q]}\cong \id$ for all $i\in[q]$.
		\end{enumerate}
		Note that this is indeed a subfunctor since both inert and active healthy morphisms are closed under composition (the latter by \Cref{prop:health_cat}). Finally, denote by $\twar'(\cE):\Delta^\op\rightarrow\cS$ the composition $\Delta^\op\xrightarrow{\twar'(\cE)}\cat\xrightarrow{|-|}\cS$, where the second functor is the geometric realization.
	\end{construction}
	\begin{prop}\label{prop:twar'_seg}
		$\twar'(\cE)$ satisfies the Segal condition.
	\end{prop}
	\begin{proof}
		Temporarily denote by $\twar'_0(\cE)([q])$ the subcategory of $\mor_\cat(\Delta^\el_{/[n]},\cat_{n,/\cE})$ satisfying the same conditions as in \Cref{constr:twar'}. In other words, the objects of $\twar'_0(\cE)([q])$ are given by strings of cospans 
		\[\begin{tikzcd}[sep=huge]
			& {x_{0,1}} && {...} && {x_{q-1,q}} \\
			{c_n} && {c_n} && {c_n} && {c_n}
			\arrow["{i_0}", tail, from=2-1, to=1-2]
			\arrow["{a_1}"', two heads, from=2-3, to=1-2]
			\arrow[tail, from=2-3, to=1-4]
			\arrow[two heads, from=2-5, to=1-4]
			\arrow["{i_{q-1}}", tail, from=2-5, to=1-6]
			\arrow["{a_q}"', two heads, from=2-7, to=1-6]
		\end{tikzcd}\]
		over $\cE$ and morphisms are given by commutative diagrams
		\[\begin{tikzcd}[sep=huge]
			& {x_{0,1}} && {...} && {x_{q-1,q}} \\
			{c_n} && {c_n} && {c_n} && {c_n} \\
			& {x'_{0,1}} && {...} && {x'_{q-1,q}}
			\arrow["{f_1}"', from=1-2, to=3-2]
			\arrow[from=1-4, to=3-4]
			\arrow["{f_q}", from=1-6, to=3-6]
			\arrow["{i_0}", tail, from=2-1, to=1-2]
			\arrow["{i'_0}"', tail, from=2-1, to=3-2]
			\arrow["{a_1}"', two heads, from=2-3, to=1-2]
			\arrow[tail, from=2-3, to=1-4]
			\arrow["{a'_1}", two heads, from=2-3, to=3-2]
			\arrow[tail, from=2-3, to=3-4]
			\arrow[two heads, from=2-5, to=1-4]
			\arrow["{i_{q-1}}", tail, from=2-5, to=1-6]
			\arrow[two heads, from=2-5, to=3-4]
			\arrow["{i'_{q-1}}"', tail, from=2-5, to=3-6]
			\arrow["{a_q}"', two heads, from=2-7, to=1-6]
			\arrow["{a'_q}", two heads, from=2-7, to=3-6]
		\end{tikzcd}.\]
		It follows that $\twar'(\cE)([q])$ satisfies the Segal condition and moreover we have
		\begin{align}\label{eq:twar_seg}
			\twar_0'(\cE)([q])&\cong |\twar_0''(\cE)([1])\times_{\twar_0''(\cE)([0])}...\times_{\twar'_0(\cE)([0])}\twar'_0(\cE)([1])|\\
			&\cong |\twar_0''(\cE)([1])|\times_{|\twar_0''(\cE)([0])|}...\times_{|\twar'_0(\cE)([0])|}|\twar'_0(\cE)([1])|\\
			&\cong \twar_0'(\cE)([1])\times_{\twar_0'(\cE)([0])}...\times_{\twar'_0(\cE)([0])}\twar'_0(\cE)([1]),
		\end{align}
		where the first isomorphism follows from the Segal condition, the second from iterated application of \cite[Lemma 5.17.]{ayala2017fibrations} (which applies in our case since $\twar'_0(\cE)([0])$ is an $\infty$-groupoid). \par
		Denote by $j:\Delta^\el_{/[q]}\hookrightarrow\Delta^\inrt_{/[q]}$ the natural inclusion, note that we have a restriction functor $j^*:\twar'(\cE)([q])\rightarrow\twar'_0(\cE)([q])$ given by a precomposition with $j$. We claim that it has a left adjoint which is given by $j_!$ -- the left Kan extension along $j$. Indeed, it suffices to prove that for $F:\Delta^\el_{/[q]}\rightarrow\cat_{n,/\cE}$ in $\twar'_0(\cE)([q])$ the value
		\begin{equation}\label{eq:col_twar}
			j_!F([l]\xinert{i}[q])\cong \underset{[e]\xinert{i'}[l]}{\colim}F(i\circ i')\in \cat_{n,/\cE}
		\end{equation}
		belongs to $\ctree^h_{n,/\cE}$, however this follows since the colimit \eqref{eq:col_twar} is an iterated pushout of healthy active morphisms along inert morphisms, and those exist by \Cref{prop:health_push}. Since left adjoints induce isomorphisms on geometric realizations, it follows that $|\twar'(\cE)([q])|\cong |\twar'_0(\cE)([q])|$, and the claim follows from \eqref{eq:twar_seg}.
	\end{proof}
	\begin{lemma}\label{lem:twar_col1}
		We have 
		\[\twar'(\cE)\cong \underset{(x\xrightarrow{f}\cE)\in\stn_{n,/\cE}}{\colim}\twar'(x).\]
	\end{lemma}
	\begin{proof}
		It suffices to show that 
		\[\twar'(\cE)\cong \underset{x\xrightarrow{f}\cE}{\colim}\twar'(x).\]
		To prove this it suffices to show that for $\Delta^\inrt_{/[q]}\xrightarrow{F}\cat_{n,/\cE}$ in $\twar'(\cE)([q])$ the space of factorizations 
		\[\Delta^\inrt_{/[q]}\xrightarrow{F'}\cat_{n,/x}\xrightarrow{f_!}\cat_{n,/\cE}\]
		of $F$ for $x\in\ctree_n$ is contractible, however note that it in fact has an initial object given by 
		\[\Delta^\inrt_{/[q]}\xrightarrow{F'}\cat_{n,/s(F([q]\eq[q]))}\xrightarrow{F([q]\eq[q])_!}\cat_{n,/\cE},\]
		where $s(F([q]\eq[q]))$ denotes the source of the corresponding object of $\cat_{n,/\cE}$ viewed as a morphism in $\cat_n$.
	\end{proof}
	\begin{construction}\label{constr:D_obj}
		We will describe a certain functor $D_n^p:\Delta^p\rightarrow \ctree_n^\act$ for $p\leq n$, we will denote $D^n_n$ simply by $D_n$, we will start by describing its value on objects. For that we will first need some notation: recall from \Cref{constr:C_delta} the categories $C_n^j(1)$ for $0\leq j\leq n-1$: they can be described as a sequence of 3 $n$-cells composable along their $j$-boundaries, we will denote those cells $\{-,*,+\}$ in order of composition and define $i_{j,n}^\sigma:c_n\xinert{} C^j_n(1)$ to be the inclusion of the corresponding cell.\par
		Define $D_n^p(0,q_2,...,q_p)\bydef c_n$ and if $q_1>0$ define $T_m(D_n^p(q_1,...,q_p))$ to be the poset of strings of symbols $S\bydef l^{\sigma_0}_{i_0}...l^{\sigma_{N-1}}_{i_{N-1}}l_{i_N}$ of finite length (including the empty string) such that for $0\leq k\leq N$ we have $0\leq i_k\leq p-1$ and $\sigma_k\in\{-,*,+\}$ and such that the following conditions are satisfies:
		\begin{enumerate}
			\item\label{it:f1} $i_0 = 0$;
			\item\label{it:f2} if $\sigma_k\neq *$, then $i_s < i_k$ for $s>k$;
			\item\label{it:f3} $i_{k+1} = i_k+1$ unless $i_k = \underset{s<k, \sigma_s\neq *}{\inf}i_s - 1$ in which case $i_{k+1} = 0$;
			\item\label{it:f4} call a term $l^{\sigma_k}_{i_k}$ with $i_k=j$ \textit{extremal} if either $i_k=p-1$ or there exists a term $l^\pm_{i_s}$ with $s<k$ and $i_s = j+1$, then the number of extremal terms $l^{\sigma_k}_j$ is $\leq q_{p-j}$, if $q_{p-j} = 0$ we interpret this to mean that the string immediately terminates after the term $l^\pm_{j+1}$.
		\end{enumerate}
		We declare that $S <_D S'$ is the string $S'$ can be obtained from $S$ by adding symbols on the right. It is clear that $T_m(D_n^p(q_1,...,q_p))$ admits a minimal element (namely, the empty string) and that for all $S\in T_m(D_n^p(q_1,...,q_p))$ the category $T_m(D_n^p(q_1,...,q_p))_{/S}$ is a linearly ordered set $[\mathrm{len}(S)]$, so $T_m(D_n^p(q_1,...,q_p))$ is indeed a tree.\par
		It remains to define markings on the nodes and edges of $T_m(D_n^p(q_1,...,q_p))$, we mark the empty string with $c_n$ and a string $l^{\sigma_0}_{i_0}...l_{i_N}$ with $C^{p-i_N-1}_n(1)$, the edge $\varnothing<l^{\sigma_0}_0$ is marked with $c_n\eq c_n$ and an edge $l^{\sigma_0}_{i_0}...l_{i_N}< l^{\sigma_0}_{i_0}...l_{i_N}^{\sigma_N}l_{i_{N+1}}$ is marked with $i^{\sigma_N}_{p-i_N-1,n}:c_n\xinert{}C^{p-i_N-1}_n(1)$.
	\end{construction}
	\begin{prop}\label{prop:D_tree}
		The object $D_n^p(q_1,...,q_p)\in\ctree'_n$ described in \Cref{constr:D_obj} belongs to $\ctree^h_n$.
	\end{prop}
	\begin{proof}
		The fact that $D_n^p(q_1,...,q_p)\in\ctree^h_n$ would follow immediately if we show $D_n^p(q_1,...,q_p)\in\ctree_n$ since all $C_n^j(1)$ lie in $\Theta_n^h$. To prove $D_n^p(q_1,...,q_p)\in\ctree_n$ we need to show that it satisfies the condition \eqref{it:tree_cond}, for that we first need figure out which marked cells in $T_m(D_n^p(q_1,...,q_p))$ share a boundary. By definition such a marked cell corresponds to a string $S\bydef l^{\sigma_0}_{i_0}...l_{i_N}$ with $0\leq i_j\leq p-1$ and $\sigma_i\in \{-,*,+\}$ together with an $n$-cell in $C_n^{p-i_N-1}(1)$, by definition $C_n^{p-i_N-1}(1)$ contains three non-trivial $n$-cells $i^\sigma_{p-i_N-1, n}$ for $\sigma\in \{-,*,+\}$ such that
		\begin{equation}\label{eq:d_bound1}
			d^\pm_{p-1-i_N}i^\mp_{i_N + n - p, n} = d^\mp_{p-1-i_N}i^*_{i_N + n - p, n},
		\end{equation}
		cells in the $j$-boundary of $i^\sigma_{i_N + n - p, n}$ are distinct for $j>p-1-i_N$ and for $j<p-1-i_N$ we have
		\begin{equation}\label{eq:d_bound2}
			d^\pm_{j}i^-_{i_N + n - p, n} = d^\pm_{j}i^*_{i_N + n - p, n} = d^\pm_{j}i^+_{i_N + n - p, n}.
		\end{equation} 
		Consequently, the claim would be proved if we could show that 
		\[d^\pm_j \tau_{\geq t([S,i^\pm_{p-i_N-1,n}])} D_n^p(q_1,...,q_p)\cong c_j\]
		for $j\leq p-1-i_N$, since then the condition \eqref{it:tree_cond} would be vacuous. Note however that all the nodes in $T_m(\tau_{\geq t([S,i^\pm_{p-i_N-1,n}])} D_n^p(q_1,...,q_p))$ are marked with $C^{p-k-1}_n(1)$ for $k<i_N$, so in particular they all share a common $j$ boundary for $j\leq p-1-i_N< p-1-k$ by the observations above.
	\end{proof}
	\begin{construction}\label{constr:str_nodes}
		Given a string $S = l^{\sigma_0}_{i_0},...,l_{i_N}$ marked with $C_n^{n-i_N-1}(1)$ and a cell $i_{\sigma_N}:c_n\xinert{} C_n^{n-i_N-1}(1)$, which together define an edge in $T_m(D_n(q_1,...,q_n))$, denote by $\widetilde{S}\bydef l^{\sigma_0}_{i_0},...,l_{i_N}^{\sigma_N}$ the extended string and associate to it a triple $((\delta_1,...,\delta_n),(s_0,...,s_{t_v}), M)$, where $\delta_i\in \{-,*,+\}$ and $\delta_i=\pm$ if and only if $\widetilde{S}$ contains the term $l^\pm_{n-i}$, $t_0 = 0$ and the indices $t_k$ for $1\leq k\leq v$ are all the indices for which $\delta_{t_k}\neq *$, $s_i\leq q_{i+1}$ counts the number of extremal terms $l^*_{n-i-1}$ (note that this number is greater than 0 only if $\delta_i\neq *$) and $M\bydef \max(t_v-i_N-1,0)$. Given $m<n$ and $x\bydef ((\delta_1,...,\delta_n),(s_0,...,s_{t_v}))$ as above denote by $\tau_{\leq m}x$ to be the pair $((\delta_1,...,\delta_m),(s_0,...,s_{t_{v'}}), M')$, where $v'$ is the maximal index for which $t_{v'}\leq m$ and $M'=\max(M-m,0)$.
	\end{construction}
	\begin{prop}\label{prop:cell_el}
		\begin{enumerate}
			\item\label{it:str1} \Cref{constr:str_nodes} establishes an isomorphism between the set of edges of $T_m(D_n(q_1,...,q_n))$ and the set of triples $((\delta_1,...,\delta_n),(s_0,...,s_{t_v}), M)$ described in \Cref{constr:str_nodes} such that additionally if for some $i$ we have $s_i=q_i$, then all $\delta_j=*$ for $j>i$ and $M=0$;
			\item\label{it:str2} a pair $((\delta_1,...,\delta_n),(s_0,...,s_{t_v}))$ corresponds to an elementary cell of $D_n(q_1,...,q_n)$ if $M=0$ and either $s_k=q_{k+1}$ for some $k$ or $\delta_n\neq *$;
			\item\label{it:str3} the set of elementary $(n-1)$-cells of $D_n(q_1,...,q_n)$ is isomorphic to the set of triples $(n-1, x, \sigma)$, where $x=((\delta_1,...,\delta_n),(s_0,...,s_{t_v}))$ corresponds to an elementary $n$-cell and $\sigma\in \{-,+\}$ is such that $\sigma = \delta_n$ if $\delta_n\neq *$;
			\item\label{it:str4} more generally, the set of elementary $l$-cells of $D_n(q_1,...,q_n)$ is isomorphic to the set of elementary $l$-cells in $D_{l+1}(q_1,...,q_{l+1})$;
			\item\label{it:str5} given an elementary $l$-cell $y$ and ann elementary $k$-cell $z$ for $l<k$ we have $y< z$ if and only if $y$ belongs to $d_l^\pm(\tau_{\leq l+1}z)$, where $\tau_{\leq l+1}z$ denotes the subcomplex $\tau_{\geq t(\tau_{\leq l+1}z)}D_{l+1}(q_1,...,q_{l+1})$ of $D_{l+1}(q_1,...,q_{l+1})$.
		\end{enumerate}
	\end{prop}
	\begin{proof}
		We start by proving \eqref{it:str1}: note that any $x\bydef ((\delta_1,...,\delta_n),(s_0,...,s_{t_v}), M)$ corresponding to an edge in $T_m(D_n(q_1,...,q_n))$ satisfies the condition of \eqref{it:str1} by \eqref{it:f4}. Conversely, given $x$ we can define a string 
		\[S\bydef s_0*S_{n-1} + \sum_{j=1}^k(S^{\delta_{t_j}}_{n-t_j} + s_{t_j}*S_{n-t_j-1}) + S_+,\]
		where $+$ denotes the concatenation of strings, $m*S_k$ denotes the string $l_0^*l_!^*...l_k^*$ repeated $m$ times, $S^\pm_{k}$ denotes the string $l_0^*...l_k^\pm$ and the term $S_+$ is either $l_0^\pm$ if $\delta_n=\pm$ or the string $l_0^*...l^*_{t_v-M}$ if $\delta_n =*$, it is easy to see that this defines an inverse to the map of \Cref{constr:str_nodes}.\par
		To prove \eqref{it:str2} note that a string $\widetilde{S}\bydef l_{i_0}^{\sigma_0}...l_{i_N}^{\sigma_N}$ corresponds to an elementary cell if and only if it cannot be extended to a string $S' = l_{i_0}^{\sigma_0}...l_{i_N}^{\sigma_N} l_{i_{N+1}}$ satisfying the conditions of \Cref{constr:D_obj}, note that this is only the case if either $i_N = 0$ and $\sigma_N\in\{-,+\}$ (since then \ref{it:f2} can not be satisfied) or $l_{i_N}^{\sigma_N}$ is the $q_{n-i_N}$th extremal term $l_{i_N}$ (in which case the condition \ref{it:f4} fails), and in both cases we have $M=0$, under the isomorphism of \eqref{it:str1} those conditions clearly correspond to the ones described in \eqref{it:str2}. In what follows we will drop $M=0$ from the notation for an elementary $n$-cell and denote it simply $((\delta_1,...,\delta_n),(s_),...,s_{t_v}))$.\par
		For \eqref{it:str3} note that by construction every unmarked $(n-1)$-cell in $D_n(q_1,...,q_n)$ lies in the boundary of an unmarked $n$-cell, recall from the proof of \Cref{prop:D_tree} that $C_n^{n-i_k-1}(1)$ contains 3 elementary $n$-cells $\{-,*,+\}$ and for $i_k>0$ their $(n-1)$-boundaries are all distinct, while for $i_k=0$ we have $d^+_{n-1}c_n^- = d^-_{n-1}c_n^*$ and $d^-_{n-1}c_n^+ = d^+_{n-1}c_{n}^*$. We may encode this information by a pair $(\sigma_k, \sigma)$, where $\sigma\in\{-,+\}$ such that $\sigma = \sigma_k$ if $i_k=0$ and $\sigma_k\in\{-,+\}$. Combining this with the previous description of $n$-cells we see that an elementary $(n-1)$-cell can be identified with the data $(n-1,(\delta_1,...,\delta_n),(s_0,...,s_{t_v}),\sigma)$ satisfying the conditions of \eqref{it:str3}.\par
		To prove the claim about the $l$-cells observe that 
		\begin{equation}\label{eq:l_label1}
			\mor_{\Theta^\inrt_{n}}(c_l, C_n^i(1))\cong \mor_{\Theta^\inrt_n}(c_l,c_{n})\cong \mor_{\Theta^\inrt_{l+1}}(c_l,c_{l+1})\text{ for $i>l$}
		\end{equation}
		and
		\begin{equation}\label{eq:l_label2}
			\mor_{\Theta^\inrt_n}(c_l, C^i_n(1))\cong \mor_{\Theta^\inrt_{l+1}}(c_l, C^i_{l+1}(1))\text{ for $i\leq l$}.
		\end{equation}
		Given an object $D_n(q_1,...,q_n)$ denote by $X_l\in\ctree_n^h$ the object with the same underlying tree such that the string $l^{\sigma_0}_{i_0}...l^{\sigma_k}_{l_k}$ is marked with $C^{n-i_k-1}_{l+1}(1)$ if $i_k\geq n-l-1$ and with $c_{l+1}$ otherwise, then it follows from \eqref{eq:l_label1} and \eqref{eq:l_label2} that 
		\[\mor_{\stn_n^\inrt}(c_l,D_n(q_1,...,q_n))\cong \mor_{\stn_{l+1}^\inrt}(c_l,X).\]
		Note that by contracting all the edges of $X_l$ marked with $c_{l+1}=c_{l+1}$ we arrive at the isomorphism $X_l\cong D_{l+1}(q_1,...,q_{l+1})$, which proves \eqref{it:str4}.\par
		Finally, given an elementary $(k+1)$ cell $z\bydef ((\delta_1,...,\delta_{k+1}),(s_0,...,s_{t_v}))$ denote by $z_{l+1}$ the object $((\delta'_1,...,\delta'_{k+1}),(s_0,...,s_{t_{v'}}),0)$, where by definition $\delta'_i = \delta_i$ for $i\leq l+1$ and $\delta'_i=*$ for $i>(l+1)$ and $v'$ is the maximal index for which $t_{v'}\leq l+1$, then $z_{l+1}$ defines a subcomplex $\tau_{\geq z_{l+1}}D_{k+1}(q_1,...,q_{k+1})$ and it follows from \eqref{eq:l_label1} that $d^\pm_{l}z = d^\pm_{l}\tau_{\geq z_{l+1}}D_{k+1}(q_1,...,q_{k+1})$. It remains to prove that $d^\pm_{l+1}\tau_{\geq z_{l+1}}D_{k+1}(q_1,...,q_{k+1}) = d^\pm_l\tau_{\leq l+1}z$, which follows from \eqref{eq:l_label2}.
	\end{proof}
	\begin{lemma}\label{lem:str_bound}
	Given $x\bydef ((\delta_1,..,\delta_n),(s_0,...,s_{t_v}), M)$ encoding an edge in $T_m(D_n(q_1,...,q_n))$ we have an isomorphism
	\begin{equation}\label{eq:edge_iso}
		d^\sigma_{n-1}(\tau_{\geq t(x)} D_n(q_1,...,q_n))\cong D^{M}_{n-1}(1,0,...,0)
	\end{equation}
	such that the elementary $(n-1)$-cells appearing in $d^\sigma_{n-1}(\tau_{\geq t(x)} D_n(q_1,...,q_n))$ are of the form $d^\sigma_{n-1}i_n$ for elementary $n$-cells $i_n:c_n\xinert{} D_{n}(q_1,...,q_n)$ corresponding to pairs $((\delta'_1,..,\delta'_n),(s'_0,...,s'_{t'_{v'}}))$ such that either $\delta'_i = \delta_i$ for all $i\neq \sigma$, $s'_{t_v} = s_{t_v}+1$ and $\delta_n=\sigma$ if $s_{t_v}+1<q_{t_v}$ and $\delta_n=*$ otherwise or $v' = v+1$, $\delta'_i = \delta_i$ for $i\notin \{t'_{v+1},n\}$, $n-1-M\leq t'_{v+1}\leq n-1$, $s_{t'_{v+1}}=0$ and $\delta_n=\sigma$.
	\end{lemma}
	\begin{proof}
	Note that $d^\sigma_{n-1} C_n^{n-i_s-1}(1) = C^{n-i_s-1}_{n-1}(1)$ if $i_s>0$ and $d^\sigma_{n-1} C_n^{n-1}(1) = c_{n-1}$, moreover note that if a node $c$ is marked with $C_n^{n-1}(1)$, then 
	\begin{equation}\label{eq:bound_n-1}
		d^\sigma_{n-1}\tau_{\geq c}D_n(q_1,...,q_n)\cong d^\sigma_{n-1} C_n^{n-1}(1)\cong c_{n-1}
	\end{equation}
	since the boundary $d^\sigma_{n-1} i_n^*$ of the marked cell $i_n^*: c_n\xinert{} C_n^{n-1}(1)$ is disjoint from the boundary $d^\sigma_{n-1} C_n^{n-1}(1)$. Denote by $X\in\ctree_n$ the object whose underlying subtree is obtained from $T\bydef T_m(\tau_{\geq t(e_S)} D_n(q_1,...,q_n))$ by removing all nodes in $T_m(\tau_{\geq x} T)$ marked with $C_n^{n-1}(1)$ and whose markings are induced from $\tau_{\geq t(e_S)} D_n(q_1,...,q_n)$, then it follows from \eqref{eq:bound_n-1} that $d^\sigma_{n-1}(\tau_{\geq t(e_S)} D_n(q_1,...,q_n))\cong d^\sigma_{n-1}X.$
	Finally, observe that the underlying tree of $X$ is a linear tree with $M$ nodes $\{1,...,M\}$ such that the node $s$ is marked with $C_n^{M-s}(1)$, from which the claim easily follows.
	\end{proof}
	\begin{prop}\label{prop:pos}
	The partially ordered set $\stn^\el_{n,/D_n(q_1,...,q_n)}$ has objects given by triples $(l,x,\sigma)$, where $x=((\delta_1,...,\delta_{r(l)}),(s_0,...,s_{t_v}))$ and $r(l)\bydef \min(n,l+1)$, given $(k,y,\sigma')$ with $k>l$ and $y=((\delta'_1,...,\delta'_{r(k)}),(s'_0,...,s'_{t'_{v'}}))$, denote by $u$ the maximal index $j$ such that $t'_j\leq r(l)$, then $x<y$ if and only if one of the following mutually exclusive conditions holds:
	\begin{enumerate}
		\item $x=(l,\tau_{\leq r(l)}y,\sigma)$;
		\item $\delta_{r(l)}=\sigma$, $v= u+1$, $\delta'_i = \delta_i$ for $i\neq t_v$ and $s_{t_i} = s'_{t_i}$ for $i\leq u$
		\item $v=u$, $\delta'_{i} = \delta_{i}$ for $i\leq r(l)$, $s'_{t_i} = s_{t_i}$ for $i<v$ and $s_{t_v} = s'_{t_v}+1$.
	\end{enumerate}
	Moreover, for $p\leq n$ we have $\stn^\el_{n,/D^p_n(q_1,...,q_p)}$ is a Cartesian fibration over $\stn^\el_{n,/D_p(q_1,...,q_p)}$ whose fiber over $c_i\xinert{}D_p(q_1,...,q_p)$ is $*$ if $i<p$ and $\Theta^\el_{n-p/c_{n-p}}$ if $i=p$.
	\end{prop}
	\begin{proof}
	The first claim follows immediately from \Cref{lem:str_bound} and \eqref{it:str5} of \Cref{prop:cell_el}, to prove the second claim note that by definition we have $T_m(D^p_n(q_1,...,q_p))\cong T_m(D_p(q_1,...,q_p))$, but in $D^p_n(q_1,...,q_p)$ the cell $l_{i_0}^{\sigma_0}...l_{i_n}^{\sigma_N}$ is marked with $C_n^{p-i_N-1}(1)$ while in $D_p(q_1,...,q_p)$ it is marked with $C_p^{p-i_N-1}(1)$, the claim now follows since the natural surjective morphism $C_n^{k}(1)\rightarrow C_p^k(1)$ for $k<p$ induces a Cartesian fibration $\stn^\el_{n,/C_n^{k}(1)}\rightarrow\stn^\el_{p,/C_p^{k}(1)}$ whose fiber over $c_i\xinert{}C_p^k(1)$ is $*$ if $i<p$ and $\Theta^\el_{n-p/c_{n-p}}$ if $i=p$.
	\end{proof}
	\begin{defnprop}\label{prop:D_func}
	The assignment $(q_1,...,q_p)\mapsto D^p_n(q_1,...,q_p)$ extends to a functor $\Delta^{\times n}\rightarrow\stn_n$, moreover $D_n(f_1,...,f_n)$ is active unless $f_1(0)>0$ and if $i^r$ is an inert morphism preserving the maximal element, then $D_n(i^r,\id,...,\id)$ is inert.
	\end{defnprop}
	\begin{proof}
	Our strategy for defining the morphism $D(f)$ for any morphism $f\bydef (f_1,...,f_p):(q_1,...,q_p)\rightarrow(d_1,...,d_p)$ would be to first define a functor
	\[D(f):\stn^\el_{/D^p_n(q_1,...,q_p)}\rightarrow\psh_\Omega(\stn^\el_{/D^p_n(d_1,...,d_p)})\]
	and then to prove it lies in the image of $\stn^\inrt_{n,//D^p_n(d_1,...,d_p)}\hookrightarrow\psh_\Omega(\stn^\el_{/D^p_n(d_1,...,d_p)})$. It follows from the second claim of \Cref{prop:pos} that we may assume $n=p$, which we will do from now on.\par
	Note that any morphism $f$ in $\Delta^{\times n}$ decomposes uniquely as $f\cong j\circ s$ for a componentwise injective $j$ and surjective $s$. We will start by describing $D_n(j)$ for some injective $j$ with components $j_i:[q_i]\rightarrow [p_i]$. 
	We will start by describing the image of an elementary $n$-cell $x\bydef ((\delta_1,...,\delta_n),(s_0,...,s_{t_v}))$: assume first that $\delta_n=*$, then there is some index $k$ such that $s_k=q_{k+1}$, we define its image under $D(j)_\Omega$ to contain all $n$-cells $((\delta'_1,...,\delta'_n),(s'_0,...,s'_{t_{v'}}))$ such that $\delta'_i = \delta_i$ for $i\leq k$, $s'_i = j(s_i)$ for $i<k$ and $s'_{k}\geq j(s_{k})$, as well as their $t$-boundaries for $t<n$. If $\delta_n=\pm$ (so in particular $s_i<q_{i+1}$ for all $i$) we define the image of $x$ to contain all $n$-cells $((\delta'_1,...,\delta'_n),(s'_0,...,s'_{t_{v'}}))$ and their $t$-boundaries for $t<n$ such that $\delta'_{t_i} = \delta_{t_i}$ and $j_{t_i}(s_{t_i})\leq s'_{t_i}<j_{t_i}(s_{t_i}+1)$ for $0\leq i\leq v$ and $s'_{i'}< j_{i'}(0)$ (if $j_{i'}(0)=0$, we take it to mean $\delta_{i'}=*$) for $i'\neq t_i$. We will now describe the image of an $(n-1)$-cell $x'\bydef(n-1,(\delta_1,...,\delta_n),(s_0,...,s_{t_{v}}),\sigma)$: assume first that $\delta_n = *$ and denote by $k$ the index for which $s_k = q_{k+1}$, in that case if $j(q_{k+1}) = p_{k+1}$ we define $j(x')$ to be $(n-1,(\delta'_1,...,\delta'_n),(s'_0,...,s'_{t_{v'}}),\sigma)$, where $\delta'_i = \delta_i$ and $s'_i = j(s_i)$ for all $i$; if $j(q_{k+1})<p_{k+1}$ then we define its image to be $(n-1,(\delta'_1,...,\delta'_n),(s'_0,...,s'_{t_{v'}}),\sigma)$ with $s'_i = j(s_i)$ for all $i$ and $\delta'_i = \delta_i$ for $i<n$ with $\delta'_n = \sigma$. Assume now that $\delta_n = \sigma$, in that case we define its image to contain $(n-1,(\delta'_1,...,\delta'_n),(s'_0,...,s'_{t_{v'}}),\sigma)$ and their $t$-boundaries for $t<(n-1)$ such that $\delta'_{t_i} = \delta_{t_i}$, $s'_{t_v} = j(s_{t_v})$, $j_{t_i}(s_{t_i})\leq s'_{t_i}< j_{t_i}(s_{t_i}+1)$ for $0\leq i<v$ and $s'_{i'}< j_{i'}(0)$ for $i'\neq t_i$. Finally, for $l<(n-1)$ an $l$-cell is defined by the data $(l, (\delta_1,...,\delta_{l+1}),(s_0,...,s_{t_v}),\sigma)$, this can be identified with an $l$-cell in $D_n(q_1,...,q_{l+1})$, so we can use the formulas above wo define its image in $D_{l+1}(p_1,...,p_{l+1})$, which has the same $l$-cells as $D_n(p_1,...,p_n)$.\par
	We first need to show that this indeed defines a morphism 
	\[D(j):\stn^\el_{n,/D_n(q_1,...,q_n)}\rightarrow\psh_\Omega(\stn^\el_{n,/D_n(p_1,...,p_n)}),\]
	in other words we need to show that if some $l$-cell $i_l$ for $l<k$ lies in the boundary of a $k$-cell $i_k$, then $D(j)(i_l)< D(j)(i_k)$. By construction it suffices to prove this for $l=(n-1)$ and $k=n$, assume first that we have an $n$-cell $x=((\delta_1,...,\delta_n),(s_0,..,s_{t_v}))$ as above such that $\delta_n = *$, its boundary consists of two $(n-1)$-cells $(n-1,(\delta_1,...,\delta_n),(s_0,..,s_{t_v}),\sigma)$, by construction $D(j)(x)$ is just $\tau_{\geq j(x)}D_n(p_1,...,p_n)$, where $j(x)$ denotes the object $((\delta_1,...,\delta_n),(j_0(s_0),...,j_{t_v}(s_{t_v})))$ viewed as a node in $T_n(D(p_1,...,p_n))$, then it follows from \Cref{lem:str_bound} that $D(j)(n-1,(\delta_1,...,\delta_n),(s_0,..,s_{t_v}),\sigma)$ is exactly the $\sigma$-boundary of $D(j)(x)$. Now assume that $\delta_n \neq * $ and $\sigma = \delta_n$, in that case note that the $(n-1)$-cell $(n-1,(\delta'_1,...,\delta'_n),(s'_0,...,s'_{t_{v'}}),\sigma)$  such that $\delta'_{t_i} = \delta_{t_i}$, $s'_{t_v} = j(s_{t_v})$, $j_{t_i}(s_{t_i})\leq s'_{t_i}< j_{t_i}(s_{t_i}+1)$ for $i< v$ and $s'_{i'}<j_{i'}(0)$ for $i'\neq t_i$ that lies in the image of $(n-1,(\delta_1,...,\delta_n),(s_0,..,s_{t_v}),\sigma)$ is isomorphic to $d^\sigma_{n-1} y'$ for the $n$-cell $y'\bydef (\delta'_1,...,\delta'_n),(s'_0,...,s'_{t_{v'}})$ which lies in the image of $x$ by construction. Finally, by \Cref{prop:pos} it remains to consider the cases of cells $(n-1,(\delta''_1,...,\delta''_n),(s''_0,..,s''_{t_{v''}}),\sigma)$ such that either $v''=v+1$, $\delta''_i = \delta_i$ for $i\neq t_{v''}$, $s''_{t_i} = s_{t_i}$ for $i\leq v$ and $s''_{t_{v''}}=0$ or $v''=v$, $\delta''_i = \delta_i$ for all $i$ and $s''_{t_i} = s_{t_i}$ for $i<v$ and $s''_{t_v} = s_{t_v}+1$ with $s''_{t_v} = q_{t_v}$ or $\delta_n = \sigma$. Consider the first case, then the image of $(n-1,(\delta''_1,...,\delta''_n),(s''_0,..,s''_{t_{v''}}),\sigma)$ would contain $(n-1)$-cells $(n-1,(\delta'_1,...,\delta'_n),(s'_0,...,s'_{t_{v'}}),\sigma)$  such that $\delta'_{t_i} = \delta_{t_i}$ for $i\leq v+1$, $j_{t_k}(s_{t_k})\leq s'_{t_k}< j_{t_k}(s_{t_k}+1)$ for $k\leq v$, $s'_{t_{v+1}} = j_{t_{v+1}}(0)$ and $s'_{i'}<j_{i'}(0)$ for $i'\neq t_{i}$, note that such an $(n-1)$-cell lies in the boundary of an $n$-cell $((\delta'_1,...,\delta'_n),(s'_0,...,s'_{t_{v'}}))$ which lies in the image of $x$. In the last case the image of $(n-1,(\delta''_1,...,\delta''_n),(s''_0,..,s''_{t_{v''}}),\sigma)$ contains $(n-1)$-cells $(n-1,(\delta'_1,...,\delta'_n),(s'_0,...,s'_{t_{v'}}),\sigma)$  such that $\delta'_{t_i} = \delta_{t_i}$ for $i\leq v$, $j_{t_i}(s_{t_i})\leq s'_{t_i}< j_{t_i}(s_{t_i}+1)$ for $i< v$, $s'_{t_{v}} = j_{t_{v}}(s_{t_v}+1)$ and $s'_{i'}< j_{i'}(0)$ for $i'\neq t_{i}$, note that such an $(n-1)$-cell lies in the boundary of an $n$-cell $((\delta'_1,...,\delta'_n),(s'_0,...,s'_{t_{v'}}))$ which lies in the image of $x$.\par
	We have thus shown that $j$ defines a morphism $D(j):\stn^\el_{/D_n(q_1,...,q_n)}\rightarrow\psh_\Omega(\stn^\el_{n,/D_n(p_1,...,p_n)})$, denote by $D(j)_\Omega:\psh_\Omega(\stn^\el_{/D_n(q_1,...,q_n)})\rightarrow\psh_\Omega(\stn^\el_{n,/D_n(p_1,...,p_n)})$ its unique pushout-preserving extensions, we will now show that $D(j)_\Omega$ is functorial injective morphisms, i.e. that for a composable pair $(q_1,...,q_n)\xrightarrow{j_0}(p_1,...,p_n)\xrightarrow{j_1}(d_1,...,d_n)$ we have $D(j_1\circ j_0)_\Omega\cong D(j_1)_\Omega\circ D(j_0)_\Omega$. More explicitly, we need to show that for any $l$-cell $x$ we have 
	\[D(j_1\circ j_0)(x)\cong \bigcup_{y\in D(j_0)(x)} D(j_1)(y),\]
	where the union is taken over all elementary cells in $D(j_0)(x)$. By construction it suffices to prove the claim for $l=n$ and $l=(n-1)$, assume first that $x$ is an $n$-cell $((\delta_1,...,\delta_n),(s_0,...,s_{t_v}))$, by construction it suffices to show that every $n$-cell in $D(j_1\circ j_0)(x)$ is in the image of some $n$-cell in $D(j_0)(x)$. Assume first that $\delta_n = *$ and fix $k$ such that $s_k = q_{k+1}$, then the image of $x$ contains the $n$-cells $z\bydef ((\delta'_1,...,\delta'_n),(s'_0,...,s'_{t_{v'}}))$ such that $\delta'_i = \delta_i$ for $i\leq k$, $s'_i = j_1\circ j_0(s_i)$ for $i<k$ and $s'_{k}\geq j_1\circ j_0(s_{k})$, given such a cell define $y$ to be the $n$-cell $((\delta_1,...,\delta_n),(s''_0,...,s''_{t(v)}))$ such that $s''_i = j_0(s_i)$ for $i\leq k$, then clearly $z$ lies in the image of $y$. Now assume that $\delta_n=\pm$, then by construction the image of $x$ under $j_1\circ j_0$ contains $z$ as above such that $\delta'_{t_i} = \delta_{t_i}$ and $(j_1\circ j_0)_i(s_{t_i})\leq s'_{t_i}< (j_1\circ j_0)_i(s_{t_I}+1)$ for $i\leq v$ and $s'_{i'}< (j_1\circ j_0)_{i'}(0)$ for $i'\neq t_i$. In that case denote by $t'_k$ for $0\leq k\leq v'$ the indices for which $s'_{t'_k}\geq j_1(0)$ (in particular they include all $t_i$), note that for each $k$ we also have $s'_{t'_k}\leq j_1(p_{t'_k})$, so in particular we can find $s''_{t'_k}$ such that $j_1(s''_{t'_k})_{t'_k}\leq s'_{t'_k}< j_1(s''_{t'_k}+1)_{t'_k}$, we can then define $y\bydef ((\delta''_1,...,\delta''_n),(s''_0,...,s''_{t'_{v'}}))$, where $s''_{t'_k}$ are as defined above and $\delta''_{t'_k}=\delta'_{t'_k}$ and $\delta''_{i'}=*$ for $i'\neq t'_k$, then $z$ lies in the image of $y$. In the case of $(n-1)$-cell $x'\bydef (n-1,x,\sigma)$, where $x=((\delta_1,...,\delta_n),(s_0,...s_{t_v}))$ and $z'\bydef (n-1,z,\sigma)$ in the image of $x'$ under $D(j_1\circ j_0)$, it is easy to see that $y'\bydef (n-1,y,\sigma)$, where $y$ is an $n$-cell constructed above such that $D(j_1)(y)$ contains $z$, defines an $(n-1)$-cell such that $z'$ belongs to $D(j_1(y')$.\par
	We now need to prove that $D(j)$ factors through $\stn^\inrt_{n,/D_n(p_1,...,p_n)}$, for which we will represent its image as an iterated pushout of representables. Note that we can decompose $j$ as $j_1\circ ...\circ j_n$, where each $j_k$ denotes the morphism $(\id,...,j_k,...,\id)$ with the unique non-trivial component in position $k$, by the functoriality proved in the previous paragraph it suffices to prove the claim for each $j_i$ individually. We can further decompose $j_k\cong i^l_k\circ i^t_k\circ a_k$, where $a_k$ is an active morphism, $i_k^r$ is the inert morphism preserving the maximal element and $i_k^l$ an inert morphism preserving the minimal element, again it suffices to prove the claim for each of them. We will start with $i^l_k$, in that case by definition the image of an elementary $l$-cell is either an elementary $l$-cell or $\tau_{\geq c}D_n(p_1,...,p_n)$ for some node $c$ of $T_m(D_n(p_1,...,p_q))$, in both cases those lie in  $\stn^\inrt_{n,/D_n(p_1,...,p_n)}$ -- in the first case this is immediate, in the second follows from \Cref{cor:tree_stn}. Also note that the induced morphism is active: it suffices to prove that every elementary $n$-cell lies in the image of $D(i^l)$, for that note that $((\delta_1,...,\delta_n),(s_0,...,s_{t_v}))$ lies in the image of $((\delta_1,...,\delta_n),(s_0,...,s_{t_v}))$ if $k\neq t_i$ for any $i$ and in the image of $((\delta_1,...,\delta_n),(s_0,...,\min(s_k,q_k),...,s_{t_v}))$ otherwise. Now consider the case of an active morphism: given $x=((\delta_1,...,\delta_n),(s_0,...,s_{t_v}))$ with $\delta_n=*$, its image is again representable, so we may assume $\delta_n\in\{-,+\}$. In that case its image contains $n$-cells $((\delta_1,...,\delta_n),(s'_0,...,s'_{t_v}))$ such that $s'_{t_i} = a(s_{t_i})$ unless $t_i=k$, in which case $a(s_{k})\leq s'_{k}<a(s_k+1)$, as well as all their boundaries. Note that this object is representable if $a(s_k+1) = a(s_k)+1$, in general we may argue by induction on $(a(s_k+1) - a(s_k))$: denote $y\bydef a(x)$ and $y'$ the subset of $y'$ containing $((\delta_1,...,\delta_n),(a(s_0),...,s'_k,...,a(s_{t_v})))$ with $a(s_{k})\leq s'_{k}<a(s_k+1)-1$, then $y'\in\stn^\inrt_{n,/D_n(p_1,...,=_n)}$ by induction and $y\cong x'\coprod_{d^\sigma_{n-1}x'}y'$, where $x'$ denotes the $n$-cell $((\delta_1,...,\delta_n),(a(s_0),...,a(s_k+1)-1,...,a(s_{t_v})))$, hence $y\in\stn^\inrt_{n,/D_n(p_1,...,=_n)}$ as well. Given an $(n-1)$-cell $(n-1,x,\delta_n)$, note that its image is representable by $(n-1,((\delta_1,...,\delta_1),(a(s_0),...,a(s_k),...,a(s_{t_v}))),\sigma)$. Note again that the image of $D(a)$ contains all the elementary $n$-cells: $((\delta_1,...,\delta_n),(s_0,...,s_{t_v}))$ lies in the image of $((\delta_1,...,\delta_n),(s_0,...,s_{t_v}))$ if $k\neq t_i$ for any $i$, if not then we can find some index $s'_k$ such that $a(s'_k)\leq s_k<a(s_k+1)$ (since $a$ is active), and then $((\delta_1,...,\delta_n),(s_0,...,s_{t_v}))$ lies in the image of $((\delta_1,...,\delta_n),(s_0,...,s'_k,...,s_{t_v}))$. Finally, consider the case of $i^r$, then the image of $x$ is again representable if $\delta_n=*$, so we may assume $\delta_n=\pm$. In that case the image $y$ of $x$ contains $n$-cells $((\delta_1,...,\delta_n),(s_0,...,s'_k,...,s_{t_v}))$ with $s'_k< i^r(0)$, once again we can prove the claim by induction on $i^r(0)$: if it is 0, then this object is representable, if not we can consider $y'$ containing $((\delta_1,...,\delta_n),(s_0,...,s''_k,...,s_{t_v}))$ with $s''_k < i^r(0)-1$, so that we have $y\cong x'\coprod_{d^\sigma_{n-1}x'}y'$, where $x'$ denotes the cell $((\delta_1,...,\delta_n),(s_0,...,i^r(0),...,s_{t_v}))$. Similarly, for the $(n-1)$-cell $(n-1,x,\delta_n)$ we see that its image under $D(i^r)$ is again representable. Note also that if $k>0$, then $D(i^r)$ is active: $((\delta_1,...,\delta_n),(s_0,...,s_{t_v}))$ lies in the image of $((\delta_1,...,\delta_n),(s_0,...,s_{t_v}))$ if either $k\neq t_i$ for any $i$ or $s_k\geq i^r(0)$, if not then it lies in the image of $((\delta'_1,...,\delta'_n),(s'_0,...,s'_{t_{v-1}}))$, where $\delta'_i = \delta_i$ for $i\neq k$, $\delta_k = *$ and $s'_{t_j} = s_{t_j}$ for all $t_j\neq k$. Finally, note that if $k=0$, then $D(i^r)$ is isomorphic to the inert inclusion $\tau_{\geq x}D_n(q_1,q_2,...,q_n)\xinert{}D_n(p_1,q_2,...,q_n)$, where $x=((*,...,*),(i^r(0),0,...,0))$.\par
	This concludes the proof in the special case of injective morphisms, assume now that $f=b=(b_1,...,b_n)$ with surjective $b_i$, once again it suffices to define the image of $n$ and $(n-1)$-cells. Given an $n$-cell $x=((\delta_1,...,\delta_n),(s_0,...,s_{t_v}))$ denote by $k$ the minimal index for which $(s_{t_k}+1)\leq q_{t_k}$ and $b_{t_k}(s_{t_k}+1) = b_{t_k}(s_{t_k})$ - if such an index does not exist we define the image of $x$ to be $((\delta_1,...,\delta_n),(b_0(s_0),...,b_{t_v}(s_{t_v})))$. Assuming it does, note that there must be some index $r>k$ such that $\delta_r\neq *$, define $t$ to be the minimal such index, we define the image of $x$ to be the $(t-1)$-cell $(t-1,((\delta_1,...,\delta_{t}),(b_0(s_0),...,b_k(s_k)+1)),\delta_t)$. Similarly, for the $(n-1)$-cell $(n-1,x,\sigma)$ we define its image to be $(n-1,b(x),\sigma)$ if there is no $k$ such that $b_{t_k}(s_{t_k}+1) = b_{t_k}(s_{t_k})$ and $b(x)$ otherwise. We need to prove that it is well-defined: it is clear from construction that for an $(n-1)$-cell $(n-1,x,\sigma)$ we have $b(n-1,x,\sigma)\leq b(x)$, so by \Cref{prop:pos} it suffices to consider two other cases: if $\delta_n=\sigma$ and $y=(n-1,x',\sigma)$ with $((\delta'_1,...,\delta'_n),(s'_0,..,s'_{t_{v'}}))$ such that $v'=v+1$, $\delta'_i = \delta_i$ for $i\neq t_{v'}$, $s'_{t_i} = s_{t_i}$ for $i\leq v$ and $s'_{t_{v'}}=0$, then by construction $b(y)=b(x)$ since $t_{v'}> t_v \geq t$. In the other case we have $v'=v$, $\delta_i = \delta'_i$ and $s'_{t_v} = s_{t_v}+1$, in that case we immediately have $b(y) = b(x)$ if $t_n>k$, in the case $k= t_v$ we must consider two separate cases: if $b_k(s_k+1) = b_k(s_k+2)$, then immediately by construction we see that $b(x) = b(x') = b(y)$, if on the other hand $b_k(s_k+2) = b_{k}(s_k)+1$ we see that $b(x') = ((\delta_1,...,\delta_n),(b_0(s_0),..., b_k(s_k)+1))$ and $b(x) = b(y) = (n-1,(\delta_1,...,\delta_n),(b_0(s_0),..., b_k(s_k)+1), \delta_n)$. Since the images of elementary cells are representable, it automatically follows that $b$ factors through $\stn^\inrt_{n,/D_n(p_1,...,p_n)}$ and it is easy to see that $D(b\circ b') = D(b')\circ D(b)$. Finally, it is easy to see by construction that $D(b)$ is active.\par
	It remains to complete the proof of functoriality of $D(f)$, to do that it remains to prove that for any factorization $s\circ j= j'\circ s'$ we have $D(s\circ j) = D(j'\circ s')$, however using the construction of $D(s)$ and $D(j)$ we see that this claim follows from the commutativity of the following square for all $l<n$ and any injective $j$:
	\[\begin{tikzcd}[sep=huge]
		{\stn^\inrt_{n,/D_n(q_1,...,q_n)}} & {\stn^\inrt_{n,/D_n(p_1,...,p_n)}} \\
		{\stn^\inrt_{n,/D_n(q_1,...,q_l)}} & {\stn^\inrt_{n,/D_n(p_1,...,p_l)}}
		\arrow["{D(j)_\Omega}", from=1-1, to=1-2]
		\arrow["{\tau_{\leq l}}"', from=1-1, to=2-1]
		\arrow["{\tau_{\leq l}}", from=1-2, to=2-2]
		\arrow["{D(j)_\Omega}"', from=2-1, to=2-2]
	\end{tikzcd}.\]
	\end{proof}
	\begin{construction}\label{constr:twar_final}
	Denote by $F_{s_1}(s_2,...,s_n);\Delta^\inrt_{/[s_1]}\rightarrow \stn_n$ the functor sending $i:[l]\xinert{}[s_1]$ to $D_n(l,s_2,...,s_n)$ and an inclusion $[l]\xinert{i}[m]$ to the induced functor $D(i): D_n(l,...,s_n)\rightarrow D_n(l,...,s_n)$, note that it follows from \Cref{prop:D_func} that $D(i)$ is inert if $i$ preserves the minimal element. Given a morphism $f:[s_1]\rightarrow [s'_1]$ we have an induced functor $f_!:\Delta^\inrt_{/[s_1]}\rightarrow \Delta^\inrt_{/[s'_1]}$ together with a natural transformation $\alpha_f:F_{s_1}(s_1,...,s_n)\rightarrow f^* F_{s'_1}(s'_1,...,s_n)$ whose component at $F_{s_1}(s_1,...,s_n)([l]\xinert{i}[m])\cong D_n(l,...,s_n)$ is $D(a):D(l,...,s_n)\rightarrow D(l',...,s_n)$, where $a$ appears in the factorization square
	\[\begin{tikzcd}[sep=huge]
		{[l]} & {[l']} \\
		{[s_1]} & {[s'_1]}
		\arrow["a", two heads, from=1-1, to=1-2]
		\arrow["i"', tail, from=1-1, to=2-1]
		\arrow["{i'}", tail, from=1-2, to=2-2]
		\arrow["f"', from=2-1, to=2-2].
	\end{tikzcd}\]
	We can view $F_{s_1}(s_2,...,s_n)$ as a value at $(s_2,...,s_n)$ of a functor $F_{s_1}:\Delta^{\times(n-1)}\rightarrow \mor_\cat(\Delta^\inrt_{/[s_1]},\stn_n)$, where a morphism $g:(s_2,...,s_n)\rightarrow(s'_2,...,s'_n)$ gets sent to the natural transformation $\beta_g$ with components $D_n(l,s_2,...,s_n)\xrightarrow{D_n(\id,g_2,...,g_n)}D_n(l,s'_2,...,s'_n)$ - that this defines a natural transformation follows easily from the functoriality of \Cref{prop:D_func}. Similarly. it is easy to see that for any $f$ and $g$ as above and $i:[l]\xinert{}[s_1]$ we have a commutative diagram
	\[\begin{tikzcd}[sep=huge]
		{D_n(l,s_2,...,s_n)} & {D_n(l',s_2,...,s_n)} \\
		{D_n(l,s'_2,...,s'_n)} & {D_n(l',s'_2,...,s'_n)}
		\arrow["{\alpha_f}", from=1-1, to=1-2]
		\arrow["{\beta_g}"', from=1-1, to=2-1]
		\arrow["{\beta_g}", from=1-2, to=2-2]
		\arrow["{\alpha_f}"', from=2-1, to=2-2]
	\end{tikzcd}\]
	Given $\cE\in\cat_n$ and $[q]\in\Delta$ denote by $*_\cE$ the constant functor $\Delta^\inrt_{/[q]}\rightarrow\cat_n$ with value $\cE$, then we define $\twar_D(\cE)([q])\in\mor_\cat(\Delta^{\times(n-1),\op}, \cS)$ to be the functor sending $(s_2,...,s_n)$ to $\mor_{\mor_\cat(\Delta^\inrt_{/[q]}.\cat_n)}(F_q(s_2,...,s_n), *_\cE)$, where $\Delta^{\times(n-1),\op}$ acts by precomposition. By functoriality in the first variable described above we can view $\twar_D(\cE)$ as a functor
	\[\twar_D(\cE):\Delta^\op\rightarrow \mor_\cat(\Delta^{\times(n-1)}, \cS),\]
	we define $\twar(\cE):\Delta^\op\rightarrow\cS$ as a composition
	\[\Delta^\op\xrightarrow{\twar_D(\cE)} \mor_\cat(\Delta^{\times(n-1)}, \cS)\xrightarrow{\underset{\Delta^{\times(n-1),\op}}{\colim}}\cS.\]
	\end{construction}
	\begin{prop}\label{prop:twar_seg}
	$\twar(\cE)$ satisfies the Segal condition.
	\end{prop}
	\begin{proof}
	It follows from \Cref{lem:push_tree} and the construction of $D_n(q_1,...,q_n)$ that we have pushout diagrams 
	\begin{equation}\label{eq:D_push}
		\begin{tikzcd}[sep=huge]
			{c_n\cong D_n(0,q_1,...,q_n)} & {D_{n}(1,q_2,...,q_n)} \\
			{D_{n}(q_1-1,q_2,...,q_n)} & {D_n(q_1,...,q_n)}
			\arrow["{D_n(\{0\},\id,...,\id)}", two heads, from=1-1, to=1-2]
			\arrow["{D_n(\{q_1-1\},\id,...,\id)}"', tail, from=1-1, to=2-1]
			\arrow[tail, from=1-2, to=2-2]
			\arrow[two heads, from=2-1, to=2-2]
			\arrow["\ulcorner"{anchor=center, pos=0.125, rotate=180}, draw=none, from=2-2, to=1-1]
		\end{tikzcd}.    
	\end{equation}
	By an iterated application of \eqref{eq:D_push} we see that the functor $F_{q}(s_2,...,s_n):\Delta^\inrt_{/[q]}\rightarrow\cat_n$ described in \Cref{constr:twar_final} is the left Kan extension of its restriction to $\Delta^\el_{/[q]}$, since that also holds for $*_\cE$ we get
	\begin{align*}
		\twar_D(\cE)([q])(s_2,...,s_n)&\cong \mor_{\mor_\cat(\Delta^\inrt_{/[q]},\cat)}(i^\el_!i^{\el,*}F_q(s_2,...,s_n),i^\el_!i^{\el,*}*_\cE)\\
		&\cong \mor_{\mor_\cat(\Delta^\el_{/[q]},\cat)}(i^{\el,*}F_q(s_2,...,s_n),i^{\el,*}*_\cE)\\
		&\cong \underset{([e]\xinert{i}[q])\in\Delta^\el_{/[q]}}{\lim}\mor_{\mor_\cat(\Delta^\el_{/[q]},\cat)}(i^*i^{\el,*}F_q(s_2,...,s_n),i^*i^{\el,*}*_\cE)\\
		&\cong \underset{([e]\xinert{i}[q])\in\Delta^\el_{/[q]}}{\lim}\twar_D(\cE)([e])(s_2,...,s_n),
	\end{align*}
	where the first isomorphism follows by definition and our preceding observations, the second since $i^\el$ is fully faithful, the third since 
	\[\Delta^\el_{/[q]}\cong \underset{([e]\xinert{i}[q])\in\Delta^\el_{/[q]}}{\colim}\Delta^\el_{/[e]}\]
	and the last again by definition. Finally, note that $\twar_D(\cE)([0]):\Delta^{\times(n-1)}\rightarrow\cS$ is a constant functor, so we can conclude by the same argument as in \Cref{prop:twar'_seg}.
	\end{proof}
	\begin{prop}
	For any pair of morphisms $f:c_n\rightarrow\cE$ and $g:c_n\rightarrow\cE$ the space $\mor_{\twar(\cE)}(f,g)$ is the geometric realization of an $(n-1)$-tuple Segal space.
	\end{prop}
	\begin{proof}
	By construction $\mor_{\twar(\cE)}(f,g)$ is the colimit of a functor $\Delta^{\times(n-1)}\rightarrow\cS$ sending $(s_2,...,s_n)$ to the space of cospans 
	\[\begin{tikzcd}[sep=huge]
		{c_n\cong D_n(0,...,0)} & {D_n(1,s_2,...,s_n)} & {c_n\cong D_n(0,...,0)} \\
		& \cE
		\arrow["{D_n(\{0\},...,\{0\})}", tail, from=1-1, to=1-2]
		\arrow["f"', from=1-1, to=2-2]
		\arrow["F"{description}, from=1-2, to=2-2]
		\arrow["{D_n(\{1\},\{s_2\}...,\{s_n\})}"', two heads, from=1-3, to=1-2]
		\arrow["g", from=1-3, to=2-2]
	\end{tikzcd}.\]
	We need to prove that it satisfies the Segal condition separately in each variable, for that it suffices to prove that
	\[D_n(1,i_2,...,i_{k-1},s_k,...,s_n)\cong \underset{([e]\xinert{i}[s_k])\in\Delta^\el_{/[s_k]}}{\colim}D_n(1,i_2,...,i_{k-1},e,...,s_n)\]
	for every $k$ and all $i_j\in\{0,1\}$. This clearly follows by an iterated application of the following observation: for every $2\leq t\leq n$ we have a pushout diagram 
	\[\begin{tikzcd}[sep=huge]
		{c_n} & {D_n(1,s_2,...,1,...,s_n)} \\
		{D_n(1,s_2,...,s_t-1,...,s_n)} & {D_n(1,s_2,...,s_t,...,s_n)}
		\arrow[two heads, from=1-1, to=1-2]
		\arrow[tail, from=1-1, to=2-1]
		\arrow[from=1-2, to=2-2]
		\arrow[from=2-1, to=2-2]
		\arrow["\ulcorner"{anchor=center, pos=0.125, rotate=225}, draw=none, from=2-2, to=1-1]
	\end{tikzcd}\]
	in $\cat_n$, which in turn follows from the definition of $D_n(-)$ and \Cref{lem:push_tree}.
	\end{proof}
	\begin{lemma}\label{lem:inrt_cont}
	Assume we are given a $t$-active morphism $f:c_{t+1}\rightarrow x$ for some $x\in\stn_n$, denote by $C$ the subcategory of $\stn^\inrt_{n,/x}$ containing $y\xinert{i}x$ for which $d^-_{t+1}y = \ima(f)$ and $d^+_{t+1}y  = d^+_{t+1}x$, then $C$ is contractible.
	\end{lemma}
	\begin{proof}
	We have seen in the proof of \Cref{lem:jact} that there is a morphism $F:c_n\rightarrow x$ whose image contains all elementary cells $b$ for which there exists a basis element $b'$ in the image of $f$ such that $b\geq_{t+1} b'$, it is easy to see that $\ima(F)\in C$ and that it is a final object of this category.  
	\end{proof}
	\begin{prop}\label{prop:twar_stein}
	For $x\in\stn_n$ we have
	\[\twar(x)\cong \stn^\inrt_{n,/x}.\]
	\end{prop}
	\begin{proof}
	First, note that for any pair of morphisms $f:c_n\rightarrow x$ and $g:c_n\rightarrow x$ such that $\ima(f)\leq\ima(g)$ in $\stn^\inrt_{n/x}$ the space $\mor_{\twar(x)}(f,g)$ is non-empty -- this follows from \Cref{cor:D_ext}. It remains to prove that it is contractible.\par
	We will first need to introduce some notation. Call two morphisms $f:c_n\rightarrow x$ and $g:c_n\rightarrow x$ $k$-parallel for $0\leq k<n$ if they have the same $k$-boundary, we will also call any two morphisms (-1)-parallel by convention. Next, define $\Sigma^k_\theta D^p_{l}(s_1,...,s_{p})$ for $(s_1,...,s_p)\in\Delta^{\times p}$ and $p\leq l$ by induction on $k$ as follows: for $k=0$ we set $\Sigma^0_\theta D^p_l(s_1,...,s_p)\bydef D^p_l(s_1,...,s_p)$, for $k\geq 1$ we define $\Sigma^k_\theta D^p_{l}(s_1,...,s_{p})$ to be the $(k+l)$-category with two objects $\{0,1\}$ such that
	\[\mor_{\Sigma^k_\theta D^p_l(s_1,...,s_{p})}(0,1)\bydef \Sigma^{k-1}_\theta D^p_{l}(s_1,..., s_{p}).\]
	Note that $\Sigma^k_\theta c_l\cong c_{k+l}$. It follows from the definition and \Cref{prop:D_func} that the assignment $(s_1,...,s_{p})\mapsto \Sigma^k_{\theta}D^p_l(s_1,...,s_{p})$ extends to a functor from $\Delta^{\times p}$ to $\cat_n$. Note also that $\Sigma^k_\theta D^p_l(s_1,...,s_{l})\in\ctree^h_n$: indeed, it can be represented by a tree with nodes given by strings $S=l_{i_0}^{\sigma_0}...l_{i_N}$ with $0\leq i_j< p$ and $\sigma_j\in \{-,*,+\}$ satisfying the conditions described in \Cref{constr:D_obj} such that the node corresponding to $S$ is marked with $C^{k+p-i_N-1}_n(1)$.\par
	Define 
	\[F^k_{s_1}([m]\xinert{i}[s_1],s_2,...,s_{n-k})\bydef \Sigma^k_\theta D_{n-k}(m,s_2,...,s_{n-k}),\] 
	then by the same argument as in \Cref{constr:twar_final} this extends to a functor $F^k_{s_1}:\Delta^{(n-k-1)}\times \Delta^\inrt_{/[s_1]}\rightarrow\stn_n$ and we define $\twar^k_D(x)([q])\in\mor_\cat(\Delta^{\times(n-k-1),\op}, \cS)$ to be the functor sending $(s_2,...,s_{n-k})$ to $\mor_{\mor_\cat(\Delta^\inrt_{/[q]}.\cat_n)}(F_q(s_2,...,s_{n-k}), *_x)$, where $\Delta^{\times(n-1),\op}$ acts by precomposition. By functoriality in the first variable we can view $\twar^k_D(x)$ as a functor
	\[\twar^k_D(x):\Delta^\op\rightarrow \mor_\cat(\Delta^{\times(n-k-1)}, \cS),\]
	we define $\twar^k(x):\Delta^\op\rightarrow\cS$ as a composition
	\[\Delta^\op\xrightarrow{\twar^k_D(x)} \mor_\cat(\Delta^{\times(n-k-1)}, \cS)\xrightarrow{\underset{\Delta^{\times(n-1),\op}}{\colim}}\cS.\] 
	The same argument as in \Cref{prop:twar_seg} then shows that this is a Segal space.\par
	It is easy to see by construction that for a pair of $n$-morphisms $f$ and $g$ as above we have 
	\[\mor_{\twar^k(x)}(f,g) = \varnothing\]
	unless they are $(k-1)$-parallel, we will prove by downward induction on $k$ that for any $(k-1)$-parallel morphisms $f$ and $g$ $\mor_{\twar^k(x)}(f,g)$ is either empty or contractible, and the latter holds if and only if $\ima(f)\leq\ima(g)$ in $\stn^\inrt_{n,/x}$ - for $k=0$ this will prove the proposition. We start with the case $k=n-1$, in that case it is easy to see that $\Sigma^{n-1}_\theta D_1(s)\cong C_n^{n-1}(s)$, so that we have $\twar^{n-1}(x) = \twar_\theta^{n-1}(x)$ in the notation of \Cref{constr:C_delta}. The claim now follows either by untangling the definitions or using \Cref{prop:stn_inrt}.\par
	Assume we have proved the claim for $(k+1)$, we will proceed to prove it for $k$. First, we will need some more notation: observe that $\Sigma^k_\theta D_{n-k}(1,0,...,0)$ has $(2(n-k)+1)$ elementary cells corresponding to the strings $l_0^*...l_t^\pm$ for $t\leq (n-k-1)$ and $l_0^*...l_{n-k-1}^*$, we will denote by $i_t^\sigma:e_t^\sigma\xinert{}\Sigma^{k}_\theta D_{n-k}(1,...,0)$ for $0\leq t\leq n-k-1$ with $\sigma\in\{-,+\}$ for $t<n-k-1$ and $\sigma\in \{-,*,+\}$ otherwise the inclusions of those elementary cells. It follows from \Cref{lem:str_bound} that for $t<(n-k-1)$ we have a pushout square
	\begin{equation}\label{eq:el_push}
		\begin{tikzcd}[sep=huge]
			{c_{n-1-t}} & {\Sigma^k_\theta D_{n-1-t}(1,0,...,0)} \\
			{c_n} & {e_t^\pm}
			\arrow[two heads, from=1-1, to=1-2]
			\arrow["{i_{n-k-1-t}^\mp}"', tail, from=1-1, to=2-1]
			\arrow[tail, from=1-2, to=2-2]
			\arrow[two heads, from=2-1, to=2-2]
			\arrow["\ulcorner"{anchor=center, pos=0.125, rotate=180}, draw=none, from=2-2, to=1-1]
		\end{tikzcd}
	\end{equation}  
	and for $t=(n-k-1)$ we have $e^\sigma_{n-k-1} = c_n$. For any $(s_2,...,s_{n-k})$ we have an active morphism
	\[a:\Sigma^k_\theta D_n(1,...,0)\xactive{(\id,\{0\},...,\{0\})}\Sigma^k_\theta D_{n-k}(1,s_2,...,s_{n-k}),\]
	it is easy to see from the construction in \Cref{prop:D_func} that the restriction of $a$ to the elementary cells of dimension $<n$ is equal to identity, denote by $A^\sigma_t(s_{n-k-t+1},...,s_{n-k})$ the image of the elementary cell $e^\sigma_t$ under $a$, we then have the pushout diagram
	\begin{equation}\label{eq:D_push1}
		\begin{tikzcd}[sep=huge]
			{c_{n-1-t}} & {\Sigma^k_\theta D_{n-1-t}(1,0,...,0)} \\
			{\Sigma^{n-t}_\theta D_t(s_{n-k-t+1},...,s_{n-k})} & {A_t^\pm(s_{n-k-t+1},...,s_{n-k})}
			\arrow[two heads, from=1-1, to=1-2]
			\arrow["{i_{n-k-1-t}^\mp}"', tail, from=1-1, to=2-1]
			\arrow[tail, from=1-2, to=2-2]
			\arrow[two heads, from=2-1, to=2-2]
			\arrow["\ulcorner"{anchor=center, pos=0.125, rotate=225}, draw=none, from=2-2, to=1-1]
		\end{tikzcd}
	\end{equation}
	By \Cref{cor:rel_seg} we see that 
	\begin{equation}
		\Sigma^k_\theta D_{n-k}(1,...,s_{n-k})\cong \underset{(e\xinert{i}\Sigma^k_\theta D_{n-k}(1,...,0))\in\stn^\el_{/\Sigma^k_\theta D_n(1,...,0)}}{\colim} X_e,
	\end{equation}
	where $X_e$ denotes the image of $e$ under $a$, so $X_e\cong e$ if $\dim(e)<n$ and $X_{e_t^\sigma} = A_t^\sigma(s_{n-k-t+1},...,s_{n-k})$. \par
	Note that for any $p\leq n-k$ we have an inert morphism $\cI_p: D^p_{n-k}(1,0,...,0)\xinert{} D_{n-k}(1,...,0)$ whose image contains the cells $e^\sigma_t$ with $t\geq n-k-p$ for $p\geq 1$ and just the cell $e^*_{n-k-1}$ for $p=0$, we will denote
	\[Y_p(1,...,s_{n-k})\bydef\underset{(e\xinert{i}\Sigma^k_\theta D^p_{n-k}(1,...,0))\in\stn^\el_{/\Sigma^k_\theta D^p_n(1,...,0)}}{\colim} X_e,\]
	where we identify $\stn^\el_{/\Sigma^k_\theta D^p_n(1,...,0)}$ with a full subcategory of $\stn^\el_{/\Sigma^k_\theta D_n(1,...,0)}$ by means of $\cI_p$ described above, note that any morphism $\Sigma^k_\theta D_{n-k}(\id,g):\Sigma^k_\theta D_{n-k}(1,...,s_{n-k})\rightarrow\Sigma^k_\theta D_{n-k}(1,...,s_{n-k})$ restricts to an $(k+p-1)$-active morphism $Y_p(\id,g):Y_p(1,...,s_{n-k})\rightarrow Y_p(1,...,s'_{n-k})$. By definition $\mor_{\twar^k(x)}(f,g)$ is a geometric realization of the category with objects given by diagrams 
	\[\begin{tikzcd}[sep=huge]
		{c_n} & {\Sigma^k_\theta D_{n-k}(1,s_2,...,s_{n-k})} & {c_n} \\
		& x
		\arrow[tail, from=1-1, to=1-2]
		\arrow["f"', from=1-1, to=2-2]
		\arrow["F"{description}, from=1-2, to=2-2]
		\arrow[two heads, from=1-3, to=1-2]
		\arrow["g", from=1-3, to=2-2]
	\end{tikzcd}\]
	with morphisms induced by morphisms $\Sigma^k_\theta D_{n-k}(\id,g)$ for $g:(s_2,...,s_{n-k})\rightarrow (s'_2,...,s'_{n-k})$ making the diagram commute. We will define $Z_p(f,g)$ to be the geometric realization of the category with objects given by commutative diagrams
	\[\begin{tikzcd}[sep=huge]
		{c_n} & {Y_p(1,s_2,...,s_{n-k})} \\
		& x
		\arrow[tail, from=1-1, to=1-2]
		\arrow["f"', from=1-1, to=2-2]
		\arrow["{F_p}"{description}, from=1-2, to=2-2]
	\end{tikzcd}\]
	for which the composition $c_n\xactive{}Y_p(1,s_2,...,s_{n-k})\xrightarrow{F}x$ is $(k+p-1)$-parallel to $g$, with morphisms given by $Y(\id,g)$ that make the diagram commute. Note that by definition $Z_{n-k}(f,g)\cong \mor_{\twar^k(x)}(f,g)$ and that the inclusion $Y_p(1,...,s_{n-k})\xinert{} Y_{p'}(1,...,s_{n-k})$ for $p'>p$ induce morphisms $Z_{p'}(f,g)\xrightarrow{\gamma_{p,p'}} Z_p(f,g)$. We will prove that all $Z_p(f,g)$ are isomorphic to $\mor_{\twar^k(x)}(f,g)$ by downward induction on $p$, starting with the trivial case $p=n-k$. Note that $Y_0(1,...,s_{n-k})\cong c_n$, meaning that $Z_0(f,g)$ is a singleton, so proving this claim would conclude the proof of the proposition.\par
	Assume we have proved the claim for $p$, observe that there is an active morphism $a_p:C_n^{k+p-1}\xactive{}\Sigma^k D^p_{n-k}(1,0,...,0)$ sending the cells $i^\pm_{k+p-1}$ to the elementary cells corresponding to the strings $l_0^\pm$ and $i^*_{k+p-1}$ to the composition of all other cells, composing it with $\Sigma^k D^p_{n-k}(1,0,...,0)\xactive{a\circ \cI_p} Y_p(1,...,s_{n-k})$ gives us a morphism $C_n^{k+p-1}(1)\xactive{} Y_p(1,...,s_{n-k})$, using \Cref{cor:rel_seg} we see that $Y_p(1,...,s_{n-k})$ is isomorphic to the colimit of the diagram 
	\begin{equation}\label{eq:A_push}
		\begin{tikzcd}
			& {D_{p+k-1}(1,0,...,0)} && {D_t(1,0,...,0)} \\
			{A^-_{n-k-p}(s_{p+1},...,s_{n-k})} && {Y_{p-1}(1,...,s_{n-k})} && {A^+_{n-k-p}(s_{p+1},...,s_{n-k})}
			\arrow[tail, from=1-2, to=2-1]
			\arrow[tail, from=1-2, to=2-3]
			\arrow[tail, from=1-4, to=2-3]
			\arrow[tail, from=1-4, to=2-5]
		\end{tikzcd}.
	\end{equation}
	Note that by definition $Z_p(f,g)$ is the colimit over $(s_2,...,s_{n-k})\in\Delta^{\times (n-k-1)}$ of $\mor_{\cat_n}^p(Y_p(1,...,s_{n-k}),x)$, where $\mor_{\cat_n}^p(Y_p(1,...,s_{n-k}),x)$ denotes the subset of $\mor_{\cat_n}(Y_p(1,...,s_{n-k}),x)$ containing morphisms $h$ for which the composition $c_n\xinert{}Y_p(1,...,s_{n-k})\xrightarrow{h}x$ is isomorphic to $f$ and $c_n\xactive{} Y_p(1,...,s_{n-k})\xrightarrow{h}x$ is $(k+p-1)$-parallel to $g$. Using \eqref{eq:A_push} we can rewrite it as
	\begin{equation}\label{eq:x_pull}
		\mor_{\cat_n}^-(A^-_{n-k-p},x)\times_{\mor_{\cat_n}(D_{p+k-1},x)} \mor_{\cat_n}^{p-1}(Y_{p-1},x)\times_{\mor_{\cat_n}(D_{p+k-1},x)}\mor_{\cat_n}^-(A^-_{n-k-p},x)
	\end{equation}
	(where we have omitted the variables $s_i$ for typographical reasons), where $\mor_{\cat_n}^\pm(A^\pm_{n-k-p},x)$ denotes the subset of morphisms 
	\[A^-_{n-k-p}(s_{p+1},...,s_{n-k})\xrightarrow{h_\pm} x\]
	whose image contains $d^\pm_{k+p-1}\ima(g)$ and $d^\mp_{k+p-1}\ima(Y_{p-1}(1,...,s_{n-k}))$. Note that for any morphism $v:(s_{p+1},..., s_{n-k})\rightarrow(s'_{p+1},..., s'_{n-k})$ the morphisms $A^\pm_{n-k-p}(v)$ and $Y_{p-1}(\id,v)$ define a natural transformation of diagrams \eqref{eq:A_push}. Denote $\overline{s}\bydef(s_2,...,s_p)$, $\widetilde{s}\bydef (s_{p+1},...,s_{n-k})$ and by $B^\pm$ the terms $\mor_{\cat_n}(D_{p+k-1}(1,...,0),x)$ appearing in \eqref{eq:x_pull}, we will now consider the variables $\overline{s}$ as fixed and calculate the colimit of \eqref{eq:x_pull} over $\widetilde{s}$:
	\begin{align*}
		&\underset{\widetilde{s}\in\Delta^{\times (n-k-p-1),\op}}{\colim}\mor_{\cat_n}^-(A^-_{n-k-p}(\widetilde{s}),x)\times_{B^-} \mor_{\cat_n}^{p-1}(Y_{p-1}(\overline{s}, \widetilde{s}),x)\times_{B^+}\mor_{\cat_n}^-(A^-_{n-k-p}(\widetilde{s}),x)\\
		\cong & \underset{(\widetilde{s^-}, \widetilde{s^*}, \widetilde{s^+})\in\Delta^{\times 3(n-k-p-1),\op}}{\colim}\mor_{\cat_n}^-(A^-_{n-k-p}(\widetilde{s^-}),x)\times_{B^-} \mor_{\cat_n}^{p-1}(Y_{p-1}(\overline{s}, \widetilde{s^*}),x)\times_{B^+}\mor_{\cat_n}^-(A^-_{n-k-p}(\widetilde{s^+}),x)\\
		\cong & \underset{\widetilde{s^-}}{\colim}\;\mor_{\cat_n}^-(A^-_{n-k-p}(\widetilde{s^-}),x)\times_{B^-} \underset{\widetilde{s^*}}{\colim}\;\mor_{\cat_n}^{p-1}(Y_{p-1}(\overline{s}, \widetilde{s^*}),x)\times_{B^=+}\underset{\widetilde{s^+}}{\colim}\;\mor_{\cat_n}^-(A^-_{n-k-p}(\widetilde{s^+}),x),
	\end{align*}
	where the first isomorphism follows since $\Delta^\op$ is sifted and the second since products distribute over colimits in the topos $\cS_{/B_0\times B_1}$. It now remains to prove that
	\[\underset{\widetilde{s^\pm}\in\Delta^{\times(n-k-p-1),\op}}{\colim}\;\mor_{\cat_n}^-(A^\pm_{n-k-p}(\widetilde{s^\pm}),x)\cong B^\pm,\]
	since then using the previous equation we would get
	\begin{align*}
		Z_p(f,g)&\cong \underset{(\widetilde{s},\overline{s})\in\Delta^{\times(n-k),\op}}{\colim} B^-\times_{B^-}\mor_{\cat_n}^{p-1}(Y_{p-1}(\overline{s}, \widetilde{s^*}),x)\times_{B^+} B^+\\
		&\cong \underset{(\widetilde{s},\overline{s})\in\Delta^{\times(n-k),\op}}{\colim} \mor_{\cat_n}^{p-1}(Y_{p-1}(\overline{s}, \widetilde{s^*}),x)\cong Z_{p-1}(f,g).
	\end{align*}
	Now fix some $g'_\pm\in\mor_{\cat_n}(D_{p+k-1}(1,0,...,0),x)$ and denote by $X^\pm$ the fiber of
	\[\underset{\widetilde{s^\pm}\in\Delta^{\times(n-k-p-1),\op}}{\colim}\;\mor_{\cat_n}^-(A^\pm_{n-k-p}(\widetilde{s^\pm}),x)\rightarrow\mor_{\cat_n}(D_{p+k-1}(1,0,...,0),x)\]
	over $g'_\pm$, so we now need to prove that $X^\pm$ is contractible. Denote by $g_\pm$ the composition
	\[c_{p+k-1}\xactive{}D_{p+k-1}(1,0,...,0)\xrightarrow{g'_\pm}x,\]
	note that $g_\pm$ is $(p+k-2)$-active. It follows from \eqref{eq:D_push1} and the definitions that $X^\pm$ is isomorphic to the geometric realization of the subcategory of $\twar^{k+p+1}(x)$ (which we identify with a subcategory of $\stn^\inrt_{n,/x}$ using the inductive assumption on $k$) containing morphisms $y\xinert{i}x$ for which $d^\pm_{p+k-1}y\cong d^\pm_{p+k-1}\ima(g)$ and $d^\mp_{p+k-1}y\cong \ima(g_\pm)$, however we have seen in \Cref{lem:inrt_cont} that this category is contractible.
	\end{proof}
	\begin{cor}\label{cor:twar_iso}
	There is an isomorphism
	\[\twar(\cE)\cong \twar'(\cE)\]
	for any $\cE\in\cat_n$.
	\end{cor}
	\begin{proof}
	Sending the natural transformation $\alpha:F_{s_1}(s_2,...,s_n)\rightarrow*_\cE$ from $\twar(\cE)([s_1])$ to itself viewed as an object of $\twar'(\cE)([s_1])$ defines a functor $F:\twar(\cE)\rightarrow \twar(\cE')$, we need to prove that this defines an isomorphism. Since both the source and target of $F$ satisfy the Segal condition, it suffices to prove that it induces an isomorphism on the spaces of objects and of morphisms. That $F$ induces an isomorphism on the space of objects is clear, to prove the second claim recall that the space of morphisms $\mor_{\twar'(\cE)}(f,g)$ between $f,g:c_n\rightarrow\cE$ is isomorphic to the geometric realization of the category $C(f,g)$ of cospans 
	\begin{equation}\label{eq:twar_mor}
		\begin{tikzcd}[sep=huge]
			{c_n} & x & {c_n} \\
			& \cE
			\arrow["i", tail, from=1-1, to=1-2]
			\arrow["f"', from=1-1, to=2-2]
			\arrow["G"{description}, from=1-2, to=2-2]
			\arrow["a"', two heads, from=1-3, to=1-2]
			\arrow["g", from=1-3, to=2-2]
		\end{tikzcd}
	\end{equation}
	over $\cE$ with $x\in\ctree^h_n$ and morphisms are induced by $f:x\rightarrow y$ over $\cE$, while $\mor_{\twar(\cE)}(f,g)$ is the geometric realization of the subcategory $\cI:c'\hookrightarrow C$ containing objects given by diagrams \eqref{eq:twar_mor} with $x=D_n(1,s_2,...,s_n)$ and morphisms by $D(\id_{[1]},f_2,...,f_n)$ with $f_i:[s_i]\rightarrow[s'_i]$, to prove the claim it suffices to prove that $\cI$ is cofinal. Fix some object $G$ represented by diagram \eqref{eq:twar_mor}, then it is easy to see that $|(\cI/G)|$ is isomorphic to $\mor_{\twar(x)}(i,a)$, and this latter space is contractible by \Cref{prop:twar_stein}.
	\end{proof}
	\begin{cor}\label{cor:twar_col}
	We have 
	\[\twar(\cE)\cong \underset{(\theta\xrightarrow{f}\cE)\in\Theta_{n,/\cE}}{\colim}\;\twar(\theta).\]
	\end{cor}
	\begin{proof}
	Combine \Cref{lem:twar_col1}, \Cref{cor:twar_iso} and \Cref{cor:stein_inrt}.
	\end{proof}
	\begin{defn}\label{def:cot}
	For $\cE\in\cat$ define $L_\cE$ to be the constant functor $\twar(\cE)\rightarrow\spc$ with value $\bS$, given $f:\cE\rightarrow\cD$ define
	\[L_f\bydef \cok(f_!L_\cE\rightarrow L_\cD).\]
	\end{defn}
	\begin{theorem}\label{thm:main_twar}
	A morphism $f:\cE\rightarrow \cD$ is an isomorphism if and only if the following conditions hold:
	\begin{enumerate}
		\item\label{it:tau1} $\tau_{\leq n+1}f:\tau_{\leq n+1}\cE\rightarrow\tau_{\leq n+1}\cD$ is an isomorphism;
		\item\label{it:cot1} $L_{f}\cong 0$.
	\end{enumerate}
	\end{theorem}
	\begin{proof}
	It follows from \Cref{prop:def} that it suffices to show that $\mor_\cat(\twar(\cE),\spc)\cong \stab(\cat_{n,/\cE})$ and that the cotangent complex $\cE$ of \Cref{def:cot} is isomorphic to the one defined in \Cref{prop:def}. To prove the first claim note that it follows from \Cref{prop:main_theta_stab} that it suffices to prove that
	\begin{equation}\label{eq:main_iso}
		\mor_\cat(\twar(\cE),\spc)\cong\mor_\cat(\twar_\theta(\cE),\spc),
	\end{equation}
	\Cref{cor:twar_iso} further implies that we may replace $\twar(\cE)$ with $\twar'(\cE)$ in \eqref{eq:main_iso}. Note that there is a natural morphism $F:\twar_\theta(\cE)\rightarrow \twar'(\cE)$ sending the natural transformation $\alpha:C_q^\inrt\rightarrow*_\cE$ of functors $\Delta^\inrt_{/[q]}\rightarrow\cat_n$, where $*_\cE$ is the constant functor with value $\cE$ and $C^\inrt_q$ is defined in \Cref{constr:twar_C}, corresponding to an element of $\twar^\theta(\cE)([q])$, to itself considered as an element of $\twar'(\cE)([q])$. We need to prove that this is an isomorphism, by the combination of \Cref{lem:twar_theta1}, \Cref{lem:twar_col1}, \Cref{prop:twar_stein} and \Cref{cor:stein_inrt} it suffices to prove this if $\cE\cong \theta$, in which case it follows from \Cref{lem:theta_twar}. To prove the second claim observe that under the isomorphism
	\[\mor_\cat(\twar(\cE),\spc)\cong\underset{(\theta\xrightarrow{f}\cE)\in\Theta_{n,/\cE}}{\lim}\mor_\cat(\twar(\theta),\spc)\]
	$L_\cE$ corresponds to the collection of $f^*L_\cE\cong L_\theta:\twar(\theta)\cong \Theta^\inrt_{n,/\theta}\rightarrow\spc$ of constant functors with values $\bS$, it is easy to see by tracing through various isomorphisms that this is isomorphic to $L_\theta$ defined in \Cref{prop:def}.
	\end{proof}
	\begin{cor}\label{cor:cot_coinit}
		Given a morphism $f:\cE\rightarrow \cD$ we have $L_f\cong 0$ if $f$ is coinitial.
	\end{cor}
	\begin{proof}
		By construction $L_\cE\cong f^*L_\cD$, hence 
		\[L_f\cong \cok(f_!f^*L_\cD\rightarrow L_\cD),\]
		the claim now follows because $f_!f^*\cong \id$ for coinitial morphisms.
	\end{proof}
	\begin{prop}\label{prop:main_prop}
		Assume that $f:\cE\rightarrow \cD$ is such that $L_f\cong 0$ and $\tau_{\leq n+1}f$ induces a monomorphism
		\[\tau_{\leq n+1}f^*:\mor_{\cat_{(n+1,n)}}(\tau_{\leq n+1}\cD,A)\rightarrow \mor_{\cat_{(n+1,n)}}(\tau_{\leq n+1}\cE,A)\]
		for any $A\in\cat_{(n+1,n)}$, then $f^*$ is also a monomorphism and moreover we have a pullback square
		\begin{equation}\label{eq:mono_pull}
			\begin{tikzcd}[sep=huge]
				{\mor_{\cat_n}(\cD,\cA)} & {\mor_{\cat_{(n+1,n)}}(\tau_{\leq n+1}\cD,\tau_{\leq n+1}\cA)} \\
				{\mor_{\cat_n}(\cE,\cA)} & {\mor_{\cat_{(n+1,n)}}(\tau_{\leq n+1}\cE,\tau_{\leq n+1}\cA)}
				\arrow["{\tau_{\leq n+1}}", from=1-1, to=1-2]
				\arrow["{f^*}"', from=1-1, to=2-1]
				\arrow["\ulcorner"{anchor=center, pos=0.125, rotate=45}, draw=none, from=1-1, to=2-2]
				\arrow["{\tau_{\leq n+1}f^*}", from=1-2, to=2-2]
				\arrow["{\tau_{\leq n+1}}"', from=2-1, to=2-2]
			\end{tikzcd}
		\end{equation}
		for any $\cA\in\cat_n$.
	\end{prop}
	\begin{proof}
		We will prove that \eqref{eq:mono_pull} is a pullback square and $f^*$ is a monomorphism for $A\in\cat_{(m,n)}$ by induction on $m$, starting with the case $m=n+1$ where this follows by assumption. Assume we have proved the claim for $m$ and $A\in\cat_{(m+1,n)}$, then we can form a pullback square
		\begin{equation}\label{eq:m_pull}
			\begin{tikzcd}[sep=huge]
				A & {\tau_{\leq n+1}A} \\
				{\tau_{\leq m}A} & {\Omega^\infty(\Sigma^{m+1} \mathrm{H}\pi_m(A))}
				\arrow[from=1-1, to=1-2]
				\arrow[from=1-1, to=2-1]
				\arrow["\ulcorner"{anchor=center, pos=0.125}, draw=none, from=1-1, to=2-2]
				\arrow[from=1-2, to=2-2]
				\arrow[from=2-1, to=2-2]
			\end{tikzcd}
		\end{equation}
		using \cite[Theorem 5.2.]{harpaz2020k}. We can then form the following diagram
		\begin{equation}\label{eq:big_diag}
			\begin{tikzcd}[sep=huge]
				{\mor_{\cat_n}(\cD,A)} & {\mor_{\cat_n}(\cD,A)} & {\mor_{\cat_n}(\cD,\tau_{\leq n+1}A)} \\
				{\mor_{\cat_n}(\cD,A)} & {\mor_{\cat_n}(\cE,A)} & {\mor_{\cat_n}(\cE,\tau_{\leq n+1}A)} \\
				{\mor_{\cat_n}(\cD,\tau_{\leq m}A)} & {\mor_{\cat_n}(\cE,\tau_{\leq m}A)} & {\mor(L_\cE, \Sigma^{m+1} \mathrm{H}\pi_m(A))} \\
				{\mor_{\cat_n}(\cD,\tau_{\leq n+1}A)} & {\mor_{\cat_n}(\cE,\tau_{\leq n+1}A)}
				\arrow[equal, from=1-1, to=1-2]
				\arrow[equal, from=1-1, to=2-1]
				\arrow["{(a)}"{description}, draw=none, from=1-1, to=2-2]
				\arrow["{\tau_{\leq n+1}}", from=1-2, to=1-3]
				\arrow["{f^*}"{description}, from=1-2, to=2-2]
				\arrow["{(b)}"{description}, draw=none, from=1-2, to=2-3]
				\arrow["{\tau_{\leq n+1}f^*}", from=1-3, to=2-3]
				\arrow["{f^*}", from=2-1, to=2-2]
				\arrow["{\tau_{\leq m}}"', from=2-1, to=3-1]
				\arrow["{(c)}"{description}, draw=none, from=2-1, to=3-2]
				\arrow["{\tau_{\leq n+1}}"{description}, from=2-2, to=2-3]
				\arrow["{\tau_{\leq m}}"{description}, from=2-2, to=3-2]
				\arrow["{(d)}"{description}, draw=none, from=2-2, to=3-3]
				\arrow[from=2-3, to=3-3]
				\arrow["{\tau_{\leq m}f^*}"{description}, from=3-1, to=3-2]
				\arrow["{\tau_{\leq n+1}}"', from=3-1, to=4-1]
				\arrow["{(e)}"{description}, draw=none, from=3-1, to=4-2]
				\arrow[from=3-2, to=3-3]
				\arrow["{\tau_{\leq n+1}}"{description}, from=3-2, to=4-2]
				\arrow["{\tau_{\leq n+1}f^*}"', from=4-1, to=4-2]
			\end{tikzcd},
		\end{equation}
		we will prove that every square in it is a pullback, for that we will use the $\infty$-categorical pasting law for pullbacks of \cite[Lemma 4.4.2.1.]{lurie2009higher} without further mention. The square $(d)$ in \eqref{eq:big_diag} is obtained by applying $\mor_{\cat_n}(\cE,-)$ to \eqref{eq:m_pull}, hence it is a pullback. Applying $\mor_{\cat_n}(\cD,-)$ to it we see that the square
		\[\begin{tikzcd}[sep=huge]
			{\mor_{\cat_n}(\cD,A)} & {\mor_{\cat_n}(\cD,\tau_{\leq n+1}A)} \\
			{\mor_{\cat_n}(\cD,\tau_{\leq m}A)} & {\mor(L_\cD, \Sigma^{m+1} \mathrm{H}\pi_m(A))}
			\arrow[from=1-1, to=1-2]
			\arrow[from=1-1, to=2-1]
			\arrow["\ulcorner"{anchor=center, pos=0.125, rotate=45}, draw=none, from=1-1, to=2-2]
			\arrow[from=1-2, to=2-2]
			\arrow[from=2-1, to=2-2]
		\end{tikzcd}\]
		is a pullback, however using that $L_\cE\cong L_\cD$ by assumption and $\tau_{\leq m}f^*$ is a monomorphism by induction, we see that the rectangle $(bd)$ is also a pullback, meaning that the square $(b)$ must now be a pullback too. This prove that \eqref{eq:mono_pull} is a pullback and also immediately implies that the rectangle $(ce)$ in \eqref{eq:big_diag} is a pullback, the square $(e)$ is a pullback by induction, hence the square $(c)$ must be a pullback too. Note that the outer square $(abcd)$ is a pullback since it is obtained by applying $\mor_{\cat_n}(\cD,-)$ to \eqref{eq:mono_pull}, since we have already seen that $(bd)$ is a pullback this implies that the rectangle $(ac)$ is a pullback. Finally, combining this with the earlier observation that the square $(c)$ is a pullback we obtain that the square $(a)$ is a pullback, which means exactly that $f^*$ is a monomorphism.
	\end{proof}
	\begin{ex}\label{ex:lowD}
		Despite the formidable definition of $\twar(\cE)$, the idea behind it is quite simple: its objects are $n$-morphisms $X$ in $\cE$ and the space of morphisms is such that its points correspond to decompositions
		\[X\cong A^{n-1}_-*_{n-1}(A^{n-2}_-*_{n-2}(...*_1(A^0_- *_0 Y *_0 A^0_+)*_1...)*_{n-2} A^{n-2}_+)*_{n-1} A^{n-1}_+\]
		and the paths in this space correspond to similar decompositions of various $A^k_\pm$. In this example we will provide explicit descriptions of this category in low dimensions:
		\begin{enumerate}
			\item if $\dim(\cE)=1$, then $\twar(\cE)$ is simply the ordinary twisted arrows category with objects given by morphisms $f:x\rightarrow y$ and morphisms by diagrams 
			\[\begin{tikzcd}[sep=huge]
				z & w \\
				x & y
				\arrow["g", from=1-1, to=1-2]
				\arrow["t", from=1-2, to=2-2]
				\arrow["h", from=2-1, to=1-1]
				\arrow["f"', from=2-1, to=2-2]
			\end{tikzcd}\]
			with source $g$ and target $f$;
			\item in dimension 2 the objects of $\twar(\cE)$ are 2-morphisms $\alpha$ and morphisms are given by the geometric realization of a category with objects given by diagrams
			\begin{equation}\label{eq:2d_twar}
				\begin{tikzcd}[sep=huge]
					x & z & w & y
					\arrow[""{name=0, anchor=center, inner sep=0}, "{a_1}"{description}, curve={height=12pt}, from=1-1, to=1-2]
					\arrow[""{name=1, anchor=center, inner sep=0}, "{a_0}"{description}, curve={height=-12pt}, from=1-1, to=1-2]
					\arrow[""{name=2, anchor=center, inner sep=0}, "f"{description}, curve={height=-30pt}, from=1-1, to=1-4]
					\arrow[""{name=3, anchor=center, inner sep=0}, "g"{description}, curve={height=30pt}, from=1-1, to=1-4]
					\arrow[""{name=4, anchor=center, inner sep=0}, "{b_1}"{description}, curve={height=12pt}, from=1-2, to=1-3]
					\arrow[""{name=5, anchor=center, inner sep=0}, "{b_0}"{description}, curve={height=-12pt}, from=1-2, to=1-3]
					\arrow[""{name=6, anchor=center, inner sep=0}, "{c_1}"{description}, curve={height=12pt}, from=1-3, to=1-4]
					\arrow[""{name=7, anchor=center, inner sep=0}, "{c_0}"{description}, curve={height=-12pt}, from=1-3, to=1-4]
					\arrow["\alpha"{description}, shorten <=3pt, shorten >=3pt, Rightarrow, from=1, to=0]
					\arrow["\epsilon"', shorten <=2pt, shorten >=2pt, Rightarrow, from=2, to=5]
					\arrow["\beta"{description}, shorten <=3pt, shorten >=3pt, Rightarrow, from=5, to=4]
					\arrow["\eta", shorten <=2pt, shorten >=2pt, Rightarrow, from=4, to=3]
					\arrow["\gamma"{description}, shorten <=3pt, shorten >=3pt, Rightarrow, from=7, to=6]
				\end{tikzcd},
			\end{equation}
			which we will denote by $(\epsilon|\alpha,\gamma|\eta)$, where the target of \eqref{eq:2d_twar} is $\eta*_1 (\alpha*_0\beta*_0\gamma)*_1 \epsilon$, and morphisms by pairs of commutative diagrams
			\begin{equation}\label{eq:2d_mor}
			\begin{tikzcd}[sep=huge]
				{a'_0} & {a'_1} & {c'_0} & {c'_1} \\
				{a_0} & {a_1} & {c_0} & {c_1}
				\arrow["{\alpha'}", from=1-1, to=1-2]
				\arrow["{\alpha'_1}", from=1-2, to=2-2]
				\arrow["{\gamma'}", from=1-3, to=1-4]
				\arrow["{\alpha_0'}", from=2-1, to=1-1]
				\arrow["\alpha"', from=2-1, to=2-2]
				\arrow["{\gamma'_0}", from=2-3, to=1-3]
				\arrow["\gamma"', from=2-3, to=2-4]
				\arrow["{\gamma'_1}"', from=1-4, to=2-4]
			\end{tikzcd}
			\end{equation} 
			such that the source of \eqref{eq:2d_mor} is $(\epsilon*_1(\alpha'_0*_0 b_0*_0 \gamma'_0)|\alpha', \gamma'|(\alpha'_1*_0 b_1*_0 \gamma'_1)*_1 \eta)$;
			\item in dimension 3 a morphism in $\twar(\cE)$ can be visualized as the diagram \eqref{eq:2d_twar} with 2-morphisms "thickened" to 3-morphisms that is sandwiched between two other 3-morphisms. Sadly, it is quite difficult to depict this, so we will have to settle for a series of two-dimensional diagrams: the objects of $\twar(\cE)$ are 3-morphisms and the morphisms are given by the geometric realization of the double category $D(i,j$) such that $D(0,0)$ consists of diagrams
			\begin{equation}\label{eq:3d_obj}
				\begin{tikzcd}[sep=huge]
					x & y && x & z & w & w \\
					\\
					\\
					x & y && x & z & w & y
					\arrow[""{name=0, anchor=center, inner sep=0}, "g"', curve={height=12pt}, from=1-1, to=1-2]
					\arrow[""{name=1, anchor=center, inner sep=0}, "f", curve={height=-12pt}, from=1-1, to=1-2]
					\arrow["X", Rightarrow , scaling nfold=3, from=1-2, to=1-4]
					\arrow[""{name=2, anchor=center, inner sep=0}, "{a_1}"{description}, curve={height=12pt}, from=1-4, to=1-5]
					\arrow[""{name=3, anchor=center, inner sep=0}, "{a_0}"{description}, curve={height=-12pt}, from=1-4, to=1-5]
					\arrow[""{name=4, anchor=center, inner sep=0}, "f", curve={height=-30pt}, from=1-4, to=1-7]
					\arrow[""{name=5, anchor=center, inner sep=0}, "g"{description}, curve={height=30pt}, from=1-4, to=1-7]
					\arrow[""{name=6, anchor=center, inner sep=0}, "{b_1}"{description}, curve={height=12pt}, from=1-5, to=1-6]
					\arrow[""{name=7, anchor=center, inner sep=0}, "{b_0}"{description}, curve={height=-12pt}, from=1-5, to=1-6]
					\arrow[""{name=8, anchor=center, inner sep=0}, "{c_1}"{description}, curve={height=12pt}, from=1-6, to=1-7]
					\arrow[""{name=9, anchor=center, inner sep=0}, "{c_0}"{description}, curve={height=-12pt}, from=1-6, to=1-7]
					\arrow[""{name=10, anchor=center, inner sep=0}, "f", curve={height=-12pt}, from=4-1, to=4-2]
					\arrow[""{name=11, anchor=center, inner sep=0}, "g"', curve={height=12pt}, from=4-1, to=4-2]
					\arrow["W", Rightarrow, scaling nfold=3, from=4-4, to=4-2]
					\arrow[""{name=12, anchor=center, inner sep=0}, curve={height=-12pt}, from=4-4, to=4-5]
					\arrow[""{name=13, anchor=center, inner sep=0}, curve={height=12pt}, from=4-4, to=4-5]
					\arrow[""{name=14, anchor=center, inner sep=0}, "f"{description}, curve={height=-30pt}, from=4-4, to=4-7]
					\arrow[""{name=15, anchor=center, inner sep=0}, "g"', curve={height=30pt}, from=4-4, to=4-7]
					\arrow[""{name=16, anchor=center, inner sep=0}, curve={height=12pt}, from=4-5, to=4-6]
					\arrow[""{name=17, anchor=center, inner sep=0}, curve={height=-12pt}, from=4-5, to=4-6]
					\arrow[""{name=18, anchor=center, inner sep=0}, curve={height=12pt}, from=4-6, to=4-7]
					\arrow[""{name=19, anchor=center, inner sep=0}, curve={height=-12pt}, from=4-6, to=4-7]
					\arrow["\zeta"{description}, shorten <=3pt, shorten >=3pt, Rightarrow, from=1, to=0]
					\arrow["{\alpha_0}"{description}, shorten <=3pt, shorten >=3pt, Rightarrow, from=3, to=2]
					\arrow["{\epsilon_0}"{description}, shorten <=2pt, shorten >=2pt, Rightarrow, from=4, to=7]
					\arrow["{Y*_1(A*_0B*_0C)*_1Z}", shorten <=5pt, shorten >=5pt, Rightarrow, scaling nfold=3, from=5, to=14]
					\arrow["{\beta_0}"{description}, shorten <=3pt, shorten >=3pt, Rightarrow, from=7, to=6]
					\arrow["{\eta_0}"{description}, shorten <=2pt, shorten >=2pt, Rightarrow, from=6, to=5]
					\arrow["{\gamma_0}"{description}, shorten <=3pt, shorten >=3pt, Rightarrow, from=9, to=8]
					\arrow["\kappa"{description}, shorten <=3pt, shorten >=3pt, Rightarrow, from=10, to=11]
					\arrow["{\alpha_1}"{description}, shorten <=3pt, shorten >=3pt, Rightarrow, from=12, to=13]
					\arrow["{\epsilon_1}"{description}, shorten <=2pt, shorten >=2pt, Rightarrow, from=14, to=17]
					\arrow["{\beta_1}"{description}, shorten <=3pt, shorten >=3pt, Rightarrow, from=17, to=16]
					\arrow["{\eta_1}"{description}, shorten <=2pt, shorten >=2pt, Rightarrow, from=16, to=15]
					\arrow["{\gamma_1}"{description}, shorten <=3pt, shorten >=3pt, Rightarrow, from=19, to=18]
				\end{tikzcd}
			\end{equation}
			such that the source of \eqref{eq:3d_obj} is $B$ and the target $W*_2(Y*_1(A*_0 B*_) C)*_1 Z)*_2 X$, we will denote such object by $(X|Y|A,C|Z|W)$, the space $D(0,1)$ is given by pairs of commutative diagrams
			\begin{equation}\label{eq:3d01}
				\begin{tikzcd}[sep=huge]
					{\epsilon'_0} & { \epsilon'_1} & { \eta'_0} & { \eta'_1} \\
					{ \epsilon_0} & { \epsilon_1} & { \eta_0} & { \eta_1}
					\arrow["{Y'}", from=1-1, to=1-2]
					\arrow["{Y_1'}", from=1-2, to=2-2]
					\arrow["{Z'}", from=1-3, to=1-4]
					\arrow["{Z'_1}", from=1-4, to=2-4]
					\arrow["{Y_0'}", from=2-1, to=1-1]
					\arrow["Y"', from=2-1, to=2-2]
					\arrow["{Z_0'}", from=2-3, to=1-3]
					\arrow["Z"', from=2-3, to=2-4]
				\end{tikzcd}
			\end{equation}
			with source $(X*_2 (Y'_0*_1(\alpha_0*_0\beta_0 *_0\gamma_0)*_1 Z'_0)|Y|A,C|Z|(Y'_1*_1(\alpha_1*_0\beta_1 *_0\gamma_1)*_1 Z'_1)*_2 W)$, the space $D(1,0)$ contains pairs of commutative diagrams
			\begin{equation}\label{eq:3d10}
				\begin{tikzcd}[sep=huge]
					{a_0} & {a'_0} & {a'_1} & {a_1} & {c_0} & {c'_0} & {c'_1} & {c_1}
					\arrow[""{name=0, anchor=center, inner sep=0}, "{\alpha_0^0}"{description}, curve={height=-12pt}, from=1-1, to=1-2]
					\arrow[""{name=1, anchor=center, inner sep=0}, "{\alpha_1^0}"{description}, curve={height=12pt}, from=1-1, to=1-2]
					\arrow[""{name=2, anchor=center, inner sep=0}, "{\alpha_0}", curve={height=-30pt}, from=1-1, to=1-4]
					\arrow[""{name=3, anchor=center, inner sep=0}, "{\alpha_1}"', curve={height=30pt}, from=1-1, to=1-4]
					\arrow[""{name=4, anchor=center, inner sep=0}, "{\alpha_0^1}"{description}, curve={height=-12pt}, from=1-2, to=1-3]
					\arrow[""{name=5, anchor=center, inner sep=0}, "{\alpha^1_1}"{description}, curve={height=12pt}, from=1-2, to=1-3]
					\arrow[""{name=6, anchor=center, inner sep=0}, "{\alpha_0^2}"{description}, curve={height=-12pt}, from=1-3, to=1-4]
					\arrow[""{name=7, anchor=center, inner sep=0}, "{\alpha_1^2}"{description}, curve={height=12pt}, from=1-3, to=1-4]
					\arrow[""{name=8, anchor=center, inner sep=0}, "{\gamma^0_0}"{description}, curve={height=-12pt}, from=1-5, to=1-6]
					\arrow[""{name=9, anchor=center, inner sep=0}, "{\gamma^0_1}"{description}, curve={height=12pt}, from=1-5, to=1-6]
					\arrow[""{name=10, anchor=center, inner sep=0}, "{\gamma_0}", curve={height=-30pt}, from=1-5, to=1-8]
					\arrow[""{name=11, anchor=center, inner sep=0}, "{\gamma_1}"', curve={height=30pt}, from=1-5, to=1-8]
					\arrow[""{name=12, anchor=center, inner sep=0}, "{\gamma^1_0}"{description}, curve={height=-12pt}, from=1-6, to=1-7]
					\arrow[""{name=13, anchor=center, inner sep=0}, "{\gamma_1^1}"{description}, curve={height=12pt}, from=1-6, to=1-7]
					\arrow[""{name=14, anchor=center, inner sep=0}, "{\gamma^2_0}"{description}, curve={height=-12pt}, from=1-7, to=1-8]
					\arrow[""{name=15, anchor=center, inner sep=0}, "{\gamma^2_1}"{description}, curve={height=12pt}, from=1-7, to=1-8]
					\arrow["{A^0_1}"{description}, shorten <=3pt, shorten >=3pt, Rightarrow, from=0, to=1]
					\arrow["{A_0}"{description}, shorten <=2pt, shorten >=2pt, Rightarrow, from=2, to=4]
					\arrow["{A^1_1}"{description}, shorten <=3pt, shorten >=3pt, Rightarrow, from=4, to=5]
					\arrow["{A_2}"{description}, shorten <=2pt, shorten >=2pt, Rightarrow, from=5, to=3]
					\arrow["{A^2_1}"{description}, shorten <=3pt, shorten >=3pt, Rightarrow, from=6, to=7]
					\arrow["{C^0_1}"{description}, shorten <=3pt, shorten >=3pt, Rightarrow, from=8, to=9]
					\arrow["{C_0}"{description}, shorten <=2pt, shorten >=2pt, Rightarrow, from=10, to=12]
					\arrow["{C^1_1}"{description}, shorten <=3pt, shorten >=3pt, Rightarrow, from=12, to=13]
					\arrow["{C_2}"{description}, shorten <=2pt, shorten >=2pt, Rightarrow, from=13, to=11]
					\arrow["{C_1^2}"{description}, shorten <=3pt, shorten >=3pt, Rightarrow, from=14, to=15]
				\end{tikzcd}
			\end{equation}
			such that the source of \eqref{eq:3d10} is $(X*_2(A_0 *_0 \beta_0*_0 C_0)|Y*_1(A^2_1 *_0 b_0 *_0 C^2_1)|A^1_1,C^1_1|(A^0_1 *_0 b_1 *_0 C^0_1)*_1 Z|(A_2 *_0 \beta_1*_0 C_2)*_0 W)$. Finally, the space $D(1,1)$ corresponds to the data of diagrams \eqref{eq:3d01} and \eqref{eq:3d10} as above together with decompositions
			\[A^i_1\cong A^i_{R,1}*_2 A^{i}_{*,1} *_2 A^i_{L,1}\text{ and }C^i_1\cong C^i_{R,1}*_2 C^{i}_{*,1} *_2 C^i_{L,1}\]
			for $i\in\{0,2\}\}$, we are confident in the reader's ability to discern what the sources and targets of those objects are.
		\end{enumerate}
	\end{ex}
	\section{Lax-idempotent monads}
	Given a monad on an $\infty$-category, it is generally quite difficult to describe its category of algebras since giving an object a structure of an algebra required an infinite amount of "coherence data". One exception to this are idempotent monads -- the category of algebras for an idempotent monad is simply a full subcategory of the category in question. As similar effect can be observed in $(\infty,2)$-category -- consider for example the category $\mathrm{Cart}(\cC)$ of Cartesian fibrations over $\cC$, it admits a forgetful functor to $\cat_{/\cC}$  and this functor is monadic. Yet the category $\mathrm{Cart}(\cC)$ admits a relatively simple description -- by the results of  \cite[Section 2.4.]{lurie2009higher} it includes those functors $F:\cD\rightarrow \cC$ for which every morphism $x\xrightarrow{g}F(y)$ lefts to a Cartesian morphism $F^*x\xrightarrow{g_*}y$ and a functor $G:\cD\rightarrow \cE$ between the objects of $\mathrm{Cart}(\cC)$ defines a morphism of algebras precisely if it preserves Cartesian morphisms.\par 
	This features are possessed more generally by \textit{lax-idempotent} monads introduced in \cite{kock1995monads}, which include free (co)fibrations monads as well as free (co)completions. A monad $T$ is lax idempotent in the unit morphism $T\xrightarrow{Ts}TT$ is left adjoint to the multiplication $m:TT\rightarrow T$ and an object $x$ admits a (unique) structure of an algebra if and only if $x\xrightarrow{s}Tx$ admits a left adjoint. All of those facts are proved in \cite{kock1995monads} in the setting of ordinary categories, however to the knowledge of the author none of those facts have been generalized to $\infty$-categories.\par 
	In this final section we will use deformation theory developed in the previous section to generalize some of the results mentioned above. A proper treatment of this subject will probably require a separate work, so we will limit ourselves to characterizing lax-idemoptent monads as those for which multiplication is adjoint to units.
	\begin{lemma}\label{lem:n_colim}
		The categories of the form $[n]\in \Delta$ admit all limits and colimits and moreover any morphism $f:[n]\rightarrow [m]$ preserves limits and colimits of non-empty categories.
	\end{lemma}
	\begin{proof}
		Note that $[n]$ admits both an initial and a final object, meaning that it admits both a limit and a colimit of an empty diagram. Now suppose we have a functor $F:S\rightarrow [n]$, in that case $\colim F$ must be an object of $[n]$ such that for $i\in [n]$ we have $i\geq \colim F$ if and only if $i\geq F(s)$ for all $s\in S$. It follows that we can set 
		\begin{equation}\label{eq:n_colim}
			\colim F\bydef \sup_{s\in S}F(s),
		\end{equation}
		it is easy to see that such an object always exists. Similarly, we have
		\begin{equation}\label{eq:n_lim}
			\lim F\cong \inf_{s\in S}F(s),
		\end{equation}
		note that any order preserving functor necessarily preserves maxima and minima, from which the last claim of the lemma follows.  
	\end{proof}
	\begin{lemma}\label{lem:n_adj}
		Given a morphism $f:[n]\rightarrow [m]$, it admits a left adjoint (as a functor in $\cat$) if and only if it preserves the minimal element and a right adjoint if and only if it preserves the minimal element, moreover when they exist the adjoints are given by 
		\begin{equation}\label{eq:left}
			f^L(x)\cong \underset{x\leq f(i)}{\inf}i
		\end{equation}
		and
		\begin{equation}\label{eq:right}
			f^R(x)\cong \underset{x\geq f(i)}{\sup}i.
		\end{equation}
	\end{lemma}
	\begin{proof}
		Note that for all categories of the form $[n], [m]\in \Delta$ and any functor $f:[n]\rightarrow [m]$, it satisfies the solution set condition of \cite[Definition 3.2.1.]{nguyen2020adjoint} (since $[n]$ and $[m]$ are themselves small, and hence so are all $f_{i/}$ for $i\in [m]$). Since $[n]$ and $[m]$ are moreover complete and cocomplete by \Cref{lem:n_colim}, we can apply the generalized adjoint functor theorem of \cite{nguyen2020adjoint} which states that $f$ admits a left (resp. right) adjoint if and only if it preserves all (co)limits. BY the second claim of \Cref{lem:n_colim}, $f$ always preserves non-empty limits and colimits, so it preserves all (co)limits if and only if it preserves the maximal (minimal) object.\par
		Finally, to prove the last claim observe (by untangling the constructions of loc. cit.) that for $F:\cD\rightarrow \cC$ satisfying the conditions of \cite[Theorem 3.2.5.]{nguyen2020adjoint}, its left adjoint is given by
		\begin{equation}\label{eq:L_adj}
			F^L(c)\cong \underset{(c\xrightarrow{f}Fd)\in F'_{c/}}{\colim}d,
		\end{equation}
		where $F'_{c/}\hookrightarrow F_{c/}$ is the small weakly initial subcategory (which exists by the solution set condition), and a dual statement holds for the right adjoint. Equations \eqref{eq:left} and \eqref{eq:right} now follow from \eqref{eq:L_adj} and its dual together with \eqref{eq:n_lim} and \eqref{eq:n_colim}.
	\end{proof}
	\begin{lemma}
		For a morphism $f:[n]\rightarrow [m]$ the following claims are equivalent:
		\begin{enumerate}
			\item\label{it:sur1} $f$ is surjective;
			\item\label{it:sur2} $f$ admits a left adjoint and $f\circ f^L\cong \id$;
			\item\label{it:sur2'} $f$ admits a right adjoint and $f\circ f^R\cong \id$;
			\item\label{it:sur3} $f$ admits a left adjoint and $f^L$ is injective;
			\item\label{it:sur3'} $f$ admits a right adjoint and $f^R$ is injective;
			\item\label{it:sur4} $f$ induces a comonadic adjunction;
			\item\label{it:sur4'} $f$ induces a monadic adjunction.
		\end{enumerate}
	\end{lemma}
	\begin{proof}
		Assume we have a left-adjoint morphism $j:[l]\rightarrow [p]$, we claim that it is injective if and only if $j^R\circ j\cong \id$. Indeed, if $j$ is injective then using \eqref{eq:right} we get
		\[j^Rj(x)\cong \underset{j(i)\geq j(x)}{\sup}j(i)\cong \underset{i\in j^{-1}(j(x))}{\sup}j(i)\cong j(x).\]
		Conversely, if $j^R\circ j\cong\id$, then for all $x\in [l]$ the fiber $j^{-1}(j(x))$ is non-empty and $x$ is its maximal element, which is obviously equivalent to $j$ being injective. Dualizing this argument, we see that a right-adjoint morphism $j$ is injective if and only if $j^L\circ j\cong \id$.\par 
		If follows from this (applied to $f^L$ and $f^R$ respectively) that $\eqref{it:sur2}\Leftrightarrow \eqref{it:sur3}$ and $\eqref{it:sur2'}\Leftrightarrow \eqref{it:sur3'}$. We will now prove the equivalence of $\eqref{it:sur3'}$ and \eqref{it:sur4'}: by \cite[Theorem 4.7.3.5.]{luriehigher} $f^R$ is monadic if and only if it is conservative and preserves colimits of $f^R$-split simplicial object. The second condition is automatic in light of \Cref{lem:n_colim}, so it remains to observe that a functor $g:[m]\rightarrow [n]$ is conservative if and only if it is injective. Dualizing this argument we obtain $\eqref{it:sur3}\Leftrightarrow \eqref{it:sur4}$.\par
		Finally, note that if $f$ is surjective, then it admits both a left and a right adjoint by \Cref{lem:n_adj} and moreover we have
		\[f\circ f^R(x)\cong \underset{i\in f^{-1}(x)}{\sup}f(i)\cong x\]
		and
		\[f\circ f^L(x)\cong \underset{i\in f^{-1}(x)}{\inf}f(i)\cong x,\]
		so \eqref{it:sur1} implies \eqref{it:sur2} and \eqref{it:sur2'}. Conversely, if $f$ admits a left (resp. right) adjoint and $f\circ f^L\cong \id$ (resp. $f\circ f^R\cong \id$), it follows that the fibers $f^{-1}(x)$ are non-empty for $x\in [m]$, meaning that $f$ is surjective.
	\end{proof}
	\begin{lemma}\label{lem:cont}
		Given a pair $f,g:[n]\rightrightarrows[m]$ of active morphisms and a pair $u, v:[p]\rightrightarrows[q]$ of active morphisms together with $\alpha:f\rightarrow g$ and $\beta:u\rightarrow v$, denote by $C$ the category with object given by diagrams of active morphisms
		\begin{equation}\label{eq:c-mor}
			\begin{tikzcd}[sep=huge]
				{[p]} & {[q]} \\
				{[n]} & {[m]}
				\arrow[""{name=0, anchor=center, inner sep=0}, "v", curve={height=-6pt}, from=1-1, to=1-2]
				\arrow[""{name=1, anchor=center, inner sep=0}, "u"', curve={height=6pt}, from=1-1, to=1-2]
				\arrow[""{name=2, anchor=center, inner sep=0}, "a"', curve={height=6pt}, from=1-2, to=2-2]
				\arrow[""{name=3, anchor=center, inner sep=0}, "b", curve={height=-6pt}, from=1-2, to=2-2]
				\arrow[""{name=4, anchor=center, inner sep=0}, "t"', curve={height=6pt}, from=2-1, to=1-1]
				\arrow[""{name=5, anchor=center, inner sep=0}, "s", curve={height=-6pt}, from=2-1, to=1-1]
				\arrow[""{name=6, anchor=center, inner sep=0}, "f"', curve={height=6pt}, from=2-1, to=2-2]
				\arrow[""{name=7, anchor=center, inner sep=0}, "g", curve={height=-6pt}, from=2-1, to=2-2]
				\arrow["\leq"{description}, draw=none, from=1, to=0]
				\arrow["\leq"{description}, draw=none, from=2, to=3]
				\arrow["\leq"{description}, draw=none, from=5, to=4]
				\arrow["\leq"{description}, draw=none, from=6, to=7]
			\end{tikzcd}
		\end{equation}
		such that $f\leq a\circ u\circ s\leq b\circ v\circ t\leq g$ and morphisms given by pairs
		\begin{equation}\label{eq:c-mor2}
			(s\leq s'\leq t'\leq t, a\leq a'\leq b'\leq b),
		\end{equation}
		then the geometric realization of $C$ is either empty or contractible.
	\end{lemma}
	\begin{proof}
		Note that $C$ admits a natural forgetful functor to $\twar(\mor_\cat([n],[p]))$ given by sending the diagram \eqref{eq:c-mor} to $s\leq t$ and a morphism \eqref{eq:c-mor2} to $s\leq s'\leq t'\leq t$, we claim that it is a Cartesian fibration. Indeed, given \eqref{eq:c-mor} and $(s\leq s'\leq t'\leq t)$, it is clear that it admits a Cartesian lifting with source
		\[\begin{tikzcd}[sep=huge]
			{[p]} & {[q]} \\
			{[n]} & {[m]}
			\arrow[""{name=0, anchor=center, inner sep=0}, "v", curve={height=-6pt}, from=1-1, to=1-2]
			\arrow[""{name=1, anchor=center, inner sep=0}, "u"', curve={height=6pt}, from=1-1, to=1-2]
			\arrow[""{name=2, anchor=center, inner sep=0}, "a"', curve={height=6pt}, from=1-2, to=2-2]
			\arrow[""{name=3, anchor=center, inner sep=0}, "b", curve={height=-6pt}, from=1-2, to=2-2]
			\arrow[""{name=4, anchor=center, inner sep=0}, "{t'}"', curve={height=6pt}, from=2-1, to=1-1]
			\arrow[""{name=5, anchor=center, inner sep=0}, "{s'}", curve={height=-6pt}, from=2-1, to=1-1]
			\arrow[""{name=6, anchor=center, inner sep=0}, "f"', curve={height=6pt}, from=2-1, to=2-2]
			\arrow[""{name=7, anchor=center, inner sep=0}, "g", curve={height=-6pt}, from=2-1, to=2-2]
			\arrow["\leq"{description}, draw=none, from=1, to=0]
			\arrow["\leq"{description}, draw=none, from=2, to=3]
			\arrow["\leq"{description}, draw=none, from=5, to=4]
			\arrow["\leq"{description}, draw=none, from=6, to=7]
		\end{tikzcd}.\]
		We now claim that the fibers of this fibration are either empty or contractible. Indeed, assume we are given a diagram \eqref{eq:c-mor}, note first that since all morphisms are presumed active, they admit left and right adjoints by \Cref{lem:n_adj}. In particular, we can consider $f\circ s^R\circ u^R:[q]\rightarrow[m]$, however it is not in general active, so we define 
		\[u'(i)\bydef 
		\begin{cases}
			0\text{ if $i=0$}\\
			u^R(i)\text{ otherwise}
		\end{cases}\]
		and
		\[s'(i)\bydef 
		\begin{cases}
			0\text{ if $i=0$}\\
			s^R(i)\text{ otherwise}
		\end{cases}\]
		(note that $s^R$ and $u^R$ preserve maximal elements as right adjoints). In that case $a_0\bydef f\circ s'\circ u'$ is an active morphism, moreover if for $x\in [n]$ we have $u\circ s(x)=0$, then
		\[f(x)\leq a\circ u\circ s(x) = 0 = a_0\circ s\circ u(x),\]
		and if $u\circ s(x)>0$, then
		\[f(x)\leq f\circ (u\circ s)^R\circ (u\circ s)(x) = f\circ s'\circ u'\circ u\circ s(x),\]
		where the first inequality follows by adjointness and the second by definition, so in general $f\leq a_0\circ u\circ s$. Similarly we have $a_0(0) = 0 = a(0)$ and for $i>0$ we have
		\[a_0(i) = f\circ s^R\circ u^R(i)\leq a\circ u\circ s\circ s^R\circ u^R\leq a,\]
		where the first equality follows by definition, the first inequality since $f\leq a\circ u\circ s$ by assumption and the last inequality by adjointness. Similarly, we can define 
		\[v'(i)\bydef 
		\begin{cases}
			p\text{ if $i=q$}\\
			v^L(i)\text{ otherwise}
		\end{cases},\]
		\[t'(i)\bydef 
		\begin{cases}
			n\text{ if $i=p$}\\
			t^L(i)\text{ otherwise}
		\end{cases}\]
		and set $b_0\bydef g\circ t'\circ v'$, then by dual arguments one can show $g\geq b_0\circ v\circ t$ and $b\leq b_0$. It follows that the diagram
		\[\begin{tikzcd}[sep=huge]
			{[p]} & {[q]} \\
			{[n]} & {[m]}
			\arrow[""{name=0, anchor=center, inner sep=0}, "v", curve={height=-6pt}, from=1-1, to=1-2]
			\arrow[""{name=1, anchor=center, inner sep=0}, "u"', curve={height=6pt}, from=1-1, to=1-2]
			\arrow[""{name=2, anchor=center, inner sep=0}, "{a_0}"', curve={height=6pt}, from=1-2, to=2-2]
			\arrow[""{name=3, anchor=center, inner sep=0}, "{b_0}", curve={height=-6pt}, from=1-2, to=2-2]
			\arrow[""{name=4, anchor=center, inner sep=0}, "t"', curve={height=6pt}, from=2-1, to=1-1]
			\arrow[""{name=5, anchor=center, inner sep=0}, "s", curve={height=-6pt}, from=2-1, to=1-1]
			\arrow[""{name=6, anchor=center, inner sep=0}, "f"', curve={height=6pt}, from=2-1, to=2-2]
			\arrow[""{name=7, anchor=center, inner sep=0}, "g", curve={height=-6pt}, from=2-1, to=2-2]
			\arrow["\leq"{description}, draw=none, from=1, to=0]
			\arrow["\leq"{description}, draw=none, from=2, to=3]
			\arrow["\leq"{description}, draw=none, from=5, to=4]
			\arrow["\leq"{description}, draw=none, from=6, to=7]
		\end{tikzcd}\]
		defines the final object of the fiber over $s\leq t$, so in particular it is contractible.\par
		The geometric realization of $C$ is then isomorphic to the geometric realization of the full subcategory $D\hookrightarrow\twar(\mor_\cat([n],[p]))$ on such $s\leq t$ that there exist some diagram of the form \eqref{eq:c-mor}, our next goal is to determine the conditions on $s\leq t$ for this fiber to be non-empty. The above considerations imply that if such a diagram exists, then there is one with $a\leq b$ substituted with $a_0\leq b_0$ in the notation above, so $(s\leq t)\in D$ if and only if $f\leq a_0\circ u\circ s$, $g\geq b_0\circ v\circ t$ and $a_0\leq b_0$. We claim that this is equivalent to the condition
		\begin{equation}\label{eq:cond_main}
			g^Lfs^Ru^Rv\leq t^L\leq s^L.
		\end{equation}
		Indeed, by definition $a_0\leq b_0$ is equivalent to the condition
		\begin{equation}\label{eq:cond_aux}
			f s^R u^R(i)\leq g t^L v^L(i)
		\end{equation}
		for $0<i<q$. Additionally, to ensure $f\leq a_0\circ u\circ s$ we must have $f(x) = 0$ for $x\leq s^R u^R(0)$ and to ensure $g\geq b_0\circ v\circ t$ we similarly need $g(y) = m$ for $y\geq g t^L v^L(q)$, however those two conditions are equivalent to requiring that \eqref{eq:cond_aux} holds for all $i$. Finally, note that \eqref{eq:cond_aux} is equivalent to the first inequality in \eqref{eq:cond_main} by properties of adjoints and $s\leq t$ is equivalent to the second.\par
		We first focus on the condition for the category of $(s,t)$ satisfying \eqref{eq:cond_main} to be non-empty, note that this is equivalent to
		\begin{equation}\label{eq:cond_nonempt}
			g^Lfs^Ru^Rv\leq s^L
		\end{equation}
		since if \eqref{eq:cond_nonempt} holds, then the pair $(s,s)$ satisfies \eqref{eq:cond_main}. Denote $F\bydef u^Rv:[p]\rightarrow [p]$ and $G\bydef g^L f:[n]\rightarrow [n]$, note that $F = u^Rv\geq u^Ru\geq \id$ and $G = g^Lf\leq f^lf\leq \id$. Define a sequence $a_i\in[p]$ inductively by $a_0 \bydef 0$ and $a_{i+1}\bydef \min(p, F(a_i)+1)$; since $F\geq \id$ we have $a_{i+1}>a_i$ unless $a_i = p$, define $N_F$ to be the least index such that $F(a_{N_F}) = p$. Similarly, define $b_0 = 0\in [n]$ and $b_{i+1}\bydef \min(n, G^R(b_i) + 1)$ (note that $G\leq \id$ preserves the minimal element, hence admits a right adjoint by \Cref{lem:n_adj}), again define $N_G$ to be the least index such that $G^R(b_{N_G}) = n$. We claim that:
		\begin{enumerate}
			\item[(*)]\label{it:*} Category $D$ is non-empty if and only if $N_G\leq N_F$.
		\end{enumerate}
		We will also need the following observation: assume we have 
		\begin{equation}\label{eq:xy_ineq}
			s^R(x)\leq y
		\end{equation}
		for some $x\in[p]$ and $y\in [n]$, if $s^R(x)=p$ then \eqref{eq:xy_ineq} is equivalent to $y = n$, we claim that for $s^R(x) < p$ \eqref{eq:xy_ineq} is equivalent to
		\begin{equation}\label{eq:xy2}
			x < s(y+1) \Leftrightarrow x+1\leq s(y+1).
		\end{equation}
		Indeed, using \eqref{eq:right} we immediately see that $s^R(x) < p$ is the least $z\in [n]$ such that $s(z+1)>x$, from which the claim follows.\par
		We claim that we have
		\begin{equation}\label{eq:ab_ineq}
			a_i \leq s(b_i)
		\end{equation}
		if $s^R F(a_{i-1}) < n$: we will prove the claim by induction, it holds for $i=0$ since $s$ is active. Assume we have proved \eqref{eq:ab_ineq} for some $i$, then we have a series of inequalities
		\begin{equation}\label{eq:series_ineq}
			s^R F(a_i)\leq s^R(F(s(b_i)))\leq G^R s^Ls(b_i)\leq G^R b_i,
		\end{equation}
		where the first inequality follows from \eqref{eq:ab_ineq}, the second from \eqref{eq:cond_nonempt} (and the definition of $G$ and $F$) and the last one is the counit for $s$. Since we have assumed that $s^R F(a_{i}) < n$, this is equivalent to 
		\[a_{i+1} = F(a_i) + 1 \leq s(G^R( b_i) + 1) = s(b_{i+1})\]
		by \eqref{eq:xy2}. Assume now that $s^R F(a_i) = n$, note that this is equivalent to $F(a_i) = p$, then we still have \eqref{eq:series_ineq}, but now they imply $G^R(b_i) = n$. It follows that $N_G\leq N_F$, so the "only if" part of \eqref{it:*} holds.\par
		Conversely, assume $N_G\leq N_F$, then define $\epsilon_F\in \{0,1\}$ to be 0 if $a_{N_F} = p$ and 1 otherwise and denote $m_F\bydef N_F + \epsilon_F$, define $j_F: [m_F]\xactive{} [p]$ to be the morphism such that $j_F(k)\bydef a_k$, note in particular that it is injective and active, and similarly define $j_G:[N_G + \epsilon_G]\xactive{}[n]$, also denote $i_0:[m_G]\xactive{}[m_F]$ the morphism such that $i_0(k)\bydef k$ for $k< m_G$ and $i_0(m_G)\bydef m_F$, so in particular it is injective and active; we claim that $s_0\bydef j_F\circ i_0\circ j^R_G$ belongs to $D$ (note that $j_G^R(0) = j_G^R (j_G(0)) = 0$ since $j_G$ is injective, so in particular $s_0$ is active). We need to check that
		\begin{equation}\label{eq:s0cond}
			G\circ s_0^R\circ F(x)\leq s_0^L(x)
		\end{equation}
		for all $x\in[p]$. Denote $k\bydef j_F^R(F(x))$, then $i_0^R\circ j_F^R(x) \leq k$ and hence 
		\[j_G^{RR}\circ i_0^R\circ j_F^R(x)\leq j_G^{RR}(k) = b_{k+1} - 1 = G^R(b_k),\]
		so
		\begin{equation}\label{eq:G1}
			G\circ s_0^R\circ F(x) \leq G(G^R(b_k))\leq b_k.
		\end{equation}
		We claim that $x > a_{k-1}$: indeed, if $x\leq a_{k-1}$, then 
		\[F(x)\leq F(a_{k-1}) < F(a_{k-1}) + 1 =a_k,\]
		which contradicts $j_F^R(F(x)) = k$. It follows that $j_F^L(x)\geq k$, hence 
		\[s_0^L(x) = j_G\circ i_0^L\circ j_F^L(x)\geq j_G\circ i_0^L(k)\geq  j_G(k) = b_k,\]
		combined with \eqref{eq:G1} this implies that \eqref{eq:s0cond} holds, proving the "if" direction of \eqref{it:*}.\par
		We conclude the proof by showing that $D$ is contractible whenever it is non-empty. Assume we have $(s\leq t)\in D$, then in particular we have $s(0) = t(0) = 0 = s_0(0)$ in the notation above, denote by $D_i$ for $i\in [n]$ the full subcategory of $D$ on $s$ and $t$ such that $s(k) = t(k) = s_0(k)$ for $k\leq i$, we will prove by induction that $|D_i|\cong |D|$ - we have seen that this holds for $i=0$, observe that $D_n\cong \{s_0\}$, so this will prove that $D$ is contractible. Assume first that $b_k < i < b_{k+1}$ for some $k$, denote by $D_i^S$ the subcategory of $D_{i-1}$ on $(s,t)$ such that additionally $s(i) = s_0(i)$, we claim that $D_i^s\hookrightarrow D_{i-1}$ admits a right adjoint. Indeed, define $s'\leq t$ by setting $s'(x) = s(x)$ for $x\neq i$ and $s'(i) = a_k$, we need to show that it still satisfies \eqref{eq:cond_main}. we have $t^L\leq s^{',L}$ since $s'\leq t$ and $s^{',R}(x) = s^R(x)$ unless $x\in [a_k, s(i)]$, in which case $s^{',R}(x) = i \geq s^R(x)$, where $s^R(x)$ is either $i$ or $i-1$. Assume now $F(y)\in [a_k,s(i)]$, then $y> a_{k-1}$, so in particular $t^L(y)\geq i > b_k$, while $G s^{',R} F(y) = G(i)\leq G(b_{k+1} - 1) = G(G^R(b_k))\leq b_k$, where the first inequality follows from $i< b_{k+1}$ and the second is the counit. This implies $(s'\leq t)\in D$, by construction it further belongs to $D^s_i$, we have a morphism $s'\leq s\leq t = t$ and it is clearly a final morphism in $D^s_{i,/(s,t)}$, so $(s,t)\mapsto (s',t)$ is indeed a right adjoint in this case. Assume now that $i = b_{k+1}$, by \eqref{eq:ab_ineq} we have $s(i)\geq a_{k+1}$, once again define $s'(i) = a_{k+1}$ and $s'(x) = s(x)$ for $x\neq i$. We have $s'^R(x) = s^R(x)$ unless $x\in [a_{k+1}, s(i)]$, in which case $s^{',R}(x) = b_{k+1} \geq s^R(x)$, where $s^R(x)$ is either $i$ or $i-1$. Assume now $F(y)\in [a_{k+1},s(i)]$, then $y> a_{k}$, so in particular $t^L(y)\geq b_{k+1}$, while $G s^{',R} F(y) = G(b_{k+1})\leq b_{k+1}$ since $G\leq \id$. By the same reasoning, we see that $(s,t)\mapsto (s',t)$ defines the required right adjoint.We will now prove that $D_i^s\hookrightarrow D_i$ admits a left adjoint: given $(s\leq t)$ define $t'(i) = s_0(i)$ and $t'(x) = t(x)$ for $x\neq i$, then we have $t'(i)\leq t(i)$, so in particular $t^{',L}\geq t^L\geq G\circ s^R\circ F$, but we also have $t'\geq s$ since $s(i) = s_0(i)$ because $s\in D_i^s$. It follows that $(s\leq t')\in D_i$ and we have a morphism $s=s\leq t'\leq t$, it is easy to see this defines a left adjoint. Finally, since left and right adjoints induce isomorphisms on geometric realizations, we have $|D|\cong |D^s_i|\cong |D_i|$ concluding the proof.
	\end{proof}
	\begin{notation}\label{not:mon}
		Denote by $\rB \Delta^\act$ the monoidal category $\Delta^\act$ with monoidal structure given by $[n]\otimes [m]\bydef [n+m]$ viewed as a bicategory with one object and by $\rB \Delta^\act_\lax$ the monoidal \textit{bicategory} $\Delta^\act$ with the same monoidal structure and bicategory structure induced from $\cat$ viewed as a tricategory with one object.
	\end{notation}
	\begin{prop}\label{prop:twar_mon}
		The category $\twar(\rB \Delta^\act_\lax)$ is a singleton.
	\end{prop}
	\begin{proof}
		By using \eqref{ex:lowD} in dimension 3 we see that $\twar(\rB \Delta^\act_\lax)$ has objects given by parallel morphisms $f,g:[n]\rightrightarrows [m]$ with $f\leq g$ and that $\mor_{\twar(\rB \Delta^\act_\lax)}((u,v),(f,g))$ is the geometric realization of a double category $B(\bullet,\bullet):\Delta^\op\times \Delta^\op\rightarrow \cD$ for which $B(\bullet, 0)$ is the category with objects given by diagrams 
		\begin{equation}\label{eq:b_mor0}
			\begin{tikzcd}[sep=huge]
				{[r^0_0 + p + r^1_0]} & {[r^2_0 + q + r^3_0]} \\
				{[n]} & {[m]}
				\arrow[""{name=0, anchor=center, inner sep=0}, "{c_{1,0}^0 + v + c_{1,0}^1}", curve={height=-6pt}, from=1-1, to=1-2]
				\arrow[""{name=1, anchor=center, inner sep=0}, "{c_{0,0}^0 + u + c_{0,0}^1}"', curve={height=6pt}, from=1-1, to=1-2]
				\arrow[""{name=2, anchor=center, inner sep=0}, "a"', curve={height=6pt}, from=1-2, to=2-2]
				\arrow[""{name=3, anchor=center, inner sep=0}, "b", curve={height=-6pt}, from=1-2, to=2-2]
				\arrow[""{name=4, anchor=center, inner sep=0}, "s", curve={height=-6pt}, from=2-1, to=1-1]
				\arrow[""{name=5, anchor=center, inner sep=0}, "t"', curve={height=6pt}, from=2-1, to=1-1]
				\arrow[""{name=6, anchor=center, inner sep=0}, "f"', curve={height=6pt}, from=2-1, to=2-2]
				\arrow[""{name=7, anchor=center, inner sep=0}, "g", curve={height=-6pt}, from=2-1, to=2-2]
				\arrow["\leq"{description}, draw=none, from=0, to=1]
				\arrow["\leq"{description}, draw=none, from=2, to=3]
				\arrow["\leq"{description}, draw=none, from=5, to=4]
				\arrow["\leq"{description}, draw=none, from=6, to=7]
			\end{tikzcd}
		\end{equation}
		such that 
		\[f\leq a\circ (c_{0,0}^0 + u + c_{0,0}^1) \circ s \leq b\circ (c_{1,0}^0 + v + c_{1,0}^1) \circ t \leq g.\]
		Morphisms in $B(\bullet, 0)$ are given by pairs of strings $(s'\leq s\leq t\leq t',a'\leq a\leq b\leq b')$. More generally, the category $B(\bullet, l)$ for $l>0$ has objects given by diagrams \eqref{eq:b_mor0} together with diagrams
		\begin{equation}\label{eq:b_morl}
			\begin{tikzcd}[sep=huge]
				{[r^i_l]} & {[r^{i+2}_l]} \\
				{[r_1^i]} & {[r_1^{i+2}]} \\
				{[r^i_0]} & {[r_0^{i+2}]}
				\arrow[""{name=0, anchor=center, inner sep=0}, "{c^i_{1,l}}", curve={height=-6pt}, from=1-1, to=1-2]
				\arrow[""{name=1, anchor=center, inner sep=0}, "{c^i_{0,l}}"', curve={height=6pt}, from=1-1, to=1-2]
				\arrow["{...}"{description}, curve={height=6pt}, from=1-2, to=2-2]
				\arrow["{...}"{description}, curve={height=-6pt}, from=1-2, to=2-2]
				\arrow["{...}"{description}, curve={height=-6pt}, from=2-1, to=1-1]
				\arrow["{...}"{description}, curve={height=6pt}, from=2-1, to=1-1]
				\arrow[""{name=2, anchor=center, inner sep=0}, "{c^i_{0,1}}"', curve={height=6pt}, from=2-1, to=2-2]
				\arrow[""{name=3, anchor=center, inner sep=0}, "{c^i_{1,1}}", curve={height=-6pt}, from=2-1, to=2-2]
				\arrow[""{name=4, anchor=center, inner sep=0}, "{k^i_{1,1}}", curve={height=-6pt}, from=2-2, to=3-2]
				\arrow[""{name=5, anchor=center, inner sep=0}, "{k^i_{1,0}}"', curve={height=6pt}, from=2-2, to=3-2]
				\arrow[""{name=6, anchor=center, inner sep=0}, "{j^i_{1,0}}", curve={height=-6pt}, from=3-1, to=2-1]
				\arrow[""{name=7, anchor=center, inner sep=0}, "{j^i_{1,1}}"', curve={height=6pt}, from=3-1, to=2-1]
				\arrow[""{name=8, anchor=center, inner sep=0}, "{c^i_{0,0}}"', curve={height=6pt}, from=3-1, to=3-2]
				\arrow[""{name=9, anchor=center, inner sep=0}, "{c^i_{1,0}}", curve={height=-6pt}, from=3-1, to=3-2]
				\arrow["\leq"{description}, draw=none, from=1, to=0]
				\arrow["\leq"{description}, draw=none, from=2, to=3]
				\arrow["\leq"{description}, draw=none, from=5, to=4]
				\arrow["\leq"{description}, draw=none, from=6, to=7]
				\arrow["\leq"{description}, draw=none, from=8, to=9]
			\end{tikzcd}
		\end{equation}
		such that $c^i_{0,x}\leq k^i_{0,x}\circ c^i_{0,x+1}\circ j^i_{0,x} \leq k^i_{1,x}\circ c^i_{1,x+1}\circ j^i_{1,x}\leq c^{i}_{1,x}$. Morphisms in $B(\bullet, l)$ are given by collections of strings $(s'\leq s\leq t\leq t',a'\leq a\leq b\leq b')$ and $(j^{i,'}_{0,x}\leq j^i_{0,x}\leq j^i_{1,x}\leq j^{i,'}_{1,x},k^{i,'}_{0,x}\leq k^i_{0,x}\leq k^i_{1,x}\leq k^{i,'}_{1,x})$. \par
		Denote by $p_1:\Delta^\op\times \Delta^\op\rightarrow \Delta^\op$ the projection to the second factor, in order to describe $B'\bydef p_{1,!}B:\Delta^\op\rightarrow \cS$ we will first need to introduce some notation: given $f,g:[n]\rightrightarrows[m]$ with $f\leq g$, denote $F_{f,g}\bydef f^R\circ g:[n]\rightarrow [n]$ and set $a_0^{f,g}\bydef 0$ and $a^{f,g}_{i}\bydef \min(n, F_{f,g}(a^{f,g}_{i-1}) + 1)$, define $N_{F_{f,g}}$ to be the minimal $i$ such that $F(a_i^{f,g}) = n$. In that case using \Cref{lem:cont} and specifically condition \eqref{it:*} we see that $B'$ is in fact a poset whose objects are pairs $(c_0, d_0:[r_0]\rightrightarrows[w_0], c_1,d_1:[r_1]\rightrightarrows[w_1])$ such that $c_i\leq d_i$ and $N_{f,g}\leq N_{c_0 + u + d_0, c_1 + v + d_1}$ and such that $(c_0, d_0:[r_0]\rightrightarrows[w_0], c_1,d_1:[r_1]\rightrightarrows[w_1])\leq (c'_0, d'_0:[r'_0]\rightrightarrows[w'_0], c'_1,d'_1:[r'_1]\rightrightarrows[w'_1])$ if and only if $N_{c_0,d_0}\leq N_{c'_0, d'_0}$ and $N_{c_1,d_1}\leq N_{c'_1, d'_1}$.\par
		To complete the proof of the claim it suffices to show that $B'$ is contractible: indeed, this would imply that the space of morphisms between any two objects in $\twar(\rB \Delta^\act_\lax)$ is contractible, which proves that it is isomorphic to $*$. To show this we will in fact prove that $B'$ is filtered. Since $B'$ is a poset, it suffices to show that for any two objects $x_0\bydef (c_0, d_0:[r_0]\rightrightarrows[w_0], c_1,d_1:[r_1]\rightrightarrows[w_1])$ and $x_1\bydef (c'_0, d'_0:[r'_0]\rightrightarrows[w'_0], c'_1,d'_1:[r'_1]\rightrightarrows[w'_1])$ there is some $y$ such that $y\geq x_0$ and $y\geq x_1$, however it is easy to see that we can take $y$ to be $(c_0 + c'_0, d_0 + d'_0:[r_0 + r'_0]\rightrightarrows[w_0 + w'_0], c_1 + c'_1, d_1 + d'_1:[r_1 + r'_1]\rightrightarrows[w_1 + w'_1])$.
	\end{proof}
	\begin{cor}\label{cor:init_mon}
		For any functor $F:\cC\rightarrow \rB \Delta^\act_\lax$ the induced functor $\twar(F):\twar(\cC)\rightarrow \twar(\rB \Delta^\act_\lax)$ is coinitial if and only if $\twar(\cC)$ is contractible.\qed
	\end{cor}
	\begin{lemma}\label{lem:twar_cont}
		For any $\cD\in\cat_n$ we have
		\[|\twar(\cD)|\cong \twar(|\cD|)\cong |\cD|.\]
	\end{lemma}
	\begin{proof}
		Note that $|-|$ preserves colimits as a left adjoint, hence it follows from \Cref{cor:twar_col} that
		\begin{align*}
			|\twar(\cD)|&\cong |\underset{(\theta\xrightarrow{f}\cD)\in\Theta_{n,/\cD}}{\colim} \twar(\theta)|\\
			&\cong\underset{(\theta\xrightarrow{f}\cD)\in\Theta_{n,/\cD}}{\colim} |\twar(\theta)|\\
			&\cong \underset{(\theta\xrightarrow{f}\cD)\in\Theta_{n,/\cD}}{\colim} |\Theta^\inrt_{n,/\theta}|\\
			&\cong \underset{(\theta\xrightarrow{f}\cD)\in\Theta_{n,/\cD}}{\colim} *\cong |\cD|,
		\end{align*}
		where the non-trivial implication uses \Cref{thm:main_theta}. Applying \Cref{cor:twar_col} to $|\cD|$ we also get
		\[\twar(|\cD|)\cong \underset{(\theta\xrightarrow{f}|\cD|)\in\Theta_{n,/|\cD|}}{\colim}\twar(\theta).\]
		Note that $\Theta_{n,/|\cD|}$ contains a cofinal subcategory on $c_0\rightarrow|\cD|$ which is isomorphic to $|\cD|$, so we can rewrite the expression above as
		\[|\cD|\otimes \twar(c_0)\cong |\cD|.\]
	\end{proof}
	\begin{theorem}\label{thm:lax_mon}
		For $\cE\in\cat_3$ the space of morphisms $F:\rB\Delta^\act_\lax\rightarrow\cE$ is isomorphic to the subspace of $\rB\Delta^\act\rightarrow\cE$ for which the image of the 2-morphism $\delta^2_1:[1]\xactive{}[2]$ is left adjoint to $\sigma_0^1:[2]\xactive{}[1]$.
	\end{theorem}
	\begin{proof}
		Denote by $\cI:\rB\Delta^\act\rightarrow \rB\Delta^\act_\lax$ the natural inclusion, it now follows from \Cref{cor:cot_coinit}, \Cref{cor:init_mon}, \Cref{lem:twar_cont} and the contractibility of $\rB\Delta^\act$ that $L_\cI\cong 0$. It now follows from \Cref{prop:main_prop} that to prove the claim it suffices to prove it for $\cE\in\cat_{(4,3)}$, in which case the result is classical -- see \cite{kock1995monads}.
	\end{proof}
	\newpage
	\appendix
	\section{Dold-Kan correspondence for $\Theta_n$}\label{sect:DK}
	Dold-Kan correspondence is a classical result of \cite{dold1958homology} and \cite{kan1958functors} providing an isomorphism between the category of simplicial abelian groups and chain complexes. Numerous generalizations of this result have been described since, in this paper we will be interested in describing the generalization of this result to stable $\infty$-categories.\par
	The principal idea behind the Dold-Kan correspondence is the observation that the simplex category $\Delta$ admits many split idempotents, and in an abelian category $A$ every such idempotent on an object $X$ gives rise to a direct sum decomposition $X=X_0\oplus X_1$. Using this observation then allows us to locate a smaller subcategory $C\hookrightarrow\mor_\cat(\Delta^\op,A)$ isomorphic to the category of chain complexes such that every object in the original category decomposes as a direct sum of objects in $C$. Exploiting this insight, the paper \cite{lack2015combinatorial} introduced a combinatorial structure that gives rise to equivalences of similar type and provided a number of examples. Lastly, this notion has been generalized to $\infty$-categories in \cite{walde2022homotopy} under the name of \textit{DK-triples}.\par 
	The paper \cite{walde2022homotopy} provides a number of examples of DK-triples, however the one we need in the present work is missing, namely the structure of a DK-triple on $\Theta_n$ and more generally on $\Theta_{n,/\theta}$ for $\theta\in\Theta_n$. The goal of the present section is to provide an explicit description of this structure and the resulting DK-equivalence which will take the form
	\[\psh_\spc(\Theta_{n,/\theta})\cong \psh_\spc(\widetilde{\Theta}_{n,/\theta})\]
	whose terms will be defined below. Sadly, as the reader will soon find out, working with $\psh_\spc(\widetilde{\Theta}_{n,/\theta})$ is hardly any easier than with the original category $\psh_\spc(\Theta_{n,/\theta})$. Nevertheless it does help in our particular case since as will be demonstrated in \Cref{sect:stab} the subcategory $\stab(\cat_{n,/\theta})\hookrightarrow\psh_\spc(\Theta_{n,/\theta})$ admits a much simpler description after the application of the Dold-Kan equivalence.
	\begin{notation}
		Given $\theta\in \Theta_n$ and $i\in \obj(\theta)$ we will denote in this section $\theta_i\bydef \mor_\theta(i,i+1)\in \Theta_{n-1}$. Additionally, for a morphism $f:\theta\rightarrow \theta'$ and $f(i)\leq k< f(i+1)$ we will denote $f_i^k:\theta_i\rightarrow \theta'_k$ the induced functor of morphism categories.
	\end{notation}

	\begin{defn}\label{def:inj}
		Call a morphism $f:\theta\rightarrow\theta'$ in $\Theta_n$ \textit{injective} if it is injective on the set of $n$-morphisms and \textit{surjective} if it is surjective on the set of $n$-morphisms.
	\end{defn}
	\begin{prop}\label{prop:fact_inj}
		\begin{enumerate}
			\item If $f:\theta\rightarrow\theta'$ is injective, then it is also injective on the set of $j$-morphisms for all $j<n$;
			\item if $f:\theta\rightarrow\theta'$ is surjective, then it is also surjective on the set of $j$-morphisms for all $j<n$;
			\item injective and surjective morphisms form a factorization system on $\Theta_n$.
		\end{enumerate}
	\end{prop}
	\begin{proof}
		Assume that a morphism $f:\theta\rightarrow\theta'$ is injective on the set of $n$-morphisms, but not on the set of $j$-morphisms for some $j<n$, then there is a pair of $j$-morphisms $(a,a')$ such that $f(a)=f(a')$, but then we also have $f(\id^n_a)=f(\id^n_{a'})$, where $\id^n_a$ denotes the identity $n$-morphism on $a$, which means $\id^n_a=\id^n_{a'}$ by injectivity of $f$, which in turn implies $a=a'$, contradicting our assumption. This proves the first claim, the second one follows since for any $j$-morphism $b$ the morphism $\id^n_b$ must lie in the image of $f$, which implies that $b$ also lies in the image of $f$.\par
		Finally, we will prove the last claim by induction on $n$. More specifically, we will show that any morphism $f:\theta\rightarrow\prod_{0\leq j\leq m} \theta'_j$ factors uniquely as
		\[\theta\xrightarrow{s}\theta_0\xrightarrow{q}\prod_{0\leq j\leq m}\theta'_j,\]
		where $s$ is surjective and $q$ is injective (by which we mean that it is injective on $n$-morphisms). We start with $n=1$, in this case a morphism $[n]\xrightarrow{f}\prod_j [m_j]$ is injective (resp. surjective) if and only if it is injective (resp. surjective) on the set of objects. We can factor the induced morphism of sets $|n+1|\rightarrow \prod_j |m_j+1|$ as
		\[|n+1|\xrightarrow{s_)} |l+1|\xrightarrow{i_0} \prod_j|m_j+1|,\]
		where $s_0$ is surjective and $i_0$ is injective. Note that the linear order on $|n+1|$ induces one on $|l+1|$ and hence we can factor $f$ as
		\[[n]\xrightarrow{s}[l]\xrightarrow{i}\prod_j [m_j],\]
		where $s$ is surjective and $i$ injective, the uniqueness of such factorization follows from the uniqueness of the surjective/injective factorization of sets. In the general case, We can again factor the induced morphism on objects as
		\begin{equation}\label{eq:fact_sur}
			\obj(\theta)\xrightarrow{s_0}S\xrightarrow{i_0}\prod_{0\leq j\leq m} \obj(\theta'_j),
		\end{equation}
		\[\]
		where $s_0$ is surjective and $i_0$ is injective; note that the linear order on $\obj(\theta)$ induces one on $S$. We define an object $\theta_s\in \Theta_n$ as follows: its set of objects is the set $S$ of \eqref{eq:fact_sur}, for $j\in S$ we define $\mor_{\theta_s}(j,j+1)$ to be $\mor_\theta(i_m, i_m+1)$, where $i_m$ is the maximal element in $s^{-1}(j)$. With this definition we have a factorization
		\[\theta\xrightarrow{s'}\theta_s\xrightarrow{j'}\prod_j \theta'_j,\]
		where $s'$ is surjective on objects and $j'$ is injective on objects. Now, fix $i\in \obj(\theta_s)$, then the morphism $j'$ induces a morphism
		\[j'_i:\mor_{\theta_s}(i,i+1)\rightarrow \mor_{\prod_j\theta'_j}(j'(i), j'(i+1))\cong \prod_{k}\theta''_k\]
		for some $\theta_k''\in \Theta_{n-1}$. By induction, $j'_i$ admits an injective/surjective factorization of the form 
		\[\mor_{\theta_s}(i,i+1)\xrightarrow{s_i}\theta_i^0\xrightarrow{q_i}\prod_{k}\theta''_k,\]
		we then define $\theta'_s$ to have the same set of objects as $\theta_s$ and set $\mor_{\theta'_s}(i,i+1)\bydef \theta_i^0$, then the morphisms $s_i$ and $q_i$ induce a factorization of $j'$ as
		\[\theta_s\xrightarrow{s}\theta'_s\xrightarrow{j}\prod_j \theta'_j,\]
		where $s$ is surjective and identity-on-objects and $j$ is injective. Precomposing it with $\theta\xrightarrow{s}\theta_s$ gives us the required factorization, the fact that it is unique again follows from the uniqueness of the injective/surjective factorization on sets.
	\end{proof}
	\begin{notation}\label{not:inj}
		Denote by $\Theta^\inj_{k,/\theta}$ the subcategory of $\Theta_{k,/\theta}$ on injective morphisms and by $p_\inj:\Theta_{k,/\theta}\rightarrow\Theta^\inj_{k,/\theta}$ the functor obtained by sending $f:\theta'\rightarrow \theta$ to the injective part of its injective/surjective factorization of \Cref{prop:fact_inj}. 
	\end{notation}
	\begin{construction}\label{constr:dk_fact}
		Denote by $E_n([0])$ the subcategory of $\Theta_n$ on surjective morphisms (in the sense of \Cref{def:inj}). We will define the subcategory $E^\lor_n([0])$ of $\Theta_n$ by induction on $n$ as follows: for $n=1$ we define $E_1^\lor([0])$ to be the category of injective morphisms preserving the minimal element, for general $n$ we define a morphism $j:\theta\rightarrow \theta'$ to be in $E_n^\lor([0])$ if $j(0)=0$ and the following conditions are satisfied:
		\begin{enumerate}[label=$(\arabic*)_j$]
			\item\label{it:j1} $j$ is injective;
			\item\label{it:j2} $j$ preserves the minimal element;
			\item\label{it:j3} for $k<j(i+1)-1$ the morphism $j_i^k$ factors as
			\[\theta_i\xactive{}[0]\xhookrightarrow{\{0\}}\theta'_k,\]
			where the second morphism is the inclusion of the minimal element, and the morphism $j_i^{j(i+1)-1}$ lies in $E_{n-1}^\lor([0])$.
		\end{enumerate}
		Finally, we define the set of morphisms $M_n$ by induction, starting with $M_1$ which contains all identity morphisms as well as $[n-1]\xinert{\lambda^n_0}[n]$ - the inert morphism preserving the maximal element. In general, define a morphism $u:\theta\rightarrow\theta''$ to be in $M_n$ if the following conditions are satisfied:
		\begin{enumerate}[label=$(\arabic*)_M$]
			\item\label{it:u1} $u$ is injective;
			\item\label{it:u2} either $u(0)=0$ or $u(0)=1$, in the latter case $\mor_{\theta''}(0,1)=[0]$, in the former $\dim(\theta)>1$;
			\item\label{it:u3} the induced morphisms $u_i^k:\theta_i\rightarrow \theta''_k$ for $u(i)\leq k<u(i+1)-1$ all have the form $v_i^k\circ s_i^k$, where $s_i^k$ are surjective and $v_i^k$ lie in $M_{n-1}$, additionally either none of $\theta''_k$ equal $[0]$ or $u(i+1)=u(i)+1$ and the morphism $u_i^i$ is $[0]\eq [0]$.
		\end{enumerate}
		We will now define the subcategories $E_n(\theta)$ and $E^\lor_n(\theta)$ of $\Theta_{n,/\theta}$ as well as a set of morphisms $M_n(\theta)$ in $\Theta_{n,/\theta}$ for all $\theta\in \Theta_n$, we will do so by induction on $\dim(\theta)$ with the base case $\theta=[0]$ having been treated above. Now, assume that we have defined $F(\theta')$ for $F\in \{E_n, E_n^\lor, M_n\}$ and $\dim(\theta')<m$, define $F(\prod_{k\in K}\theta'_k)$ with $\dim(\theta'_k)<m$ for a finite set $K$ to be the subcategory of $\Theta_{n,/\prod_{K} \theta'_k}$ containing all objects and only those morphisms whose image in $\Theta_{n,/\theta'_k}$ for all $k\in K$ lies in $F(\theta'_k)$. We will now define $F(\theta)$ assuming $\dim(\theta)=m$. We define $E_n(\theta)$ to be the subcategory of surjective morphisms, we define $E_n^\lor(\theta)$ to consists of injective morphisms $j:\theta_f\rightarrow\theta_g$ such that:
		\begin{enumerate}[label=$(\arabic*)_j^\theta$]
			\item\label{it:1} viewed as a morphism in $\Theta_n$, $j$ lies in $E_n^\lor([0])$;
			\item\label{it:2} for any $i\in\obj(\theta_f)$ we have $g(j(i))=g(j(i+1)-1)$;
			\item\label{it:3} for any $i\in\obj(\theta_f)$ the morphism $j_i^{j(i+1)-1}$ belongs to $E^\lor_{n-1}(\prod_{s=f(i)}^{f(i+1)-1} \theta_s)$.
		\end{enumerate}
		Finally, define $M_n(\theta)$ to be the set of morphisms $u:\theta_f\rightarrow\theta_g$ such that:
		\begin{enumerate}[label=$(\arabic*)_M^\theta$]
			\item\label{it:u'1} $u$ is injective;
			\item\label{it:u'2} either $u(0)=0$ or $u(0)=1$, in the latter case $\mor_{\theta_g}(0,1)=[0]$, in the former $\dim(\mor_{\theta_g}(0,1))>1$;
			\item\label{it:u'3} the induced morphisms $u_i^k:\theta_{f,i}\rightarrow \theta_{k,g}$ for $u(i)\leq k<u(i+1)-1$ all have the form $v_i^k\circ s_i^k$, where $s_i^k$ are surjective and $v_i^k$ lie in $M_{n-1}(\Theta_{n,/\prod_{s=g(k)}^{g(k+1)-1}\theta_s})$;
			\item\label{it:u'4} if the target of some $u_i^k$ equals $[0]$, then either $g(k)<g(k+1)$ or $u(i+1)=u(i)+1$ and and the morphism $u_i^i$ is $[0]\eq [0]$.
		\end{enumerate}
		We will also denote by $\mathrm{Mono}_n(\theta)$ the category of injective morphisms, by $\mathrm{Reg}_n(\theta)$ the set of morphisms of the form $m\circ e$, where $m\in M_n(\theta)$ and $e\in E_n(\theta)$ and by $\mathrm{Sing}_n(\theta)$ the category of morphisms of the form $e'\circ g$ for some $e'\in E^\lor_n(\theta)$ that is not an isomorphism.
	\end{construction}
	\begin{lemma}\label{lem:sur_sect}
		There are inverse equivalences 
		\[G_n^\theta:E_n(\theta)\overset{\sim}{\rightleftarrows} E_n^\lor(\theta)^\op:H_n^\theta.\]
	\end{lemma}
	\begin{proof}
		We will prove the claim by double induction on $n$ and the dimension of $\theta$. For the base case of $n=1$ and $\theta=[0]$ note that any surjective $s:[l]\xactive{} [q]$ admits a left adjoint $j(k)\bydef \min_{s(t)=k} t$ that preserves the minimal element and is injective and conversely any injective minimal-element-preserving morphism $j$ admits a left adjoint $s(k)\bydef \max_{j(t)<k} t$, which is surjective, this correspondence is functorial by functoriality of adjoints and defines the required equivalence.\par
		To prove the equivalence for general $n$ we first need to show that $E_n^\lor([0])$ is a category: given a composable pair $\theta\xrightarrow{j}\theta'\xrightarrow{w}\theta''$ we need to show that $(w\circ j)_i^k$ satisfies conditions \ref{it:j1}, \ref{it:j2} and \ref{it:j3}. The first two are easy since injective and minimum-preserving morphisms are closed under composition, to prove \ref{it:j3} note that for $k<w\circ j(i+1)-1$ it is either given by the composition 
		\[\theta_i\rightarrow[0]\xhookrightarrow{\{0\}}\theta'_s\xrightarrow{w_s^k}\theta''_k,\]
		where $w_s^k$ preserves the minimal element or by 
		\[\theta_i\xrightarrow{j_i^{j(i+1)-1}}\theta'_{j(i+1)-1}\rightarrow[0]\xhookrightarrow{\{0\}}\theta''_k,\]
		in either case it factors through $[0]\xhookrightarrow{\{0\}}\theta''_k$, and for $k=w\circ j(i+1)-1$ it is given by the composition $\theta_i\xrightarrow{j_i^{j(i+1)-1}}\theta'_{j(i+1)-1}\xrightarrow{w_{j(i+1)-1}^{w(j(i+1))-1}}\theta''_{w(j(i+1))-1}$ in which both morphisms lie in $E^\lor_{n-1}([0])$, hence the composition itself also belongs to it by inductive assumption. Assume now that we have a surjective morphism $\theta\xactive{s}\theta'$, assume $|\obj(\theta)|=(l+1)$ and $|\obj(\theta')|=m+1$, then we get an induced morphism $\widetilde{s}:[l]\xactive{}[m]$ which admits a section $\widetilde{j}\bydef G_1^{[0]}(\widetilde{s})$. Note that for $\widetilde{j}(i)\leq k<\widetilde{j}(i+1)$ we have $\widetilde{s}(k)=i$, denote by $s_i$ the surjective morphism $\theta_{\widetilde{j}(i+1)-1}\xactive{} \theta'_i$ induced by $s$. To upgrade $\widetilde{j}$ to an element of $E^\lor_n([0])$ we need to define morphisms $j_i^k:\theta'_i\rightarrow\theta_k$ for $\widetilde{j}(i)\leq k<\widetilde{j}(i+1)$, we define it to be the composition $\theta'_i\rightarrow[0]\xhookrightarrow{\{0\}}\theta_k$ for $k<\widetilde{j}(i+1)-1$ and for $k=\widetilde{j}(i+1)-1$ we define it to be the section $G^{[0]}_{n-1}(s_i)$. Conversely, given a morphism $j:\theta\rightarrow\theta''$ which induces $\widetilde{j}$ on objects, denote $\widetilde{s}\bydef H^{[0]}_1(\widetilde{j})$, in order to define a surjective section $s$ of $j$ it suffices to define surjective morphisms $s_i:\theta''_{\widetilde{j}(i+1)-1}\xactive{} \theta_i$, we define them to be the image of $H_n^{[0]}(j_i^{j(i+1)-1})$ (note that those morphisms lie in $E^{\lor}_{n-1}([0])$ by definition). It now immediately follows from the inductive assumptions and definition of $E_n^\lor([0])$ that those functors are inverse to each other.\par
		Assume now that we have constructed the equivalence $G_n^{\theta'}$ for all $\theta'$ with $\dim(\theta')<m$ and moreover assume that the underlying morphism of $G_n^{\theta'}(s)$ coincides with $G_n^{[0]}(s)$ for any $s\in E_n(\theta)$ and similarly for $H^{\theta'}_n$. Assume we have $\theta\in \Theta_n$ with $\dim(\theta)=m$ and a surjective $s:\theta_f\xactive{}\theta_g$ over $\theta$ we define $j\bydef G_n^\theta(s)$ to be $G_n^{[0]}(s)$, note that $f\circ j\cong g\circ s\circ j\cong g$ since $j$ is a section of $s$, hence $j$ is automatically a morphism over $\theta$, it remains to show that it lies in $E^\lor_n(\theta)$. \ref{it:1} follows since $j$ lies in $E_n^\lor([0])$ by construction, for \ref{it:2} note that $f(j(i+1)-1)=g(s(j(i+1)-1))=g(s(j(i)))=g(i)=f(j(i))$ and \ref{it:3} follows from the inductive assumption. Conversely, given $j:\theta_f\rightarrow\theta_h$ with $j\in E_n^\lor(\theta)$ we need to prove that $s\bydef H^{[0]}_n(j)$ lies in $E_n(\theta)$. Since $s$ is surjective by construction, it suffices to show that it is a morphism over $\theta$, which follows immediately from \ref{it:2}, \ref{it:3} and the inductive assumption.
	\end{proof}
	\begin{prop}\label{prop:dk_fact}
		Every morphism $\theta_f\xrightarrow{h}\theta_g$ in $\Theta_{n,/\theta}$ factors uniquely as
		\[\theta_f\xrightarrow{e\in E_n(\theta)}\theta_s\xrightarrow{m\in M_n(\theta)}\theta_t\xrightarrow{e'\in E^\lor_n(\theta)}\theta_g.\]
	\end{prop}
	\begin{proof}
		We have already seen in \Cref{prop:fact_inj} that any morphism uniquely factors as $j\circ s$ with surjective $s$ and injective $j$, so it remains to show that any injective morphism uniquely factors as $e'\circ m$ as above. We will once again prove it by induction on $n$ and $\dim(\theta)$. For the base case $n=1$, $\theta=[0]$ observe that any injective morphism $j:[p]\rightarrow[l]$ that does not lie in $E^\lor_1([0])$ (i.e. does not preserve the minimal element) can be uniquely factored as
		\[[p]\xinert{\lambda_0^p}[p+1]\xrightarrow{j'}[l],\]
		where $j'(0)=0$ and $j'(k)=j(k-1)$ for $k>0$, so the claim holds in this case. For the general $n$ we will first construct a factorization and then show that it is unique: given $j:\theta\rightarrow\theta'$ denote by $j_i^k:\theta_i\rightarrow \theta'_k$ the induced functors on morphism categories. By induction we can factor  each $j_k^i$ as
		\begin{equation}\label{eq:fact_ind}
			\theta_i\xrightarrow{e_i^k}\widetilde{\theta}_i^k\xrightarrow{m_i^k}\widehat{\theta}_i^k\xrightarrow{e_i^{',k}}\theta'_k.
		\end{equation}
		Denote by $\widehat{\theta}$ the object of $\Theta_n$ constructed as follows: if $j$ preserves the minimal object, then the objects of $\widehat{\theta}$ are pairs $(i,k)$ with $i\in \obj(\theta)$ and $j(i)\leq k < j(i+1)$, otherwise we add an additional minimal object $*$, we also set $\widetilde{\theta}_{(i,k)}\bydef \widehat{\theta}_i^k$ in the notation of \eqref{eq:fact_ind} and $\widehat{\theta}_* = [0]$ if applicable, then $j$ factors as
		\begin{equation}\label{eq:fact1}
			\theta\xrightarrow{j'}\widehat{\theta}\xrightarrow{j''}\theta',
		\end{equation}
		where $j'(i)\bydef (i,j(i))$ and $j_i^{',k}$ is the composition $m_i^k\circ e_i^k$ while $j''(i,k)=k$ and $j_{(i,k)}^k$ is $e_i^{',k}$ in the notation of \eqref{eq:fact_ind}, if $j$ does not preserve the minimal element we also set $j''(0) = 0$ and define all the morphisms $j_*^{'',r}$ to be the inclusions $[0]\hookrightarrow\theta'_r$ of the minimal element. The morphism $j''$ lies in $E^\lor_n([0])$, but $j'$ does not belong to $M_n([0])$ since it does not satisfy the second part of \ref{it:u3}, so we will factor it further. Assume that $|\obj(\widehat{\theta})\setminus \{*\}|=m+1$, then we can define a surjective morphism $s':[m]\xactive{}[q]$ as follows: to define such a morphism it suffices to describe which elementary intervals $q<q+1$ are sent to identity morphisms, we define $s'(i,k) = s'(i,k+1)$ if $\widetilde{\theta}_{(i,k)}=[0]$, denote by $l:[q]\rightarrow[m]$ the left adjoint of $s'$. Now, denote by $\widetilde{\theta}\in\Theta_n$ the $n$-category with the set of object given by $\{0,...,q\}$ or $\{*,0,...,q\}$ such that $\widetilde{\theta}_i = \widehat{\theta}_{l(i+1)-1}$ and $\widetilde{\theta}_* = [0]$, note that $j'$ factors as $\theta\xrightarrow{m'}\widetilde{\theta}\xrightarrow{w}\widehat{\theta}$, where $m'$ sends $i$ to $s'(j'(i))$ and $m_i^{',k} = j_i^{', l(k+1)-1}$, while $w$ sends $k$ to $l(k)$ and $w_k^{t}$ are the unique morphisms $\widetilde{\theta}_k\rightarrow[0]$ for $t\neq l(k+1)-1$ and the identity morphism otherwise (and sends $*$ to the minimal element if it is an object of $\widetilde{\theta}$). It is easy to see from this description that the morphism $m'$ now satisfies the second part of \ref{it:u3} and that $w$ lies in $E_n^\lor([0])$, concluding our construction.\par
		We now need to show that such a factorization is unique, so assume we have a different factorization $\theta\xrightarrow{n'}\overline{\theta}\xrightarrow{w'}\theta'$ of $j$. We start by showing that $n'$ and $w'$ agree with $m'$ and $w'$ on objects: first, note that $n'$ preserves the minimal object if and only if $j$ does (since $w'$ preserves it by definition) and if it does not, then $w'(0)=0$ and the components $w_0^{'',k}$ are all inclusions of the minimal element since $w'$ lies in $E-n^\lor([0])$. Now fix some $i\in \theta$, we claim that 
		\[n'(i+1) - n'(i) = \max(1, |k:\;\theta_i\xrightarrow{j^k_i}\theta'_k\text{ does not factor through }[0]\xhookrightarrow{\{0\}}\theta'_k|).\]
		indeed, if all $j_i^k$ factor through $[0]$, then by \ref{it:u3} we must have $n'(i+1) = n'(i) + 1$, so assume this is not the case. Note that $j_i^k = w_t^{',k}\circ n_i^{',t}$, none of the $n_i^{',t}$ factor through $[0]$ and by \ref{it:j3} for a given $t$ there is exactly one $k$ such that $w_t^{'',k}$ does not factor through $[0]$, from which the claim follows. This also uniquely defines the value of $n'$ on objects and the value of $w'$ on objects is now uniquely determined by \ref{it:j3}. It remains to show that the values of $n'$ and $w'$ on morphism categories are uniquely defined, however this follows from \ref{it:u3}, \ref{it:j3} and the inductive hypothesis.\par
		It remains to construct the factorization over general $\theta$, assuming it has been constructed over $\theta'$ with $\dim(\theta')<\dim(\theta)$. First, note that the factorization of \Cref{prop:fact_inj} also exists in $\Theta_{n,/\theta}$, so it suffices to provide a factorization of an injective morphism. The construction in the relative case will be very similar to the construction in the absolute case described above, so we will give slightly less details. First, we can factor  the underlying morphism of $j:\theta_f\rightarrow\theta_g$ as $j''\circ j'$ as in \eqref{eq:fact1}, it is easy to see by construction that $j''$ lies in $E_n^\lor(\theta)$, so it remains to factor $j'$. Assume that $|\obj(\widehat{\theta})\setminus \{*\}|=m+1$, then we define a surjective morphism $s:[m]\rightarrow[q]$ by sending $(i,k)<(i,k+1)$ to identity if $\theta_{g\circ j,(i,k)}=[0]$ and $g\circ j(i,k) = g\circ j(i,k+1)$. After this the construction goes through unchanged, the arguments proving the uniqueness of the factorization also works upon replacing references to \ref{it:u3} with \ref{it:u'4}.
	\end{proof}
	\begin{prop}\label{prop:dk_triple}
		The data of $(E_n(\theta), E^\lor_n(\theta), M_n(\theta))$ of \Cref{constr:dk_fact} defines a DK-triple in the sense of \cite{walde2022homotopy}.
	\end{prop}
	\begin{proof}
		Our construction of the subcategories $M_n(\theta)$, $\mathrm{Mono}_n(\theta)$ and $\mathrm{Reg}_n(\theta)$ is different from the one given in \cite{walde2022homotopy}, however it follows from \Cref{prop:dk_fact} that they coincide. More specifically, for a category $B$ endowed with subcategories $E$ and $E^\lor$ $\mathrm{Mono}$ is defined as a subset of arrows not of the form $f\circ e$ for some arrow $f$ and $e\in E$, $\mathrm{Sing}$ as a set of arrows of the form $e'\circ g$ for $e'\neq \id\in E^\lor$, $\mathrm{Reg}$ as the complement to $\mathrm{Sing}$  and $M\bydef \mathrm{Reg}\bigcap \mathrm{Mono}$, it is easy to see from the existence of the factorization \Cref{prop:dk_fact} that this coincides with our definition.\par
		It follows that we need to prove the conditions outlined in \cite[Definition 3.1.1.]{walde2022homotopy}, we reproduce them here in our notation for convenience of the reader:
		\begin{enumerate}
			\item[(T1)]\label{it:t1} every morphism $f$ in $\Theta_{n,/\theta}$ uniquely decomposes as
			\[\theta_f\xrightarrow{e\in E_n(\theta)}\theta_s\xrightarrow{m\in M_n(\theta)}\theta_t\xrightarrow{e'\in E^\lor_n(\theta)}\theta_g;\]
			\item[(T2)]\label{it:t2} for any $\theta_f\in \Theta_{,/\theta}$ the pairing $E^\lor_n(\theta)_{/\theta_f}\times E_n(\theta)_{\theta_f/}\rightarrow\arr(\Theta_{n,/\theta})$ given by sending $\theta_g\xrightarrow{e'}\theta_f\xrightarrow{e}\theta_h$ can be described by a square matrix of the form 
			\[\left( 
			\begin{array}{cccc}
				=    & ?  & \dots & ?\\
				\neq & \ddots & \ddots  & \vdots \\
				\vdots & \ddots & \ddots  & ? \\
				\neq & \neq & \neq & = 
			\end{array}
			\right)\]
			with isomorphisms on the diagonal and non-isomorphisms below it;
			\item[(T3)]\label{it:t3} the set $(E^\lor_n(\theta)\circ E_n(\theta))$ is closed under composition;
			\item[(T4)]\label{it:t4} the set $M_n(\theta)\circ M_n(\theta)$ belongs to $\mathrm{Mono}_n(\theta)$;
			\item[(T5)]\label{it:t5} $\mathrm{Mono}_n(\theta)\circ\mathrm{Sing}_n(\theta)\subset \mathrm{Sing}_n(\theta)$. 
		\end{enumerate}
		The claims \ref{it:t4} and $\ref{it:t5}$ are immediate since injective morphisms are closed under composition and \ref{it:t1} is just \Cref{prop:dk_fact}, so it remains to prove \ref{it:t2} and \ref{it:t3}. For \ref{it:t3} it suffices to show that given a surjective/injective factorization square
		\[\begin{tikzcd}[sep=huge]
			{\theta_f} & {\theta_g} \\
			{\theta_s} & {\theta_h}
			\arrow["{e_0'}", from=1-1, to=1-2]
			\arrow["{e_1}"', two heads, from=1-1, to=2-1]
			\arrow["{e_0}", two heads, from=1-2, to=2-2]
			\arrow["{e_1'}"', from=2-1, to=2-2]
		\end{tikzcd}\]
		with $e_0\in E_n(\theta)$ and $e'_0\in E^\lor_n(\theta)$ we also have $e_1\in E_n(\theta)$ and $e_1'\in E^\lor_n(\theta)$. That $e_1\in E_n(\theta)$ follows immediately from the definition, so it remains to prove $e_1'\in E^\lor_n(\theta)$, which we will do by induction on $n$ and $\dim(\theta)$. For $n=1$ the claim follows since the composition of a surjective and minimum-preserving morphism preserves minimum, so assume we have proved it for $n-1$. In this case $e_1'$ satisfies \ref{it:j1} by construction and \ref{it:j2} by earlier observation, so it remains to show \ref{it:j3}. Take an object $i\in \theta_f$ such that $e_0\circ e'_0(i<i+1)$ is not an identity, in this case $e_1(i+1) = e_1(i)+1$, denote $\theta''_i\bydef \mor_{\theta_s}(e_1(i), e_1(i)+1)$. Denote by 
		\[\theta_i\xrightarrow{\prod w_k}\prod^{e_0\circ e'_0(i+1)-1}_{k = e_0\circ e'_0(i)}\theta'_k\]
		the morphism induced by $e_0\circ e'_0$ on morphism categories, so by construction and \ref{it:j3} we have that $w_k$ factor through $[0]\xrightarrow{\{0\}}\theta'_k$ for $k<e_0\circ e'_i(i+1)-1$ and $w_{e_0\circ e'_i(i+1)-1}\cong s_i\circ j_i$ with surjective $s_i$ and $j_i\in E^\lor_{n-1}([0])$. We can factor $\prod_k w_k$ as 
		\[\theta_i\xrightarrow{s'}\theta''_i\xrightarrow{\prod w'_k}\prod^{e_0\circ e'_0(i+1)-1}_{k = e_0\circ e'_0(i)}\theta'_k,\]
		where $s'$ is surjective and $w'_i$ are injective, so to prove \ref{it:j3} it remains to show that $w'_k$ factors through $[0]\xrightarrow{\{0\}}\theta'_k$ for $k<e_0\circ e'_0(i+1)-1$ and lies in $E^\lor_{n-1}([0])$ otherwise. For that note that $w_k$ factors through the inclusion of the minimal element for $k<e_0\circ e'_0(i+1)-1$, which implies the first part of the claim, and the second follows by inductive assumption since $w'_{e_0\circ e'_0(i+1)-1}\circ s'$ is the injective/surjective factorization of $w_{e_0\circ e'_i(i+1)-1}\cong s_i\circ j_i$. Finally, in the case of general $\theta$ \ref{it:1} follows by what we just proved, \ref{it:3} by induction and \ref{it:2} follows since for any object $i\in \theta_g$ such that $e_0(i)=e_0(i+1)$ we also have $g(i)=g(i+1)$.\par
		It remains to prove \ref{it:t2}. Given $\theta_f\xrightarrow{e}\theta_g$ we will denote by $e^\lor$ the image of $e$ under the isomorphism of \Cref{lem:sur_sect}. We will prove the claim by induction on $n$, starting with $n=1$; the object $\theta$ over which the construction is performed will play no role, so we will suppress mentioning it. Note that if we have two morphisms $f,g:[m]\rightarrow[l]$ in $E_1$ such that for some $i$ we have $f(i)<g(i)$, then $f^\lor \circ g(i)>i$, so in particular it is not an isomorphism. It follows that if we order $E^\lor_{1,/[l]}$ lexicographically (which is a linear order), then the matrix of the pairing from \ref{it:t2} has the required form. More generally, assume we have provided a linear order on $E^\lor_{n-1,/\theta''}$ with the required property. Any $\theta'\xrightarrow{e'}\theta''$ in $E_n^\lor$ such that $|\obj(\theta')|=\{0,...,m\}$ and $|\obj(\theta'')|=\{0,...,l\}$ is uniquely determined by the underlying morphism $e'_\obj:[m]\rightarrow[l]$ on objects together with morphisms $e'_i:\theta'_i\rightarrow \theta''_{e_\obj'(i+1)-1}$, we have a lexicographical order on $E^\lor_{1,/[l]}$ and linear orders on $E^\lor_{n-1,/\theta''_{e_\obj'(i+1)-1}}$ by induction, so we can order the set of tuples $(e'_\obj, e'_0,...,e'_{m-1})$ lexicographically, it follows by induction that with this ordering the matrix in \ref{it:t2} has the required property.
	\end{proof}
	\begin{notation}\label{not:dk}
		Given $\theta_f\in \Theta_{n,/\theta}$ denote 
		\[\widetilde{\theta}_f\bydef \cok(\underset{\theta_g\xrightarrow{e'}\theta_f}{\colim}\theta_g\rightarrow\theta_f),\]
		where the colimit is taken over all non-identity morphisms $\theta_g\xrightarrow{e'}\theta_f$ in $E^\lor_n(\theta)$. Also denote by $\widetilde{\Theta}_{n,/\theta}$ the pointed category with the same objects as $\Theta_{n,/\theta}$ such that $\mor_{\widetilde{\Theta}_{n,/\theta}}(\theta_f,\theta_g)\cong \{0\}\bigcup\{M_n(\theta)(\theta_f,\theta_g)\}$, where $M_n(\theta)(\theta_f,\theta_g)$ denotes the set of morphisms in $M_n(\theta)$ with source $\theta_f$ and target $\theta_g$ such that for a composable pair $m'\circ m''\cong m$ if their composition in $\Theta_{n,/\theta}$ equals $m$ and $m\in M_n(\theta)$ and $m'\circ m''\cong 0$ otherwise.
	\end{notation}
	\begin{cor}\label{cor:dk_theta}
		There is an equivalence
		\[\widetilde{\theta}_f\cong \ker(\theta_f\rightarrow\underset{\theta_f\xrightarrow{e}\theta_h}{\lim}\theta_h),\]
		where the limit is taken over non-identity morphisms $\theta_f\xrightarrow{e}\theta_h$ in $E_n(\theta)$, moreover
		\begin{equation}\label{eq:dk_sum}
			\theta_f\cong \bigoplus_{\theta_f\xrightarrow{e}\theta_h}\widetilde{\theta}_h,
		\end{equation}
		where the sum is taken over all morphisms in $E_n(\theta)$ (including the identity morphism). Additionally, there exists a morphism $F:\widetilde{\Theta}_{n,/\theta}\rightarrow \psh_\spc(\Theta_{n,/\theta})$ sending $\theta_f$ to $\widetilde{\theta}_f$ which induces an isomorphism
		\[\dk:\psh_\spc(\Theta_{n,/\theta})\leftrightarrows \psh_\spc(\widetilde{\Theta}_{n,/\theta}):\dk'\]
	\end{cor}
	\begin{proof}
		In light of \Cref{prop:dk_triple}, this follows from the main result of \cite{walde2022homotopy}.
	\end{proof}
	\newpage
	\bibliographystyle{plain}
	\bibliography{ref.bib}

\begin{thebibliography}{10}

\bibitem{aoki2023posets}
Ko~Aoki.
\newblock Posets for which verdier duality holds.
\newblock {\em Selecta Mathematica}, 29(5):78, 2023.

\bibitem{ara2023categorical}
Dimitri Ara, Andrea Gagna, Viktoriya Ozornova, and Martina Rovelli.
\newblock A categorical characterization of strong steiner $\omega$-categories.
\newblock {\em Journal of Pure and Applied Algebra}, 227(7):107313, 2023.

\bibitem{ayala2017fibrations}
David Ayala and John Francis.
\newblock Fibrations of $\infty $-categories.
\newblock {\em arXiv preprint arXiv:1702.02681}, 2017.

\bibitem{ayala2014configuration}
David Ayala and Richard Hepworth.
\newblock Configuration spaces and $\theta_n$.
\newblock {\em Proceedings of the American Mathematical Society},
  142(7):2243--2254, 2014.

\bibitem{ayala2019stratified}
David Ayala, Aaron Mazel-Gee, and Nick Rozenblyum.
\newblock Stratified noncommutative geometry.
\newblock {\em arXiv preprint arXiv:1910.14602}, 2019.

\bibitem{campion2023infty}
Timothy Campion.
\newblock An $(\infty, n) $-categorical pasting theorem.
\newblock {\em arXiv preprint arXiv:2311.00200}, 2023.

\bibitem{chu2019homotopy}
Hongyi Chu and Rune Haugseng.
\newblock Homotopy-coherent algebra via segal conditions.
\newblock {\em arXiv preprint arXiv:1907.03977}, 2019.

\bibitem{dold1958homology}
Albrecht Dold.
\newblock Homology of symmetric products and other functors of complexes.
\newblock {\em Annals of Mathematics}, 68(1):54--80, 1958.

\bibitem{forest2019unifying}
Simon Forest.
\newblock Unifying notions of pasting diagrams.
\newblock {\em arXiv preprint arXiv:1903.00282}, 2019.

\bibitem{gepner2016universality}
David Gepner, Moritz Groth, and Thomas Nikolaus.
\newblock Universality of multiplicative infinite loop space machines.
\newblock {\em Algebraic \& Geometric Topology}, 15(6):3107--3153, 2016.

\bibitem{harpaz2018abstract}
Yonatan Harpaz, Joost Nuiten, and Matan Prasma.
\newblock The abstract cotangent complex and quillen cohomology of enriched
  categories.
\newblock {\em Journal of Topology}, 11(3):752--798, 2018.

\bibitem{harpaz2020k}
YONATAN HARPAZ, JOOST NUITEN, and MATAN PRASMA.
\newblock On k-invariants for ($\infty$, n)-categories.
\newblock {\em arXiv preprint arXiv:2011.12723}, 2020.

\bibitem{joyal1997disks}
Andr{\'e} Joyal.
\newblock Disks, duality and $\theta$-categories.
\newblock {\em preprint}, 2:12--16, 1997.

\bibitem{joyal2007quasi}
Andr{\'e} Joyal and Myles Tierney.
\newblock Quasi-categories vs segal spaces.
\newblock {\em Contemporary Mathematics}, 431(277-326):10, 2007.

\bibitem{kan1958functors}
Daniel~M Kan.
\newblock Functors involving css complexes.
\newblock {\em Transactions of the American Mathematical Society},
  87(2):330--346, 1958.

\bibitem{kock1995monads}
Anders Kock.
\newblock Monads for which structures are adjoint to units.
\newblock {\em Journal of Pure and Applied Algebra}, 104(1):41--59, 1995.

\bibitem{lack2015combinatorial}
Stephen Lack and Ross Street.
\newblock Combinatorial categorical equivalences of dold--kan type.
\newblock {\em Journal of Pure and Applied Algebra}, 219(10):4343--4367, 2015.

\bibitem{luriehigher}
Jacob Lurie.
\newblock Higher algebra, september 2017.
\newblock {\em available at his webpage https://www. math. ias. edu/\~{}
  lurie}.

\bibitem{lurie2008classification}
Jacob Lurie.
\newblock On the classification of topological field theories.
\newblock {\em Current developments in mathematics}, 2008(1):129--280, 2008.

\bibitem{lurie2009higher}
Jacob Lurie.
\newblock {\em Higher topos theory}.
\newblock Princeton University Press, 2009.

\bibitem{lurie2018spectral}
Jacob Lurie.
\newblock Spectral algebraic geometry.
\newblock {\em preprint}, 2018.

\bibitem{manin1989arrangements}
Yu~I Manin and VV~Schechtman.
\newblock Arrangements of hyperplanes, higher braid groups and higher bruhat
  orders.
\newblock 1989.

\bibitem{nguyen2020adjoint}
Hoang~Kim Nguyen, George Raptis, and Christoph Schrade.
\newblock Adjoint functor theorems for $\infty$-categories.
\newblock {\em Journal of the London Mathematical Society}, 101(2):659--681,
  2020.

\bibitem{nuiten2019quillen}
JJ~Nuiten et~al.
\newblock Quillen cohomology of ($\infty$, 2)-categories.
\newblock {\em Higher Structures}, 3(1):17--66, 2019.

\bibitem{steiner2004omega}
Richard Steiner.
\newblock Omega-categories and chain complexes.
\newblock 2004.

\bibitem{steiner2006simple}
Richard Steiner.
\newblock Simple omega-categories and chain complexes.
\newblock {\em arXiv preprint math/0608680}, 2006.

\bibitem{street1987algebra}
Ross Street.
\newblock The algebra of oriented simplexes.
\newblock {\em Journal of Pure and Applied Algebra}, 49(3):283--335, 1987.

\bibitem{walde2022homotopy}
Tashi Walde.
\newblock Homotopy coherent theorems of dold--kan type.
\newblock {\em Advances in Mathematics}, 398:108175, 2022.

\end{thebibliography}
	
\end{document}